%% file: 3-manifold-groups-new-version-032613.tex
\documentclass[12pt]{amsart}

\usepackage{hyperref}
\usepackage{color}
\usepackage{pinlabel}
\usepackage{enumerate}
\usepackage{xypic}
\usepackage{graphicx}
\usepackage{booktabs}

\usepackage{amsmath,amsfonts,amssymb}
\usepackage{psfrag}

\theoremstyle{plain}
\newtheorem*{theorem*}{Theorem}
\newtheorem*{lemma*}{Lemma}
\newtheorem*{corollary*}{Corollary}
\newtheorem*{corollary-1}{Corollary 1}
\newtheorem*{corollary-2}{Corollary 2}

\newtheorem*{proposition*}{Proposition}
\newtheorem*{proposition-1}{Proposition 1}
\newtheorem*{proposition-2}{Proposition 2}
\newtheorem*{proposition-3}{Proposition 3}

\newtheorem{conjecture*}{Conjecture}
\newtheorem{theorem}{Theorem}[section]
\newtheorem{lemma}[theorem]{Lemma}
\newtheorem{observation}[theorem]{Observation}
\newtheorem{corollary}[theorem]{Corollary}
\newtheorem{proposition}[theorem]{Proposition}

\newtheorem{conjecture}[theorem]{Conjecture}
\newtheorem{question}[theorem]{Question}
\newtheorem{questions}[theorem]{Questions}

\theoremstyle{remark}

\newtheorem*{remark}{Remark}
\newtheorem*{convention}{Convention}
\newtheorem*{remarks}{Remarks}
\newtheorem*{definition}{Definition}

\newtheorem*{example}{Example}

\theoremstyle{definition}

\def\charc{\mbox{char}}

\def\out{\operatorname{Out}}

\def\PP{\mathcal{P}}

\def\MM{\mathcal{M}}
\def\pd{\operatorname{PD}}
\def\Isom{\operatorname{Isom}}
\def\Wh{\operatorname{Wh}}

\def\op{\operatorname}
\def\rp{\R\operatorname{P}}

\def\rp{\mathbb{R}P}

\def\gl{\operatorname{GL}} \def\sl{\operatorname{SL}} \def\Q{\mathbb{Q}} \def\F{\mathbb{F}} \def\Z{\mathbb{Z}} \def\R{\mathbb{R}} \def\C{\mathbb{C}} \def\H{\mathbb{H}}
\def\N{\mathbb{N}}  \def\l{\lambda} \def\ll{\langle} \def\rr{\rangle}
 \def\a{\alpha}   \def\bp{\begin{pmatrix}}
\def\smallsetminus{\setminus} \def\ep{\end{pmatrix}} \def\bn{\begin{enumerate}} 
   \def\en{\end{enumerate}}
\def\ba{\begin{array}} \def\ea{\end{array}}    \def\a{\alpha}  \def\ti{\tilde} \def\wti{\widetilde}
\def\id{\operatorname{id}} \def\Aut{\operatorname{Aut}} \def\im{\operatorname{Im}} 
  
\def\ker{\operatorname{Ker}}\def\be{\begin{equation}} \def\ee{\end{equation}} 
   
 \def\hom{\operatorname{Hom}}  
 \def\aut{\operatorname{Aut}}  \def\eps{\epsilon}
 \def\dim{\operatorname{dim}} 

   \def\K{\mathbb{K}}
\def\w{\omega}

\def\ol{\overline}
\def\G{\Gamma}

\def\ti{\tilde}
\def\G{\Gamma}

\def\roots{\operatorname{roots}}
\def\psl{\operatorname{PSL}}

\def\Sol{\mathrm{Sol}}
\def\Nil{\mathrm{Nil}}

\renewcommand\epsilon{\varepsilon}
\DeclareMathAlphabet{\mathbf}{OML}{cmm}{b}{it}

\numberwithin{equation}{section}

\makeindex

\begin{document}

\title{$3$-manifold groups}
\author{Matthias Aschenbrenner}
\address{University of California, Los Angeles, California, USA}
\email{matthias@math.ucla.edu}
\thanks{The first author was partially supported by NSF grant DMS-0556197.}

\author{Stefan Friedl}
\address{Mathematisches Institut\\ Universit\"at zu K\"oln\\   Germany}
\email{sfriedl@gmail.com}

\author{Henry Wilton}
\address{Department of Mathematics, University College London, UK}
\email{hwilton@math.ucl.ac.uk}
\thanks{The third author was partially supported by an EPSRC Career Acceleration Fellowship.}

\begin{abstract}
We  summarize properties of $3$-manifold groups, with a particular focus on the consequences of the recent results  of Ian Agol,
Jeremy Kahn, Vladimir Markovic and  Dani Wise.
\end{abstract}

\maketitle

\section*{Introduction}

\noindent
In this survey we give an overview of properties of fundamental groups of compact $3$-manifolds.
This class of groups sits between the class of fundamental groups of surfaces, which for the most part are well understood, and the class of fundamental groups of higher dimensional manifolds, which are very badly understood for the simple reason that given any finitely presented group $\pi$ and any $n\geq 4$, there exists a closed $n$-manifold with fundamental group $\pi$. (See \cite[Theorem~5.1.1]{CZi93} or \cite[Section~52]{SeT80} for a proof.)
This basic fact about high-dimensional manifolds is the root of many problems; for example, the  unsolvability of the isomorphism problem for finitely presented groups \cite{Ady55,Rab58} implies that closed manifolds of dimensions greater than three cannot be classified \cite{Mav58,Mav60}.

The study of $3$-manifold groups is also of great interest since for the most part, $3$-manifolds are determined by their fundamental groups.
More precisely, a closed, irreducible, non-spherical $3$-manifold is uniquely determined by its fundamental group (see Theorem~\ref{thm:p1determinen}).

Our account of $3$-manifold groups is based on the following building blocks:
\newcounter{itemcounter}
\begin{list}
{{(\arabic{itemcounter})}}
{\usecounter{itemcounter}\leftmargin=2em}
\item If $N$ is an irreducible $3$-manifold with infinite fundamental group, then the Sphere Theorem (see (C.\ref{C.sphere}) below), proved by Papakyriakopoulos \cite{Pap57a}, implies that $N$ is in fact an
 Eilenberg--Mac~Lane space.
 It follows, for example, that $\pi_1(N)$ is torsion-free.
\item The work of Waldhausen \cite{Wan68a,Wan68b} produced many results on the fundamental groups of Haken $3$-manifolds, e.g., the solution to the word problem.
\item The Jaco--Shalen--Johannson (JSJ) decomposition \cite{JS79,Jon79a} of an irreducible $3$-manifold with incompressible boundary gave insight into the subgroup structure of the fundamental groups of Haken $3$-manifolds and prefigured Thurston's Geometrization Conjecture.
\item The formulation of the Geometrization Conjecture and its proof for Haken $3$-manifolds by Thurston \cite{Thu82a} and in the general case by Perelman \cite{Per02,Per03a,Per03b}.  In particular, it became possible to prove that $3$-manifold groups share many properties with linear groups: they are residually finite~\cite{Hem87}, they satisfy the Tits Alternative (see (C.\ref{C.tits}) and (K.\ref{K.tits}) below), etc.
\item The solutions to Marden's Tameness Conjecture by Agol \cite{Ag07} and Cale\-ga\-ri--Gabai \cite{CaG06}, combined with Canary's Covering Theorem \cite{Cay96} implies the Subgroup Tameness Theorem (see Theorem~\ref{thm:subgroupdichotomy} below), which describes the finitely generated, geometrically infinite subgroups of fundamental groups of finite-volume hyperbolic $3$-manifolds. As a result, in order to understand the finitely generated subgroups of such hyperbolic $3$-manifold groups, one can  mainly restrict attention to the geometrically finite case.
\item The results announced by  Wise \cite{Wis09}, with proofs provided in the preprint \cite{Wis12a} (see also \cite{Wis12b}), revolutionized the field.  First and foremost,
 together with Agol's Virtual Fibering Theorem \cite{Ag08} they imply the Virtually Fibered Conjecture for Haken hyperbolic $3$-manifolds.  Wise in fact proves something stronger, namely that if $N$ is a hyperbolic $3$-manifold with an embedded geometrically finite surface, then $\pi_1(N)$ is \emph{virtually compact special}---see Section~\ref{section:specialcc} for the definition. As well as virtual fibering, this also implies that $\pi_1(N)$ is LERF and large, and has some unexpected corollaries: for instance, $\pi_1(N)$ is linear over $\Z$.
\item Agol \cite{Ag12}, building on the proof of the Surface Subgroup Conjecture by Kahn--Markovic \cite{KM12} and the aforementioned work of Wise, recently gave a proof of the Virtually Haken Conjecture. Indeed, he proves that the fundamental group of any closed hyperbolic $3$-manifold is virtually compact special.
\item Przytycki--Wise \cite{PW12a} showed that fundamental groups of compact irreducible $3$-manifolds with empty or toroidal boundary which are neither graph manifolds nor Seifert fibered are virtually special.
In particular such manifolds are virtually fibered and their fundamental groups are linear over $\Z$. The combination of the results of Agol and Przytycki--Wise and a theorem of Liu \cite{Liu11} implies that
the fundamental group of a compact, orientable, aspherical $3$-manifold $N$ with empty or toroidal boundary is virtually special if and only if $N$ is non-positively curved.
\end{list}

\noindent
Despite the great interest in $3$-manifold groups, survey papers seem to be few and far between.
We refer to \cite{Neh65}, \cite{Sta71}, \cite{Neh74}, \cite{Hem76}, \cite{Thu82a}, \cite[Section~5]{CZi93}, \cite{Ki97} for some results on $3$-manifold groups
and lists of open questions.

The goal of this survey is to fill what we perceive as a gap in the literature, and to give an extensive overview of results on fundamental groups of compact $3$-manifolds with a particular emphasis on the
impact of the Geometrization Theorem of Perelman, the Tameness Theorem of Agol, Calegari-Gabai, and the Virtually Compact Special Theorem of  Agol \cite{Ag12}, Kahn--Markovic \cite{KM12} and Wise \cite{Wis12a}. Our approach is to summarize many of the results in several diagrams and to provide detailed references for each implication appearing in these diagrams.
We will mostly consider results of a `combinatorial group theory' nature that hold for fundamental groups of $3$-manifolds which are either closed or have toroidal boundary. We do not make any claims to originality---all results are either already in the literature, or simple consequences of established facts, or well known to the experts.


\medskip

As with any survey, this one  reflects the tastes and biases of the authors.  The following lists some of the topics which we leave basically untouched:
\newcounter{itemcountera}
\begin{list}
{{(\arabic{itemcountera})}}
{\usecounter{itemcountera}\leftmargin=2em}
\item Fundamental groups of non-compact $3$-manifolds. Note though that Scott \cite{Sco73b} showed that given a $3$-manifold $M$ with finitely generated fundamental group, there exists a compact $3$-manifold with the same fundamental group as $M$.
\item  `Geometric' and `large scale' properties of $3$-manifold groups; see, e.g., \cite{Ge94a,KaL97,KaL98,BN08,BN10,Sis11} for some results in this direction.  We also leave aside automaticity, formal languages, Dehn functions and combings: see, for instance, \cite{Brd93,BrGi96,Sho92,CEHLPT92}.
\item  Three-dimensional Poincar\'e duality groups; see, e.g.,~\cite{Tho95,Davb00,Hil11} for further information.
\item Specific properties of fundamental groups of knot complements (known as `knot groups').
We note that in general, irreducible $3$-manifolds with non-trivial boundary are not determined by their fundamental groups, but interestingly, prime knots in $S^3$ are in fact determined by their groups  \cite{CGLS85,CGLS87,GLu89,Whn87}. Knot groups were some of the earliest and most popular examples of $3$-manifold groups to be studied.
\item Fundamental groups of distinguished classes of $3$-manifolds.  For instance, arithmetic hyperbolic $3$-manifold groups exhibit many special features. (See, for example, \cite{MaR03, Lac11, Red07} for more on arithmetic $3$-mani\-folds).
\item The representation theory of $3$-manifolds is a substantial field in its own right, which fortunately is served well by Shalen's survey paper \cite{Shn02}.
\end{list}
We conclude the paper with a discussion of some outstanding open problems in the theory of $3$-manifold groups.

\medskip

This survey is not intended as a leisurely introduction to $3$-manifolds. Even though most terms will be defined, we will assume that the reader is already somewhat acquainted with $3$-manifold topology. We refer to \cite{Hem76, Hat, JS79, Ja80} for background material.
Another gap we perceive is the lack of a post-Geometrization-Theorem $3$-manifold book. We hope that  somebody else will step forward and fill this gaping hole.

\subsection*{Conventions and notations}
All spaces are assumed to be connected and compact and all groups are assumed to be finitely presented, unless it is specifically stated otherwise. All rings have an identity.
We denote the cyclic group with $n$ elements by $\Z/n$.
If $N$ is a $3$-manifold and $S\subseteq N$ a submanifold, then we denote by $\nu S\subseteq N$ a tubular neighborhood of $S$. When we write `a manifold with boundary' then we also include the case that the boundary is empty. If we want to ensure that the boundary is in fact non-empty, then we will write `a manifold with non-empty boundary'.

\subsection*{Acknowledgments}
 The authors would like to thank Ian Agol, Igor Belegradek, Mladen Bestvina, Michel Boileau, Steve Boyer, Martin Bridson, Jack Button, Danny Calegari, Jim Davis,
Daryl Cooper, Dave Futer, Cameron Gordon, Pierre de la Harpe, Matt Hedden, John Hempel, Jonathan Hillman, Jim Howie,
Thomas Koberda, Tao Li, Viktor Kulikov, Marc Lackenby, Mayer A. Landau, Wolfgang L\"uck,
Curtis McMullen, Piotr Przytycki, Alan Reid, Saul Schleimer, Dan Silver, Stefano Vidussi, Liam Watson and Susan Williams for helpful comments, discussions and suggestions.  We are also grateful for the extensive feedback we got from many other people on earlier versions of this survey.
Finally we also would like to thank Anton Geraschenko for bringing the authors together.

\newpage

\tableofcontents

\newpage

\section{Decomposition Theorems}

\subsection{The Prime Decomposition Theorem}\label{section:prelim1}

A $3$-manifold $N$ is called \emph{prime} \index{$3$-manifold!prime} if $N$ cannot be written as a non-trivial connected sum of two manifolds,
i.e., if $N=N_1\# N_2$, then $N_1=S^3$ or $N_2=S^3$.
Furthermore $N$ is called \emph{irreducible}  \index{$3$-manifold!irreducible} if every embedded $S^2$ bounds a $3$-ball. Note that an irreducible $3$-manifold is prime.
Also, if $N$ is an orientable prime $3$-manifold with no spherical boundary components, then by \cite[Lemma~3.13]{Hem76} either $N$ is irreducible or $N=S^1\times S^2$. The following theorem is due to Kneser \cite{Kn29}, Haken \cite[p.~441f]{Hak61} and Milnor \cite[Theorem~1]{Mil62} (see also \cite[Chapter~III]{Sco74}, \cite[Chapter~3]{Hem76} and \cite{HM08}).
We also refer to \cite{Grs69,Grs70,Swp70,Prz79} for more decomposition theorems in the bounded cases.
 \index{decomposition!prime} \index{theorems!Prime Decomposition Theorem}

\begin{theorem} \label{thm:prime}\textbf{\emph{(Prime Decomposition Theorem)}}
Let $N$ be a compact, oriented $3$-manifold with no spherical boundary components.
\newcounter{itemcounterd}
\begin{itemize}
\item[(1)] There exists a decomposition $N\cong N_1\# \cdots \# N_r$ where the $3$-manifolds $N_1,\dots,N_r$ are oriented prime $3$-manifolds.
\item[(2)] If $N\cong N_1\# \cdots \# N_r$ and $N\cong N_1'\# \cdots \# N_s'$ where the $3$-manifolds $N_i$ and $N_i'$  are oriented prime $3$-manifolds, then $r=s$ and \textup{(}possibly after reordering\textup{)} there exist orientation-preserving diffeo\-mor\-phisms $N_i\to N_i'$.
\end{itemize}
In particular, $\pi_1(N)=\pi_1(N_1)*\dots * \pi_1(N_r)$ is the free product of fundamental groups of prime $3$-manifolds.
\end{theorem}

Note that the uniqueness concerns the homeomorphism types of the prime components. The decomposing spheres are not unique up to isotopy, but two different sets of decomposing spheres
are related by `slide homeomorphisms'. We refer to \cite[Theorem~3]{CdS79}, \cite{HL84} and \cite[Section~3]{McC86} for details.

\subsection{The Loop Theorem and the Sphere Theorem}
The life of $3$-manifold topology as a flourishing subject started with the proof of the Loop Theorem and the Sphere Theorem by Papakyriakopoulos. We first state the Loop Theorem.

\begin{theorem} \textbf{\emph{(Loop Theorem)}}\label{thm:loop}
Let $N$ be a compact $3$-manifold and $F\subseteq \partial N$ a subsurface.
If $\ker(\pi_1(F)\to \pi_1(N))$ is non-trivial, then there exists a proper embedding $g\colon (D^2,\partial D^2)\to (N,F)$
such that $g(\partial D^2)$ represents a non-trivial element in $\ker(\pi_1(F)\to \pi_1(N))$.
\end{theorem}

\index{theorems!Loop Theorem}
\index{theorems!Dehn's Lemma}

A somewhat weaker version (usually called `Dehn's Lemma') of this theorem was first stated by Dehn \cite{De10,De87} in 1910, but
Kneser \cite[p.~260]{Kn29} found a gap in the proof provided by Dehn. The Loop Theorem was finally proved by
 Papakyriakopoulos  \cite{Pap57a,Pap57b} building on work of Johansson \cite{Jos35}. We refer to  \cite{Hom57,SpW58,Sta60,Wan67b,Gon99,Bin83,Jon94,AiR04} and \cite[Chapter~4]{Hem76} for more details and several extensions.
We now turn to the Sphere Theorem.

\begin{theorem} \textbf{\emph{(Sphere Theorem)}} \label{thm:sphere}
Let $N$ be an orientable $3$-manifold with $\pi_2(N)\ne 0$.  Then $N$ contains an embedded $2$-sphere which is homotopically non-trivial.
\end{theorem}

\index{theorems!Sphere Theorem}

This theorem was proved by Papakyriakopoulos \cite{Pap57a} under a technical assumption which was removed by Whitehead \cite{Whd58a}.
(We also refer to \cite{Whd58b,Bat71,Gon99,Bin83} and \cite[Theorem~4.3]{Hem76} for extensions and more information.)
Gabai (see \cite[p.~487]{Gab83a} and \cite[p.~79]{Gab83b}) proved that for $3$-manifolds the Thurston norm equals the Gromov norm. (See Section~\ref{section:thurstonnorm} below for more on the Thurston norm.)
 This result can be viewed as a higher-genus analogue of the Loop Theorem and the Sphere Theorem.

\subsection{Preliminary observations about $3$-manifold groups}\label{section:prelim2}

The main subject of this survey  are the properties of fundamental groups of compact $3$-manifolds.
In this section we argue that for most purposes it suffices to study the fundamental groups of compact, orientable, irreducible $3$-manifolds whose boundary is either empty or toroidal.

\medskip

We start out with  the following  basic observation.

\begin{observation}\label{obs:s2}
Let $N$ be a compact $3$-manifold.
\newcounter{itemcounterc}
\begin{itemize}
\item[(1)] Denote by $\widehat{N}$ the $3$-manifold obtained from $N$ by gluing $3$-balls to all spherical components of $\partial N$.
Then $\pi_1(\widehat{N})=\pi_1(N)$.
\item[(2)] If $N$ is non-orientable, then there exists a  double cover which is orientable.
\end{itemize}
\end{observation}

Most properties of groups of interest to us are preserved under going to free products of groups
(see, e.g., \cite{Nis40} and \cite[Proposition~1.3]{Shn79} for linearity and \cite{Rom69,Bus71} for being LERF) and similarly most properties of groups are preserved under passing to an index-two supergroup
(see, e.g., (H.\ref{H.retracttozgreen}) to  (H.\ref{H.csgreen}) below).
Note though that this is not true for all properties; for example, conjugacy separability does not in general pass to degree-two extensions \cite{CMi77,Goa86}.

In light of Theorem~\ref{thm:prime} and Observation~\ref{obs:s2},
we therefore generally restrict ourselves to the study of orientable, irreducible $3$-manifolds with no spherical boundary components.

\medskip

An embedded surface $\Sigma\subseteq N$ with components $\Sigma_1,\dots,\Sigma_k$ is \emph{incompressible}
if for each $i=1,\dots,k$ we have $\Sigma_i\ne S^2,D^2$ and the map $\pi_1(\Sigma_i)\to \pi_1(N)$ is injective.  The following lemma is a well known consequence of the Loop Theorem.

\index{surface!incompressible}

\begin{lemma} \label{lem:incomp}
Let $N$ be a compact $3$-manifold. Then there exist $3$-manifolds $N_1,\dots,N_k$ whose boundary components are incompressible, and a free group $F$
such that $\pi_1(N)\cong \pi_1(N_1)*\dots *\pi_1(N_k)*F$.
\end{lemma}

\begin{proof}
By the above observation we can without loss of generality assume that $N$ has no spherical boundary components.
Let $\Sigma\subseteq \partial N$ be a  component such that $\pi_1(\Sigma)\to \pi_1(N)$ is not injective. By the Loop Theorem
(see Theorem~\ref{thm:loop})
there exists a properly  embedded disk $D\subseteq N$ such that the curve $c=\partial D \subseteq \Sigma$ is  essential.
 Here a curve $c$ is called \emph{essential} if $c$ does not bound an embedded disk in $\Sigma$.
 \index{curve!essential}

Let $N'$ be the result of capping off the spherical boundary components of $N\smallsetminus \nu D$  by $3$-balls.  If $N'$ is connected, then $\pi_1(N)\cong \pi_1(N')\ast \Z$; otherwise $\pi_1(N)\cong \pi_1(N_1)*\pi_1(N_2)$ where $N_1$, $N_2$ are the two components of $N'$. The lemma now follows by induction on the lexicographically ordered pair $(-\chi(\partial N),b_0(\partial N))$ since we have either that $-\chi(\partial N')<-\chi(\partial N)$ (in the case that $\Sigma$ is not a torus), or that $\chi(\partial N')=\chi(\partial N)$ and $b_0(\partial N')<b_0(\partial N)$ (in the case that $\Sigma$ is a torus).
\end{proof}

We say that a group $A$ is a \emph{retract} of a group $B$ if there exist group homomorphisms $\varphi\colon A\to B$ and $\psi\colon B\to A$ such that $\psi\circ\varphi=\id_A$.  In particular, in this case $\varphi$ is injective and we can then view $A$ as a subgroup of $B$.

\index{subgroup!retract}

\begin{lemma} \label{lem:closed}
Let $N$ be a compact $3$-manifold with non-empty boundary. Then $\pi_1(N)$ is a retract of the fundamental group of a  closed $3$-manifold.\end{lemma}

\begin{proof}
 Denote by $M$ the double of $N$, i.e., $M=N\cup_{\partial N}N$. Note that $M$ is a closed $3$-manifold.  Let $f$ be the canonical inclusion of $N$ into $M$ and let $g \colon M\to N$ be the map which restricts to the identity on the two copies of $N$ in $M$.
 Clearly $g\circ f=\id_N$ and hence $g_*\circ f_*=\id_{\pi_1(N)}$.
 \end{proof}

Many properties of groups are preserved under retracts and taking free products; this way, many problems on $3$-manifold groups can be reduced to the study of fundamental groups of closed $3$-manifolds. Due to the important role played by $3$-manifolds with toroidal boundary components we will be slightly less restrictive, and in the remainder we study fundamental groups of compact, orientable, irreducible $3$-manifolds $N$ such that the boundary is either empty or toroidal.

\subsection{The JSJ Decomposition Theorem} \index{decomposition!JSJ}\label{decomposition!JSJ}

In the previous section we saw that an oriented, compact $3$-manifold with no spherical boundary components admits a decomposition along spheres such that the set of resulting pieces are unique up to diffeomorphism.  In the following we say  that a $3$-manifold $N$ is  \emph{atoroidal} \index{$3$-manifold!atoroidal} if any map $T\to N$ from a torus to~$N$ which induces a monomorphism $\pi_1(T)\to \pi_1(N)$ can be homotoped into the boundary of $N$.
(Note that in the literature some authors refer to a $3$-manifold as atoroidal if the above condition holds for any embedded torus. These two notions differ only for certain Seifert fibered $3$-manifolds where the base orbifold is a genus $0$ surface such that the number of boundary components together with the number of cone points equals three.)
 There exist orientable irreducible $3$-manifolds which cannot be cut  into atoroidal pieces in a unique way (e.g., the $3$-torus). Nonetheless, any orientable irreducible $3$-manifold  admits a canonical decomposition along tori, but to formulate this result we need the notion of a Seifert fibered manifold.

A \emph{Seifert fibered manifold} \index{Seifert fibered manifold} \index{$3$-manifold!Seifert fibered}
is a $3$-manifold $N$ together with a decomposition into disjoint simple closed curves (called \emph{Seifert fibers}) such that each Seifert fiber has a tubular neighborhood that forms a standard fibered torus.
The \emph{standard fibered torus} corresponding to a pair of coprime integers $(a,b)$ with $a>0$ is the surface bundle of the automorphism of a disk given by rotation by an angle of $2\pi b/a$, equipped with the natural fibering by circles.
If $a>1$, then the middle Seifert fiber is called \emph{singular}. A compact Seifert fibered manifold has only a finite number of singular fibers.
It is often useful  to think of a Seifert fibered manifold as a circle bundle over a $2$-dimensional orbifold.
We refer to \cite{Sei33a,Or72,Hem76,Ja80,JD83,Sco83a,Brn93,LRa10} for   further information and for the classification of Seifert fibered manifolds. \index{Seifert fibered manifold!standard fibered torus} \index{Seifert fibered manifold!Seifert fiber} \index{Seifert fibered manifold!Seifert fiber!singular}

Some $3$-manifolds (e.g., lens spaces) admit distinct Seifert fibered structures; generally, however, this will not be of importance to us (but see, e.g., \cite[Theorem~VI.17]{Ja80}).  Sometimes, later in the text, we will slightly abuse language and say that a $3$-manifold is Seifert fibered if it admits the structure of a Seifert fibered manifold.

\begin{remark}
\mbox{}

\bn
\item
The only orientable non-prime  Seifert fibered manifold is $\rp^3\#\rp^3$
(see, e.g., \cite[Proposition~1.12]{Hat} or \cite[Lemma~VI.7]{Ja80}).
\item
By Epstein's Theorem \cite[p.~81]{Ep72}, a $3$-manifold $N$ which is not homeomorphic to the solid Klein bottle
 admits a Seifert fibered structure if and only if it admits a foliation by circles. \index{theorems!Epstein's Theorem}
 \en
\end{remark}

The following theorem was first announced by Waldhausen \cite{Wan69} and was proved independently by Jaco--Shalen \cite[p.~157]{JS79} and Johannson \cite{Jon79a}. \index{theorems!JSJ Decomposition Theorem} \index{JSJ!decomposition}
In the case of knot complements the JSJ decomposition theorem was foreshadowed by the work of Schubert \cite{Sct49,Sct53,Sct54}.

\begin{theorem}\label{thm:jsj}  \textbf{\emph{(JSJ Decomposition  Theorem)}}
Let $N$ be a compact, orientable, irreducible $3$-manifold with empty or toroidal boundary. Then there exists a collection of disjointly embedded  incompressible tori $T_1,\dots,T_k$ such that each component of $N$ cut along $T_1\cup \dots \cup T_k$ is  atoroidal  or Seifert fibered. Furthermore any such collection of tori with a minimal number of components is unique up to isotopy.
\end{theorem}

In the following we refer to the tori $T_1,\dots,T_k$ as the \emph{JSJ tori} \index{JSJ!tori} and we will refer to the components of $N$ cut along $\bigcup_{i=1}^k T_i$ as the \emph{JSJ components of~$N$}. \index{JSJ!components}
Let~$M$ be a JSJ component of $N$. After picking base points for $N$ and $M$ and a path connecting these base points, the inclusion $M\subseteq N$ induces a map  on the level of fundamental groups.  This map is injective since the tori we cut along are incompressible.
(We refer to \cite[Chapter~IV.4]{LyS77} for details.)
  We can thus view $\pi_1(M)$ as a subgroup of $\pi_1(N)$, which is well defined up to the above choices, i.e., well defined up to conjugacy.
Furthermore we can view $\pi_1(N)$ as the fundamental group of a graph of groups with vertex groups the fundamental groups of the JSJ components and with edge groups the fundamental groups of the JSJ tori. We refer to \cite{Ser80,Bas93} for more on graphs of groups.

\medskip

We need the following definition due to Jaco--Shalen:

\begin{definition}
Let $N$ be a compact, orientable, irreducible $3$-manifold with empty or toroidal boundary.
The \emph{characteristic submanifold of $N$} \index{submanifold!characteristic} is the union of the following submanifolds:
\bn
\item all Seifert fibered pieces in the JSJ decomposition;
\item all boundary tori which cobound an atoroidal JSJ component;
\item all JSJ tori which do not cobound a Seifert fibered JSJ component.
\en
\end{definition}

The following theorem is a consequence of the `Characteristic Pair Theorem' of Jaco--Shalen~\cite[p.~138]{JS79}.
\index{theorems!Characteristic Pair Theorem}

\begin{theorem} \label{thm:charsub}
Let $N$ be a compact, orientable, irreducible $3$-manifold with empty or toroidal boundary which admits at least one JSJ torus.
If  $f\colon M\to N$ is a map from a Seifert fibered manifold $M$ to $N$ which is $\pi_1$-injective and
if $M\ne S^1\times D^2$ and $M\ne S^1\times S^2$,  then $f$ is homotopic to a  map $g \colon M\to N$
such that $g(M)$ lies in a component of the characteristic submanifold of $N$.
\end{theorem}

We refer to \cite[Lemma~IX.10]{Ja80} for a somewhat stronger statement.
The next proposition is  an immediate consequence of Theorem~\ref{thm:charsub}, and gives a useful criterion for showing that a collection of tori are  the JSJ tori of a given $3$-manifold.

\begin{proposition}\label{prop:charpair}
Let $N$ be a compact, orientable, irreducible $3$-manifold with empty or toroidal boundary.
Let $T_1,\dots,T_k$ be disjointly embedded tori in $N$. Suppose the following hold:
\begin{itemize}
\item[(1)] the components $M_1,\dots,M_l$ of $N$ cut along $T_1\cup \dots \cup T_k$ are either Seifert fibered or atoroidal; and
\item[(2)] if a torus $T_i$ cobounds two Seifert fibered components $M_r$ and $M_s$ \textup{(}where it is possible that $r=s$\textup{)}, then the regular fibers of $M_r$ and $M_s$ do not define the same element in $H_1(T_i)$.
\end{itemize}
Then $T_1,\dots,T_k$ are the JSJ tori of $N$.
\end{proposition}

\subsection{The Geometrization Theorem}
We now turn to the study of atoroidal $3$-manifolds.  We say that a closed $3$-manifold is \emph{spherical} \index{$3$-manifold!spherical} if it admits a complete metric of constant positive curvature.  Note that fundamental groups of spherical $3$-manifolds are finite; in particular spherical $3$-manifolds are atoroidal.

\medskip

In the following  we say that a compact $3$-manifold is \emph{hyperbolic} \index{$3$-manifold!hyperbolic}  if its interior admits a complete metric of constant negative curvature $-1$.
The following theorem is due to Mostow \cite[Theorem~12.1]{Mos68} in the closed case and due to  Prasad \cite[Theorem~B]{Pra73} and Marden \cite{Man74} independently in the case of non-empty boundary.
(See also \cite[Section~6]{Thu79}, \cite{Mu80}, \cite[Chapter~C]{BP92}, \cite[Chapter~11]{Rat06} and \cite[Corollary~1]{BBI12} for alternative proofs.)

\index{theorems!Mostow--Prasad--Marden Rigidity Theorem}

\begin{theorem}\emph{\textbf{(Mostow--Prasad--Marden Rigidity Theorem)}}\label{thm:mostow-prasad}
Let $M$ and $N$ be finite volume hyperbolic $3$-manifolds.  Any isomorphism $\pi_1(M)\to \pi_1(N)$
is induced by a unique isometry $M\to N$.
\end{theorem}

\begin{remarks} \mbox{}

\bn
\item
 This theorem implies in particular that the geometry of finite volume hyperbolic $3$-manifolds is determined by their topology.
 This is not the case if we drop the finite-volume condition.
 More precisely, the Ending Lamination Theorem states that hyperbolic $3$-manifolds with finitely generated fundamental groups
 are determined by their topology and by their `ending laminations'. The Ending Lamination Theorem
 was conjectured by Thurston~\cite{Thu82a} and was proved by Brock--Canary--Minsky~\cite{BCM04, Miy10}.
 We also refer to \cite{Miy94,Miy03,Miy06,Ji12} for more background information and to \cite{Bow11a,Bow11b}, \cite{Ree08} and \cite{Som10} for alternative approaches. \index{theorems!Ending Lamination Theorem}
\item If we apply the Rigidity Theorem to a $3$-manifold equipped with two different finite volume hyperbolic structures, then the theorem says that the two hyperbolic structures are the same up to an isometry
which is \emph{homotopic} to the identity. This does not imply that the set of hyperbolic metrics on a finite volume $3$-manifold is path connected. The path connectedness was later shown
by Gabai--Meyerhoff--N. Thurston \cite[Theorem~0.1]{GMT03} building on earlier work of Gabai \cite{Gab94a,Gab97}.
\item Gabai \cite[Theorem~1.1]{Gab01} showed that if $N$ is a closed hyperbolic $3$-manifold, then the inclusion of the isometry group $\op{Isom}(N)$ into the diffeomorphism group $\op{Diff}(N)$ is a homotopy equivalence.
For Haken manifolds, and in particular for non-compact finite volume hyperbolic $3$-manifolds, the statement was proved
by Hatcher \cite{Hat76,Hat83} and Ivanov \cite{Iva76}.
\en
\end{remarks}

\medskip

A hyperbolic $3$-manifold has finite volume if and only if it is either closed or has toroidal boundary (see \cite[Theorem~5.11.1]{Thu79} or \cite[Theorem~2.9]{Bon02}). Since in this survey we are mainly interested in $3$-manifolds with empty or toroidal boundary,
we henceforth restrict ourselves to hyperbolic $3$-manifolds with finite volume.
We will therefore work with the following understanding.

\begin{convention}
Unless we say explicitly otherwise, in the remainder of the survey, a hyperbolic $3$-manifold is always  understood to have finite volume.
\end{convention}

With this convention, hyperbolic $3$-manifolds are atoroidal; in fact, the following slightly stronger statement holds (see \cite[Proposition~6.4]{Man74}, \cite[Proposition~5.4.4]{Thu79} and also \cite[Corollary~4.6]{Sco83a}):

\begin{theorem}\label{thm:hypatoroidal}
Let $N$ be a hyperbolic $3$-manifold. If $\G\leq \pi_1(N)$ is abelian and not cyclic,
then  there exists a boundary torus $S$ and $h\in \pi_1(N)$ such that
\[ \G\subseteq h\,\pi_1(S)\,h^{-1}.\]
\end{theorem}

The Elliptization  Theorem and the Hyperbolization Theorem (Theorems~\ref{thm:ellips} and \ref{thm:hyp} below) together imply  that every atoroidal $3$-manifold is either spherical or hyperbolic.  Both theorems  were conjectured by Thurston \cite{Thu82a,Thu82b} and the latter was foreshadowed by the work of Riley \cite{Ril75a,Ril75b,Ril13}.
The Hyperbolization Theorem was proved by Thurston for Haken manifolds (see \cite{Thu86c,Mor84,Su81,McM96,Ot96,Ot01} for the fibered case and
\cite{Thu86b,Thu86d,Mor84,McM92,Ot98,Kap01} for the non-fibered case).  The full proof of both theorems was first given by Perelman in his seminal papers \cite{Per02,Per03a,Per03b} building on earlier work of R. Hamilton \cite{Hamc82,Hamc95,Hamc99}.  We refer to \cite{MTi07} for full details and  to  \cite{CZ06a,CZ06b,KlL08,BBBMP10} for further information on the proof. Finally we refer to \cite{Mil03,Anb04,Ben06,Bei07,McM11} for expository accounts.

\begin{theorem}\label{thm:ellips}  \textbf{\emph{(Elliptization Theorem)}}
Every closed,  orientable $3$-manifold with finite fundamental group is spherical.
\end{theorem}

\index{theorems!Elliptization Theorem}

It is well known that $S^3$ equipped with the canonical metric is the only spherical simply connected $3$-manifold.
It follows  that the Elliptization Theorem  implies the Poincar\'e Conjecture: the $3$-sphere $S^3$ is the only  simply connected, closed $3$-manifold. We thus see  that a $3$-manifold $N$ is spherical if and only if  it is the quotient of $S^3$ by a finite group, which acts freely and isometrically.  In particular, we can view $\pi_1(N)$ as a finite subgroup of $\operatorname{SO}(4)$
which acts freely on $S^3$. By Hopf \cite[\S~2]{Hop26} (see also \cite{SeT30,SeT33},\cite[Theorem~2]{Mil57} and \cite[Chapter~6,~Theorem~1]{Or72})
such a group is isomorphic to precisely one of the following types of groups:
\bn
\item the trivial group,
\item $Q_{4n}:= \ll x,y\,|\,x^2=(xy)^2=y^n\rr$, $n\geq 2$, which is an extension of the dihedral group $D_{2n}$ by $\Z/2$,
\item $P_{48}:=\ll x,y\,|\, x^2=(xy)^3=y^4,x^4=1\rr$, which is an extension of the octahedral group by $\Z/2$,
\item $P_{120}:=\ll x,y\,|\, x^2=(xy)^3=y^5,x^4=1\rr$, which is an extension of the icosahedral group by $\Z/2$,
\item the dihedral group $D_{2^k(2n+1)}:=\ll x,y\,|\,x^{2^k}=1,y^{2n+1}=1,xyx^{-1}=y^{-1}\rr$, where $k\geq 2$ and $n\geq 1$,
\item $P'_{8\cdot 3^k}:=\ll x,y,z\,|\,x^2=(xy)^2=y^2,zxz^{-1}=y,zyz^{-1}=xy,z^{3^k}=1\rr$, where $k\geq 1$,
\item the direct product of any of the above groups with a cyclic group of relatively prime order.
\en
Note that spherical $3$-manifolds are in fact Seifert fibered (see \cite[\S 7, Hauptsatz]{SeT33}, \cite[Chapter~6,~Theorem~2]{Or72}, \cite[\S~4]{Sco83a} or \cite[Theorem~2.8]{Bon02}).
By \cite[Theorem~3.1]{EvM72} the fundamental group of a spherical $3$-manifold $N$
is solvable unless $\pi_1(N)$  is isomorphic to the binary dodecahedra1 group $P_{120}$
or  the direct sum of $P_{120}$ with a cyclic group of order relatively prime to 120.
In particular the Poincar\'e homology sphere,
\index{$3$-manifold!Poincar\'e homology sphere}
 the $3$-manifold with fundamental group  $P_{120}$, is the only homology sphere with finite fundamental group. Finally we  refer to  \cite{Mil57,Lee73,Tho79,Dava83,Tho86,Rub01}
for some `pre-Geometrization' results on the classification of finite fundamental groups of $3$-manifolds.
\index{$3$-manifold group!finite}

\medskip

We now turn to atoroidal $3$-manifolds with infinite fundamental groups.

\begin{theorem}\label{thm:hyp}  \textbf{\emph{(Hyperbolization Theorem)}}
Let $N$ be a compact, orientable, irreducible $3$-manifold with empty or toroidal boundary. If $N$ is atoroidal and $\pi_1(N)$ is infinite, then $N$ is hyperbolic.
\end{theorem}

\index{theorems!Hyperbolization Theorem}

Combining the JSJ Decomposition Theorem with the Elliptization  Theorem and the Hyperbolization Theorem we now obtain the following:

\begin{theorem}\label{thm:geom}  \textbf{\emph{(Geometrization Theorem)}}
Let $N$ be a compact, orientable, irreducible $3$-manifold with empty or toroidal boundary.  Then there exists a collection of disjointly embedded  incompressible tori $T_1,\dots,T_k$ such that each component of $N$ cut along $T_1\cup \dots \cup T_k$ is  hyperbolic  or Seifert fibered. Furthermore any such collection of tori with a minimal number of components is unique up to isotopy.
\end{theorem}

\begin{remark}
The Geometrization Conjecture has also been formulated for non-orien\-table $3$-manifolds; we refer to \cite[Conjecture~4.1]{Bon02} for details.
To the best of our knowledge this has not been fully proved yet.
Note though that by (D.\ref{D.nonorientable}), a non-orientable $3$-manifold has infinite fundamental group, i.e., it can not be spherical.
It follows from \cite[Theorem~H]{DL09} that a closed atoroidal $3$-manifold with infinite fundamental group is hyperbolic.
\end{remark}

\index{theorems!Geometrization Theorem}

We finish this subsection with two further theorems related to the JSJ decomposition.
The first theorem says that the JSJ decomposition behaves well under passing to finite covers:

\begin{theorem}\label{thm:jsjlift}
Let $N$ be a compact, orientable, irreducible $3$-manifold with empty or toroidal boundary. Let $N'\to N$ be a finite cover.
Then $N'$ is irreducible and the pre-images of the JSJ tori of $N$ under the projection map are the JSJ tori of $N'$.
Furthermore $N'$ is hyperbolic \textup{(}respectively Seifert fibered\textup{)} if and only if  $N$ is  hyperbolic \textup{(}respectively Seifert fibered\textup{)}.
\end{theorem}

The fact that $N'$ is again irreducible follows from the Equivariant Sphere Theorem (see \cite[p.~647]{MSY82} and see also \cite{Duw85,Ed86,JR89}).
(The assumption that $N$ is orientable is necessary, see, e.g., \cite[Theorem~5]{Row72}.) The other statements are straightforward consequences of Proposition~\ref{prop:charpair} and the Hyperbolization Theorem.  Alternatively we  refer to  \cite[p.~290]{MeS86} and \cite{JR89} for details.

\medskip

Finally, the following theorem, which is an immediate consequence of Proposition~\ref{prop:charpair},  often allows us to reduce proofs to the closed case:

\begin{theorem}\label{thm:jsjdouble}
Let $N$ be a compact, orientable, irreducible $3$-manifold with non-trivial toroidal boundary. We denote the boundary tori by $S_1,\dots,S_k$ and we denote the JSJ tori by $T_1,\dots,T_l$. Let $M=N\cup_{\partial N}N$ be the double of $N$ along the boundary.  Then the two copies of $T_i$  for $i=1,\dots,l$ together with the $S_i$ which bound hyperbolic components are the JSJ tori for $M$.
\end{theorem}

\subsection{The Geometric Decomposition Theorem} \label{section:jsj}
The decomposition in Theorem~\ref{thm:geom} can be viewed as somewhat \emph{ad hoc} (`Seifert fibered vs.~hyperbolic').
The geometric point of view introduced by Thurston gives rise to an elegant reformulation of Theorem~\ref{thm:geom}.
Thurston introduced the notion of a \emph{geometry of a $3$-manifold} and of a \emph{geometric $3$-manifold}.
 We will now give a quick summary of the definitions and the most relevant results. We refer to the expository papers by Scott \cite{Sco83a} and  Bonahon \cite{Bon02} and to Thurston's book \cite{Thu97} for proofs and further references.



A \emph{$3$-dimensional geometry} $X$ is a smooth, simply connected $3$-manifold which is equipped with a smooth, transitive action of a Lie group $G$ by diffeo\-mor\-phisms on $X$, with compact point stabilizers.  The Lie group $G$ is called the \emph{group of isometries} of~$X$.  A \emph{geometric structure} on a $3$-manifold $N$ is a diffeomorphism from the interior of $N$ to $X/\pi$, where $\pi$ is a discrete subgroup of $G$ acting freely on~$X$.  The geometry~$X$ is said to \emph{model} $N$, and $N$ is said to \emph{admit an $X$-structure}, or just to be an \emph{$X$-manifold}.   There are also two technical conditions, which rule out redundant examples of geometries: the group of isometries is required to be maximal among Lie groups acting transitively on $X$ with compact point stabilizers; and  $X$ is required to have a compact model.

\index{geometry!$3$-dimensional}
\index{$3$-manifold!geometric structure}



Thurston showed that, up to a certain equivalence,  there exist precisely eight $3$-dimensional geometries that model compact $3$-manifolds. These geometries are: the $3$-sphere, Euclidean $3$-space, hyperbolic $3$-space, $S^2\times \R$, $\H^2\times \R$, the universal cover $\widetilde{\sl(2,\R)}$ of $\sl(2,\R)$, and two further geometries called $\Nil$ and $\Sol$.  We refer to \cite{Sco83a} for details.  Note that spherical and hyperbolic manifolds are precisely the type of manifolds we introduced in the previous section.  (It is well known that a $3$-manifold equipped with a complete spherical metric has to be closed.)  A $3$-manifold is called \emph{geometric} if it is an $X$-manifold for some geometry~$X$.

\index{$3$-manifold!geometric}

\medskip

The following theorem summarizes the relationship between Seifert fibered manifolds and geometric $3$-manifolds.

\begin{theorem}\label{thm:geoms}
Let $N$ be a compact, orientable $3$-manifold with empty or toroidal boundary.
We assume that $N\ne S^1\times D^2$, $N\ne S^1\times S^1\times I$, and that $N$ does not equal the twisted $I$-bundle over the Klein bottle
\textup{(}i.e., the total space of the unique non-trivial interval bundle over the Klein bottle\textup{)}. Then
$N$  is Seifert fibered if and only if $N$ admits a geometric structure based on one of the following geometries:
the $3$-sphere, Euclidean $3$-space,  $S^2\times \R$, $\H^2\times \R$, $\widetilde{\sl(2,\R)}$, $\Nil$.
\end{theorem}


We refer to \cite[Theorem~4.1]{Bon02} and \cite[Theorems~2.5,~2.7,~2.8]{Bon02} for the proof and for references (see also \cite[Theorem~5.3]{Sco83a} and \cite[Lecture~31]{FoM10}).  (Note that in \cite{Bon02} the geometries $\widetilde{\sl(2,\R)}$ and $\Nil$  are referred to as $H^2 \wti{\times} E^1$ and $E^2 \wti{\times} E^1$, respectively.)

\medskip


By a \emph{torus bundle} we mean  an oriented $3$-manifold which is a fiber  bundle over $S^1$ with fiber the $2$-torus $T$.  The action of the monodromy on $H_1(T;\Z)$ defines an element in $\op{SAut}(H_1(T;\Z))\cong \sl(2,\Z)$. Note that if $A\in \sl(2,\Z)$ is a matrix, then it follows from an elementary linear algebraic argument that one of the following occurs: \index{torus bundle}
\bn
\item $A^n=\id$ for some $n\in \{1,2,4,6\}$,  or
\item $A$ is non-diagonalizable but has eigenvalue $\pm 1$, or
\item $A$ has two distinct real eigenvalues.
\en
In the first case we say that the matrix is \emph{periodic}, in the second case we say it is \emph{nilpotent} and in the remaining case we say it is \emph{Anosov}.
If $N$ is a torus bundle with monodromy $\varphi$, then  $N$ is Seifert fibered if and only if $\varphi_*\in \op{SAut}(H_1(T;\Z))$  is periodic or nilpotent (see \cite{Sco83a}).

The following theorem (see \cite[Theorem~5.3]{Sco83a}  or \cite{Dub88}) now gives a complete classification of $\Sol$-manifolds.  \index{monodromy!abelian} \index{monodromy!nilpotent} \index{monodromy!Anosov}

\begin{theorem} \label{thm:sol}
Let $N$ be a compact, orientable $3$-manifold. Then $N$ is a $\Sol$-manifold if and only if one of the following occurs:
\begin{itemize}
\item[(1)] $N$ is a torus bundle with Anosov monodromy, or
\item[(2)] $N$ is a double of the twisted $I$-bundle $M$ over the Klein bottle with Anosov gluing map, i.e.,
\[ N= M\times 1\cup_{\varphi} M\times 2\]
such that the map $H_2(\partial M\times 2;\Z)=H_1(\partial M\times 1;\Z)\xrightarrow{\varphi_*} H_1(\partial M\times 2;\Z)$ is Anosov.
\end{itemize}
\end{theorem}


The following is now the `geometric version' of Theorem~\ref{thm:geom} (see \cite[Conjecture~2.2.1]{Mor05} and \cite[p.~5]{FoM10}).

\index{decomposition!geometric}
\index{theorems!Geometric Decomposition Theorem}

\begin{theorem}\label{thm:geom2}  \textbf{\emph{(Geometric Decomposition Theorem)}}
Let $N$ be a compact, orientable, irreducible $3$-manifold with empty or toroidal boundary.
We assume that $N\ne S^1\times D^2$, $N\ne S^1\times S^1\times I$, and that $N$ does not equal the twisted $I$-bundle over the Klein bottle. Then there exists a  collection of disjointly embedded  incompressible
surfaces $S_1,\dots,S_k$ which are either tori or Klein bottles,  such that each component of  $N$ cut along $S_1\cup \dots \cup S_k$  is geometric.
Furthermore, any such collection with a minimal number of components is unique up to isotopy.
\end{theorem}

We will quickly outline the existence of such a decomposition, assuming Theorems~\ref{thm:geom},~\ref{thm:geoms} and~\ref{thm:sol}.

\begin{proof}
Let $N$ be an orientable, irreducible $3$-manifold with empty or toroidal boundary such that $N\ne S^1\times D^2$, $N\ne S^1\times S^1\times I$, and such that $N$ does not equal the twisted $I$-bundle over the Klein bottle.
 By Theorem~\ref{thm:geom} there exists a minimal collection of disjointly embedded  incompressible tori $T_1,\dots,T_k$ such that each component of $N$ cut along $T_1\cup \dots \cup T_k$  is either hyperbolic or Seifert fibered.
 We denote the components of $N$ cut along $T_1\cup \dots \cup T_k$ by $M_1,\dots,M_r$.
 Note that $M_i\ne S^1\times D^2$ since the JSJ tori are incompressible and since $N\ne S^1\times D^2$.
 Now suppose that one of the $M_i$ is $S^1\times S^1\times I$. By the minimality of the number of tori and by our assumption that $N\ne S^1\times S^1\times I$, it follows easily
 that $N$ is a torus bundle with a non-trivial JSJ decomposition. By Theorem~\ref{thm:sol} and the discussion preceding it we see that $N$ is a $\Sol$ manifold, hence already geometric.

In view of  Theorem~\ref{thm:sol} we can  assume that $N$ is not the double of the twisted $I$-bundle over the Klein bottle.
For any $i$ such that the JSJ torus $T_i$ bounds a twisted $I$-bundle over the Klein bottle we now replace $T_i$ by the Klein bottle which is the core of the twisted $I$-bundle.

It is now straightforward to verify (using Theorem~\ref{thm:geoms}) that  the resulting collection of tori and Klein bottles decomposition has the required properties.
\end{proof}

\begin{remark}
The proof of Theorem~\ref{thm:geom2} also shows  how to obtain the decomposition postulated by Theorem~\ref{thm:geom2} from the decomposition given by Theorem~\ref{thm:geom}. Let $N$ be a compact, orientable, irreducible $3$-manifold with empty or toroidal boundary.
If $N$ is a $\Sol$-manifold, then $N$ has one JSJ torus, namely a surface fiber, but $N$ is geometric.
Now suppose that $N$ is not a $\Sol$-manifold.
Denote by  $T_1,\dots,T_l$ the JSJ tori of $N$. We assume that they are ordered
such that $T_1,\dots,T_r$ are precisely the tori which do not bound twisted $I$-bundles over a Klein bottle.
For $i={r+1},\dots,l$, each $T_i$ cobounds a twisted $I$-bundle over a Klein bottle $K_i$ and a hyperbolic JSJ component.
The decomposition of Theorem~\ref{thm:geom2} is then given by $T_1\cup \dots \cup T_{r}\cup K_{r+1}\cup \dots \cup K_l$.
\end{remark}

\begin{remark}
Let $\Sigma$ be a compact surface. We denote by $M(\Sigma)$ the mapping class group
\index{mapping class group}
of $S$, that is, the group of isotopy classes of orientation preserving self-diffeo\-mor\-phisms of $\Sigma$.
If $\Sigma$ is a torus, then $M(\Sigma)$ is canonically isomorphic to $\op{SAut}(H_1(\Sigma;\Z))$, see, e.g., \cite[Theorem~2.5]{FaM12}.
In the discussion before Theorem~\ref{thm:sol} we saw that the elements of $\op{SAut}(H_1(\Sigma;\Z))$ fall naturally into three distinct classes.
If $\chi(\Sigma)<0$ then
the Nielsen--Thurston Classification Theorem
\index{theorems!Nielsen--Thurston Classification Theorem}
says that a similar trichotomy appears for $M(\Sigma)$.
More precisely,  any class $f\in M(\Sigma)$ is either
\bn
\item periodic, i.e., $f$ is represented by $\varphi$ with $\varphi^n=\id_\Sigma$ for some $n\geq 1$, or
\index{surface self-diffeomorphism!periodic}
\item pseudo-Anosov, i.e., there exists $\varphi\colon \Sigma \to \Sigma$ which represents $f$ and a pair of transverse measured foliations and a $\l>1$ such that $\varphi$ stretches one measured foliation
 by  $\l$ and the other one by $\l^{-1}$, or
 \index{surface self-diffeomorphism!pseudo-Anosov}
\item reducible, i.e.\
there exists $\varphi\colon \Sigma \to \Sigma$ which represents $f$ and a minimal non-empty embedded $1$-manifold $\G$ in $\Sigma$
with a $\varphi$-invariant tubular neighborhood $\nu \G$ such that on each $\varphi$-orbit of $\Sigma\setminus \nu \G$ the restriction of $\varphi$
is either finite order or pseudo-Anosov.
\index{surface self-diffeomorphism!reducible}
\en
We refer to \cite{Nie44,BlC88,Thu88}, \cite[Chapter~13]{FaM12}, \cite{FLP79a,FLP79b} and \cite[Theorem~2.15]{CSW11} for details,
to \cite[Theorem~1]{Iva92} for an extension, and to  \cite{Gin81,Milb82,HnTh85} for the connection between the work of Nielsen and Thurston.  Thurston also determined the geometric structure of
the mapping torus $N$ of $\psi\colon \Sigma\to \Sigma$ in terms of $[\psi]\in M(\Sigma)$ as follows (see \cite{Thu86c} and \cite{Ot96,Ot01}).
\bn
\item If $[\psi]$ is periodic, then $N$ admits an $\H^2\times \R$ structure.
\item If $[\psi]$ is pseudo-Anosov, then $N$ is hyperbolic.
\item If $[\psi]$ is reducible, then $N$ admits a non-trivial JSJ decomposition
where the JSJ tori are given by the $\varphi$-orbits of the $1$-manifold $\G$, here $\varphi$ and $\G$ are as in the definition of a reducible element in the mapping class group.
\en
\end{remark}

\begin{remark}
A `generic' $3$-manifold is hyperbolic. This statement can be made precise in various ways.
\bn
\item Let $N$ be a hyperbolic $3$-manifold with one boundary component.
Thurs\-ton's Hyperbolic Dehn Surgery Theorem  \cite{Thu79} \index{theorems!Hyperbolic Dehn Surgery Theorem} says that
at most finitely many Dehn fillings are exceptional, i.e., do  not give hyperbolic $3$-manifolds.
Considerable effort has been expended on computing bounds for the number of exceptional fillings (see, e.g., \cite{Ag00,Ag10a,BGZ01,BCSZ08,FP07,BlH96,Lac00,Ter06,HK05} and the survey papers \cite{Boy02,Gon98}).  Lackenby--Meyerhoff \cite{LaM13} showed that there exist at most 10 Dehn fillings that are not hyperbolic.
\item Maher \cite{Mah11}, Rivin (see \cite[Section~8]{Riv12} and \cite{Riv08,Riv09,Riv10}), Lubotzky--Meiri \cite{LM11},
 Atalan--Korkmaz \cite{AK10} and Malestein--Souto \cite{MlS12} made precise the statement that a generic element in the mapping class group is pseudo-Anosov.
 \item The work of Maher \cite[Theorem~1.1]{Mah10} together with work of Hempel \cite{Hem01} and Kobayashi \cite{Koi88}
 and the Geometrization Theorem implies that `most' closed manifolds produced from
Heegaard splittings of a fixed genus are hyperbolic.
\en
\end{remark}

Before we continue our discussion of geometric $3$-manifolds we introduce a definition. Given a property $\PP$ of groups we say that a group $\pi$ is \emph{virtually $\PP$} if $\pi$ admits a (not necessarily normal) subgroup of finite index that satisfies $\PP$. \index{group!virtually $\PP$}

\medskip

In Table~\ref{tablegeoms} we  summarize some of the key properties of geometric $3$-manifolds.
Given a geometric $3$-manifold, the first column lists the geometry type, the second describes the fundamental group of $N$ and  the third describes the topology of $N$ (or a finite-sheeted cover).

If the geometry is neither $\Sol$ nor hyperbolic, then by Theorem~\ref{thm:geoms} the manifold $N$ is Seifert fibered. One can think of a Seifert fibered manifold as an $S^1$-bundle over an orbifold.
We denote by $\chi$ the orbifold Euler characteristic of the base orbifold and we denote by $e$ the Euler number. We refer to \cite[p.~427 and p.~436]{Sco83a} for the precise definitions.


\begin{table}[htbp]
   \centering
   \begin{tabular}{lllcc} 
      \toprule
 Geometry & Fundamental group & Topology &$ \chi$ & $e$ \\
      \midrule
	Spherical              & $\pi$ is finite & finitely covered by $S^3$ & $>0$ & $\neq 0$\\[1mm]
	$S^2\times \R$         & $\pi=\Z$ or $\pi$ is the & $N$ or a double cover & $>0$ & $=0$\\
	& infinite dihedral group & equals $S^1\times S^2$ & & \\[1mm]
	Euclidean               & $\pi$ is virtually $\Z^3$ & $N$ finitely covered by &0&0\\
	&& $S^1\times S^1\times S^1$ &&\\[1mm]
	Nil                    & $\pi$ is virtually nilpotent & $N$ finitely covered by & 0 & $\ne 0$ \\
& but not virtually $\Z^3$ & a torus bundle with &&\\
&& nilpotent monodromy &&\\[1mm]  
\mbox{Sol}                     & $\pi$ is solvable& $N$ or a double cover&& \\
& but not virtually & is a torus bundle with &&\\
& nilpotent & Anosov monodromy &&\\[1mm]
$\H^2\times \R$                   &  $\pi$ is virtually a & $N$ finitely covered by & $<0$ &0\\
& product $\Z\times F$ with &  $S^1\times \Sigma$ where $\Sigma$ is a &&\\
& $F$ a non-cyclic free group & surface with $\chi(\Sigma)<0$ &&\\[1mm]
$\widetilde{\sl(2,\R)}$         &  $\pi$ is a non-split extension & $N$ finitely covered by &  $<0$ &$\ne 0$\\
& of a non-cyclic & a non-trivial $S^1$-bundle &&\\
& free group $F$ by $\Z$   & over a surface $\Sigma$ with &&\\
&& $\chi(\Sigma)<0$ && \\[1mm]
hyperbolic              & $\pi$ infinite and $\pi$ does &$N$ is atoroidal &&\\
& not contain a non-trivial &&&\\
& abelian normal subgroup &&&\\[1mm]

      \bottomrule \\[2mm]
   \end{tabular}
   \label{tab:booktabs}
   \caption{Geometries of $3$-manifolds.}\label{tablegeoms}
\end{table}

We now give the references for Table~\ref{tablegeoms}.
We refer to \cite[p.~478]{Sco83a} for the last two columns.  For the first three rows we refer to \cite[p.~449,~p.~457,~p.~448]{Sco83a}.  We refer to \cite[p.~467]{Sco83a} for details regarding Nil
and we refer to \cite[Theorem~2.11]{Bon02} and \cite[Theorem~5.3]{Sco83a} for details regarding $\Sol$.
Finally we refer to \cite[p.~459,~p.~462,~p.~448]{Sco83a} for details regarding the last three geometries.
The fact that the fundamental group of a hyperbolic $3$-manifold does not contain a non-trivial abelian normal subgroup will be shown in Theorem~\ref{thm:infcyclicnormal}.

If $N$ is a non-spherical Seifert fibered manifold, then the Seifert fiber subgroup is infinite cyclic and normal in $\pi_1(N)$ (see \cite[Lemma~II.4.2]{JS79} for details). It now follows from the above table that the geometry of a geometric manifold can be read off from its fundamental group. In particular, if a $3$-manifold admits a geometric structure, then the type of that geometric structure is unique (see  also \cite[Theorem~5.2]{Sco83a} and  \cite[Section~2.5]{Bon02}).  Some of these geometries are very rare:  there exist only finitely many $3$-manifolds with  Euclidean  geometry or $S^2\times \R$ geometry \cite[p.~459]{Sco83a}.
Finally, note that the geometry of a geometric $3$-manifold with non-empty boundary  is either $\H^2\times \R$ or hyperbolic.

\subsection{$3$-manifolds with (virtually) solvable fundamental group}\label{section:solv}

The above discussion can be used to classify the abelian, nilpotent and solvable groups which appear as fundamental groups of $3$-manifolds.

\index{$3$-manifold group!solvable}
\index{$3$-manifold group!virtually solvable}

\begin{theorem}\label{thm:solv}\label{thm:virtsolv}
Let $N$ be an  orientable, non-spherical $3$-manifold which is either closed or has toroidal boundary.
Then the following are equivalent:
\begin{itemize}
\item[(1)]  $\pi_1(N)$ is solvable;
\item[(2)]  $\pi_1(N)$ is virtually solvable;
\item[(3)] $N$ is one of the following six types of manifolds:
\begin{itemize}
\item[(a)] $N=\rp^3\# \rp^3$;
\item[(b)] $N=S^1\times D^2$;
\item[(c)] $N=S^1\times S^2$;
\item[(d)] $N$ admits a finite solvable cover which is a torus bundle;
\item[(e)] $N=S^1\times S^1\times I$, where $I$ is the standard interval $I=[0,1]$;
\item[(f)] $N$ is the  twisted $I$-bundle over the Klein bottle.
\end{itemize}
\end{itemize}
\end{theorem}

Before we prove the theorem, we state a useful lemma.

\begin{lemma}\label{lem:free subgroup}
Let $\pi$ be a group.
If $\pi$ decomposes non-trivially as an amalgamated free product $\pi = A*_C B$, then $\pi$ contains a non-cyclic free subgroup unless $[A:C],[B:C]\leq 2$.  Similarly, if $\pi$ decomposes non-trivially as an HNN extension $\pi = A*_C$, then $\pi$ contains a non-cyclic free subgroup unless one of the inclusions of $C$ into $A$ is an isomorphism.
\end{lemma}

The proof of the lemma is a standard application of Bass--Serre theory~\cite{Ser80}.  We are now ready to prove the theorem.

\begin{proof}[Proof of Theorem~\ref{thm:solv}]
The implication $(1) \Rightarrow (2)$ is obvious.
Note that the group $\pi_1(\rp^3\# \rp^3)=\Z/2*\Z/2$ is isomorphic to the infinite dihedral group, so is solvable.  It is clear that if $N$ is one of the remaining types (b)--(f) of $3$-manifolds, then $\pi_1(N)$ is also solvable. This shows $(3) \Rightarrow (1)$.

Finally, assume that $(2)$ holds.   We will show that $(3)$ holds.
Let $A$ and $B$ be two non-trivial groups. By Lemma~\ref{lem:free subgroup}, $A*B$ contains a non-cyclic free group (in particular it is not virtually solvable) unless $A=B=\Z/2$.
Note that by the Elliptization Theorem, any $3$-manifold $M$ with $\pi_1(M)\cong \Z/2$  is diffeomorphic to $\rp^3$.
 It follows that if $N$ is a compact $3$-manifold with solvable fundamental group, then either
$N=\rp^3\# \rp^3$ or $N$ is prime.

Since $S^1\times S^2$ is the only orientable prime $3$-manifold which is not irreducible we can henceforth assume that $N$ is irreducible.
Now let $N$ be an irreducible $3$-manifold which is either closed or has  toroidal boundary and such that $\pi=\pi_1(N)$ is infinite and solvable.
We furthermore assume that $N\ne S^1\times S^1\times I$ and that~$N$ does not equal the twisted $I$-bundle over the Klein bottle. It now follows from Theorem~\ref{thm:geom2}
that $\pi_1(N)$ is the fundamental group of a graph of groups where the vertex groups are fundamental groups of geometric $3$-manifolds. By Lemma~\ref{lem:free subgroup}, $\pi_1(N)$ contains a non-cyclic free group unless the $3$-manifold is already geometric.
If $N$ is geometric, then by the discussion preceding this theorem, $N$ is either a Euclidean manifold, a $\Sol$-manifold or a
$\Nil$-manifold,
and  $N$ is finitely covered by a torus bundle. It follows from the discussion of these geometries in \cite{Sco83a} that the finite cover
 is in fact a finite solvable cover.  (Alternatively we could have applied \cite[Theorems~4.5~and~4.8]{EvM72}, \cite[Corollary~4.10]{EvM72} and  \cite[Section~5]{Tho79} for a proof of the theorem without using the full Geometrization Theorem and only requiring the Elliptization Theorem.)
\end{proof}

\begin{remark}
It follows from the proof of the above
theorem that every compact $3$-manifold with nilpotent fundamental group is either a spherical, a Euclidean, or a $\Nil$-manifold. Using the discussion
of these geometries in \cite{Sco83a} one can then  determine the list of nilpotent groups which can appear as fundamental groups of compact $3$-manifolds. \index{$3$-manifold group!nilpotent}

This list of nilpotent groups was already determined `pre-Geometrization' by Thomas \cite[Theorem~N]{Tho68} for the closed case
 and by Evans--Moser  \cite[Theorem~7.1]{EvM72} in the general case.
\end{remark}

\begin{remark}
In Table~\ref{tableabelian} we give the complete list of all compact $3$-manifolds with abelian fundamental groups.
\begin{table}[htbp]
   \centering
   \begin{tabular}{ll} 
      \toprule
abelian group $\pi$&compact $3$-manifolds with fundamental group $\pi$\\
\midrule
$ \Z$&$S^2\times S^1, \, D^2\times S^1$ (the twisted sphere bundle over the circle)\\
$ \Z^3$&{$S^1\times S^1\times S^1$}\\
$ \Z/n$& the lens spaces  $L(n,m)$, $m\in \{1,\dots,n\}$ with  $(n,m)=1$\\
$ \Z\oplus \Z$ &$ S^1\times S^1\times I$\\
$ \Z\oplus \Z/2$& $S^1\times \rp^2$\\
 \bottomrule\\[2mm]
 \end{tabular}
   \label{tab:booktabs2}
   \caption{Abelian fundamental groups  of $3$-manifolds.}\label{tableabelian}
\end{table}
The table can be obtained in a straightforward fashion from the Prime Decomposition Theorem and the Geometrization Theorem.
The fact that the groups in the table are indeed the only abelian groups that appear as fundamental groups of compact $3$-manifolds
is in fact a classical `pre-Geometrization' result. The list of abelian fundamental groups of closed $3$-manifolds  was first determined by Reidemeister \cite[p.~28]{Rer36} and in the general case by
   Epstein  (\cite[Satz~IX']{Sp49}, \cite[Theorem~3.3]{Ep61a} and \cite[Theorem~9.1]{Ep61b}). (See also \cite[Theorems~9.12~and~9.13]{Hem76}.) \index{$3$-manifold group!abelian} \index{$3$-manifold!lens space}
   \end{remark}


\section{The classification of $3$-manifolds by their fundamental groups}

\noindent
In this section we will discuss the degree to which the fundamental group determines a $3$-manifold and its topological properties.  By Moise's Theorem  \cite{Moi52,Moi77} (see also \cite{Bin59,Hama76,Shn84}) any topological $3$-manifold also admits a smooth structure, and two $3$-manifolds are homeomorphic if and only if they are diffeomorphic. We can therefore freely go back and forth between the topological and the smooth categories. \index{$3$-manifold!smooth} \index{theorems!Moise's Theorem}
(Note that this also holds for surfaces by work of Rad\'o \cite{Rad25} but not for manifolds of dimension greater than three,
see \cite{Mil56,Ker60,KeM63,KyS77,Fre82,Do83,Lev85,Mau13}.)

\begin{remark}
\mbox{}

\bn
\item By work of Cerf \cite{Ce68} and Hatcher \cite[p.~605]{Hat83} (see also \cite{Lau85}), given any closed $3$-manifold $M$ the map $\operatorname{Diff}(M)\to \operatorname{Homeo}(M)$ between the space of diffeo\-mor\-phisms of $M$ and the space of homeomorphisms of $M$ is in fact a weak homotopy equivalence.
\item Bing \cite{Bin52} gives an example of a continuous involution on $S^3$ with fixed point set a wild $S^2$.  In particular, this involution cannot be smoothed.
\en
\end{remark}

\subsection{Closed $3$-manifolds and fundamental groups}

It is well known that closed, compact surfaces are determined by their fundamental groups, and
compact surfaces with non-empty boundary are determined by their fundamental groups together with the number of boundary components.
In $3$-manifold theory a similar, but more subtle, picture emerges.

\medskip

One quickly notices that there are three ways to construct pairs of closed, orientable, non-diffeomorphic $3$-manifolds with
isomorphic fundamental groups.
\newcounter{itemcountere}
\begin{list}
{{(\Alph{itemcountere})}}
{\usecounter{itemcountere}\leftmargin=2em}
\item Consider lens spaces $L(p_1,q_1)$ and $L(p_2,q_2)$. They are diffeomorphic if and only if $p_1=p_2$ and $q_1q_2^{\pm 1}\equiv \pm 1\mbox{ mod }p_i$, but
they are homotopy equivalent  if and only if $p_1=p_2$ and $q_1q_2^{\pm 1}\equiv \pm t^2\mbox{ mod }p_i$ for some $t$, and
their fundamental groups are isomorphic if and only if $p_1=p_2$.
\item Let $M$ and $N$ be two oriented $3$-manifolds.  Denote by $\ol{N}$ the same manifold as $N$ but with opposite orientation.
Then $\pi_1(M\# N)\cong \pi_1(M\# \ol{N})$ but if neither $M$ nor $N$ admits an orientation reversing diffeomorphism,
then  $M\# N$ and $M\# \ol{N}$ are not diffeomorphic.
\item  Let $M_1$, $M_2$ and $N_1$, $N_2$ be $3$-manifolds with $\pi_1(M_i)\cong \pi_1(N_i)$ and such that $M_1$ and $N_1$ are not diffeomorphic. Then $\pi_1(M_1\# M_2)\cong \pi_1(N_1\# N_2)$ but in general $M_1\# M_2$ is not diffeomorphic to $N_1\# N_2$.
\end{list}
Reidemeister \cite[p.~109]{Rer35} and Whitehead \cite{Whd41a} classified lens spaces in the PL-category.
The classification of lens spaces up to homeomorphism, i.e. the first statement above,
then follows from Moise's proof \cite{Moi52} of the `Hauptvermutung' in dimension three.
We refer to \cite{Mil66} and \cite[Section~2.1]{Hat} for more modern accounts and we refer to  \cite[p.~455]{Fo52}, \cite[p.~181]{Bry60},
\cite{Tur76}
 and \cite{PY03} for different approaches.
The fact that lens spaces with the same fundamental group are not necessarily homeomorphic was first observed by
Alexander \cite{Ale19,Ale24}.
 The other two statements follow from the uniqueness of the prime decomposition.  In the subsequent discussion we will see that (A), (B) and~(C) form in fact a complete list of methods for finding examples of pairs of closed, orientable, non-diffeomorphic $3$-manifolds with isomorphic fundamental groups.

\medskip

Recall that Theorem~\ref{thm:prime} implies that the fundamental group of a compact, orientable $3$-manifold
is isomorphic to a free product of fundamental groups of prime $3$-manifolds.
The Kneser Conjecture implies that the converse holds.

\begin{theorem} \textbf{\emph{(Kneser Conjecture)}}\label{thm:kneserconj}
Let $N$ be a  compact, orientable $3$-mani\-fold with incompressible boundary.
If $\pi_1(N)\cong \G_1*\G_2$, then there exist compact, orientable $3$-manifolds $N_1$ and $N_2$
with $\pi_1(N_i)\cong \G_i$ and $N\cong N_1\# N_2$.
\end{theorem}

\index{theorems!Kneser Conjecture}

The Kneser Conjecture was first proved by Stallings \cite{Sta59a,Sta59b} in the closed case, and by
Heil \cite[p.~244]{Hei72} in the bounded case. (We also refer to \cite{Ep61c} and \cite[Section~7]{Hem76} for details.)

\medskip

The following theorem  is a consequence of the Geometrization Theorem, the Mostow--Prasad Rigidity Theorem \ref{thm:mostow-prasad}, work of Waldhausen \cite[Corollary~6.5]{Wan68a} and
Scott \cite[Theorem~3.1]{Sco83b} and classical work on spherical $3$-manifolds (see~\cite[p.~113]{Or72}).

\begin{theorem}\label{thm:primepi}
Let $N$ and $N'$ be two   orientable, closed, prime $3$-manifolds and let  $\varphi\colon \pi_1(N)\to \pi_1(N')$ be an isomorphism.
\begin{itemize}
\item[(1)] If $N$ and $N'$ are not lens spaces, then  $N$ and $N'$ are homeomorphic.
\item[(2)] If  $N$ and $N'$ are not spherical, then there exists a homeomorphism which induces $\varphi$.
\end{itemize}
\end{theorem}

\begin{remark}
The Borel Conjecture states that if $f\colon N\to N'$ is a homotopy equivalence between closed and aspherical topological manifolds,
then $f$ is homotopic to a homeomorphism.
In dimensions greater than four the Borel Conjecture is known to hold for large classes of fundamental groups, e.g.,  if the fundamental group is word-hyperbolic \cite[Theorem~A]{BaL12}. The high-dimensional results also extend to dimension four if the fundamental groups are furthermore good in the sense of Freedman \cite{Fre84}.
The Borel Conjecture holds for all $3$-manifolds.
In the case where $N$ and $N'$ are orientable,  this is an immediate consequence of Theorem~\ref{thm:primepi};
the case where $N$ or $N'$ is non-orientable was proved by Heil~\cite{Hei69a}.
\end{remark}

Summarizing, Theorems~\ref{thm:prime},~\ref{thm:kneserconj} and~\ref{thm:primepi}
show that fundamental groups determine closed $3$-manifolds up to orientation of the prime factors and up to the indeterminacy arising from lens spaces. More precisely, we have the following:

\begin{theorem}\label{thm:p1determinen}
Let $N$ and $N'$ be two closed, oriented $3$-manifolds with isomorphic fundamental groups.
Then there exist
natural numbers $p_1,\dots,p_r$, $q_1,\dots,q_r$ and $q'_1,\dots,q'_r$  and  oriented manifolds $N_1,\dots,N_s$ and $N'_1,\dots,N'_s$
such that the following three conditions hold:
\begin{itemize}
\item[(1)] we have homeomorphisms
\begin{align*}
N &\cong L(p_1,q_1)\# \cdots \# L(p_r,q_r)\, \#\, N_1\#  \cdots \# N_s \mbox{ and }\\
N'&\cong L(p_1,q_1')\# \cdots \# L(p_r,q_r')\, \#\, N_1'\#  \cdots \# N_s';
\end{align*}
\item[(2)] $N_i$ and $N_i'$ are homeomorphic \textup{(}but possibly with opposite orientations\textup{)}; and
\item[(3)] for $i=1,\dots,r$ we have $q_i'\not\equiv \pm q_i^{\pm 1} \bmod p_i$.
 \end{itemize}
\end{theorem}

\subsection{Peripheral structures and $3$-manifolds with boundary}

According to The\-o\-rem~\ref{thm:primepi}, orientable, prime $3$-manifolds with infinite fundamental groups are determined by their fundamental groups if they are closed.  The same conclusion does not hold if we allow boundary.  For example, if $K$ is the trefoil knot with an arbitrary orientation, then $S^3\smallsetminus \nu (K\# K)$ and $S^3\smallsetminus \nu (K\# -K)$ (i.e., the exteriors of the granny knot and the square knot)  have isomorphic fundamental groups, but the spaces are not homeomorphic (which can be seen by studying the
linking form (see \cite[p.~826]{Sei33b}) or the Blanchfield form \cite{Bla57},
which in turn can be studied using Levine--Tristram signatures, see \cite{Kea73,Lev69,Tri69}).

\medskip

We will need the following definition to formulate the classification theorem.

\begin{definition}
Let $N$ be a $3$-manifold with incompressible boundary.  The fundamental group $\pi_1(N)$ of $N$ together with the set of conjugacy classes of its subgroups determined by the boundary components is called the \emph{peripheral structure of $N$}.
\end{definition}

\index{structure!peripheral}

We now have the following theorem.

\begin{theorem}
Let $N$ and $N'$ be two  compact, orientable, irreducible $3$-manifolds with non-spherical, non-trivial incompressible boundary.
\begin{itemize}
\item[(1)] If $\pi_1(N)$ and $\pi_1(N')$ are isomorphic, then $N$ can be turned into $N'$
using finitely many `Dehn flips'.
\item[(2)] There exist only finitely many  compact, orientable, irreducible $3$-manifolds with non-spherical, non-trivial incompressible boundary
such that the fundamental group is isomorphic to $\pi_1(N)$.
\item[(3)]
If there exists an isomorphism $\pi_1(N)\to \pi_1(N')$ which sends the peripheral structure of $N$ isomorphically to the peripheral structure of $N'$,
then $N$ and $N'$ are homeomorphic.
\end{itemize}
\end{theorem}

The first two statements of the theorem were proved by  Johannson \cite[Theorem~29.1~and~Corollary~29.3]{Jon79a}.
We  refer to \cite[Section~X]{Jon79a} for the definition of Dehn flips.
(See also \cite{Swp80a} for a proof of the second statement.)
The third statement was   proved by Waldhausen.  We refer to \cite[Corollary~7.5]{Wan68a} and \cite{JS76} for details.
Note that if the manifolds $N$ and $N'$ have no Seifert fibered JSJ components, then any isomorphism of fundamental groups
is in fact induced by a homeomorphism (this follows, e.g., from \cite[Theorem~1.3]{Jon79c}).
\medskip

We conclude this section with a short discussion of knots.
A knot is a simple closed curve in $S^3$. A knot is called prime if it is not the connected sum of two non-trivial knots.
Somewhat surprisingly, in light of the above discussion,
prime knots  are in fact determined by their fundamental  groups.
More precisely, if $K_1$ and $K_2$ are two prime knots with $\pi_1(S^3\setminus \nu K_1)\cong \pi_1(S^3\setminus K_2)$,
then there exists a homeomorphism $f$ of $S^3$ with $f(K_1)=K_2$.
This was first proved by Gordon--Luecke \cite[Corollary~2.1]{GLu89}
extending earlier work of Culler--Gordon--Luecke--Shalen \cite{CGLS85,CGLS87} and Whitten  \cite{Whn86,Whn87}.
See \cite{Tie08,Fo52,Neh61a,Sim76b,FW78,Sim80,Swp80b} for earlier discussions and work on this result.
Non-prime knots are determined by their `quandles', see \cite{Joy82} and \cite{Mae82}, and their `2-generalized knot groups', see \cite{LiN08,NN08,Tuf09}. \index{knot!prime}

\subsection{Submanifolds and subgroups}

Let $M$ be a connected submanifold of a $3$-manifold $N$.  If $M$ has incompressible boundary, then the inclusion-induced map $\pi_1(M)\to \pi_1(N)$ is injective, and $\pi_1(M)$ can be viewed as a subgroup of $\pi_1(N)$, which is well defined up to conjugacy.  In the previous two sections we have seen that $3$-manifolds are, for the most part, determined by their fundamental groups.  The following theorem, due to Jaco and Shalen \cite[Corollary~V.2.3]{JS79}, now says that submanifolds of $3$-manifolds
are, under mild assumptions, completely determined by the subgroups they define.

\begin{theorem}
Let $N$ be a compact irreducible $3$-manifold and let $M$,~$M'$ be two compact connected submanifolds in the interior of $N$
whose boundaries are incompressible. Then $\pi_1(M)$ and $\pi_1(M')$ are conjugate if and only if
there exists a homeomorphism $f\colon N\to N$ which is isotopic, relative to $\partial N$, to the identity, such that $f(M)=f(M')$.
\end{theorem}

\subsection{Properties of $3$-manifolds and their fundamental groups}

In the previous section we saw that orientable, closed irreducible $3$-manifolds with infinite fundamental groups are determined by their fundamental groups. We also saw that the fundamental group
determines the fundamental groups of the prime factors of a given  compact, orientable $3$-manifold. It is interesting to ask which topological properties of $3$-manifolds can be `read off' from the fundamental group.

\medskip

Let $N\ne S^1\times D^2$ be a compact, irreducible $3$-manifold which is not a line bundle.  Let $\mathrm{Diff}_0(N)$ be the identity component of the group $\mathrm{Diff}(N)$ of diffeo\-mor\-phisms of $N$.  The quotient $\mathrm{Diff}(N)/\mathrm{Diff}_0(N)$ is denoted by $\MM(N)$.  Furthermore, given a group $\pi$, we denote by $\out(\pi)$ the group of outer automorphisms
\index{outer automorphism group}
of $\pi$ (i.e., the quotient of the group of isomorphisms of $\pi$ by its normal subgroup of inner automorphisms of $\pi$).  It follows from the Rigidity Theorem \ref{thm:mostow-prasad}, from Waldhausen \cite[Corollary~7.5]{Wan68a} and from the Geometrization Theorem, that the canonical map
\[ \Phi\colon \MM(N)\to \big\{ \varphi\in \out(\pi) : \mbox{$\varphi$ preserves the peripheral structure}\big\}\]
is an isomorphism.
\bn
\item If $N$ is hyperbolic, then it is a consequence of the Rigidity Theorem (see \cite[Theorem~C.5.6]{BP92} and also \cite{Jon79b},~\cite[p.~213]{Jon79a})
that $\out(\pi)$ is finite and canonically isomorphic to the isometry group of $N$.
\item If $N$ is a Seifert fibered space, i.e. if $N$ admits a fixed-point free $S^1$-action, then $\MM(N)$ contains torsion elements of arbitrarily large order. On the other hand Kojima \cite[Theorem~4.1]{Koj84}
showed that if $N$ is a closed irreducible $3$-manifold which is not Seifert fibered, then there is a bound on the order of finite subgroups of $\MM(N)$.
\item
 If $N$ is a closed irreducible $3$-manifold which is not Seifert fibered,
then it follows from the above discussion of the hyperbolic case, from Zimmermann \cite[Satz~0.1]{Zim82}, and from the Geometrization Theorem,
that any finite subgroup of $\out(\pi_1(N))$ can be represented by a finite group of diffeo\-mor\-phisms of $N$ (see also \cite{HeT87}).
 The case of Seifert fibered $3$-manifolds is somewhat more complicated and is treated by Zieschang and Zimmermann
\cite{ZZ82,Zim79} and Raymond \cite[p.~90]{Ray80} (see also \cite{RaS77,HeT78,HeT83}).
\en

Note that for $3$-manifolds which are spherical or not prime the map $\Phi$ is in general neither
injective nor surjective. We refer to \cite{Gab94b,McC90,McC95} for more information.
\medskip

We now give a few more situations in which topological information can be `directly' obtained from the fundamental group.

\bn
\item Let $N$ be a compact $3$-manifold and $\phi\in H^1(N;\Z)=\hom(\pi_1(N),\Z)$ a non-trivial class.
Work of Stallings (see \cite[Theorem~2]{Sta62} and (K.\ref{K.stallings62})), together with the resolution of the Poincar\'e Conjecture, shows that $\phi$ is a fibered class (i.e., can be realized by a surface bundle $N\to S^1$)
if and only if $\ker\{\phi:\pi_1(N)\to \Z\}$ is finitely generated. (If $\pi_1(N)$ is a one-relator group the latter condition can be verified easily using Brown's criterion,
see \cite[\S~4]{Brob87} and \cite{Mas06a,Dun01} for details.)
We refer to \cite[p.~381]{Neh63b} for an alternative proof for knots,
to \cite[p.~86]{Siv87} and \cite[p.~953]{Bie07} for a homological reformulation of Stallings' criterion, and to \cite[Theorem~5.2]{Zu97} for a group-theoretic way to detect semibundles.
\item
  Let $N$ be an orientable, closed, irreducible $3$-manifold such that $\pi_1(N)$
  is an amalgamated product $A_1*_BA_2$ where $B$ is the fundamental group of a closed surface $S\ne S^2$.
   Feustel \cite[Theorem~1]{Fe72a} and Scott \cite[Theorem~2.3]{Sco72} showed that this splitting of $\pi_1(N)$ can be realized geometrically, i.e., there exists an embedding of $S$ into $N$
  such that $N\smallsetminus S$ consists of two components $N_1,N_2$ such that $(B,B\to A_1,B\to A_2)$
  and $(\pi_1(S),\pi_1(S)\to \pi_1(N_1),\pi_1(S)\to \pi_1(N_2))$ are triples which are isomorphic in the obvious sense.
  We refer to  \cite{Fe73} and Scott \cite[Theorem~3.6]{Sco74} for more details.
  We furthermore refer to Feustel--Gregorac \cite[Theorem~1]{FG73} and \cite[Corollary~1.2~(a)]{Sco80}
  (see also \cite{TY99})
  for a similar result corresponding to HNN extensions where the splitting is given by closed surfaces or annuli.

More generally, if $\pi_1(N)$ admits a non-trivial decomposition as a graph of groups (e.g., as an amalgamated product or an HNN extension), then this decomposition gives rise to a decomposition along incompressible surfaces of $N$
with the same underlying graph.  (Some care is needed here: in the general case the edge and vertex groups of the new decomposition may be different from the edge and vertex groups of the original decomposition.)
We refer to Culler--Shalen \cite[Proposition~2.3.1]{CuS83} for details and for \cite{Hat82,HO89,CoL92,SZ01,ChT07,HoSh07,Gar11,DG12} for extensions of this result.
\item If $N$ is a geometric $3$-manifold, then by the discussion of Section~\ref{section:jsj} the geometry of $N$ is determined by the properties of $\pi_1(N)$.
\item The Thurston norm $H^1(N;\R)\to \R^{\geq 0}$ measures the minimal complexity of surfaces dual to cohomology classes. We refer to \cite{Thu86a} and Section~\ref{section:thurstonnorm} for a precise definition
 and for details.
 \bn
 \item If $N$ is a closed $3$-manifold with $b_1(N)=1$, then it follows from \cite[Theorem~1]{FG73} that the Thurston norm can be recovered in terms of splittings of fundamental groups along surface groups.
  By work of Gabai \cite[Corollary~8.3]{Gab87} this also gives a group-theoretic way to recover the genus of a knot in $S^3$.
  \item If $N$ is a $3$-manifold with non-empty boundary and with $H_2(N;\Z)=0$, i.e., $N$ is the exterior of a knot in a rational homology sphere,
then  Calegari \cite[Proof~of~Proposition~4.4]{Cal09} gives a group-theoretic interpretation of the Thurston seminorm  of $N$ in terms of the `stable commutator length' of a longitude.
\en
However, there does not seem to be a good group-theoretical equivalent to the Thurston norm for general $3$-manifolds.
Nevertheless, the Thurston norm and the hyperbolic volume can be recovered from the fundamental group alone using the Gromov norm; see \cite{Grv82}, \cite[Corollary~6.18]{Gab83a},~\cite[p.~79]{Gab83b}, and \cite[Theorem~6.2]{Thu79}, for background and details.
For most $3$-manifolds, the Thurston norm can be obtained from the fundamental group using twisted Alexander polynomials, see
\cite{FKm06,FV12b,DFJ12}.
\item Scott and Swarup gave an algebraic characterization of the JSJ decomposition of a compact, orientable $3$-manifold with incompressible boundary~\cite[Theorem~2.1]{SS01} (see also \cite{SS03}).
\en
In many cases, however, it is difficult  to obtain topological information about $N$ by just applying group-theoretical methods to $\pi_1(N)$.  For example, it is obvious that given a closed $3$-manifold $N$,  the minimal number
$r(N)$ of generators of $\pi_1(N)$ is a lower bound on the Heegaard genus $g(N)$ of $N$. It has been a long standing question of Waldhausen's when $r(N)=g(N)$ (see \cite[p.~149]{Hak70} and \cite{Wan78b}),
the case $r(N)=0$ being equivalent to the Poincar\'e conjecture.
It has been known for a while that $r(N)\neq g(N)$ for graph manifolds~\cite{BoZ83,BoZ84, Zie88, Mon89, Wei03, ScW07,Wo11},
and evidence for the inequality for some hyperbolic $3$-manifolds was given in~\cite[Theorem~2]{AN12}.
In contrast to this,  work of Souto \cite[Thoerem~1.1]{Sou08} and Namazi--Souto \cite[Theorem~1.4]{NS09} yields that $r(N)=g(N)$ for `sufficiently complicated' hyperbolic $3$-manifolds (see also \cite{Mas06a} for more examples).
Recently Li  \cite{Lia11} showed that there also exist hyperbolic $3$-manifolds with $r(N)<g(N)$.
See \cite{Shn07} for some background.

\section{Centralizers}

\subsection{The centralizer theorems}

Let $\pi$ be a group. The \emph{centralizer} of a subset  $X\subseteq \pi$ is defined to be the subgroup \[ C_{\pi}(X):=\{g\in \pi : \text{$gx=xg$ for all $x\in X$}\}.\]
Determining the  centralizers is often one of the key steps in understanding a group. \index{group!centralizer}
In the world of $3$-manifold groups, thanks to the Geometrization Theorem,   an almost complete picture emerges.
In this section we will only consider $3$-manifolds to which Theorem~\ref{thm:geom} applies, i.e., $3$-manifolds
with empty or toroidal boundary which are compact, orientable and irreducible. But many of the results of this section also generalize fairly easily to fundamental groups of compact $3$-manifolds in general, using the arguments of Sections~\ref{section:prelim1} and~\ref{section:prelim2}.

\medskip

The following theorem reduces the determination of centralizers to the case of Seifert fibered manifolds.

\begin{theorem}\label{thm:centgen}
Let $N$ be a compact, orientable, irreducible $3$-manifold  with empty or toroidal boundary. We write $\pi=\pi_1(N)$.
Let $g\in \pi$ be non-trivial. If $C_{\pi}(g)$ is non-cyclic, then
one of the following holds:
\begin{itemize}
\item[(1)] there exists a JSJ torus $T$ and $h\in \pi$ such that $g\in h\,\pi_1(T)\,h^{-1}$ and such that
\[ C_{\pi}(g)=h\,\pi_1(T)\,h^{-1};\]
\item[(2)] there exists a boundary component $S$  and $h\in \pi$ such that
 $g\in h\,\pi_1(S)\,h^{-1}$ and such that
\[ C_{\pi}(g)=h\,\pi_1(S)\,h^{-1};\]
\item[(3)] there exists a Seifert fibered component $M$ and $h\in \pi$ such that
 $g\in h\pi_1(M)h^{-1}$ and such that
\[ C_{\pi}(g)=h\,C_{\pi_1(M)}\,(h^{-1}gh)h^{-1}.\]
\end{itemize}
\end{theorem}

\begin{remark}
Note that one could formulate the theorem more succinctly: if $g$ is non-trivial and $C_\pi(g)$ is non-cyclic,
then there exists a component $C$ of the characteristic submanifold and $h\in \pi$ such that $g\in h\,\pi_1(C)\,h^{-1}$ and such that
\[ C_{\pi}(g)=h\,C_{\pi_1(C)}\,(h^{-1}gh)h^{-1}.\]
\end{remark}

\index{$3$-manifold group!centralizers}

We will provide a short proof of Theorem~\ref{thm:centgen} which makes use of the deep results of Jaco--Shalen and Johannson and of the Geometrization Theorem for non-Haken manifolds.  Alternatively the theorem can be proved using the Geometrization Theorem much more explicitly---we refer to \cite{Fri11} for details.

\begin{proof}
We first consider the case that $N$ is hyperbolic. In Section~\ref{section:diagram} we will see that we can
 view  $\pi=\pi_1(N)$ as a discrete, torsion-free subgroup of $\psl(2,\C)$.
Note that the centralizer of any non-trivial matrix in $\psl(2,\C)$ is abelian and isomorphic to either $\Z$ or $\Z^2$; this can be seen easily using the Jordan normal form of such a matrix.
Now let $g\in \pi\subseteq \psl(2,\C)$ be non-trivial. If
$C_{\pi}(g)$ is not  infinite cyclic,
then it is  a free abelian group of rank two.
It now follows from Theorem~\ref{thm:hypatoroidal} that either (1) or (2) holds.

If $N$ is Seifert fibered, then the theorem is trivial.
It follows from Theorem~\ref{thm:geom} that it remains to consider the case where  $N$ admits a non-trivial JSJ decomposition.
 In that case  $N$ is in particular Haken (see Section~\ref{section:diagram} for the definition) and the theorem follows from  \cite[Theorem~VI.1.6]{JS79} (see also \cite[Theorem~4.1]{JS78}, \cite[Proposition~32.9]{Jon79a} and \cite[Theorem~1]{Sim76a}).
\end{proof}

We now turn to the study of centralizers in Seifert fibered manifolds.
Let $N$ be a Seifert fibered manifold with a given Seifert fiber structure. Then there exists a projection map $p\colon N\to B$ where $B$ is the base orbifold. We denote by $B'\to B$ the orientation cover; note that this is either the identity or a $2$-fold cover. Following \cite{JS79} we refer to $(p_*)^{-1}(\pi_1(B'))$ as the \emph{canonical subgroup} of $\pi_1(N)$. If $f$ is a regular Seifert fiber of the Seifert fibration, then we refer to the subgroup of $\pi_1(N)$ generated by $f$ as the \emph{Seifert fiber subgroup}. Recall that if $N$ is non-spherical, then the Seifert fiber subgroup is infinite cyclic and normal. (Note that the fact that the Seifert fiber subgroup is normal implies in particular that it is well defined, and not just up to conjugacy.)

\index{Seifert fibered manifold!Seifert fiber subgroup}
\index{Seifert fibered manifold!canonical subgroup}

\begin{remark}
The definition of the canonical subgroup and of the Seifert fiber subgroup depend on the Seifert fiber structure.
By \cite[Theorem~3.8]{Sco83a} (see also \cite{OVZ67} and \cite[II.4.11]{JS79}) a Seifert fibered manifold $N$ admits a unique Seifert fiber structure unless $N$ is either covered by $S^3$, $S^2\times \R$, or the $3$-torus,
or if either $N=S^1\times D^2$, $N$ is an $I$-bundle over the torus or the Klein bottle.
\end{remark}

The following theorem, together with Theorem~\ref{thm:centgen}, now classifies centralizers of compact, orientable, irreducible $3$-manifolds  with empty or toroidal boundary.

\begin{theorem}\label{thm:centsfs} Let  $N$ be  an  orientable, irreducible, non-spherical Seifert fibered manifold  with a given Seifert fiber structure.
Let $g\in \pi=\pi_1(N)$ be a non-trivial element.
Then the following hold:
\begin{itemize}
\item[(1)] if $g$ lies in the Seifert fiber group, then $C_{\pi}(g)$ equals the canonical subgroup;
\item[(2)] if $g$ does not lie in the Seifert fiber group, then the intersection of $C_{\pi}(g)$ with the canonical subgroup is abelian---in particular, $C_{\pi}(g)$ admits an abelian subgroup of index at most two;
\item[(3)] if $g$ does not lie in the canonical subgroup, then $C_{\pi}(g)$ is infinite cyclic.
\end{itemize}
\end{theorem}

\index{$3$-manifold group!centralizers}

\begin{proof}
The first statement is \cite[Proposition~II.4.5]{JS79}. The second and the third statement follow from
\cite[Proposition~II.4.7]{JS79}.
\end{proof}

Let  $N$ be  an  orientable, irreducible, non-spherical Seifert fibered manifold. It follows immediately from the theorem
that  if $g$ does not lie in the Seifert fiber group of a Seifert fiber structure, then $C_\pi(g)$ is isomorphic to one of $\Z$, $\Z\oplus \Z$, or the fundamental group of a Klein bottle. (See \cite[p.~82]{JS78} for details.)

\subsection{Consequences of the centralizer theorems}

Let $\pi$ be a group and $g\in \pi$. We say $h\in \pi$ is a \emph{root} of $g$ if a power of $h$ equals $g$.
We denote by $\roots_\pi(g)$ the set of all roots of $g$ in $\pi$.
Following \cite[p.~32]{JS79} we say that $g\in \pi$ has \emph{trivial root structure} if $\roots_\pi(g)$
lies in a cyclic subgroup of $\pi$.
We say that $g\in \pi$ has \emph{nearly trivial root structure} if $\roots_\pi(g)$ lies in a subgroup of $\pi$ which admits an abelian subgroup of index at most two.

\index{group!root structure}
\index{group!root}


\begin{theorem}\label{thm:roots}
Let $N$ be a compact, orientable, irreducible  $3$-manifold with empty or toroidal boundary. Let $g\in \pi=\pi_1(N)$.
\begin{itemize}
\item[(1)] If $g$ does not have trivial root structure,
 then there exists a Seifert fibered JSJ component $M$ of $N$ and $h\in \pi$ such that
$g$ lies in $h\pi_1(M)h^{-1}$ and
\[ \roots_\pi(g)=h\roots_{h^{-1}gh}(\pi_1(M))\,h^{-1}.\]
\item[(2)] If $N$ is  Seifert fibered  and if $g\in \pi_1(N)$ does not have nearly trivial root structure, then
$hgh^{-1}$ lies in a Seifert fiber group of $N$.
\item[(3)] If $N$ is  Seifert fibered  and if  $g\in \pi_1(N)$ lies in the Seifert fiber group, then  all roots of $hgh^{-1}$ are conjugate to an element represented by a power of a singular Seifert fiber of $N$.
\end{itemize}
\end{theorem}

If the Seifert fibered manifold $N$ does not contain any embedded Klein bottles, then by  \cite[Addendum~II.4.14]{JS79} we get the following strengthening of conclusion~(2):  either $g\in \pi_1(N)$ has trivial root structure, or $g$ is conjugate to an element in a Seifert fiber group of $N$.

\begin{proof}
Note that the roots of $g$ necessarily lie in $C_\pi(g)$. The theorem now follows immediately from  Theorem~\ref{thm:centgen} and from \cite[Proposition~II.4.13]{JS79}.
\end{proof}

\begin{remark}
Let $N$ be a $3$-manifold. Kropholler \cite[Proposition~1]{Kr90a} (see also \cite{Ja75} and \cite{Shn01})  showed, without using the Geometrization Theorem, that if $x\in \pi_1(N)$  is an element of infinite order such that  $x^n$ is conjugate to $x^m$, then $m=\pm n$. This fact also follows immediately from Theorem~\ref{thm:roots}.
\end{remark}

Given a group $\pi$ we say that an element $g$ is \emph{divisible by an integer $n$} if
there exists an $h$ with $g=h^n$. \index{group!divisibility of an element by an integer}
We now obtain the following corollary to Theorem~\ref{thm:roots}.

\begin{corollary}
Let $N$ be a compact, orientable, irreducible, non-spherical $3$-manifold  with empty or toroidal boundary.
Then $\pi_1(N)$ does not contain elements which are infinitely divisible, i.e., divisible by infinitely many integers.
\end{corollary}

\begin{remark}
For Haken $3$-manifolds this result had been proved in \cite[Corollary~3.3]{EJ73}, \cite[p.~327]{Shn75} and \cite[p.~328]{Ja75}
(see also  \cite{Wan69} and \cite{Fe76a,Fe76b,Fe76c}).
\end{remark}

As we saw earlier, the fundamental group  of a non-spherical Seifert fibered manifold has a normal infinite cyclic subgroup, namely the Seifert fiber group. The following consequence of Theorem~\ref{thm:centgen} shows that the converse holds.

\begin{theorem}\label{thm:infcyclicnormal}
Let $N$ be a compact, orientable, irreducible $3$-manifold  with empty or toroidal boundary.
If $\pi_1(N)$ admits a  normal infinite cyclic subgroup, then $N$ is Seifert fibered.
\end{theorem}

\begin{remark}
\mbox{}

\bn
\item This theorem was proved before the Geometrization Theorem:
\bn
\item Casson--Jungreis \cite{CJ94} and Gabai \cite{Gab92}, extending work of Tukia \cite{Tuk88a,Tuk88b}, showed  that if $\pi$ is a word-hyperbolic group
(see Section \ref{section:wise} for the definition of word-hyperbolic) such that its boundary (see \cite{BrH99} for details) is homeomorphic to $S^1$, then $\pi$ acts properly discontinuously and cocompactly on $\H^2$ with finite kernel.
\item Mess \cite{Mes88} showed that this result on word-hyperbolic groups implies Theorem \ref{thm:infcyclicnormal}.
\en
We also refer to \cite{Neh60,Neh63a,BZ66,Wan67a,Wan68a}, \cite{GoH75} and \cite[Theorem~VI.24]{Ja80} for partial results, \cite[Corollary~0.5]{Bow04} and \cite[Theorem~1.4]{Mac12} for alternative proofs, and \cite[Theorem~1.3]{Mai03} and \cite[Theorem~1]{Whn92} for  extensions to orbifolds and to the non-orientable case.
\item
If  $N$ is a compact orientable $3$-manifold with non-empty boundary, then
by \cite[Lemma~II.4.8]{JS79} a more precise conclusion holds:
if $\pi_1(N)$ admits a  normal infinite cyclic subgroup $\G$, then $\G$ is the Seifert fiber group for some Seifert fibration of $N$.
\en
\end{remark}

\begin{proof}[Proof of Theorem~\ref{thm:infcyclicnormal}]
Suppose
$\pi=\pi_1(N)$ admits a  normal infinite cyclic subgroup $G$. Recall that $\aut{G}$ is canonically isomorphic to $\Z/2$.
The conjugation action of $\pi$ on $G$ defines a homomorphism $\varphi\colon \pi\to \aut{G}=\Z/2$. We write $\pi'=\ker(\varphi)$.
Clearly $\pi'=C_\pi(G)$. It follows immediately from
Theorem~\ref{thm:centgen}  that either $N$ is Seifert fibered, or $\pi'=\Z$ or $\pi'=\Z^2$. But the latter case also implies that $N$ is either a solid torus, an $I$-bundle over the torus or an $I$-bundle over the Klein bottle. In particular $N$ is again Seifert fibered.
\end{proof}

Given a group $\pi$ and an element $g\in \pi$, the set of conjugacy classes of $g$ is in a canonical bijection with the set $\pi/C_g(\pi)$. We thus easily obtain the following corollary to Theorem~\ref{thm:centgen}.

\begin{theorem}
Let $N$ be a compact, orientable, irreducible $3$-manifold  with empty or toroidal boundary.
 If $N$ is not a Seifert fibered manifold, then the number of conjugacy classes is infinite for any $g\in \pi_1(N)$.
\end{theorem}

This result was proved (in slightly greater generality) by de la Harpe--Pr\'eaux \cite[p.~563]{dlHP07} using different methods.
We  refer to \cite{dlHP07} for an application of this result to the von Neu\-mann algebra $W_\l^*(\pi_1(N))$.

\medskip

The following was  shown by Hempel~\cite[p.~390]{Hem87} (generalizing work of Noga \cite{No67}) without using the Geometrization Theorem.

\begin{theorem}
Let $N$ be a compact, orientable, irreducible $3$-manifold  with toroi\-dal boundary, and let $S$ be a JSJ torus or a boundary component.
Then $\pi_1(S)$ is a maximal abelian subgroup of $\pi_1(N)$.
\end{theorem}

\begin{proof}
The result is well known to hold for Seifert fibered manifolds. The general case follows immediately from  Theorem~\ref{thm:centgen}.
\end{proof}

Recall that a subgroup $A$ of a group $\pi$ is called \emph{malnormal} if $A\cap gAg^{-1}=1$ for all $g\in \pi\setminus A$. \index{subgroup!malnormal}

\begin{theorem}\label{thm:malnormal}
Let $N$ be a compact, orientable, irreducible $3$-manifold  with empty or toroi\-dal boundary.
\begin{itemize}
\item[(1)]  Let $S$ be a boundary component. If the JSJ component which contains $S$ is hyperbolic, then $\pi_1(S)$ is a malnormal subgroup of $\pi_1(N)$.
\item[(2)] Let $T$ be a JSJ torus. If both of the JSJ components abutting $T$ are hyperbolic, then $\pi_1(T)$ is a malnormal subgroup of $\pi_1(N)$.
\end{itemize}
\end{theorem}

The first statement was proved by de la Harpe--Weber \cite[Theorem~3]{dlHW11} and can be viewed
as a strengthening of the previous theorem. We refer to \cite[Theorem~4.3]{Fri11} for an alternative proof.
The second statement can be proved using the same techniques.

\medskip

The following theorem was first proved by Epstein \cite{Ep61d,Ep62}:

\begin{theorem}
Let $N$ be a compact, orientable, irreducible $3$-manifold  with empty or toroi\-dal boundary.
If $\pi_1(N)\cong A\times B$ is isomorphic to a direct product of two non-trivial groups, then $N=S^1\times \Sigma$ with $\Sigma$ a surface.
\end{theorem}

\begin{proof}
If $\pi_1(N)\cong A\times B$ is isomorphic to a direct product of two non-trivial groups, then any element in $A$ and in $B$ has a non-trivial centralizer,
it follows easily from Theorem \ref{thm:centgen} that $N$ is a Seifert fibered space.
The case of a Seifert fibered space then follows from an elementary argument.
\end{proof}

\medskip

Given a group $\pi$ we define an \emph{ascending sequence of centralizers of length $m$} to be a sequence of subgroups of the form:
\[ C_{\pi}(g_1) \varsubsetneq C_{\pi}(g_{2}) \varsubsetneq\dots \varsubsetneq C_{\pi}(g_{m}).\]
We define $m(\pi)$ to be the maximal length of an ascending sequence of centralizers.  Note that if $m(\pi)<\infty$, then $\pi$ satisfies in particular property Max-c (maximal condition on centralizers); see \cite{Kr90a} for details. If $N$ is a compact, orientable, irreducible, non-spherical $3$-manifold  with empty or toroidal boundary, then it follows from Theorems~\ref{thm:centgen} and~\ref{thm:centsfs} that $m(\pi_1(N))\leq 3$. It follows from \cite[Lemma~5]{Kr90a} that $m(\pi_1(N))\leq 16$ for any spherical $N$. It now follows from \cite[Lemma~4.2]{Kr90a}, combined with the basic facts of Sections~\ref{section:prelim1} and~\ref{section:prelim2} and some elementary arguments that  $m(\pi_1(N))\leq 17$ for any compact $3$-manifold. We refer to \cite{Kr90a} for an alternative proof of this fact which does not require the Geometrization Theorem. We also refer to \cite{Hil06} for a different approach.

\medskip

We finish this section by illustrating how the results discussed so far can be used to quickly determine all $3$-manifolds whose fundamental groups have a given interesting group-theoretic property. As an example we describe all $3$-manifold groups which are CA and CSA.
A group is said to be \emph{CA} (short for \emph{centralizer abelian}\/) if  the centralizer of any non-identity element is abelian. Equivalently, a group is CA if and only if the intersection of any two distinct  maximal abelian subgroups is trivial, if and only if ``commuting'' is an equivalence relation on the set of non-identity elements. For this reason, CA groups are also sometimes called ``commutative transitive groups'' (or CT groups, for short).

\begin{lemma}
Let $\pi$ be a CA group and $g\in\pi$, $g\neq 1$. If $C_\pi(g)$ is infinite cyclic, then $C_\pi(g)$ is self-normalizing.
\end{lemma}
\begin{proof}
Suppose $C:=C_\pi(g)$ is infinite cyclic. Let $x$ be a generator for $C$, and let $y\in \pi$ such that
$yC=Cy$. Then $yxy^{-1} = x^{\pm 1}$ and hence $y^2 x y^{-2} = x$. Thus $x$ commutes with $y^2$, and since $y^2$ commutes with $y$, we obtain that $x$ commutes with $y$. Hence $y$ commutes with $g$ and thus $y\in C_\pi(g)=C$.
\end{proof}

The class of CSA groups was introduced by  Myasnikov and Remeslennikov \cite{MyR96} as a natural (in the sense of first-order logic, universally axiomatizable) generalization of torsion-free word-hyperbolic groups (see Section~\ref{section:wise} below).
A group is said to be \emph{CSA}
(short for \emph{conjugately separated abelian}\/) if all of its maximal abelian subgroups are malnormal.
Alternatively, a group is CSA if and only if the centralizer of every non-identity element is abelian and self-normalizing. (As a consequence, every subgroup of a CSA group is again CSA.)
It is easy to see that CSA~$\Rightarrow$~CA. There are CA groups which are not CSA, e.g., the infinite dihedral group, see \cite[Remark~5]{MyR96}. But for $3$-manifold groups, we have:

\index{group!centralizer abelian (CA)}
\index{$3$-manifold group!centralizer abelian (CA)}
\index{group!conjugately separated abelian (CSA)}
\index{$3$-manifold group!conjugately separated abelian (CSA)}

\begin{corollary}\label{cor:csa}
Let $N$ be a compact, orientable, irreducible $3$-manifold with empty or toroidal boundary, and suppose $\pi=\pi_1(N)$ is non-abelian. Then the following are equivalent:
\begin{itemize}
\item[(1)] Every JSJ component of $N$ is hyperbolic.
\item[(2)] $\pi$ is CA.
\item[(3)] $\pi$ is CSA.
\end{itemize}
\end{corollary}

\begin{proof}
We only need to show (1)~$\Rightarrow$~(3) and (2)~$\Rightarrow$~(1).
Suppose all JSJ components of $N$ are hyperbolic. Then by Theorem~\ref{thm:centgen}, the centralizer $C_\pi(g)$ of each $g\neq 1$ in $\pi$ is abelian (so $\pi$ is CA). It remains to show that each such $C_\pi(g)$ is self-normalizing. If $C_\pi(g)$ is cyclic, then this follows from the preceding lemma, and if $C_\pi(g)$ is not cyclic by  Theorems~\ref{thm:centgen} and \ref{thm:malnormal}.
This shows (1)~$\Rightarrow$~(3).
The implication (2)~$\Rightarrow$~(1) follows easily from Theorem~\ref{thm:centsfs},~(1) and the fact that subgroups of CA groups are CA.
\end{proof}

\section{Consequences of the Geometrization Theorem}\label{section:diagramgeom}

\noindent
In Section \ref{section:prelim2} we argued that for most purposes it suffices to study the fundamental groups of
compact, orientable,  $3$-manifolds~$N$ with empty or toroidal  boundary.
The following theorem for such $3$-manifolds is an immediate consequence of the Prime Decomposition Theorem and the Geometric Decomposition Theorem
(Theorems  \ref{thm:prime} and \ref{thm:geom2}) and Table \ref{tablegeoms}.

\begin{theorem}
Let $N$ be a compact, orientable $3$-manifold with empty or toroidal boundary. Then $N$ admits a decomposition
\[ N\cong  S_1\#\dots \# S_k \,\,\# \,\,T_1\#\dots \# T_l \,\,\#\,\, N_1\#\dots \# N_m\]
as a connected sum of orientable prime $3$-manifolds, where:
\begin{itemize}
\item[(1)] $S_1,\dots,S_k$ are spherical;
\item[(2)] for any $i=1,\dots,l$ the manifold $T_i$ is either of the form $S^1\times S^2$, $S^1\times D^2$, $S^1\times S^1\times I$, or it equals the twisted $I$-bundle over the Klein bottle, or it admits a finite solvable cover which is a torus bundle; and
\item[(3)] $N_1,\dots,N_m$ are irreducible $3$-manifolds which are either hyperbolic, or finitely covered by an $S^1$-bundle over a surface $\Sigma$ with $\chi(\Sigma)<0$, or they have a non-trivial geometric decomposition.
\end{itemize}
\end{theorem}

Note that the decomposition in the previous theorem can also be stated in terms of fundamental groups:
\bn
\item $S_1,\dots,S_k$ are the prime components of $N$ with finite fundamental groups,
\item $T_1,\dots,T_l$ are the prime components of $N$ with infinite solvable fundamental groups,
\item $N_1,\dots,N_m$ are the prime components of $N$ with fundamental groups
which are neither finite nor solvable.
\en
Since the first two types of $3$-manifolds are well understood, we will henceforth, for the most part,
restrict ourselves  to  the study of fundamental groups of  compact, orientable, irreducible $3$-manifolds~$N$ with empty or toroidal  boundary and such that $\pi_1(N)$ is neither finite nor solvable.
(Note that this implies that the boundary is incompressible, since the only irreducible $3$-manifold with compressible, toroidal boundary is $S^1\times D^2$.)

\medskip

In this section we will now summarize, in Diagram~1, various results on $3$-manifold groups which do not rely on the work of Agol, Kahn--Markovic and Wise.  Many of these results do, however, rely on the Geometrization Theorem.

\medskip

We first give some of the  definitions which we will use in  Diagram~1.
The definitions are roughly in the order that they appear in the diagram.
\newcounter{itemcounterm}

\begin{list}
{{(A.\arabic{itemcounterm})}}
{\usecounter{itemcounterm}\leftmargin=2em}


\item A space $X$ is the Eilenberg--Mac~Lane space for a group $\pi$, written as $X=K(\pi,1)$, if $\pi_1(X)\cong \pi$ and if $\pi_i(X)=0$ for $i\geq 2$.
\item The deficiency of a finite presentation $\ll g_1,\dots,g_k \, |\, r_1,\dots,r_l\rr$ of a group is defined to be $k-l$. \index{group!deficiency}
The \emph{deficiency of a finitely presented group $\pi$} is defined to be the maximum over the deficiencies of all finite presentations of $\pi$.  Note that some authors use the negative of this quantity.
\item A finitely presented group is called \emph{coherent} if each of its finitely generated subgroups is also finitely presented. \index{group!coherent}
\item
The $L^2$-Betti numbers $b_i^{(2)}(X,\a)$, for a given space $X$ and a homomorphism $\a\colon \pi_1(X)\to \Gamma$,  were introduced by Atiyah \cite{At76}; we refer to \cite{Lu02} for details of the definition. If the group homomorphism is the identity map, then we just write $b_i^{(2)}(X)=b_i^{(2)}(X,\id)$.
\item\label{A.cofinallimit}
A  \emph{cofinal \textup{(}normal\textup{)} filtration} of a group $\pi$ is a nested sequence $\{\pi_i\}_{i\in \N}$ of finite-index (normal) subgroups of $\pi$
such that $\bigcap_{i\in \N} \pi_i=\{1\}$.
\index{group!cofinal filtration}
 Let $N$ be a $3$-manifold. A  \emph{cofinal \textup{(}regular\textup{)} tower} of $N$ is a sequence $\{\ti{N}_i\}_{i\in\N}$ of connected covers
 of $N$ such that $\{\pi_1(\ti{N}_i)\}_{i\in \N}$ is a cofinal (normal) filtration of $N$.
 \index{$3$-manifold!cofinal tower}
Let $R$ be an integral domain.
If the limit
\[ \lim_{i\to \infty} \frac{b_1(\ti{N}_i;R)}{[N:\ti{N}_i]} \]
exists for any cofinal regular tower $\{\ti{N}_i\}$ of  $N$, and if all the limits agree,
then we denote this unique limit by
\[ \lim_{\ti{N}} \frac{b_1(\ti{N};R)}{[N:\ti{N}]}. \]
\item  We denote by $\ol{\Q}$  the algebraic closure  of $\Q$.
\item The \emph{Frattini subgroup $\Phi(\pi)$} of a group $\pi$ is the intersection of all maximal subgroups of $\pi$.  If $\pi$ does not admit a maximal subgroup, then we define $\Phi(\pi)=\pi$. \index{subgroup!Frattini}
(By an elementary argument, the Frattini subgroup of $\pi$ also agrees with the intersection of all normal subgroups  of $\pi$ which are not strictly contained in a proper normal subgroup of $\pi$.)
\item Let $R$ be a (commutative) ring. We say that a group $\pi$ is \emph{linear over $R$} if there exists an embedding $\pi\to \gl(n,R)$ for some $n$. Note that in this case, $\pi$ also admits an embedding into
    $\sl(n+1,R)$. \index{group!linear over a ring}
\item Let $N$ be a $3$-manifold. By a surface in  $N$ we will always mean an orientable,  compact surface, properly embedded  in $N$. Note that if $N$ is orientable, then a surface in our sense will always be  two-sided.
Let $\Sigma$ be a surface in $N$. Then $\Sigma$ is called
\index{surface!in a $3$-manifold}  \index{surface!in a $3$-manifold!separating} \index{$3$-manifold!Haken} \index{$3$-manifold!sufficiently large} \index{surface!in a $3$-manifold!non-fiber} \index{surface!in a $3$-manifold!separable}
    \bn[(a)]
    \item \emph{separating} if $N\smallsetminus \Sigma$ is disconnected;
    \item a \emph{\textup{(}non-\textup{)} fiber surface} if it is incompressible, connected, and (not) the fiber of a surface bundle map $N\to S^1$;
   \item   \emph{separable} if $\Sigma$ is connected  and  $\pi_1(\Sigma)\leq \pi_1(N)$ is separable. (See (A.\ref{(A.separable)}) for the definition of a separable subgroup.)
       \en
The $3$-manifold $N$ is \emph{Haken} (or \emph{sufficiently large}) if $N$ is compact, orientable, irreducible, and has an embedded incompressible surface.
\item A group is \emph{large} if it contains a finite-index subgroup which admits an epimorphism onto a non-cyclic free group. \index{group!large} \index{$3$-manifold!homologically large} \index{homology!non-peripheral}
\item Given $k\in \N$ we refer to
\[ \mbox{coker}\big\{H_1(\partial N;\Z)\to H_1(N;\Z)\big\}\]
as the \emph{non-peripheral homology of $N$}. Note that if $N$ has non-peripheral homology of rank $k$, then any finite cover of $N$ has non-peripheral homology of rank at least $k$.
We say that a $3$-manifold $N$ is \emph{homologically large}
if given any $k\in \N$ there exists a finite regular cover $N'$ of $N$ which has non-peripheral homology of rank at least $k$.
\item Given a group $\pi$ and an integral domain $R$ with quotient field $Q$ we write $vb_1(\pi;R)=\infty$ if for any $k$ there exists a finite-index (not necessarily normal) subgroup  $\pi'$ of $\pi$ such that \[ \mbox{rank}_R(H_1(\pi';R)):=\dim_Q(H_1(\pi';Q))\geq k.\]
In that case we say that $\pi$ has {\it infinite virtual first $R$-Betti number.}\/ \index{group!with infinite virtual first $R$-Betti number}
Given a $3$-manifold $N$ we write $vb_1(N;R)=\infty$ if $vb_1(\pi_1(N);R)=\infty$.
We will sometimes write $vb_1(N)=vb_1(N;\Z)$.

If $N$ is an irreducible, non-spherical, compact $3$-manifold with empty or toroidal boundary such that
$vb_1(\pi;R)=\infty$, then for any
$k$ there  exists also a finite-index \emph{normal} subgroup  $\pi'$ of $\pi$   such that $ \mbox{rank}_R(H_1(\pi';R))\geq k$.
Indeed, if $\charc(R)=0$, then this follows from elementary group-theoretic arguments, and if $\charc(R)\ne 0$, from \cite[Theorem~5.1]{Lac09}, since the Euler characteristic of $N=K(\pi_1(N),1)$ is zero.
(Here we used that our assumptions on $N$ imply that $N=K(\pi_1(N),1)$---see (C.\ref{bkpi1}).)
\item A group is called \emph{indicable} if  it admits an epimorphism onto $\Z$.
 A group is called \emph{locally indicable} if each of its finitely generated subgroups is indicable.\index{group!indicable}\index{group!locally indicable}
\item A group $\pi$ is called \emph{left-orderable} if it admits a strict total ordering ``$<$'' which is left-invariant, i.e., it has the property that if $g,h,k\in \pi$ with $g<h$, then $kg<kh$. A group  is called \emph{bi-orderable} if it admits a strict total ordering which is left- and right-invariant. \index{group!left-orderable}  \index{group!bi-orderable}
\item Given a property $\PP$ of groups we say that a group $\pi$ is \emph{virtually $\PP$} if $\pi$ admits a finite-index subgroup (not necessarily normal) which satisfies $\PP$. \index{group!virtually $\PP$}
\item Given a class $\PP$ of groups we say that a group $\pi$ is \emph{residually $\PP$} if given any non-trivial $g\in\pi$ there exists a surjective homomorphism $\a\colon \pi\to G$ onto a group $G\in \PP$ and such that $\a(g)$ is non-trivial.   A case of particular importance is when $\PP$ is the class of finite groups, in which case $\pi$ is said to be \emph{residually finite}.  Another important case is when $\PP$ is the class of finite $p$-groups for $p$ a prime (that is, the class of groups of $p$-power order), in which case  $\pi$ is said to be \emph{residually $p$}. \index{group!residually $\PP$} \index{group!residually finite} \index{group!residually $p$}
\item The \emph{profinite topology} on a group $\pi$ is the coarsest topology with respect to which every homomorphism from $\pi$ to a  finite group, equipped with the discrete topology, is continuous.  Note that $\pi$ is residually finite if and only if the profinite topology on $\pi$ is Hausdorff.  Similarly, the {\it pro-$p$ topology}\/ on $\pi$ is the coarsest topology with respect to which every homomorphism from $\pi$ to a finite $p$-group, equipped with the discrete topology, is continuous. \index{group!profinite topology on a group} \index{group!pro-$p$ topology on a group}
\item \label{(A.separable)} Let $\pi$ be a group. We say that a subset $S$ is \emph{separable} if $S$ is closed in the profinite topology on $\pi$; equivalently, for any $g\in \pi\smallsetminus S$, there exists a homomorphism $\a\colon\pi\to G$ to a finite group with $\a(g)\not\in \a(S)$.  The group $\pi$ is called \emph{locally extended residually finite \textup{(}LERF\textup{)}} (or \emph{subgroup separable}) if any finitely generated subgroup is separable, and \emph{$\pi$ is AERF} (or \emph{abelian subgroup separable}) if any finitely generated abelian subgroup of $\pi$ is separable.\index{group!separable subset of a group}\index{group!locally extended residually finite (LERF)}\index{group!subgroup separable}\index{group!abelian subgroup separable (AERF)}
\item Let $\pi$ be a group. We say that $\pi$ is \emph{double-coset separable} if given any two finitely generated subgroups $A,B\subseteq \pi$ and any $g\in \pi$, the subset $AgB\subseteq \pi$ is separable. Note that $AgB$ is separable if and only if $(g^{-1}Ag)B$ is separable, and therefore to prove double-coset separability it suffices to show that products of finitely generated subgroups are separable. \index{group!double-coset separable}
\item Let $\pi$ be a group and $\G$ a subgroup of $\pi$.  We say that \emph{$\pi$ induces the full profinite topology on $\Gamma$}
if the restriction of the profinite topology on $\pi$ to $\Gamma$ is the full profinite topology on $\Gamma$---equivalently, for any finite-index subgroup $\G'\subseteq \G$ there exists a finite-index subgroup $\pi'$ of $\pi$ such that $\pi'\cap \G\subseteq \G'$. \index{subgroup!induced topology is the full profinite topology}
\item Let $N$ be an orientable, irreducible $3$-manifold with empty or toroidal boundary. We will say that $N$ is \emph{efficient} if the graph of groups corresponding to the JSJ decomposition is efficient, i.e., if the following hold:
\bn
\item[(a)] $\pi_1(N)$ induces the full profinite topology on the fundamental groups of the JSJ tori and of the JSJ pieces; and
\item[(b)] the fundamental groups of the JSJ tori and the JSJ pieces, viewed as subgroups of $\pi_1(N)$, are separable.
\en
We refer to \cite{WZ10} for details.  \index{$3$-manifold!efficient}

\item Let $\pi$ be a finitely presentable group. We say that the \emph{word problem for $\pi$ is solvable} if given any finite presentation
for $\pi$ there exists an algorithm which can determine whether or not a given word in the generators is trivial.  Similarly, the  \emph{conjugacy problem for $\pi$ is solvable} if given any finite presentation
for $\pi$ there exists an algorithm to determine whether or not any two given words in the generators represent  conjugate elements of $\pi$.  We refer to  \cite[Section~D.1.1.9]{CZi93} for details. \index{group!word problem} \index{group!conjugacy problem} (Note that by \cite[Lemma~2.2]{Mila92} the word problem is solvable for one finite presentation if and only if it is solvable for every
finite presentation; similarly for the conjugacy problem.)
\item A group is called \emph{Hopfian} if it is not isomorphic to a proper quotient of itself.\index{group!Hopfian}
\end{list}

\noindent Before we move on to Diagram~1, we state a few conventions which we apply in the diagram.
\newcounter{itemcounteriig}
\begin{list}
{{(B.\arabic{itemcounteriig})}}
{\usecounter{itemcounteriig}\leftmargin=2em}
\item In Diagram~1, $N$ is a compact, orientable, irreducible $3$-manifold  such that
 its boundary consists of a (possibly empty) collection of tori. We furthermore assume throughout Diagram~1  that $\pi:=\pi_1(N)$ is neither solvable
 nor finite.
 Without these extra assumptions some of the implications do not hold.
For example, not every Seifert fibered manifold~$N$ admits a finite cover $N'$ with $b_1(N')>1$,  but this is the case if $\pi$ additionally is  neither solvable nor finite.
\item  Arrow (5) splits into three arrows, this means that precisely one of the three possible conclusion holds.
\item Red arrows indicate that the conclusion holds \emph{virtually}, e.g., if $N$ is a Seifert fibered space
such that $\pi_1(N)$ is neither finite nor solvable, then $N$ contains virtually an incompressible torus.
\item If a property $\PP$ of groups is written in green, then the following conclusion always holds:
If $N$ is a compact, orientable, irreducible $3$-manifold with empty or toroidal boundary such that the fundamental group of a finite (not necessarily regular) cover of $N$ is $\PP$, then $\pi_1(N)$ also is~$\PP$.
In most cases it is clear that the properties in Diagram~1 written in green satisfy this condition.
It follows from  Theorem~\ref{thm:jsjlift} that if
 $N'$ is a finite cover of a compact, orientable, irreducible $3$-manifold $N$ with empty or toroidal boundary, then
   $N'$ is hyperbolic (Seifert fibered, admits non-trivial JSJ decomposition) if and only if  $N$ has the same property.

\item Note that a concatenation of red and black arrows which leads  to a green property means that the initial group also has the green property.
\item An arrow with a condition next to it indicates that this conclusion only holds if this extra condition is satisfied.
\end{list}
Finally we give one last disclaimer: the diagram is  meant as a guide to the precise statements in the text and in the literature; it  should  not be used as a reference in its own right.


\begin{figure}
\psfrag{1}{\ref{C.1}}\psfrag{2}{\ref{C.2}}\psfrag{3}{\ref{C.3}}\psfrag{4}{\ref{C.4}}
\psfrag{5}{\ref{C.5}}\psfrag{6}{\ref{C.6}}\psfrag{7}{\ref{C.7}}\psfrag{8}{\ref{C.8}}
\psfrag{9}{\ref{C.9}}\psfrag{10}{\ref{C.10}}
\psfrag{11}{\ref{C.11}}\psfrag{12}{\ref{C.12}}\psfrag{13}{\ref{C.13}}\psfrag{14}{\ref{C.14}}
\psfrag{15}{\ref{C.15}}\psfrag{16}{\ref{C.16}}\psfrag{17}{\ref{C.17}}\psfrag{18}{\ref{C.18}}
\psfrag{19}{\ref{C.19}}\psfrag{20}{\ref{C.20}}\psfrag{21}{\ref{C.21}}\psfrag{22}{\ref{C.22}}
\psfrag{23}{\ref{C.23}}\psfrag{24}{\ref{C.24}}\psfrag{25}{\ref{C.25}}\psfrag{26}{\ref{C.26}}
\psfrag{27}{\ref{C.27}}\psfrag{28}{\ref{C.28}}\psfrag{29}{\ref{C.29}}\psfrag{30}{\ref{C.30}}
\psfrag{31}{\ref{C.31}}\psfrag{32}{\ref{C.32}}
\psfrag{33}{\ref{C.33}}
\hspace{-0.5cm}
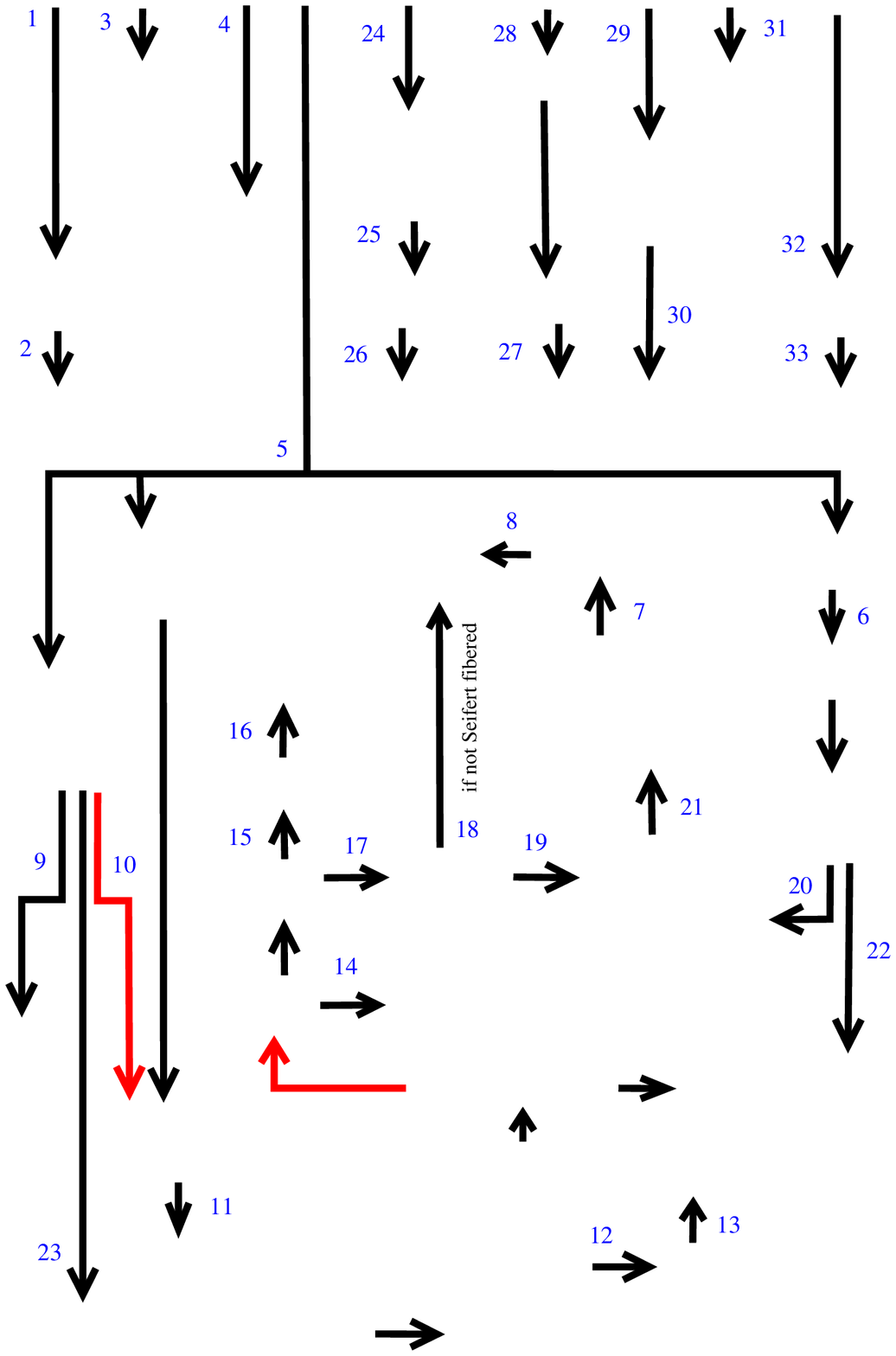
\end{figure}

\medskip

We now give the justifications for the implications of Diagram~1. In the subsequent discussion we strive for maximal generality; in particular, unless we say otherwise, we will only assume that $N$ is a connected $3$-manifold. We will give the required references and arguments for the general case, so each justification can be read independently of all the other steps.
We will also give further information and background material to put the statements in context.
\newcounter{itemcounterol}
\begin{list}
{{(C.\arabic{itemcounterol})}}
{\usecounter{itemcounterol}\leftmargin=2em}
\item \label{C.1}\label{bkpi1} \label{C.sphere}
Let $N$ be an irreducible,  orientable $3$-manifold with infinite fundamental group. It follows from the irreducibility of $N$ and the Sphere Theorem
(see Theorem~\ref{thm:sphere}) that   $\pi_2(N)=0$. Since $\pi_1(N)$ is  infinite,  it follows from the Hurewicz theorem that $\pi_i(N)=0$ for any $i>2$, i.e.,  $N$ is an Eilenberg--Mac~Lane space.
(This result was first proved by Aumann\footnote{Robert Aumann won a Nobel Memorial Prize in Economic Sciences in 2005.}
 \cite{Aum56} for the exterior of an alternating knot in $S^3$.)
If $N$ has non-trivial boundary, then $N$ admits a deformation retract to the 2-skeleton, i.e., $\pi_1(N)$ has a $2$-dimensional
Eilenberg--Mac~Lane space.

An argument as on \cite[p.~458]{FJR11} shows that
if $N$ is an irreducible $3$-manifold with non-trivial toroidal boundary and if $P$ is a presentation of deficiency one,
 then  the $2$-complex $X$ corresponding to $P$ is also an Eilenberg--Mac~Lane space for $\pi$.
 (This argument relies on the fact that $\pi_1(N)$ is locally indicable, see (C.\ref{C.15}).)  In fact $X$ is simple homotopy equivalent to $N$.

\item \label{C.2}\label{C.pi1torsionfree} Let $N$ be an orientable, irreducible $3$-manifold with infinite fundamental group.
By (C.\ref{bkpi1}) we have $N=K(\pi_1(N),1)$. Since the Eilenberg--Mac~Lane space is finite-dimensional it follows from standard arguments,
 see \cite[Proposition~2.45]{Hat02}) that $\pi_1(N)$ is torsion-free.
(Indeed, if $g\in \pi_1(N)$ is an element of finite order $k$, then consider $G=\ll g\rr \leq \pi_1(N)$ and denote by $\ti{N}$ the corresponding covering space of $N$.
Then the $3$-manifold $\ti{N}$ is an Eilenberg--Mac~Lane space for $\Z/k$, hence $H_*(\Z/k;\Z)=H_*(\ti{N};\Z)$, but the only cyclic group with finite homology is the trivial group.)

\item \label{C.3}Let $N$ be an irreducible $3$-manifold with empty or toroidal boundary and with infinite fundamental group. By
  (C.\ref{bkpi1}), $N$ is an Eilenberg--Mac~Lane space. It now follows from work of Epstein
  (see \cite[Lemmas~2.2~and~2.3]{Ep61a} and     \cite[Theorem~2.5]{Ep61a}) that the deficiency of $\pi_1(N)$ equals $1-b_3(N)$.
Hence $\pi_1(N)$ has a balanced presentation, i.e., a presentation of deficiency zero.
The question which groups with a balanced presentation are $3$-manifold groups is studied   in \cite{Neh68,Neh70,OsS74,OsS77a,OsS77b,Osb78,Sts75,Hog00}).
\item\label{C.4}\label{C.scottcore} Scott \cite{Sco73b} (see also \cite{Sco73a,Sco74,Sta77,RS90})
proved the Core Theorem, \index{theorems!Scott's Core Theorem}
which states  that if $N$ is any $3$-manifold  such that $\pi_1(N)$ is finitely generated,
then $N$ has a compact submanifold $M$ such that $\pi_1(M)\to \pi_1(N)$ is an isomorphism.
In particular, $\pi_1(N)$ is finitely presented. It now follows easily that the fundamental group of any $3$-manifold is coherent.
\item\label{C.5} The  Geometrization Theorem (see Theorem~\ref{thm:geom}) implies that any compact, orientable, irreducible $3$-manifold $N$
with empty or toroidal boundary satisfies one of the following:
\bn[(a)]
\item $N$ is Seifert fibered, or
\item $N$ is hyperbolic, or
\item $N$ admits an incompressible torus.
\en
\item \label{C.6}\label{C.sl2c} The fundamental group of an orientable hyperbolic $3$-manifold $N$ admits a faithful discrete representation $\pi_1(N)\to \operatorname{Isom}^+({\mathbb{H}}^3)$,
where $\operatorname{Isom}^+({\mathbb{H}}^3)$ denotes the group of orientation preserving isometries of $3$-dimensional hyperbolic space.
There is a well known identification of $\operatorname{Isom}^+({\mathbb{H}}^3)$ with $\psl(2,\C)$, which thus gives rise  to a faithful discrete representation $\pi_1(N)\to \psl(2,\C)$.  As a consequence of the Rigidity Theorem \ref{thm:mostow-prasad}, this representation is unique up to conjugation and complex conjugation. Another consequence of rigidity is that there exists in fact a faithful discrete representation $\rho\colon\pi_1(N)\to \psl(2,\ol{\Q})$  over the algebraic closure $\ol{\Q}$ of~$\Q$~\cite[Corollary~3.2.4]{MaR03}.
  Thurston (see \cite[Corollary~2.2]{Cu86} and \cite[Section~1.6]{Shn02}) showed that the representation $\rho$  lifts  to a faithful discrete representation $\ti{\rho}\colon  \pi_1(N)\to \sl(2,\ol{\Q})$.
The set of lifts of $\rho$  to a  representation $\pi_1(N)\to \sl(2,\ol{\Q})$
is in a natural one-to-one correspondence with the set of Spin-structures on $N$; see \cite[Section~2]{MFP11} for details.

If $T$ is a boundary component, then $\ti{\rho}(\pi_1(T))$ is a discrete subgroup isomorphic to $\Z^2$. It follows that, up to conjugation, we have
\[ \ti{\rho}(\pi_1(T)) \subseteq \left\{ \left.\bp \eps & a \\ 0&\eps \ep \, \right. : \, \eps\in \{-1,1\},\ a\in \ol{\Q}\right\}.\]
By \cite[Corollary~2.4]{Cal06} we have $\operatorname{tr}(\ti{\rho}(a))=-2$ if $a\in \pi_1(T)$ is represented by a curve on $T$ which cobounds a surface in $N$.

Button \cite{But12b} studied the question of which non-hyperbolic $3$-manifolds  admit (non-discrete) embeddings of their  fundamental groups  into $\sl(2,\C)$. For example, he showed that if $N$ is given by gluing two Figure-8 knot complements along their boundary,
then there are some gluings for which such an embedding exists, and there are some for which it doesn't.

\item \label{C.7}Long--Reid \cite[Theorem~1.2]{LoR98} showed that if a subgroup $\pi$ of $\sl(2,\ol{\Q})$
 is isomorphic to the fundamental group of a compact, orientable, non-spherical $3$-manifold, then $\pi$ is residually finite simple.
(Note that the assumption that $\pi$ is a non-spherical $3$-manifold group is necessary since  not all subgroups of $\sl(2,\ol{\Q})$ are residually simple.)
 \index{$3$-manifold group!residually finite simple}
Reading the proof of  \cite[Theorem~1.2]{LoR98} shows that under the same hypothesis as above a slightly stronger conclusion holds: $\pi$ is {\it fully residually simple,}\/ i.e., given    $1\neq g_1,\dots,g_k\in \pi$ there exists a morphism
$\a\colon \pi\to G$ onto a finite simple group with $\a(g_1),\dots,\a(g_k)\neq 1$.\index{$3$-manifold group!fully residually simple}

We refer to \cite{But11b} for more results on $3$-manifold groups (virtually) surjecting onto finite simple groups.
\item \label{C.8}It is easy to see that residually simple groups have trivial  Frattini subgroup.
\item \label{C.9}The fundamental group of a Seifert fibered manifold is well known to be linear over $\Z$.
 We provide a proof, suggested to us by Boyer, in Theorem~\ref{thm:Seifert linear}.
\item \label{C.10}Let $N$ be a Seifert fibered manifold. It follows from Lemma~\ref{lem:seifert} that  $N$ is  finitely covered by an $S^1$-bundle over a connected orientable surface $F$.
If  $\pi_1(N)$ is neither solvable nor finite, then $\chi(F)<0$. The surface $F$ thus admits an essential curve $c$, the $S^1$--bundle over~$c$ is
an incompressible torus.
\item \label{C.11}Let $N$ be a compact, orientable $3$-manifold which admits an incompressible torus $T$.
(Note that $T$ could be any incompressible boundary torus.)
By \cite[Theorem~2.1]{LoN91} (see also (C.\ref{C.AERF}) for a more general statement) the subgroup $\pi_1(T)\subseteq \pi_1(N)$ is separable,
i.e., $T$ is a separable surface.
If $\pi_1(N)$ is furthermore not  solvable, then the torus is not a fiber surface.
\item \label{C.12}\label{C.large}
Let $N$ be a compact, connected irreducible $3$-manifold with non-empty  incompressible boundary.
 Cooper--Long--Reid \cite[Theorem~1.3]{CLR97} (see also  \cite[Corollary~6]{But04} and \cite[Theorem~2.1]{Lac07a}) have shown that in that case either $N$ is covered by $S^1\times S^1\times I$ or  $\pi_1(N)$ is large.

     Now let $N$ be a closed $3$-manifold.
Let $\Sigma$ be a separable non-fiber surface, i.e., $\Sigma$ is a connected incompressible surface in $N$ which is not a fiber surface.
By Stallings' Fibration Theorem \cite{Sta62} (see also (K.\ref{K.stallings62}) and \cite[Theorem~10.5]{Hem76}) \index{theorems!Stallings' Theorem}
there exists a $g\in \pi_1(N\smallsetminus \nu \Sigma)\smallsetminus \pi_1(\Sigma)$. Since $\pi_1(\Sigma)$ is separable by assumption, we can separate $g$ from $\pi_1(\Sigma)$.
A standard argument shows that in the corresponding  finite cover $N'$ of $N$  the preimage of $\Sigma$ consists of at least two, non null-homologous and non-homologous orientable surfaces. Any two such surfaces give rise to an epimorphism from $\pi_1(N')$ onto a free group with two generators.
    (See also \cite[Proof~of~Theorem~3.2.4]{LoR05}.)

The \emph{co-rank $c(\pi)$} \index{group!co-rank} of a group $\pi$ is defined as the maximal integer $c$ such that
$\pi$ admits an epimorphism onto the free group on $c$ generators. If $N$ is a $3$-manifold, then
$c(N):=c(\pi_1(N))$  equals  the `cut number of $N$'; we refer to \cite[Proposition~1.1]{Har02} for details.
Jaco \cite[Theorem~2.3]{Ja72} showed that the cut number is additive under connected sum.
Clearly $b_1(N)\geq c(N)$. On the other hand, Harvey \cite[Corollary~3.2]{Har02} (see also \cite{LRe02} and \cite{Sik05}) showed that, given any $b$, there exists a closed hyperbolic $3$-manifold $N$ with $b_1(N)\geq b$ and $c(N)=1$. See \cite[Theorem~15.1]{GiM07} and \cite[Theorem~1.5]{Gil09} for upper bounds on $c(N)$ in terms of quantum invariants.
\item\label{C.13}\label{C.homlarge} Let $N$ be a compact $3$-manifold with trivial or toroidal boundary. Let $\varphi\colon \pi_1(N)\to F$ be a morphism onto a non-cyclic free group.
Then $\pi_1(N)$ is homologically large, i.e.,  given any $k\in \N$ there is a finite cover $N'$ of $N$ with
\[\mbox{rank}_\Z\, \mbox{coker}\{H_1(\partial N';\Z)\to H_1(N';\Z)\}\geq k.\]
Indeed, denote by $S_1,\dots,S_m$ (respectively $T_1,\dots,T_n$) the boundary components of $N$ which have the property that $\varphi$ restricted to the boundary torus is trivial (respectively non-trivial). Note that the image of $\pi_1(T_i)\subseteq F$ is a non-trivial infinite cyclic group generated by some $a_i\in F$.
Given $k\in\N$,
we now pick a prime number $p$ with $p\geq 2n+k$. Since free groups are residually~$p$~\cite{Iw43,Neh61b}, we can take an epimorphism $\a\colon F\to P$ onto a $p$-group $P$ with $\a(a_i)\neq 1$ for $i=1,\dots,n$. Let $F'=\ker(\a)$ and  denote by $q\colon N'\to N$ the covering of $N$ corresponding to $\a\circ \varphi$. If $S'$ is any boundary component of $N'$ which covers one of the $S_i$, then $\pi_1(S')\to \pi_1(N')\to F'$ is the trivial map. Using this observation we now calculate that
\begin{align*}
& \hskip1.2em \mbox{rank}_\Z\, \mbox{coker}\{H_1(\partial N';\Z)\to H_1(N';\Z)\}\\
 &\geq  \mbox{rank}_\Z\, \mbox{coker}\{H_1(\partial N';\Z)\to H_1(F';\Z)\}\\
       &\geq   b_1(F')-\textstyle\sum\limits_{i=1}^n b_1(q^{-1}(T_i)) \geq   b_1(F')-2\textstyle\sum\limits_{i=1}^n b_0(q^{-1}(T_i))\\
    &\geq  |P|(b_1(F)-1)+1-2n\textstyle\frac{|P|}{p}\geq |P|-2n\textstyle\frac{|P|}{p}= \textstyle\frac{|P|}{p}(p-2n)\geq k.
\end{align*}
    (See also  \cite[Corollary~2.9]{CLR97} for a related argument.)
\item \label{C.14}\label{C.b1atleast2}
A straightforward Thurston-norm argument (see \cite{Thu86a} or \cite[Corollary~10.5.11]{CdC03}) shows that
if $N$ is a compact, orientable, irreducible $3$-manifold with empty or toroidal boundary and  $b_1(N)\geq 2$,
then either $N$ is a torus bundle
(in which case $\pi_1(N)$ is solvable),
or $N$ admits a  homologically essential non-fiber surface.
A surface is called \emph{homologically essential} if it represents a non-trivial homology class. \index{surface!homologically essential}
(This surface is necessarily  non-separating.)

Note that since $\Sigma$ is homologically essential it follows from standard arguments (e.g., using Stallings' Fibration Theorem, see \cite{Sta62} and (K.\ref{K.stallings62}))
 that  $\Sigma$ does in fact not lift to the fiber of a surface bundle in any finite cover.
\item \label{C.15}\label{C.locallyindicable} Howie \cite[Proof~of~Theorem~6.1]{How82} (see also \cite[Lemma~2]{HoS85})
 used Scott's Core Theorem (C.\ref{C.scottcore}) and the fact that a $3$-manifold with non-trivial non-spherical boundary has positive first Betti number
 to  show that if $N$ is a compact, orientable, irreducible $3$-manifold
and $\G\subseteq \pi_1(N)$ a finitely generated subgroup of infinite index, then $b_1(\G)\geq 1$.

Furthermore a standard transfer argument shows that if $G$ is a finite-index subgroup of a group $H$, then $b_1(G)\geq b_1(H)$.
Combining these two facts it follows that if $N$ is a compact, orientable, irreducible $3$-manifold with $b_1(N)\geq 1$,
then any finitely generated subgroup $\G$ of $\pi_1(N)$ has the property that $b_1(\G)\geq 1$, i.e., $\pi_1(N)$ is locally indicable. \index{$3$-manifold group!locally indicable}

\item \label{C.16}\label{C.left-orderable} Burns--Hale \cite[Corollary~2]{BHa72} have shown that a locally indicable group is left-order\-able.
Note that left-orderability is not a `green property', i.e., there exist compact $3$-manifolds with non-left-orderable fundamental groups which admit left-orderable  finite-index subgroups (see, e.g., \cite[Proposition~9.1]{BRW05}  and  \cite{DPT05}).

\item \label{C.17}\label{C.haken} Let $N\ne S^1\times D^2$ be a compact, orientable, irreducible $3$-manifold. If~$N$ has toroidal boundary, then each boundary component is incompressible and hence $N$ is Haken. If $N$ is closed and $b_1(N)\geq 1$,
then $H_2(N;\Z)$ is non-trivial. Let $\Sigma$ be an oriented surface representing a non-trivial element $\phi\in H_2(N;\Z)$.
Since $N$ is irreducible we can assume that $\Sigma$ has no spherical components
and that $\Sigma$ has no component which bounds a solid torus. Among all such surfaces we take a surface of maximal Euler characteristic. It now follows from
an extension of the Loop Theorem to embedded surfaces (see \cite[Corollary~3.1]{Sco74}
 and Theorem~\ref{thm:loop})
 that any component of such a surface is incompressible; thus, $N$ is Haken. (See also \cite[Lemma~6.6]{Hem76}.)

If $\phi$ is a fibered class (see (E.\ref{E.fibered}) for the definition), then the surface $\Sigma$ is unique up to isotopy \cite[Lemma~5.1]{EdL83}.  On the other hand, if $\phi$ is not a fibered class, then $\Sigma$ is at times unique up to isotopy (see \cite{Ly74a,Koi89,CtC93,HiS97,Kak05,GI06,Brt08,Ju08,Ban11}),
but in general it is not;
see, e.g., \cite{Scf67,Alf70,AS70,Ly74b,Ein76b,Ein77b,ScT88,Gus81,Kak91,Kak92,HJS13,Alt12} for examples and more precise statements.

Results of Dunfield--Thurston \cite[Corollary~8.5]{DnTb06}, Kowalski \cite[Section~6.2]{Kow08} (see also \cite[p.~4]{Sar12}) and Ma \cite[Corollary~1.2]{Ma12} suggest that a `generic' closed, orientable $3$-manifold $N$ is a rational homology sphere, i.e., that $b_1(N)=0$.

The work of Hatcher \cite{Hat82} together with \cite{CJR82,Men84}, \cite[Theorem~2(b)]{HaTh85},
\cite[Theorem~1.1]{FlH82}, \cite[Theorem~A]{Lop92} and \cite[Theorem~A]{Lop93}
shows that almost all Dehn surgeries on large classes of $3$-manifolds with toroidal boundary are non-Haken.
See also \cite{Thu79,Oe84,Ag03,BRT12} for more examples of non-Haken manifolds.
\index{$3$-manifold!examples which are non-Haken}
Jaco--Oertel \cite[Theorem~4.3]{JO84} (see also \cite{BCT12}) found an algorithm which can decide whether or not a given closed irreducible $3$-manifold is Haken.

Finally, let $N$ be a compact, orientable, irreducible $3$-manifold with non-trivial boundary.
We say that a surface $\Sigma$ in $N$ is \emph{essential} \index{surface!essential} if $\Sigma$ is incompressible and not isotopic to a boundary component.
If $H_2(N;\Z)\ne 0$, then it follows from the above that $N$ contains an essential closed surface.
On the other hand, if $H_2(N;\Z)=0$, then Culler--Shalen \cite[Theorem~1]{CuS84} showed that $N$ contains an essential, separating surface with non-trivial boundary. Furthermore, if $H_2(N;\Z)=0$, then in some cases $N$ will contain an incompressible, connected, non-boundary parallel surface (see \cite{Ly71,Sht85,Gus94,FiM99,FiM00,MQ05,Lib09,Ozb09})
and in some cases it will not (see \cite[Corollary~1.2]{GLi84}, \cite{HaTh85}, \cite[Corollary~4]{Oe84}, \cite{Lop93,Mad04,QW04,Ozb08,Ozb10}).

 \item\label{C.18}
Let $N$ be a compact, orientable, irreducible $3$-manifold with empty or toroidal boundary.
It follows from the work of Allenby--Boler--Evans--Moser--Tang
\cite[Theorems~2.9~and~4.7]{ABEMT79}  that
if $N$
is Haken and not a closed Seifert fibered manifold, then the Frattini group of $\pi_1(N)$ is trivial.
On the other hand, if $N$ is a closed Seifert fibered manifold
and $\pi_1(N)$ is infinite, then the Frattini group of $\pi_1(N)$ is a (possibly trivial) subgroup of the infinite cyclic subgroup generated by a regular Seifert fiber  (see  \cite[Lemma~4.6]{ABEMT79}).

\item \label{C.19}\label{C.em72free} Evans--Moser \cite[Corollary~4.10]{EvM72} showed that
if the fundamental group of an  irreducible Haken $3$-manifold is non-solvable, then  it contains a non-cyclic free group.
\item \label{C.20}\label{C.tits}
Tits \cite{Tit72}  showed that  a group which is linear over $\C$ is either virtually solvable or contains a non-cyclic free group;
this dichotomy  is commonly referred to as the Tits Alternative. \index{theorems!Tits Alternative}
(Recall  that as in Diagram~1 we assumed that $\pi$ is neither finite nor solvable, it follows from  Theorem~\ref{thm:virtsolv} that $\pi$ is not virtually  solvable.)

The combination of the above and of (C.\ref{C.em72free}) shows that the fundamental group of a compact $3$-manifold
with empty or toroidal boundary is either virtually solvable or contains a non-cyclic free group.
This dichotomy is a weak version of the Tits Alternative for $3$-manifold groups.  We refer to
(K.\ref{K.tits}) for a stronger version of this for $3$-manifold groups, and to \cite{Par92,ShW92,KZ07} for `pre-geometrization' results on the Tits Alternative.

Aoun \cite{Ao11} showed that `most' two generator subgroups of a group which is linear over $\C$ and not virtually solvable,  are in fact free.

\item \label{C.21}A group which contains a non-abelian free group is non-amenable.
Indeed,  it is well known that  any subgroup and any finite-index supergroup of an amenable group is also amenable. On the other hand, non-cyclic free groups are not amenable. \index{group!amenable}
In (I.\ref{I.weaklyamenable}) we will, in contrast, see that most $3$-manifold groups are weakly amenable.

\item \label{C.22}\label{C.Lubotzky alternative} A consequence of the Lubotzky Alternative
(cf.~\cite[Window~9,~Corollary~18]{LuSe03}) asserts that  a finitely generated group which is linear
over $\C$ either is virtually solvable or, for any prime $p$, has infinite virtual first $\F_p$-Betti number
(see also \cite[Theorem 1.3]{Lac09}  and \cite[Section~3]{Lac11}).

We refer to \cite[Example~5.7]{CE11},  \cite[Theorems 1.7~and~1.8]{Lac09}, \cite{ShW92,Walb09} and \cite[Section~4]{Lac11} for more on the growth of  $\F_p$-Betti numbers of finite covers of hyperbolic $3$-manifolds.
See \cite[Proposition~3]{Mes90} for a `pre-Perelman' result regarding the $\F_p$-homology of finite covers of $3$-manifolds.
\item \label{C.23}\label{C.sfssubgroupseparable}
Let $N$ be a Seifert fibered manifold. Niblo \cite[Corollary~5.1]{Nib92} showed  that $\pi_1(N)$ is
 double-coset separable. In particular $\pi_1(N)$ is LERF.
 It follows from work of Hall \cite[Theorem~5.1]{Hal49} that fundamental groups of Seifert fibered spaces with non-empty boundary
 are LERF. Scott \cite[Theorem~4.1]{Sco78},~\cite{Sco85} showed that fundamental groups of closed Seifert fibered spaces are LERF.
 We  refer to \cite{BBS84,Nib90,Tre90,Lop94,Git97,LoR05,Wil07,BaC12,Pat12} for alternative proofs and extensions of Scott's theorem.

\item \label{C.24}\label{C.AF} Let $N$ be any compact $3$-manifold. In \cite{AF10} it is shown that, for all but finitely many primes $p$, the group $\pi_1(N)$ is virtually residually $p$.

If $N$ is a graph manifold
(i.e., if all its JSJ components are Seifert fibered manifolds), then by \cite[Proposition~2]{AF10}  a slightly stronger statement holds:
for any prime $p$ the group    $\pi_1(N)$ is virtually residually $p$.
Also note that for hyperbolic $3$-manifolds, or more generally for $3$-manifolds~$N$ such that $\pi_1(N)$ is linear over $\C$, it follows from
\cite{Pla68} (see also \cite[Theorem~4.7]{Weh73}) that for  all but finitely many primes~$p$, the group $\pi_1(N)$ is virtually residually $p$.
\item \label{C.25}\label{C.resfinite} The well known argument in (H.\ref{H.resfinitegreen}) below can be used to show that a group which is virtually residually $p$ is also residually finite.
The residual finiteness of fundamental groups of compact  $3$-manifolds was first shown by Hempel \cite{Hem87} and Thurston \cite[Theorem~3.3]{Thu82a}.

Some pre-Geometrization results on the residual finiteness of fundamental groups of knot exteriors were obtained by
Mayland, Murasugi, and Stebe \cite{May72,May74,May75a,May75b,MMi76,Ste68}.

Residual finiteness of $\pi$ implies that if we equip $\pi$ with its profinite topology, then  $\pi$ is homeomorphic to the rationals.
We refer to \cite{ClS84} for details.

Given an oriented hyperbolic $3$-manifold $N$ the residual finiteness of $\pi_1(N)$ can be seen using congruence subgroups.
  By (C.\ref{C.sl2c}) there is an embedding $\pi_1(N)\hookrightarrow\sl(2,\overline{\Q})$ with discrete image. We write $\G:=\rho(\pi_1(N))$.
We say that $H\leq \G$ is a \emph{congruence subgroup of $\G$} if there exists a  ring $R$
 which is  obtained from the ring of
integers of a number field by inverting a finite number of elements and a maximal ideal $\mathfrak m$ of  $R$
such that
$\G\subseteq \sl(2,R)$  and
\[ \ker\{\G\to \sl(2,R)\to \sl(2,R/\mathfrak m)\}\leq H.\]
Congruence subgroups have finite index (see, e.g., \cite[Theorem~4.1]{Weh73}) and the intersection of all congruence subgroups is trivial (see, e.g., \cite{Mal40} and \cite[Theorem~4.3]{Weh73}). This implies that $\pi_1(N)\cong \G$ is residually finite.

\index{subgroup!congruence}

 Lubotzky \cite[p.~116]{Lub83} showed that in general
  not every finite-index subgroup of $\G$ is a congruence subgroup.
 We refer to \cite{Lub95,CLT09,Lac09} and \cite[Section~3]{Lac11} for further results.

The fact that $3$-manifold groups are residually finite together with the Loop Theorem shows in particular that a non-trivial knot
admits a finite index subgroup such that the quotient is not cyclic. Broaddus \cite{Brs05} (see also \cite{Kup11}) gives an explicit upper bound on the index of such a subgroup in terms of the crossing number of the knot.
\item \label{C.26} Mal'cev \cite{Mal40} (see also \cite[Theorem~VII]{Mal65}) showed that every finitely generated residually finite group is Hopfian.
Here are some related properties of a group $\pi$: one says that $\pi$ is
\bn
\item[(a)] \emph{co-Hopfian} if it is not isomorphic to any proper subgroup of itself;
\index{group!co-Hopfian}
\item[(b)] \emph{cofinitely Hopfian}\index{group!cofinitely Hopfian} if  every endomorphism of $\pi$ whose image is of finite index in $\pi$ is in fact an automorphism;
\item[(c)]  \emph{hyper-Hopfian} \index{group!hyper-Hopfian} if every homomorphism $\varphi\colon\pi\to \pi$ such that $\varphi(\pi)$ is normal in $\pi$ with $\pi/\varphi(\pi)$ cyclic, is in fact an automorphism.
\en
If $\Sigma$ is a surface then $\pi_1(S^1\times \Sigma)=\Z\times \pi_1(\Sigma)$ is neither co-Hopfian, nor cofinitely Hopfian nor hyper-Hopfian.
\bn
\item[(a)]
Let $N$ be a compact, orientable, irreducible $3$-manifold. Wang and Yu \cite[Theorem~8.7]{WY94} showed that, if $N$ is closed, then $\pi_1(N)$ is
co-Hopfian if and only if $N$ has no finite cover that is either a direct product $S^1\times \Sigma$ or a torus bundle over $S^1$.
Gonz\'alez-Acu\~na--Whitten \cite[Theorem~2.5]{GW92} showed that, if $N$  has non-trivial toroidal boundary, then $\pi_1(N)$ is co-Hopfian if and only if $\pi_1(N)\ne \Z^2$ and if  no non-trivial Seifert fibred piece of the JSJ decomposition of $N$ meets~$\partial N$.
\item[(b)]  Bridson, Groves, Hillman and Martin \cite[Theorems A~and~C]{BGHM10} showed that
fundamental groups of hyperbolic $3$-manifolds are cofinitely Hopfian and also
that if $K\subseteq S^3$ is not a torus knot, then $\pi_1(S^3\setminus \nu K)$ is cofinitely Hopfian.
\item[(c)] Silver \cite{Sil96} (see also \cite[Corollary~7.2]{BGHM10}) showed that,
if $K\subseteq S^3$ is not a torus knot, then $\pi_1(S^3\setminus \nu K)$ is hyper-Hopfian.
\en
We also refer to  \cite{GW87,Dam91,GW92,GW94,GLW94,WW94,WY99,PV00} for more details and related results.

\item \label{C.27}We refer to \cite[Theorem~IV.4.6]{LyS77} for a proof of the fact that finitely presented groups which are residually finite have solvable word problem.

In fact, a more precise statement can be made: the fundamental group of a compact $3$-manifold has an exponential Dehn function; see \cite{CEHLPT92} for details. Waldhausen~\cite{Wan68b} showed that the word problem for fundamental groups of $3$-manifolds which are virtually Haken is solvable.
\item \label{C.28}\label{C.AERF} E. Hamilton \cite{Hamb01} showed that the fundamental group of any compact, orientable $3$-manifold  is AERF. Earlier results are in   \cite[Theorem~2]{LoN91} and  \cite{AH99}.
\item \label{C.29}\label{C.conjugacyproblem}  The conjugacy problem has been solved for all $3$-manifolds with incompressible boundary by Pr\'eaux \cite{Pre05,Pre06},
 building on ideas of Sela \cite{Sel93}.

\item \label{C.30} Every algorithm solving the conjugacy problem in a given group,  applied to the conjugacy class of the identity,  also solves the word problem.
\item \label{C.31}\label{C.efficient} Wilton--Zalesskii \cite[Theorem~A]{WZ10} showed that closed orientable prime $3$-manifolds are efficient. If $N$ is a prime $3$-manifold with toroidal boundary, then we denote by $W$ the result of gluing exteriors of hyperbolic knots to the boundary components of $N$.
    It follows from Proposition~\ref{prop:charpair} that the JSJ tori of $W$ consist of the JSJ tori of $N$ and the boundary tori of $N$. Since $W$ is efficient by \cite[Theorem~A]{WZ10} it now follows  that $N$ is also efficient.
See also  \cite[Chapter~5]{AF10} for a discussion of the question whether closed orientable prime $3$-manifolds are, for all but finitely many primes $p$, virtually $p$-efficient. (Here $p$-efficiency is the natural analogue of efficiency for the pro-$p$-topology; cf.~\cite[Section~5.1]{AF10}.) \index{$3$-manifold!$p$-efficient}

\item \label{C.32}\label{C.l2betti} Lott--L\"uck \cite[Theorem~0.1]{LoL95} (see also \cite[Section~4.2]{Lu02}) showed that if~$N$ is a compact irreducible non-spherical $3$-manifold with empty or toroidal boundary, then $b_i^{(2)}(N)=b_i^{(2)}(\pi)=0$ for any $i$. We refer to these papers for the calculation of $L^2$-Betti numbers of any compact $3$-manifold.

\item \label{C.33}\label{C.l2bettilimit}   It follows from work of L\"uck   \cite[Theorem~0.1]{Lu94} that
given a topological space $X$ and a homomorphism $\pi_1(X)\to \G$ to a residually finite group $\G$,
 the $L^2$-Betti numbers $b_i^{(2)}(X,\pi_1(X)\to \G)$ can be viewed
    as a limit of ordinary Betti numbers of finite regular covers of $X$.
    Combining this result with (C.\ref{C.l2betti}) we see that  if $N$ is a compact irreducible non-spherical $3$-manifold with empty
    or toroidal boundary, then
\[  \lim_{\ti{N}} \frac{b_1(\ti{N};\Z)}{[N:\ti{N}]} =0.\]
See \cite[Theorem~0.1]{CrW03} for more information on the rate of convergence of the limit
and see \cite[Th\'eor\`eme~0.1]{ABBGNRS11} for a generalization for closed hyperbolic $3$-manifolds.
Note that the assumption that the finite covers are regular is necessary.
In fact Gir\~{a}o \cite{Gir10} (see proof of \cite[Theorem~3.1]{Gir10}) gives an
example of a hyperbolic $3$-manifold with non-trivial boundary together with a cofinal filtration of  $\{\pi_i\}_{i\in \N}$  of $\pi=\pi_1(N)$ such that
\[ \lim_{i\to \infty} \frac{ b_1(\pi_i)}{[\pi:\pi_i]}>0.\]
See \cite{BeG04} for further results on the limits of Betti numbers in finite irregular covers.

The study of the growth of various complexities of groups (e.g., first Betti number, rank, size of torsion homology, etc.)
in filtrations of $3$-manifold groups has garnered a lot of interest in recent years.
We refer to \cite{BD13,BE06,BV13,CD06,Gir10,Gir13,KiS12,Lu12,Rai12b,Sen11,Sen12} for more results.

\end{list}

\begin{remark}
In Diagram~1, statements (C.1)--(C.4) do not rely on the Geometrization Theorem.
Statement (C.5) is a variation on the Geometrization Theorem,
whereas statements (C.6)--(C.\ref{C.sfssubgroupseparable}) gain their relevance from the Geometrization Theorem.
The general statements (C.\ref{C.AF})--(C.\ref{C.l2bettilimit}) rely directly on the  Geometrization Theorem.
In particular the results of Hempel~\cite{Hem87} and Hamilton~\cite{Hamb01} were proved for $3$-manifolds `for which geometri\-zation works'; by the work of Perelman these results then hold in the above generality.
\end{remark}

There are a few arrows and results on $3$-manifold groups which can be proved using the Geometrization Theorem,
and which we left out of the diagrams:
\newcounter{itemcounterrs}
\begin{list}
{{(D.\arabic{itemcounterrs})}}
{\usecounter{itemcounterrs}\leftmargin=2em}
\item Let $N$ be a compact, orientable, irreducible $3$-manifold. Kojima~\cite[p.~744]{Koj87}  and Luecke~\cite[Theorem~1.1]{Lue88} first showed that if $N$ contains an incompressible, non-boundary parallel torus and  $\pi_1(N)$ is not  solvable,  then $vb_1(N;\Z)=\infty$. In spirit their proof is rather similar to the  steps we provide.
\item Let $N$ be a compact $3$-manifold with incompressible boundary and no spherical boundary components, which is not a product on a boundary component.
(i.e., there does not exist a component $\Sigma$ of $\partial N$ such that $N\cong S\times [0,1]$.)
 It follows from standard
arguments (e.g., boundary subgroup separability, see (L.\ref{L.ln91}) for details) that for any $k$ there exists
a finite cover $\ti{N}\to N$ such that~$\ti{N}$ has at least $k$ boundary components.
In particular a Poincar\'e duality argument immediately implies that  $vb_1(N;\Z)=\infty$.
\index{$3$-manifold group!with infinite virtual $\Z$-Betti number}

\item Wilton \cite[Corollary~2.10]{Wil08} determined the closed $3$-manifolds with residually free fundamental group. In particular, it is shown that if $N$ is an orientable, prime $3$-manifold with empty or toroidal boundary such that $\pi_1(N)$ is residually free, then $N$ is the product of a circle with a connected surface.\index{$3$-manifold group!residually free}
\item
Boyer--Rolfsen--Wiest  \cite[Corollary~1.6]{BRW05} showed that the fundamental groups of  Seifert fibered manifolds are virtually bi-orderable.
Perron--Rolfsen \cite[Theorem~1.1]{PR03} and \cite[Corollary~2.4]{PR06} have shown that fundamental groups of many fibered $3$-manifolds
(i.e., $3$-manifolds which fiber over $S^1$) are bi-orderable.
In the other direction, Smythe (see \cite[p.~228]{Neh74}) proved that the fundamental group of the trefoil complement is not bi-orderable.  Thus not all fundamental groups of fibered $3$-manifolds are bi-orderable.
We   refer to  Clay--Rolfsen \cite{CR12} for many more examples, including some fibered hyperbolic $3$-manifolds with non bi-orderable fundamental groups. \index{$3$-manifold group!bi-orderable} \index{$3$-manifold group!virtually bi-orderable}

\item The proof of \cite[Proposition~4.16]{AF10} shows that
the fundamental group of  a Seifert fibered manifold has a finite-index residually torsion-free nilpotent subgroup.
    By (G.\ref{G.gruenberg}) and (G.\ref{G.rhemtulla}) this gives an alternative proof that  fundamental groups of  Seifert fibered manifolds are virtually bi-orderable.
\item\label{D.nonorientable} If $N$ is a closed $3$-manifold which is not orientable, then a Poincar\'e duality argument shows that $b_1(N)\geq 1$,
see, e.g., \cite[Lemma~3.3]{BRW05} for a proof.
\item
Teichner \cite{Tei97} showed that if the lower central series of the fundamental group $\pi$ of a closed $3$-manifold stabilizes,
then the maximal nilpotent quotient of $\pi$ is the fundamental group of a closed $3$-manifold (and such groups were determined in \cite[Theorem~N]{Tho68}). \index{$3$-manifold group!lower central series}
The lower central series and nilpotent quotients of $3$-manifold groups were also  studied by Cochran--Orr \cite[Corollary~8.2]{CoO98},
Cha--Orr \cite[Theorem~1.3]{ChO12}, Freedman--Hain--Teichner \cite[Theorem~3]{FHT97}, Putinar \cite{Pu98} and Turaev \cite{Tur82}.

\item \label{D.Boyer} The fact that Seifert fibered manifolds admit a geometric structure can in most cases be used to give an alternative proof of the fact that their fundamental groups are linear over $\C$.
More precisely, if $N$ admits a geometry $X$, then $\pi_1(N)$ is a discrete subgroup
of $\Isom(X)$. By \cite{Boy} the isometry groups of the following geometries are subgroups of $\gl(4,\R)$: spherical geometry, $S^2\times \R$, Euclidean geometry, Nil, $\Sol$ and hyperbolic geometry.
Furthermore, the fundamental group of an $\H^2\times \R$--manifold is a subgroup of $\gl(5,\R)$.  On the other hand, the isometry group of the universal covering group of $\sl(2,\R)$ is not linear (see, e.g., \cite[p.~170]{Di77}).

Groups which are virtually polycyclic are linear over $\Z$ by the Auslander--Swan Theorem (see \cite{Swn67} and \cite[Theorem~2]{Aus67}) and (H.\ref{H.lineargreen}). This implies in particular
that fundamental groups of $\Sol$-manifolds are linear over $\Z$.

\item The \emph{Whitehead group $\operatorname{Wh}(\pi)$} of a group $\pi$ is defined as the quotient of $K_1(\Z[\pi])$ by $\pm \pi$.
Here $K_1(\Z[\pi])$ is the abelianization of $\lim_{n\to \infty} \mbox{GL}(n,\Z[\pi])$, i.e., it is the abelianization of  the direct limit of the general linear groups over $\Z[\pi]$.
We refer to \cite{Mil66} for details. \index{group!Whitehead group of a group}

  The Whitehead group of the fundamental group of a compact, orientable, non-spherical irreducible $3$-manifold is trivial.  This follows from the Geometrization Theorem together with the work of Farrell--Jones \cite[Corollary~1]{FJ86}, Waldhausen \cite[Theorem~17.5]{Wan78a}, Farrell--Hsiang \cite{FaH81} and Plotnick \cite{Plo80}. We also refer to \cite{FJ87} for extensions of this result.\index{$3$-manifold group!Whitehead group}

Using this fact, and building on work of Turaev \cite{Tur88}, Kreck and L\"uck \cite[Theorem~0.7]{KrL09}
showed that if $f\colon M\to N$ is an orientation preserving homotopy equivalence
between closed, oriented, connected $3$-manifolds and if $\pi_1(N)$ is torsion-free, then  $f$  is homotopic to a homeomorphism.

    Two homotopy equivalent manifolds $M$ and $M'$ are simple homotopy equivalent if $\Wh(\pi_1(M'))$ is trivial.
    It follows in particular that two compact, orientable, non-spherical irreducible $3$-manifolds which are homotopy equivalent are in fact simple homotopy equivalent.
    On the other hand, homotopy equivalent lens spaces are not necessarily simple homotopy equivalent.
    We refer to \cite{Mil66,Coh73,Rou11} and \cite[p.~119]{Ki97} for more details.

     Bartels--Farrell--L\"uck \cite{BFL11}, continuing earlier investigations by Roushon \cite{Rou08a,Rou08b}, showed that the fundamental group of any $3$-manifold  satisfies the Farrell--Jones Conjecture from algebraic $K$-theory.
      The Farrell--Jones Conjecture for $3$-manifold groups implies in particular the following (see, e.g., \cite[p.~4]{BFL11}), for each $3$-manifold $N$:
      \bn
      \item[(a)] an alternative proof that   $\Wh(\pi_1(N))$ is trivial if $\pi_1(N)$ is torsion-free.
      \item[(b)] if $\pi_1(N)$ is torsion-free, then $\pi_1(N)$ satisfies the Kaplansky Conjecture, i.e., the group ring $\Z[\pi_1(N)]$ has no non-trivial idempotents. \index{theorems!Farrell--Jones Conjecture} \index{theorems!Kaplansky Conjecture}
      \item[(c)] the Novikov Conjecture holds for $\pi_1(N)$.
     \en
Matthey--Oyono-Oyono--Pitsch \cite[Theorem~1.1]{MOP08} showed that the fundamental group of any orientable $3$-manifold satisfies the Baum--Connes
\index{theorems!Baum--Connes Conjecture}
Conjecture, which gives an alternative proof for the Novikov  and  Kaplansky Conjectures for $3$-manifold groups
(see \cite[Theorem~1.13]{MOP08}).

\item
We say that a group has Property $U$ if it contains uncountably many maximal subgroups of infinite index.
Margulis--Soifer \cite[Theorem~4]{MrS81} showed that every finitely generated group which is linear over $\C$ and not virtually solvable has Property $U$.
Using the fact that free groups are linear, one can use this result to show that in fact any large group also has Property~$U$.
Tracing through Diagram~1 now implies that the fundamental group of any compact, orientable, aspherical $3$-manifold $N$ with empty or toroidal boundary has Property $U$, unless $\pi_1(N)$ is solvable.
It follows from \cite[Corollary~1.2]{GSS10} that  any maximal subgroup of infinite index of the fundamental group of a hyperbolic $3$-manifold is in fact infinitely generated.
\index{group!with Property $U$}

\item Let $\pi=\pi_1(N)$ be the fundamental group of a closed $3$-manifold which is
also the fundamental group of a  K\"ahler manifold. By Gromov \cite{Grv89} the group $\pi$ is not the free product
of non-trivial groups, which implies that $N$ is a prime $3$-manifold.  Kotschick \cite[Theorem~4]{Kot12} showed that $vb_1(N)=0$. It now follows from (C.\ref{C.11}), (C.\ref{C.12}) and (C.\ref{C.13})
 that  $\pi$ is finite.
This  result was first obtained by Dimca--Suciu \cite{DiS09} and an alternative proof is given in  \cite[Theorem~2.26]{BMS12}.
We refer to \cite{CaT89,DPS11,FS12,BiM12,Kot13} for other approaches and extensions of these results. The question which  groups are at the same time fundamental groups of $3$-manifolds and of quasi-projective manifolds is discussed in \cite{FS12}.
\item Ruberman \cite[Theorem~2.4]{Rub01} compared the behavior of the Atiyah--Patodi--Singer $\eta$-invariant \cite{APS75a,APS75b} and the Chern--Simons invariants \cite{ChS74} under finite coverings to give an obstruction
to a group being a $3$-manifold group.
\item A group is called \emph{$k$-free} if every subgroup  generated by at most $k$ elements is a free group.
For an orientable, closed hyperbolic $3$-manifold $N$ such that $\pi_1(N)$ is $k$-free for $k=3,4$ or $5$
 the results of \cite[Theorem~9.6]{ACS10} and \cite{CuS08b,Guz12} give lower bounds on the volume of $N$.
 The growth of $k$-freeness in a filtration of an arithmetic $3$-manifold was studied in \cite{Bel12}.\index{group!$k$-free}
\item Milnor \cite[Corollary~1]{Mil57} gave restrictions on  finite groups which can act freely on an integral homology sphere
(see also \cite{MZ04,MZ06,Reni01,Zim02}). On the other hand, Cooper and Long \cite{CoL00} showed that for each finite group $G$ there is a rational homology sphere with a free $G$-action.
 Kojima \cite{Koj88} (see also \cite[Theorem~1.1]{BeL05}) showed that every finite group also appears as the full isometry group of a closed hyperbolic $3$-manifold.
\item Let $N$  be  a compact orientable $3$-manifold
with no spherical boundary components.  De la Harpe and Pr\'eaux \cite[Proposition~8]{dlHP11} showed that if $N$ is neither a Seifert manifold nor a $\Sol$-manifold,
then $\pi_1(N)$ is a `Powers group', which by \cite{Pow75} implies that $\pi_1(N)$ is $C^*$--simple.
Here a group  is called \emph{$C^*$-simple} if it is infinite and if its reduced $C^*$-algebra has
no non-trivial two-sided ideals. We refer to \cite{Dan96} for background.
\item
Let $N$ be a  closed, orientable, irreducible $3$-manifold
which has $k$ hyperbolic pieces in its JSJ-decomposition. Weidmann \cite[Theorem~2]{Wei02}
showed that the minimal number of generators of $\pi_1(N)$
is bounded below by $k+1$.
\item Let $N$ be a compact orientable $3$-manifold such that every loop in $N$ is freely homotopic to a loop in a boundary component.
Brin--Johannson--Scott \cite[Theorem~1.1]{BJS85} (see also \cite[\S~2]{MMt79} with $\rho=1$) showed that there exists a boundary component $F$ such that $\pi_1(F)\to \pi_1(N)$ is surjective.
\item
Let $\pi$  be a finitely generated group and let $S$ be a finite generating set of $\pi$. The exponential growth rate of $(\pi,S)$
is defined as
\[ \omega(\pi,S):=\lim_{k\to \infty}\sqrt[k]{\# \{\mbox{elements in $\pi$ with word length $\leq k$}\}},\]
where the word length is taken with respect to $S$.
 The uniform exponential growth rate of $\pi$ is defined as
 \[ \w(\pi):=\inf \{ \omega(\pi,S) :\text{$S$ finite generating set of $\pi$}\}.\]
It follows from work of Leeb \cite[Theorem~3.3]{Leb95}
and di Cerbo \cite[Theorem~2.1]{dCe09} that there exists a $C>1$ such that for any closed irreducible $3$-manifold which is not a graph manifold
 we have $\w(\pi_1(N))>C$. This result builds on and extends earlier work of Milnor \cite{Mil68}, Avez \cite{Av70},
  Besson--Courtois--Gallot \cite{BCG11}
  and
  Bucher--de la Harpe \cite{BdlH00}.
 \item  It is a classical fact that every closed $3$-manifold is the boundary of a smoooth 4-manifold (see, e.g., \cite[p.~277]{Rol90} for a proof).
 Hausmann \cite[p.~122]{Hau81} (see also  \cite{FR12}) showed that given any closed $3$-manifold $N$ there exists in fact a smooth 4-manifold $W$
 such that $\pi_1(N)\to \pi_1(W)$ is injective.
\end{list}

\section{The Work of Agol, Kahn--Markovic, and Wise}
\label{section:akmw}

\noindent
The Geometrization Theorem resolves the Poincar\'e Conjecture and, more generally, the classification of $3$-manifolds with finite fundamental group.  For $3$-manifolds with infinite fundamental group, the Geometrization Theorem can be viewed as asserting that the key
problem is to understand hyperbolic $3$-manifolds.

\medskip

In this section we first discuss the Tameness Theorem, proved independently by Agol \cite{Ag07} and by Calegari--Gabai \cite{CaG06}, which implies an essential dichotomy for finitely generated subgroups of hyperbolic $3$-manifolds. We then turn to the Virtually Compact Special Theorem of Agol \cite{Ag12}, Kahn--Markovic \cite{KM12} and Wise \cite{Wis12a}. This theorem, together with the Tameness Theorem and further work of Agol \cite{Ag08} and Haglund \cite{Hag08} and  many others, resolves many hitherto intractable questions about hyperbolic $3$-manifolds.

\subsection{The Tameness Theorem}\label{section:tameness}

Agol \cite{Ag07} and Calegari--Gabai  \cite[Theorem~0.4]{CaG06} independently proved  the following theorem in 2004, which was first conjectured by Marden \cite{Man74} in 1974:
\index{theorems!Tameness Theorem}

\begin{theorem}\textbf{\emph{(Tameness Theorem)}}
 Let $N$ be a hyperbolic $3$-manifold, not necessarily of finite volume.  If $\pi_1(N)$ is finitely generated, then $N$ is topologically tame, i.e., $N$ is homeomorphic to the interior of a compact $3$-manifold.
\end{theorem}

We refer to \cite{Cho06,Som06,Cay08,Gab09,Bow10,Man07} for further details regarding the statement and alternative approaches
to the proof. We especially refer to \cite[Section~6]{Cay08} for a detailed discussion of earlier results leading towards the proof of the Tameness Theorem.

\medskip

In the context of this survey,  the main application of the Tameness Theorem is the Subgroup Tameness Theorem below.
In order to formulate this theorem we need a few more definitions.
\bn
\item A \emph{surface group} is  the fundamental group of a closed, orientable surface of genus at least one.
\item  Let $\Gamma$ be a \emph{Kleinian group}, i.e., a discrete subgroup of $\psl(2,\C)$.   The subgroup $\Gamma$ is called \emph{geometrically finite} if $\Gamma$ acts cocompactly on the convex hull of its limit set; see, for instance, \cite[Chapter 3]{LoR05} for details.
 (Note that a geometrically finite Kleinian group is necessarily finitely generated.) Now let $N$ be an orientable hyperbolic $3$-manifold. We can identify $\pi_1(N)$ with a discrete subgroup of $\psl(2,\C)$ which is well defined up to conjugation (see \cite[Section~1.6]{Shn02} and (C.\ref{C.sl2c})).  We say that a subgroup $\G\subseteq \pi_1(N)\subseteq \psl(2,\C)$ is \emph{geometrically finite} if $\G\subseteq \psl(2,\C)$ is a geometrically finite Kleinian group. We refer to \cite{Bow93} for a discussion of various different equivalent definitions of `geometrically finite'.
 We say that a surface $\Sigma\subseteq N$ is \emph{geometrically finite} if $\Sigma$ is incompressible
 and if the subgroup $\pi_1(\Sigma)\subseteq \pi_1(N)$ is geometrically finite. \index{surface!in a $3$-manifold!geometrically finite}
\item We say that a  $3$-manifold $N$ is \emph{fibered} if $N$ admits the structure of a surface bundle over $S^1$. \index{$3$-manifold!fibered}
By a \emph{surface fiber} in a $3$-manifold $N$ we mean the fiber of a surface bundle $N\to S^1$.
We say that $\G\subseteq \pi_1(N)$ is a \emph{surface fiber subgroup} if there exists a surface fiber $\Sigma$ such that $\G=\pi_1(\Sigma)$.
We say $\G\subseteq \pi_1(N)$ is a \emph{virtual surface fiber subgroup} if $N$ admits a finite cover $N'\to N$ such that $\G\subseteq \pi_1(N')$ and   $\G$ is a surface fiber subgroup of $N'$.
\en

\index{surface group}
\index{surface group!quasi-Fuchsian}\index{Kleinian group!geometrically finite}
\index{subgroup!surface fiber} \index{subgroup!virtual surface fiber}

\noindent
We can now state the Subgroup Tameness Theorem, which follows from combining the Tameness Theorem  with Canary's Covering Theorem (see \cite[Section~4]{Cay94}, \cite{Cay96} and \cite[Corollary~8.1]{Cay08}): \index{theorems!Subgroup Tameness Theorem} \index{theorems!Canary's Covering Theorem}

\begin{theorem}\textbf{\emph{(Subgroup Tameness Theorem)}}\label{thm:subgroupdichotomy}
 Let $N$ be a hyperbolic $3$-manifold and let
$\G\subseteq \pi_1(N)$ be a finitely generated subgroup. Then either
\begin{itemize}
\item[(1)] $\G$ is a virtual surface fiber group, or
\item[(2)] $\G$ is  geometrically finite.
\end{itemize}
\end{theorem}

The importance of this theorem will become fully apparent in Sections~\ref{section:diagram} and~\ref{section:subgroups}.

\subsection{The Virtually Compact Special Theorem}\label{section:vcsthm}

In his landmark 1982 article~\cite{Thu82a}, Thurston posed twenty-four questions, which illustrated the limited understanding of hyperbolic $3$-manifolds at that point.  These questions guided research into hyperbolic $3$-manifolds in the following years.
Huge progress towards answering these questions has been made since.
For example, Perelman's proof of  the Geometrization Theorem answered Thurston's Question~1
and the proof by Agol and Calegari--Gabai of the Tameness Theorem answered Question~5.

By early 2012, all but five of Thurston's questions had been answered.
Of the open problems, Question~23 plays a special role:
Thurston conjectured that not all volumes of hyperbolic $3$-manifolds are rationally related.
This is a very difficult question which in nature is much closer to deep problems in number theory than to topology
or differential geometry.
We  list the  remaining four questions  (with the original numbering):

\begin{questions}\text{\emph{(Thurston, 1982)}}
\begin{itemize}
\item [(15)] Are fundamental groups of hyperbolic $3$-manifolds LERF?
\item [(16)] Is  every hyperbolic $3$-manifold virtually Haken?
\item [(17)] Does every hyperbolic  $3$-manifold have a finite-sheeted cover with
positive first Betti number?
\item[(18)] Is every hyperbolic $3$-manifold virtually fibered?
\end{itemize}
\end{questions}

(It is clear that a positive answer to Question 18 implies a positive answer to Question~17,
and in (C.\ref{C.haken}) we saw that a positive answer to Question 17 implies a positive answer to Question~16.)
There has been a tremendous effort to resolve these four questions over the last three decades.  (See Section~\ref{section:history} for an overview of previous results.)  Nonetheless, progress  has been slow for the better part of the period.  In fact opinions on Question 18 were split. Regarding this particular question,  Thurston  himself famously wrote  `this dubious-sounding question seems to have a definite chance for a positive answer' \cite[p.~380]{Thu82a}.

\medskip

A stunning burst of creativity during the years 2007--2012 has lead to the following theorem, which was proved by Agol \cite{Ag12}, Kahn--Markovic \cite{KM12} and Wise \cite{Wis12a}, with major contributions from  Agol--Groves--Manning \cite{Ag12}, Bergeron--Wise \cite{BeW12}, Haglund--Wise \cite{HaW08,HaW12}, Hsu--Wise \cite{HsW12} and Sageev \cite{Sag95,Sag97}.

\index{theorems!Virtually Compact Special Theorem}

\begin{theorem}\textbf{\emph{(Virtually Compact Special Theorem)}}\label{thm:akmw}
If $N$ is a hyperbolic $3$-manifold, then $\pi_1(N)$ is virtually compact special.
\end{theorem}

\begin{remarks}

\mbox{}

\bn
\item
We will give the definition of `virtually compact special' in Section~\ref{section:specialcc}.  In that section we will also state the theorem of Haglund and Wise (see Corollary~\ref{cor:HW compact special}) which gives an alternative formulation of the Virtually Compact Special Theorem in terms of subgroups of right-angled Artin groups.
\item
In the case that $N$ is closed and admits a geometrically finite surface, a proof was first given by Wise \cite[Theorem~14.1]{Wis12a}.
Wise also gave a proof in the case that $N$ has non-empty boundary (see Theorem~\ref{thm:wiseboundary}). Finally, for the case that $N$ is closed and does not admit a geometrically finite surface,  the decisive ingredients of the proof were given by the work of Kahn--Markovic \cite{KM12} and Agol \cite{Ag12}.   The latter builds heavily on the ideas and results of \cite{Wis12a}.  See Diagram~2 for further details.
\item Recall that, according to our conventions, a hyperbolic $3$-manifold is assumed to be of finite volume.  This agrees with the theorems stated by Agol \cite[Theorems 9.1 and 9.2]{Ag12} and Wise \cite[Theorem~14.29]{Wis12a}.  In fact, it follows that the fundamental group of any compact hyperbolic $3$-manifold with (possibly non-toroidal) incompressible boundary is virtually compact special.  This is well known to the experts, but as far as we are aware does not appear in the literature.  Thus, in Section~\ref{ss: Non-toroidal} below, we explain how to deduce the infinite-volume case from Wise's results.
\en
\end{remarks}

We will discuss the consequences of the Virtually Compact Special Theorem in detail in Section~\ref{section:diagram}, but as an \emph{amuse-bouche} we mention that it gives affirmative answers to Thurston's Questions~15--18.  More precisely, Theorem~\ref{thm:akmw} together with the Tameness Theorem, work of Haglund \cite{Hag08}, Haglund--Wise \cite{HaW08} and Agol \cite{Ag08} implies the following corollary.

\begin{corollary}\label{cor:akmw}
If $N$ is a hyperbolic $3$-manifold, then
\begin{itemize}
\item[(1)] $\pi_1(N)$ is LERF;
\item[(2)] $N$ is virtually Haken;
\item[(3)] $vb_1(N)=\infty$; and
\item[(4)] $N$ is virtually fibered.
\end{itemize}
\end{corollary}

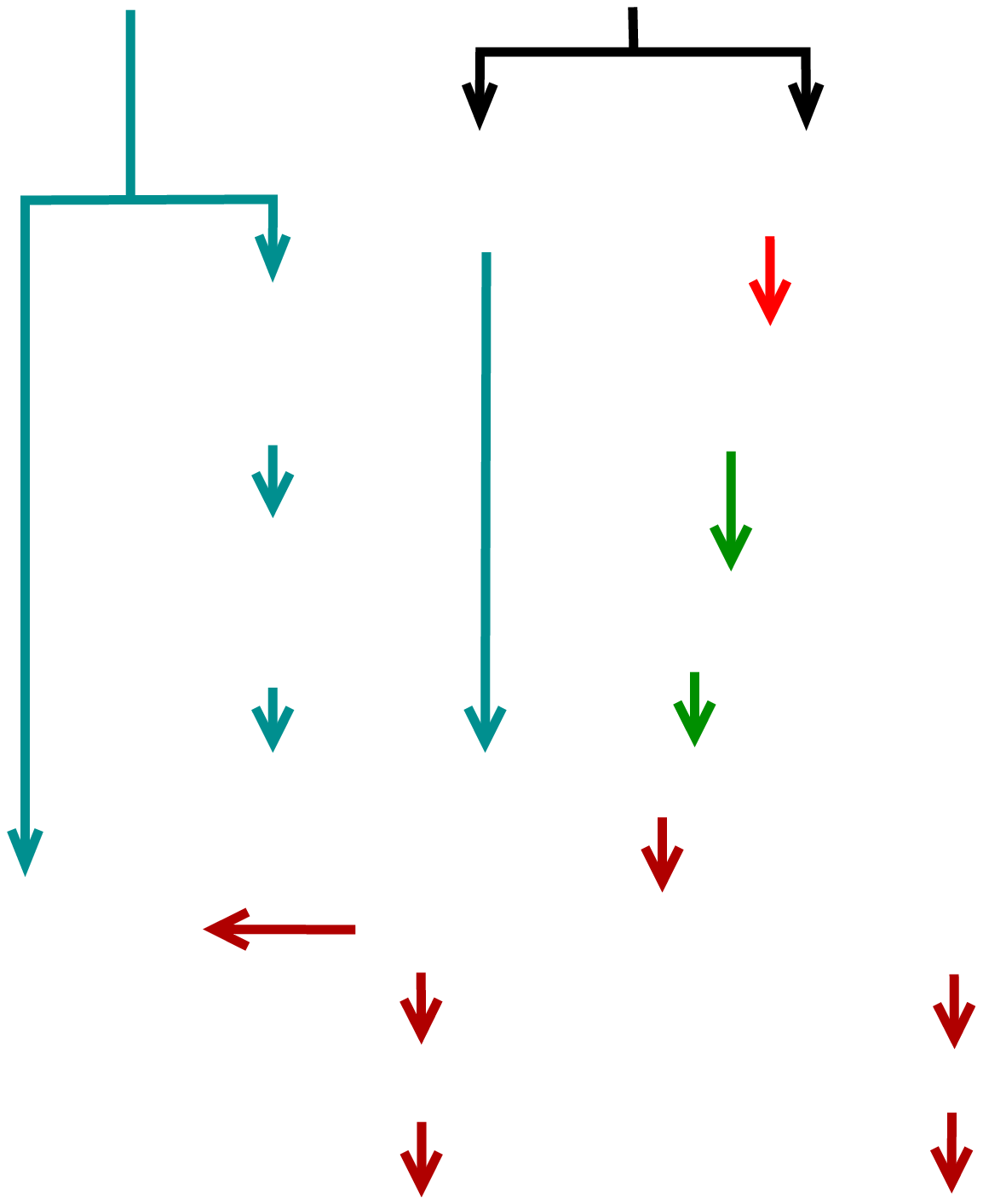

Diagram~2 summarizes the various contributions to the proof of Theorem~\ref{thm:akmw}.
The diagram can also  be viewed as a guide to the next sections.
More precisely we use the following color code.
\bn
\item Turquoise arrows correspond to Section~\ref{section:wise}.
\item The red arrow is treated in Section~\ref{section:km}.
\item The green arrows are covered in Section~\ref{section:agol}.
\item Finally, the brown arrows correspond to the consequences of Theorem~\ref{thm:akmw}. They are treated in detail in Section~\ref{section:corakmw}.
\en

\subsection{Special cube complexes}\label{section:specialcc}

The idea of applying non-positively curved cube complexes to the study of $3$-manifolds originated with the work of Sageev \cite{Sag95}.      Haglund and Wise's definition of a \emph{special} cube complex was a major step forward, and sparked the recent surge of activity \cite{HaW08}.
In this section, we give rough definitions that are designed to give a flavor of the material. The reader is referred to \cite{HaW08} for a precise treatment. For most applications, Corollary~\ref{cor: HW special} or Corollary~\ref{cor:HW compact special} can be taken as a definition.

\medskip

A \emph{cube complex} $X$ is a finite-dimensional cell complex in which each cell is a cube and the attaching maps are combinatorial isomorphisms.
 We also impose the condition, whose importance was brought to the fore by Gromov, that $ X$ should admit a locally $\op{CAT}(0)$ (i.e., non-positively curved) metric.   One of the attractions of cube complexes is that this condition can be phrased purely combinatorially.   Note that the  link of a vertex in a cube complex naturally has the structure of a simplicial complex. \index{cube complex}  \index{CAT(0)} \index{theorems!Gromov's Link Condition}

\begin{theorem}\textbf{\emph{(Gromov's Link Condition)}}\label{thm:gromov}
A cube complex $X$ admits a non-positively curved metric if and only if the link of each vertex is flag.  Recall that a simplicial cube complex is flag if every subcomplex $Y$ that is isomorphic to the boundary of an $n$-simplex \textup{(}for $n\geq 2$\textup{)} is the boundary of an $n$-simplex in $X$.
\end{theorem}

This theorem is due originally to Gromov \cite{Grv87}.  See also \cite[Theorem~II.5.20]{BrH99}
 for a proof, as well as many more details about CAT(0) metric spaces and cube complexes.  The next definition is due to Salvetti \cite{Sal87}.

\begin{example}[Salvetti complexes]
Let $\Sigma$ be any (finite) graph.  We build a cube complex $ S_\Sigma$ as follows:
\begin{enumerate}
\item $S_\Sigma$ has a single $0$-cell $x_0$;
\item $S_\Sigma$ has one (oriented) $1$-cell $e_v$ for each vertex $ v$ of $ \Sigma$;
\item $S_\Sigma$ has a square $2$-cell with boundary reading $e_ue_v\bar{e}_u\bar{e}_v$ whenever $u$ and $v$ are joined by an edge in $\Sigma$;
\item for $n>2$, the $n$-skeleton is defined inductively---attach an $n$-cube to any subcomplex isomorphic to the boundary of $n$-cube which does not already bound an $n$-cube.
\end{enumerate}
It is an easy exercise to check that $S_\Sigma$ satisfies Gromov's Link Condition and hence is non-positively curved. \index{cube complex!Salvetti complex}
\end{example}

\begin{definition}
The fundamental group of the Salvetti complex $ S_\Sigma$ is the \emph{right-angled Artin group \textup{(}RAAG\textup{)}} $A_\Sigma$.   Let $v_1,\dots,v_k$ be the distinct vertices of $\Sigma$. The corresponding RAAG is defined as
\[
A_\Sigma=\bigg\langle v_1,\dots,v_k : \text{$[v_i,v_j]=1$ if $v_i$ and $v_j$ are connected by an edge of $\Sigma$}\bigg\rangle.
\]
 Note: the definition of $A_\Sigma$ specifies a certain generating set.
\end{definition}

Right-angled Artin groups were introduced by Baudisch \cite{Bah81} under the name semi-free groups, but they are also sometimes referred to as \emph{graph groups} or  \emph{free partially commutative groups}.  We refer to \cite{Cha07} for a very readable survey paper on RAAGs.

\index{group!right-angled Artin group (RAAG)} \index{group!graph group} \index{group!free partially commutative}

\medskip

Cube complexes have natural immersed codimension-one subcomplexes, called \emph{hyperplanes}.  If an $n$-cube $C$ in $X$ is identified with $[-1,1]^n$, then a hyperplane of~$C$ is any intersection of $C$ with a coordinate hyperplane of $ \mathbb{R}^n$.  We then glue together hyperplanes in adjacent cubes whenever they meet, to get the hyperplanes of $ \{Y_i\}$ of $ X$, which naturally immerse into $ X$.  Pulling back the cubes in which the cells of $ Y_i$ land defines an interval bundle $ N_i$ over $ Y_i$, which also has a natural immersion $\iota_i\colon N_i\to X$.  This interval bundle has a natural boundary~$\partial N_i$, which is a $2$-to-$1$ cover of $Y_i$, and we let $N_i^o=N_i\setminus \partial N_i$.

Henceforth, although it will sometimes be convenient to consider non-compact cube complexes, we will always assume that the cube complexes we consider have only finitely many hyperplanes.

Using this language, we can write down a short list of pathologies for hyperplanes in cube complexes.

\index{cube complex!hyperplane}
\index{cube complex!hyperplane!one-sided}
\index{cube complex!hyperplane!self-intersecting}
\index{cube complex!hyperplane!directly self-osculating}
\index{cube complex!hyperplane!inter-osculating pair}

\begin{enumerate}
\item A hyperplane $ Y_i$ is \emph{one-sided} if $ N_i\to Y_i$ is not a product bundle.  Otherwise it is \emph{two-sided}.
\item A hyperplane $ Y_i$ is \emph{self-intersecting} if $\iota_i: Y_i\to X$ is not an injection.
\item A hyperplane $Y_i$ is \emph{directly self-osculating} if there are distinct vertices $x,y$ in the same component of $\partial N_i$ such that $\iota_i(x)=\iota_i(y)$ but, for some small neighborhoods $B_\epsilon(x)$ and $B_\epsilon(y)$, the restriction of $\iota_i$ to $(B_\epsilon(x)\sqcup B_\epsilon(y))\cap N_i^o$ is an injection.
\item A distinct pair of hyperplanes $ Y_i,Y_j$ is \emph{inter-osculating} if they both intersect and osculate; that is, the map $ Y_i\sqcup Y_j\to X$ is not an embedding and there are vertices $x\in\partial N_i$ and $y\in\partial N_j$ such that $\iota_i(x)=\iota_j(y)$ but, for some small neighborhoods $B_\epsilon(x)$ and $B_\epsilon(y)$, the restriction of $\iota_i\sqcup \iota_j$ to $(B_\epsilon(x)\cap N_i^o)\sqcup (B_\epsilon(y)\cap N_j^o)$ is an injection.
\end{enumerate}

In Figure 1 we give a schematic illustration
of directly self-osculating and inter-osculating hyperplanes in a cube complex.

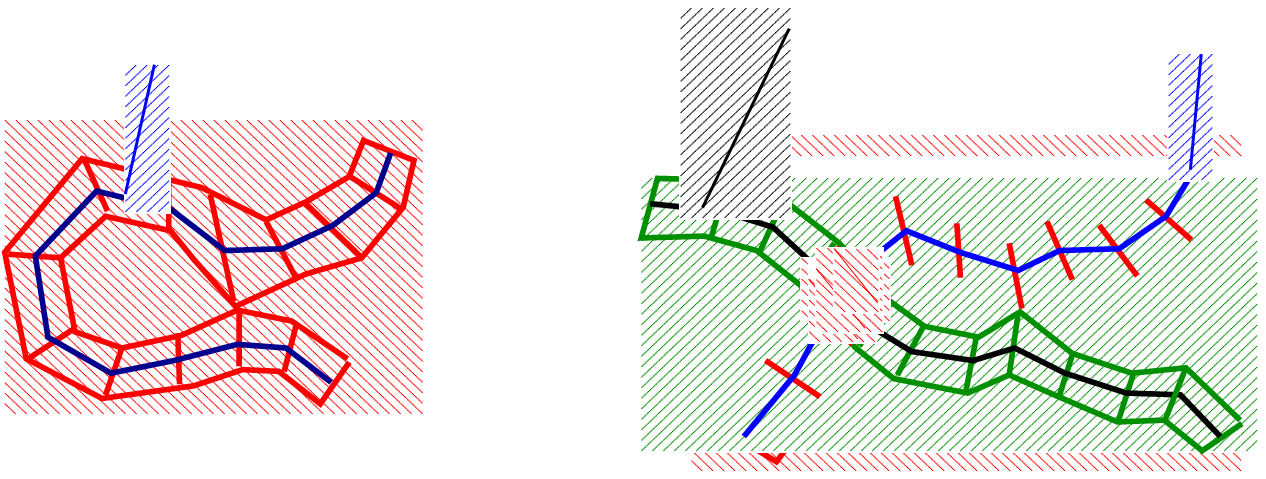

\begin{definition}[Haglund--Wise \cite{HaW08}]
A cube complex $ X$ is \emph{special} if none of the above pathologies occur.  (In fact, we have given the definition of \emph{A-special} from \cite{HaW08}.  Their definition of a \emph{special} cube complex is slightly less restrictive. However, these two definitions agree up to passing to finite covers, so the two notions of `virtually special' coincide.) \index{cube complex!special} \index{cube complex!hyperplane graph}
\end{definition}

\begin{definition}
The \emph{hyperplane graph} of a cube complex $X$ is the graph  $\Sigma(X)$ with vertex-set equal to the hyperplanes of $ X$, and with two vertices joined by an edge if and only if the corresponding hyperplanes intersect.
\end{definition}

If every hyperplane of $X$ is two-sided, then there is a natural \emph{typing map} $\phi_X\colon X\to S_{\Sigma(X)}$, which we now describe.  Each $0$-cell of $X$ maps to the unique $0$-cell $x_0$ of $S_{\Sigma(X)}$.  Each $1$-cell $e$ crosses a unique hyperplane $Y_e$ of $X$; $\phi_X$ maps $e$ to the $1$-cell $e_{Y_e}$ of $S_{\Sigma(X)}$ that corresponds to the hyperplane $Y_e$, and the two-sided-ness hypothesis ensures that orientations can be chosen consistently.  Finally, $\phi_X$ is defined inductively on higher dimensional cubes: a higher-dimensional cube $C$ is mapped to the unique cube of $S_{\Sigma(X)}$ with boundary $\phi_X(\partial C)$. \index{cube complex!typing map}

\medskip

The key observation of \cite{HaW08} is that pathologies (2)--(4) above correspond exactly to the failure of the map $\phi_X$ to be a local isometry.  We  sketch the argument.  For each $0$-cell $x$ of $X$, the typing map $\phi_X$ induces a map of links $\phi_{X*}\colon\mathrm{lk}(x)\to\mathrm{lk}(x_0)$.  This map $\phi_{X*}$ embeds $\mathrm{lk}(x)$ as an isometric subcomplex of $\mathrm{lk}(x_0)$.  Indeed, if $\phi_x$ identifies two $0$-cells of $\mathrm{lk}(x)$, then we have a self-intersection or a direct self-osculation; likewise, if there are $0$-cells $u$, $v$ of $\mathrm{lk}(x)$ that are not joined by an edge but $\phi_{X*}(u)$ and $\phi_{X*}(u)$ are joined by an edge in $\mathrm{lk}(x_0)$, then there is an inter-osculation.

\medskip

This is one direction of \cite[Theorem 4.2]{HaW08}:

\begin{theorem}\textbf{\emph{(Haglund--Wise)}}\label{thm:HW4.2}
A non-positively curved cube complex $X$ is special if and only if there is a
graph $\Sigma$ and a
local isometry $X\to S_\Sigma$.
\end{theorem}

The other direction of the theorem is a straightforward consequence of the results of \cite{Hag08}.

\medskip

Let $\Sigma$ be a graph and $\phi\colon X\to S_\Sigma$ be a local isometry.
Lifting the local isometry $\phi$ to universal covers,  we obtain a genuine isometric embedding of universal covers $\widetilde{X}\hookrightarrow \widetilde{S}_\Sigma$.
In particular, $\phi$ induces an injection $\phi_*\colon \pi_1(X)\to A_\Sigma$.
On the other hand, a covering space of a special cube complex is itself a special cube complex. Theorem~\ref{thm:HW4.2} therefore yields a characterization of subgroups of RAAGs. \index{group!special} \index{group!compact special}

\begin{definition}[Special group]
A group is called \emph{special} (respectively, \emph{compact special}) if it is the fundamental group of a non-positively curved special cube complex with finitely many hyperplanes (respectively,  a compact, non-positively curved special cube complex).
\end{definition}

\begin{corollary}\label{cor: HW special}
Every special group is a subgroup of a right-angled Artin group.  Conversely, every subgroup of a right-angled Artin group  is the fundamental group of a special cube complex $X$ \textup{(}although $X$ need not, in general, have finitely many hyperplanes\textup{)}.
\end{corollary}

\begin{proof}
Haglund--Wise \cite[Theorem~1.1]{HaW08} showed that if a group $\pi$ is  special, then $\pi$ admits a subgroup of finite index which is a subgroup of a RAAG.  Indeed, suppose that $\pi$ is the fundamental group of a special cube complex $X$.  Take a graph $\Sigma$ and a local isometry $\phi\colon X\to S_\Sigma$, by Theorem \ref{thm:HW4.2}.  The induced map on universal covers $\tilde{\phi}\colon\widetilde{X}\to\widetilde{S}_\Sigma$ is then an isometry onto a convex subcomplex of~$\widetilde{S}_\Sigma$~\cite[Lemma 2.11]{HaW08}.  It follows that $\phi_*$ is injective.

For the partial converse, if $\pi$ is a subgroup of a RAAG $A_\Sigma$, then $\pi$ is the fundamental group of a covering space $X$ of  $S_\Sigma$; the Salvetti complex $S_\Sigma$ is special and so, by \cite[Corollary 3.8]{HaW08}, is $X$.
\end{proof}

\medskip

Arbitrary subgroups of RAAGs may exhibit quite wild behavior.  However, if the cube complex $X$ is compact, then $\pi_1(X)$ turns out to be a quasi-convex subgroup of a RAAG, and hence much better behaved.

\begin{definition}
Let $X$ be a geodesic metric space. A subspace $Y$ of $X$ is said to be  \emph{quasi-convex} if there exists $\kappa\geq 0$ such that any geodesic in $X$ with endpoints in $Y$ is contained within the $\kappa$-neighborhood of $Y$. \index{subspace!quasi-convex}
\end{definition}

\begin{definition}
Let $\pi$ be a group with a fixed generating set $S$.  A subgroup of $\pi$ is said to be \emph{quasi-convex} (with respect to $S$) if it is a quasi-convex subspace of $\mathrm{Cay_S(\pi)}$, the Cayley graph of $\pi$ with respect to the generating set $S$. \index{subgroup!quasi-convex}
\end{definition}

Note that in general the notion of  quasi-convexity depends on the choice of generating set $S$.   Recall that the definition of a RAAG as given above specifies a generating set; we will always take this given choice of generating set when we talk about a quasi-convex subgroup of a RAAG.

\begin{corollary}\label{cor:HW compact special}
A group is compact special if and only if it is a quasi-convex subgroup of a right-angled Artin group.
\end{corollary}
\begin{proof}
Let $\pi$ be the fundamental group of a compact special cube complex~$X$.    Just as in the proof of Corollary \ref{cor: HW special}, there is a graph $\Sigma$ and a map of universal covers $\tilde{\phi}:\widetilde{X}\to \widetilde{S}_\Sigma$ that maps $\widetilde{X}$ isometrically onto a convex subcomplex of $\widetilde{S}_\Sigma$.   Because $\pi_1(\widehat{X})$ acts cocompactly on $\widetilde{X}$, it follows from \cite[Corollary 2.29]{Hag08} that $\phi_*\pi_1(\widehat{X})$ is a quasi-convex subgroup of $\pi_1(S_\Sigma)=A_\Sigma$.

For the converse let $\pi$ be a subgroup of a RAAG $A_\Sigma$. As in the proof of Corollary \ref{cor: HW special},  $\pi$ is the fundamental group of a covering space $X$ of the Salvetti complex $S_\Sigma$. By \cite[Corollary 2.29]{Hag08}, $\pi$ acts cocompactly on a convex subcomplex $\widetilde{Y}$ of the universal cover of~$S_\Sigma$. The quotient $Y=\widetilde{Y}/\pi$ is a locally convex, compact subcomplex of $X$ and so is special, by \cite[Corollary 3.9]{HaW08}.
\end{proof}

\subsection{Haken hyperbolic $3$-manifolds: Wise's Theorem}\label{section:wise}

In this subsection, we discuss Wise's proof that closed, Haken hyperbolic $3$-manifolds are virtually fibered.  The starting point for Wise's work is the following theorem of Bonahon~\cite{Bon86} and Thurston (see also \cite{CEG87,CEG06}), which is a special case of the Tameness Theorem.
See Section~\ref{section:tameness} for the definition of geometrically finite surfaces.

\begin{theorem}\textbf{\emph{(Bonahon--Thurston)}}\label{thm:bon86}
Let $N$ be a closed hyperbolic $3$-manifold and let $\Sigma\subseteq N$ be  an incompressible connected surface. Then either
\begin{itemize}
\item[(1)]  $\Sigma$ lifts to a surface fiber in a finite cover, or
\item[(2)] $\Sigma$ is  geometrically finite.
\end{itemize}
\end{theorem}

In particular, a closed hyperbolic Haken manifold is either virtually fibered or admits a geometrically finite surface.  Also, note that by the argument of (C.\ref{C.b1atleast2}) and by Theorem~\ref{thm:bon86}, any $3$-manifold  with $b_1(N)\geq 2$ admits a geometrically finite surface.

\medskip

Let $N$ be a closed, hyperbolic $3$-manifold that contains a  geometrically finite surface.   Thurston proved that $N$ in fact admits a hierarchy of geometrically finite surfaces (see \cite[Theorem~2.1]{Cay94}). In order to link up with Wise's results we  need to recast Thurston's result in the language of geometric group theory.

\begin{definition}
A group is called \emph{word-hyperbolic} if it acts properly discontinuously and cocompactly by isometries on a Gromov-hyperbolic space.  This notion was introduced by Gromov \cite{Grv81,Grv87}.  See \cite[Section III.$\Gamma$.2]{BrH99}, and the references therein, for details.
\end{definition}

When $\pi$ is word-hyperbolic, the quasi-convexity of a subgroup of $\pi$ does not depend on the choice of generating set \cite[Corollary III.$\Gamma$.3.6]{BrH99}, so we may speak unambiguously of a quasi-convex subgroup of a word-hyperbolic group. \index{group!word-hyperbolic}

\medskip

Next, we introduce the class $\mathcal{QH}$ of \emph{groups with a quasi-convex hierarchy}. \index{group!with a quasi-convex hierarchy}

\begin{definition}
The class $\mathcal{QH}$ is defined to be the smallest class of finitely generated groups that is closed under isomorphism and satisfies the following properties.
\begin{enumerate}
\item $1\in\mathcal{QH}$.
\item If $A,B\in \mathcal{QH}$ and the inclusion map $C\hookrightarrow A*_C B$ is a quasi-isometric embedding, then $A*_C B\in\mathcal{QH}$.
\item If $A\in \mathcal{QH}$ and the inclusion map $C\hookrightarrow A*_C $ is a quasi-isometric embedding, then $A*_C \in\mathcal{QH}$.
\end{enumerate}
\end{definition}

By, for instance, \cite[Corollary III.$\Gamma$.3.6]{BrH99}, a finitely generated subgroup of a word-hyperbolic group is quasi-isometrically embedded if and only if it is quasi-convex, which justifies the terminology.

\medskip

The next proposition now makes it possible to go from hyperbolic $3$-manifolds to the purely group-theoretic realm.

\begin{proposition}\label{prop:hypwordhyp}
Let $N$ be a closed hyperbolic $3$-manifold. Then
\begin{itemize}
\item[(1)] $\pi=\pi_1(N)$ is word-hyperbolic;
\item[(2)] a subgroup of $\pi$ is geometrically finite if and only if it is quasi-convex;
\item[(3)] if $N$ has a hierarchy of geometrically finite surfaces, then $\pi_1(N)\in\mathcal{QH}$.
\end{itemize}
\end{proposition}

\begin{proof}
For the first statement, note that $\H^3$ is Gromov-hyperbolic  and so the fundamental groups of closed hyperbolic manifolds are  word-hyperbolic (see \cite{BrH99} for details). We refer to \cite[Theorem 1.1 and Proposition 1.3]{Swp93} and also \cite[Theorem~2]{KaS96} for proofs of the second statement. The third statement follows from the second statement.
\end{proof}

We thus obtain the following reinterpretation of the aforementioned theorem of Thurston:

\begin{theorem}\textbf{\emph{(Thurston)}}
If $N$ is a closed, hyperbolic $3$-manifold containing a geometrically finite surface, then
$\pi_1(N)$ is   word-hyperbolic   and $\pi_1(N)\in\mathcal{QH}$.
\end{theorem}

The main theorem of \cite{Wis12a}, Theorem~13.3, concerns word-hyperbolic groups with a quasi-convex hierarchy. \index{theorems!Wise's Quasi-Convex Hierarchy Theorem}

\begin{theorem}\textbf{\emph{(Wise)}}\label{thm:wise}
Every word-hyperbolic group in $\mathcal{QH}$ is virtually compact special.
\end{theorem}

We immediately obtain the following corollary.

\begin{corollary}\label{cor:wisehyp}
If $N$ is a closed hyperbolic $3$-manifold that contains a geometrically finite surface,
then $\pi_1(N)$ is virtually compact special.
\end{corollary}

The proof of Theorem \ref{thm:wise} is beyond the scope of this article.  However, we will state two of the most important ingredients here.  Recall that a subgroup~$H$ of a group $G$ is called \emph{malnormal} if $gHg^{-1}\cap H=1$ for every $g\notin H$. A  finite set of subgroups $\{H_1,\ldots,H_n\}$ is called \emph{almost malnormal} if $|gH_ig^{-1}\cap H_j|<\infty$ whenever $g\notin H_i$ or $i\neq j$.\index{subgroup!malnormal}\index{subgroup!almost malnormal set of subgroups}

\medskip

The first ingredient is the Malnormal Special Combination Theorem of Hag\-lund--Wise \cite{HaW12}, which is a gluing theorem for virtually special cube complexes.  We use the notation of Section~\ref{section:specialcc}.  \index{theorems!Haglund--Wise's Malnormal Special Combination Theorem}\index{theorems!Malnormal Special Combination Theorem}

\begin{theorem}\textbf{\emph{(Haglund--Wise)}}\label{thm:msct}
Let $X$ be a compact non-positively curved cube complex with an embedded two-sided hyperplane $Y_i$. Suppose that $\pi_1X$ is word-hyperbolic and that $\pi_1Y_i$ is malnormal in $\pi_1X$. Suppose that each component of $X\setminus N_i^o$ is virtually special. Then $X$ is virtually special.
\end{theorem}

The second ingredient is Wise's Malnormal Special Quotient Theorem \cite[Theorem 12.3]{Wis12a}.  This asserts the profound fact that the result of a (group-theoretic) Dehn filling on a virtually compact special word-hyperbolic group is still virtually compact special, for all sufficiently deep (in a suitable sense) fillings.\index{theorems!Wise's Malnormal Special Quotient Theorem}\index{theorems!Malnormal Special Quotient Theorem}

\begin{theorem}\textbf{\emph{(Wise)}}\label{thm:msqt}
Suppose $\pi$ is word-hyperbolic and virtually compact special and $\{H_1,\ldots,H_n\}$ is an almost malnormal family of subgroups of $\pi$. There are subgroups of finite index $K_i\subseteq H_i$ such that, for all subgroups of finite index $L_i\subseteq K_i$, the quotient
\[
\pi/\langle\langle L_1,\ldots,L_n\rangle\rangle
\]
is word-hyperbolic and virtually compact special.
\end{theorem}

Wise also proved a generalization of Theorem~\ref{thm:wise} to the case of certain relatively hyperbolic groups, from which he deduces the corresponding result in the cusped case \cite[Theorem 16.28 and Corollary 14.16]{Wis12a}.

\begin{theorem}\textbf{\emph{(Wise)}}\label{thm:wiseboundary}
If $N$ is a non-closed hyperbolic $3$-manifold of finite volume, then $\pi_1(N)$ is virtually compact special.
\end{theorem}

This last theorem relies on extending some of Wise's techniques from the word-hyperbolic case to the relatively hyperbolic case.   Some foundational results for the relatively hyperbolic case were proved in \cite{HrW12}.

\subsection{Quasi-Fuchsian surface subgroups: the work of Kahn and Marko\-vic}\label{section:km}


As discussed in Section~\ref{section:wise}, Wise's work applies to hyperbolic $3$-manifolds with a geometrically finite hierarchy.  A non-Haken $3$-manifold, on the other hand, has no hierarchy by definition.  Likewise, although Haken hyperbolic $3$-manifolds without a geometrically finite hierarchy are virtually fibered by Theorem~\ref{thm:bon86}, Thurston's Questions 15 (LERF), as well as other important open problems such as largeness, do not follow from Wise's theorems in this case.

\medskip

The starting point for dealing with hyperbolic $3$-manifolds without a geometrically finite hierarchy is provided by Kahn and Markovic's proof of the Surface Subgroup Conjecture.  More precisely, as a key step towards answering Thurston's question in the affirmative, Kahn--Markovic \cite{KM12} showed that the  fundamental group of any  closed hyperbolic $3$-manifold contains a surface group. In fact they proved a significantly stronger statement.  In order to state their theorem precisely, we need two more definitions.  In the following discussion,~$N$ is assumed to be a closed hyperbolic $3$-manifold.
\bn
\item We refer to \cite[p.~4]{KAG86} and \cite[p.~10]{KAG86} for the definition of a \emph{quasi-Fuchsian surface group}. A surface subgroup $\Gamma$ of $\pi_1(N)$ is quasi-Fuchsian if and only if it is geometrically finite \cite[Lemma 4.66]{Oh02}.  (If $N$ has cusps then we need to add the condition that $\Gamma$ has no `accidental' parabolic elements.)
\item
We fix an identification of $\pi_1(N)$ with a discrete subgroup of $\op{Isom}(\H^3)$.
We say that $N$ \emph{contains a dense set of quasi-Fuchsian surface groups} if  for each great circle $C$ of $\partial \H^3=S^2$ there exists a sequence of $\pi_1$-injective immersions $\iota_i\colon\Sigma_i\to N$
of surfaces $\Sigma_i$ such that the following hold:
\bn
\item for each $i$ the group  $(\iota_i)_*(\pi_1(\Sigma_i))$ is a quasi-Fuchsian surface group,
\item the sequence  $(\partial \Sigma_i)$  converges to $C$ in the Hausdorff metric on $\partial \H^3$.
\en
\en
\noindent
We can now state the theorem of Kahn--Markovic \cite{KM12}.
(Note that this particular formulation is \cite[Th\'eor\`eme~5.3]{Ber12}.) \index{theorems!Kahn-Markovic Theorem}
\index{theorems!Surface Subgroup Conjecture}  \index{surface group!quasi-Fuchsian} \index{$3$-manifold!containing a dense set of quasi-Fuchsian surface groups}

\begin{theorem}\label{thm:km} \textbf{\emph{(Kahn--Markovic)}}
Every closed hyperbolic $3$-manifold contains a dense set of quasi-Fuchsian surface groups.
\end{theorem}

\subsection{Agol's Theorem}\label{section:agol}

The following theorem of Bergeron--Wise \cite[Theorem~1.4]{BeW12}, building extensively on work of Sageev \cite{Sag95,Sag97}
makes it possible to approach hyperbolic $3$-manifolds via non-positively curved cube complexes.

\begin{theorem}\label{thm:sbw} \textbf{\emph{(Sageev, Bergeron--Wise)}}
Let $N$ be a closed hyperbolic $3$-manifold which contains a dense set of quasi-Fuchsian surface groups.
Then $\pi_1(N)$ is also the fundamental group of a compact non-positively curved cube complex.
\end{theorem}

In the previous section we saw that Kahn--Markovic showed that every closed hyperbolic $3$-manifold satisfies the hypothesis of the theorem.

\medskip
The following theorem was conjectured by Wise \cite{Wis12a} and proved recently by Agol \cite{Ag12}.

\begin{theorem}\textbf{\emph{(Agol)}}\label{thm:ag12}
Let $\pi$ be word-hyperbolic and the fundamental group of a compact, non-positively curved cube complex.  Then $\pi$ is virtually compact special.
\end{theorem}

The proof of Theorem~\ref{thm:ag12} relies heavily on results in the appendix to~\cite{Ag12}, which are due to Agol, Groves and Manning.  The results of this appendix extend the techniques of \cite{AGM09} to word-hyperbolic groups with torsion, and combine them with the Malnormal Special Quotient Theorem (Theorem~\ref{thm:msqt}). \index{theorems!Agol's Virtually Compact Special Theorem}\index{theorems!Virtual Compact Special Theorem}

\medskip

Note that the combination of Theorems~\ref{thm:km},~\ref{thm:sbw} and~\ref{thm:ag12}
now implies Theorem~\ref{thm:akmw} for closed hyperbolic $3$-manifolds.

\subsection{$3$-manifolds with non-trivial JSJ decomposition}\label{section:lpw}

Although an understanding of the hyperbolic case is key to an understanding of all $3$-manifolds, a good understanding of hyperbolic $3$-manifolds and of Seifert fibered spaces does not necessarily immediately yield the answers to questions on $3$-manifolds with non-trivial JSJ decomposition.  For example, by (C.\ref{C.sl2c}) the fundamental group of a hyperbolic $3$-manifold is linear over $\C$, but it is still an open question whether or not the fundamental group of any closed irreducible $3$-manifold is linear. (See Section~\ref{sec:linear repres} below.)

In the following we  say that a $3$-manifold $N$ with empty or toroidal boundary is \emph{non-positively curved}  if the interior of $N$ admits a complete non-positively curved Riemannian metric. \index{$3$-manifold!non-positively curved}
Furthermore, a compact orientable irreducible $3$-manifold with empty or toroidal boundary is called a \emph{graph manifold} \index{$3$-manifold!graph manifold}
if all JSJ components are Seifert fibered manifolds.

These concepts are related the following theorem.

\begin{theorem}\textbf{\emph{(Leeb, \cite{Leb95})}}\label{thm:leeb}
Let $N$ be an irreducible $3$-manifold with empty or to\-roi\-dal boundary.
If $N$ is not a closed graph manifold, then $N$ is non-positively curved.
\end{theorem}

The question of which closed graph manifolds are non-positively-curved was treated in detail by Buyalo and Svetlov \cite{BuS05}.

\begin{theorem}\textbf{\emph{(Liu)}}\label{thm:liu11}
Let $N$ be an aspherical graph manifold.  Then $\pi_1(N)$ is virtually special if and only if $N$ is non-positively curved.
\end{theorem}

\begin{remarks}

\mbox{}

\bn
\item
 Liu \cite{Liu11}, building on the ideas and results of \cite{Wis12a}, proved the theorem  if  $N$ has a non-trivial JSJ decomposition.  The case that $N$ is a Seifert fibered $3$-manifold
  is well known to the experts and follows `by inspection.'  More precisely, let $N$ be an aspherical Seifert fibered $3$-manifold. If $\pi=\pi_1(N)$ is virtually special, then by the arguments of Section \ref{section:diagram} (see (G.\ref{G.special=>subgroup of a RAAG}), (G.\ref{G.RAAGvRFRS}) and (G.\ref{G.RFRSretracttoz})) it follows that  $\pi$ virtually retracts onto infinite cyclic subgroups.
 This implies easily that $\pi$ is virtually special
 if and only if its underlying geometry is either Euclidean or $\H\times \R$.
 On the other hand it is well known (see, e.g., \cite{Leb95}) that these are precisely the geometries of aspherical Seifert fibered $3$-manifolds which support a non-positively curved metric.
\item By Theorem \ref{thm:leeb}  a graph manifold with non-empty boundary is non-positively curved.
Liu thus showed in particular that fundamental groups of graph manifolds with non-empty boundary are virtually special; this was also obtained by  Przytycki--Wise \cite{PW11}.
\item
There exist closed  graph manifolds with non-trivial JSJ decompositions that are not virtually fibered (see, e.g., \cite[p.~86]{LuW93} and \cite[Theorem~D]{Nemb96}), and hence
by (G.\ref{G.special=>subgroup of a RAAG}), (G.\ref{G.RAAGvRFRS}) and (G.\ref{G.agol08})
are neither virtually special nor non-positively curved (see also \cite{BuK96a,BuK96b}, \cite[Example~4.2]{Leb95} and \cite{BuS05}).
There also exist fibered graph manifolds which are not virtually special; for instance, the fundamental groups of  non-trivial torus bundles are not virtually RFRS by (G.\ref{G.RFRSretracttoz})
and (G.\ref{G. vb_1=infty}) (cf.\ \cite[p.~271]{Ag08});
also see \cite[Section~2.2]{Liu11} and \cite{BuS05} for examples with non-trivial JSJ decomposition which are not torus bundles.
\en
\end{remarks}

Przytycki--Wise \cite[Theorem~1.1]{PW12a},  building on the ideas and results of \cite{Wis12a}, proved the following theorem,
which  complements  the Virtually Compact Special Theorem of Agol, Kahn--Markovic and Wise, and Liu's theorem.

\begin{theorem}\textbf{\emph{(Przytycki--Wise)}}\label{thm:pw12}
Let $N$ be an irreducible $3$-manifold with empty or toroidal boundary which is neither hyperbolic nor a graph manifold.
 Then $\pi_1(N)$ is virtually special.
\end{theorem}

The combination of the Virtually Compact Special Theorem of Agol, Kahn--Markovic and Wise,
the results of Liu and Przytycki--Wise and the theorem of Leeb now
gives us the following succinct and beautiful statement:

\begin{theorem}\label{thm:npcvs}
Let $N$ be a  compact orientable aspherical $3$-manifold $N$ with empty or toroidal boundary.
Then $\pi_1(N)$ is virtually special if and only if $N$ is non-positively curved.
\end{theorem}

\begin{remarks}

\mbox{}

\bn
\item The connection between $\pi_1(N)$ being virtually special and $N$ being non-positively curved is very indirect.
It is an interesting question whether one can find a more direct connection between these two notions.
\item Note that by the Virtually Compact Special Theorem of Agol, Kahn--Markovic, and Wise,
the fundamental groups of hyperbolic $3$-manifolds are in fact virtually \emph{compact} special.
It is not known whether fundamental groups of non-positively curved irreducible non-hyperbolic $3$-manifolds are also virtually compact special.
\en
\end{remarks}


\subsection{$3$-manifolds with more general boundary}\label{ss: Non-toroidal}
For simplicity of exposition, we have only considered compact $3$-manifolds with empty or toroidal boundary.  However, the virtually special theorems above apply equally well in the case of general boundary, and in this section we give some details.  We emphasize that we make no claim to the originality of any of the results of this section.

The main theorem of \cite{PW12a} also applies in the case with general boundary, and so we have the following addendum to Theorem \ref{thm:npcvs}.

\begin{theorem}\textbf{\emph{(Przytycki--Wise)}}\label{thm: Non-toroidal vs}
Let $N$ be a  compact, orientable, aspherical $3$-manifold $N$ with non-empty boundary.  Then $\pi_1(N)$ is virtually special.
\end{theorem}

\begin{remark}
Compressing the boundary and doubling along a suitable subsurface, one may also deduce that $\pi_1(N)$ is a subgroup of a RAAG directly from Theorem~\ref{thm:npcvs}.
\end{remark}

Invoking suitably general versions of the torus decomposition (for instance, \cite[Theorem 3.4]{Bon02} or \cite[Theorem 1.9]{Hat}) and Thurston's Geometrization theorem for manifolds with boundary (for instance, \cite[Theorem 1.43]{Kap01}), the proof of Theorem \ref{thm:leeb} applies equally well in this setting, and one obtains the following statement (see \cite{Bek} for further details).

\begin{theorem}\label{thm: Non-empty boundary npc}
The interior of any compact, orientable, aspherical $3$-manifold with non-empty boundary admits a complete, non-positively curved, Riemannian metric.
\end{theorem}

\begin{remark}
Bridson \cite[Theorem~4.3]{Brd01} proved that the interior of a compact, orientable, aspherical $3$-manifold $N$ with non-empty boundary admits an \emph{incomplete}, non-positively curved, Riemannian metric that extends to a non-positively curved (ie locally CAT(0)) metric on the whole of $N$.
\end{remark}

Combining Theorems \ref{thm: Non-toroidal vs} and \ref{thm: Non-empty boundary npc}, the hypotheses on the boundary in Theorem \ref{thm:npcvs} can be removed.  For completeness, we state the most general result here.

\begin{theorem}\textbf{\emph{(Agol, Liu, Przytycki, Wise)}}\label{thm: Any vs}
Let $N$ be a  compact, orientable, aspherical $3$-manifold with possibly empty boundary.  Then $\pi_1(N)$ is virtually special if and only if $N$ is non-positively curved.
\end{theorem}

Next, we turn to the case of a compact, hyperbolic $3$-manifold with at least one higher-genus boundary component.  Appealing to the theory of Kleinian groups,  the (implicit) hypotheses of Theorem \ref{thm:akmw} can be relaxed.

We start with a classical result from the theory of Kleinian groups \cite[Theorem 11.1]{Cay08}.

\begin{theorem}\label{thm: All gf}
If $N$ is compact, hyperbolic $3$-manifold $N$ with at least one higher-genus boundary component, then $\pi_1(N)$ admits a geometrically finite representation as a Kleinian group in which only the fundamental groups of toroidal boundary components are parabolic.
\end{theorem}

We can now apply a theorem of Brooks \cite[Theorem 2]{Brk86}
to deduce that $\pi_1(N)$ can be embedded in the fundamental group of a hyperbolic $3$-manifold of finite volume.

\begin{theorem}\label{thm:Brooks}
If $N$ is a compact, hyperbolic $3$-manifold $N$ with at least one higher-genus boundary component then there is a hyperbolic $3$-manifold $M$ of finite volume such that $\pi_1(N)$ embeds into $\pi_1(M)$ as a geometrically finite subgroup and only the fundamental groups of toroidal boundary components are parabolic.
\end{theorem}

It now follows that $\pi_1(N)$ is virtually compact special \cite[Corollary~14.33]{Wis12a}.

\begin{theorem}\textbf{\emph{(Wise)}}\label{thm: Non-toroidal vcs}
Let $N$ be a compact, hyperbolic $3$-manifold with at least one higher genus boundary component. Then $\pi_1(N)$ is virtually compact special.
\end{theorem}
\begin{proof}
Let $M$ be as in Theorem \ref{thm:Brooks}. By Theorem \ref{thm:wiseboundary}, $\pi_1(M)$ is virtually compact special.  We will now argue that $\pi_1(N)$ is virtually compact special as well.  Indeed, $\pi_1(N)$ is a geometrically finite, and hence relatively quasiconvex, subgroup of $\pi_1(M)$ (K.\ref{K.qc}).  As $\pi_1(N)$ contains any cusp subgroup that it intersects non-trivially, it is in fact a \emph{fully} relatively quasiconvex subgroup of $\pi_1(M)$, and is therefore virtually compact special by \cite[Proposition 5.5]{CDW12} or \cite[Theorem 1.1]{SaW12}.
\end{proof}

We now summarize  some properties of $3$-manifolds with general boundary, which are a consequence of
Theorem~\ref{thm: Non-toroidal vs} and the discussion in Section~\ref{section:diagram}.

\begin{corollary}\label{cor: Non-toroidal properties}
Let $N$ be a compact, orientable, aspherical
$3$-manifold with non-empty boundary. Then
\begin{itemize}
\item[(1)] $\pi_1(N)$ is linear over $\mathbb{Z}$;
\item[(2)] $\pi_1(N)$ is RFRS; and
\item[(3)] if $N$ is hyperbolic, then $\pi_1(N)$ is LERF.
\end{itemize}
\end{corollary}

\begin{proof}
It follows from Theorem \ref{thm: Non-toroidal vs} that $\pi_1(N)$ is virtually special. Linearity over $\mathbb{Z}$ and RFRS now both follow from Theorem \ref{thm: Non-toroidal vs}: see (G.\ref{G.special=>subgroup of a RAAG}), (G.\ref{G.RAAGvRFRS}) and (G.\ref{G.24}) in Section~\ref{section:diagram} for details.

If $N$ is hyperbolic, then by Theorem \ref{thm:Brooks}, $\pi_1(N)$ is a subgroup of $\pi_1(M)$, where $M$ is a hyperbolic $3$-manifold of finite volume.  Because $\pi_1(M)$ is LERF (G.\ref{G.tameness}), it follows that $\pi_1(N)$ is also LERF.
\end{proof}

\subsection{Summary of previous research on the virtual conjectures}\label{section:history}

Questions~15--18 of Thurston, stated in Section~\ref{section:vcsthm} above, have been a central area of research in $3$-manifold topology over the last 30~years.  The study of these questions lead to various other questions and conjectures. Perhaps the most important of these is the Lubotzky--Sarnak Conjecture (see \cite[Conjecture~4.2]{Lub96a}) that there is no  closed hyperbolic $3$-manifold~$N$ such that $\pi_1(N)$ has  Property~($\tau$).
(We refer to \cite[Definition~4.3.1]{Lub94} and \cite{LuZ03} for the definition of Property~($\tau$).) \index{conjectures!Lubotzky--Sarnak Conjecture}
\index{Property ($\tau$)}

%
%
%

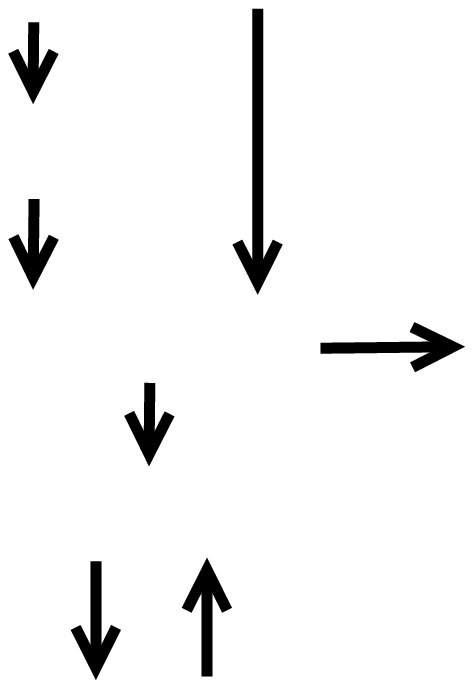

In Diagram~3 we list various (virtual) properties of $3$-manifold groups and logical implications between them.
Some of the implications are obvious, and two implications follow from
(C.\ref{C.homlarge}) and (C.\ref{C.haken}).
Also note that if a $3$-manifold~$N$ contains a surface group, then it admits a $\pi_1$-injective map $\pi_1(\Sigma)\to \pi_1(N)$
of a closed surface $\Sigma$ with genus at least one. If $\pi_1(N)$ is furthermore LERF, then there exists a finite cover of $N$ such that the immersion lifts to an embedding (see \cite[Lemma~1.4]{Sco78} for details).
Finally note that if $vb_1(N;\Z)\geq 1$, then by \cite[p.~444]{Lub96a} the group $\pi_1(N)$ does not have Property~($\tau$).

We will now survey some of the work in the past on Thurston's questions and the properties of Diagram~3.  The literature is so extensive that we cannot hope to achieve completeness.  Beyond the summary below we also refer to the survey papers by Long--Reid \cite{LoR05} and Lackenby \cite{Lac11} for further details and references.

We arrange this survey by grouping references under the question that they address.

\begin{question}\label{qu: Surface subgroups}\textbf{\emph{(Surface Subgroup Conjecture)}}
Let $N$ be a closed hyperbolic $3$-manifold.
Does $\pi_1(N)$ contain a \textup{(}quasi-Fuchsian\textup{)} surface group? \index{conjectures!Surface Subgroup Conjecture}
 \end{question}

The following papers attack Question~\ref{qu: Surface subgroups}.

\bn
\item  Cooper--Long--Reid  \cite[Theorem~1.5]{CLR94}  showed that if $N$ is a
 closed hyperbolic $3$-manifold which fibers over $S^1$, then  there exists a $\pi_1$-injective immersion of a quasi-Fuchsian surface into $N$.  We note one important consequence: if $N$ is any hyperbolic $3$-manifold such that $\pi_1(N)$ is LERF and contains a surface subgroup, then $\pi_1(N)$ is large (cf.~(C.\ref{C.large})).
\item The work of \cite{CLR94} was extended by Masters \cite[Theorem~1.1]{Mas06b}, which in turn allowed Dufour \cite[p.~6]{Duf12} to show that if $N$ is a closed hyperbolic $3$-manifold which is virtually fibered, then $\pi_1(N)$ is also the fundamental group of a compact non-positively curved cube complex. This proof does not require the Surface Subgroup Theorem \ref{thm:km} of Kahn--Markovic \cite{KM12}.
\item Li \cite{Li02}, Cooper--Long \cite{CoL01} and Wu \cite{Wu04} showed that in many cases the Dehn surgery on a hyperbolic $3$-manifold contains a surface group.  \item
  Lackenby \cite[Theorem~1.2]{Lac10} showed that closed arithmetic hyperbolic $3$-manifolds
 contain surface groups.
 \item Bowen \cite{Bowe04} attacked the Surface Subgroup Conjecture with methods which foreshadowed the approach taken by Kahn--Markovic \cite{KM12}.

\en

\begin{question}\textbf{\emph{(Virtually Haken Conjecture)}}
Is every closed hyperbolic $3$-manifold virtually Haken? \index{conjectures!Virtually Haken Conjecture}
\end{question}

Here is a summary of approaches towards the Virtually Haken Conjecture.

\bn
\item
Thurston~\cite{Thu79} showed that all but finitely many Dehn fillings of the Figure 8 knot complement are not Haken.
For this reason, there has  been  considerable interest in studying the Virtually Haken Conjecture for fillings of $3$-manifolds.
Much work in this direction was done by
Aitchison--Rubinstein \cite{AiR99b},
Aitchi\-son--Mat\-su\-moti--Rubinstein \cite{AMR97,AMR99},
Baker \cite{Bak88, Bak89, Bak90, Bak91}, Boyer--Zhang \cite{BrZ00}, Cooper--Long \cite{CoL99} (building on \cite{FF98}), Cooper--Walsh \cite{CrW06a,CrW06b}, Hempel \cite{Hem90},
Kojima--Long \cite{KL88},
Masters \cite{Mas00,Mas07}, Masters--Menasco--Zhang \cite{MMZ04,MMZ09}, Morita \cite{Moa86} and X. Zhang \cite{Zha05} and Y. Zhang \cite{Zhb12}.
\item Hempel \cite{Hem82,Hem84,Hem85a} and Wang \cite{Wag90} \cite[p.~192]{Wag93} studied the Virtually Haken Conjecture for $3$-manifolds which admit an orientation reversing involution.
\item Long \cite{Lo87} (see also \cite[Corollary~1.2]{Zha05}) showed that if $N$ is a hyperbolic $3$-manifold which  admits a totally geodesic immersion of a closed surface, then
$N$ is virtually Haken.
\item
 We refer to Millson \cite{Mis76}, Clozel \cite{Cl87}, Labesse--Schwermer \cite{LaS86}, Xue \cite{Xu92}, Li--Millson \cite{LiM93}, Rajan \cite{Raj04}, Reid \cite{Red07} and Schwermer \cite{Scr04,Scr10} for details of approaches to the Virtually Haken Conjecture for arithmetic hyperbolic $3$-manifolds using  number theoretic methods.
\item Reznikov \cite{Rez97} studied hyperbolic $3$-manifolds $N$ with $vb_1(N)=0$.
\item Experimental evidence towards the validity of the conjecture was provided by Dunfield--Thurston \cite{DnTb03}.
\item We refer to Lubotzky \cite{Lub96b} and Lackenby \cite{Lac06,Lac07b,Lac09} for work towards the stronger conjecture that fundamental groups of hyperbolic $3$-manifolds are large. (See Question~\ref{qu:vbn} below.)
\en


\begin{question}\label{qu: LERF}
Let $N$ be a hyperbolic $3$-manifold.  Is $\pi_1(N)$ LERF? \index{conjectures!LERF Conjecture}
\end{question}

The following papers gave evidence for an affirmative answer to Question~\ref{qu: LERF}.  Note that, by the Subgroup Tameness Theorem, $\pi_1(N)$ is LERF if and only if every geometrically finite subgroup is separable, i.e., $\pi_1(N)$ is \emph{GFERF}.  See (G.\ref{G.LERF}) for details.

\bn
\item Let $N$ be a compact hyperbolic $3$-manifold and $\Sigma$ a totally geodesic immersed surface in $N$. Long  \cite{Lo87} proved that $\pi_1(\Sigma)$ is separable in $\pi_1(N)$.

\item
The first examples of hyperbolic $3$-manifolds with LERF fundamental groups were given by Gitik \cite{Git99b}.
\item
Agol--Long--Reid \cite{ALR01} showed that geometrically finite subgroups of Bianchi groups are separable.
\item
Wise \cite{Wis06} showed that the fundamental group of the Figure~8 knot complement is LERF.
\item
Agol--Groves--Manning \cite{AGM09} showed that fundamental groups of hyperbolic $3$-manifolds are LERF if every word-hyperbolic group is residually finite.
\item After the definition of special complexes was given in \cite{HaW08}, it was shown that various classes of hyperbolic $3$-manifolds had virtually special fundamental groups, and hence were LERF (and virtually fibered).  The following were shown to be virtually compact special:
\begin{enumerate}
\item `standard' arithmetic $3$-manifolds \cite{BHW11};
\item certain branched covers of the figure-eight knot \cite[Theorem~1.1]{Ber08};
\item manifolds built from gluing all-right ideal polyhedra, such as augmented link complements \cite{CDW12}.
\end{enumerate}

\en

\begin{question}\textbf{\emph{(Lubotzky--Sarnak Conjecture)}}
Let $N$ be a  closed hyperbolic $3$-manifold.
Is it true that $\pi_1(N)$ does not have Property $(\tau)$? \index{conjectures!Lubotzky--Sarnak Conjecture}
\end{question}

The following represents some of the major work on the Lubotzky--Sarnak Conjecture.  We also refer to \cite[Section~7]{Lac11} and \cite{LuZ03} for further details.

\bn
\item Lubotzky \cite{Lub96a} stated the conjecture and proved that certain arithmetic $3$-manifolds have positive virtual first Betti number, extending the above-mentioned work of Millson \cite{Mis76} and Clozel \cite{Cl87}.
\item Lackenby \cite[Theorem~1.7]{Lac06} showed that the Lubotzky--Sarnak Conjecture, together with a conjecture about Heegaard gradients, implies the Virtually Haken Conjecture.
\item Long--Lubotzky--Reid \cite{LLuR08} proved that the fundamental group of every hyperbolic $3$-manifold has Property $(\tau)$ with respect to some cofinal regular filtration of $\pi_1(N)$.
\item  Lackenby--Long--Reid \cite{LaLR08b} proved that if the fundamental group of a hyperbolic $3$-manifold $N$ is LERF, then $\pi_1(N)$ does not have Property~$(\tau)$.
\en

\begin{questions}\label{qu:vbn}
Let $N$ be a hyperbolic $3$-manifold with $b_1(N)\geq 1$.
\begin{itemize}
\item[(1)] Does $N$ admit a finite cover $N'$ with $b_1(N')\geq 2$?
\item[(2)] Is $vb_1(N)=\infty$?
\item[(3)] Is $\pi_1(N)$ large?
\end{itemize}
\end{questions}

The virtual Betti numbers of hyperbolic $3$-manifolds in particular were studied by the following authors:
\bn
\item
Cooper--Long--Reid \cite[Theorem~1.3]{CLR97} have shown that if $N$ is a compact, irreducible $3$-manifold with non-trivial incompressible boundary,
then  either $N$ is covered
by a product $N= S^1\times S^1\times I$, or  $\pi_1(N)$ is large. (See also  \cite[Corollary~6]{But04} and \cite[Theorem~2.1]{Lac07a}.)
\item
Cooper--Long--Reid  \cite[Theorem~1.3]{CLR07}, Venkataramana \cite[Corollary~1]{Ve08} and Agol \cite[Theorem~0.2]{Ag06} proved the fact that if $N$ is an arithmetic $3$-manifold, then
$vb_1(N)\geq 1$ implies that $vb_1(N)=\infty$. In fact by further work of Lackenby--Long--Reid \cite{LaLR08a} it follows that if $vb_1(N)\geq 1$, then $\pi_1(N)$ is large.
\item  Long and Oertel \cite[Theorem~2.5]{LO97} gave many examples of fibered $3$-manifolds with $vb_1(N;\Z)=\infty$.  Masters \cite[Corollary~1.2]{Mas02}  showed that if $N$ is a fibered $3$-manifold such that the genus of the fiber is 2, then $vb_1(N;\Z)=\infty$.
\item Kionke--Schwermer \cite{KiS12} showed that certain arithmetic hyperbolic $3$-manifolds admit a cofinal tower  with rapid growth of first Betti numbers.
\item Cochran and Masters \cite{CMa06} studied the growth of Betti numbers in abelian covers of $3$-manifolds with Betti number equal to two or three.
\item Button \cite{But11a} gave computational evidence towards the conjecture that the fundamental group of any hyperbolic $3$-manifold $N$ with $b_1(N)\geq 1$ is large.
\item Koberda \cite{Kob12a,Kob12b} gives a detailed study of Betti numbers of finite covers of fibered $3$-manifolds.
\en

\begin{question}\textbf{\emph{(Virtually Fibered Conjecture)}}
Is every hyperbolic $3$-mani\-fold virtually fibered? \index{conjectures!Virtually Fibered Conjecture}
\end{question}

The following papers deal with the Virtually Fibered Conjecture:
\bn
\item
An affirmative answer to the question was given for specific classes of $3$-manifolds, e.g., certain knot and link complements,
 by  Agol--Boyer--Zhang~\cite{ABZ08}, Aitchison--Rubinstein~\cite{AiR99a}, DeBlois~\cite{DeB10},
 Ga\-bai~\cite{Gab86}, Guo--Zhang~\cite{GZ09},  Reid~\cite{Red95}, Leininger~\cite{Ler02}, and  Walsh~\cite{Wah05}.
\item
 Button \cite{But05} gave  computational evidence towards  an affirmative answer to the Virtually Fibered Conjecture.
 \item Long--Reid \cite{LoR08b} (see also  Dunfield--Ramakrishnan \cite{DR10} and \cite[Theorem~7.1]{Ag08}) showed that  arithmetic hyperbolic $3$-manifolds
which are fibered admit in fact finite covers with arbitrarily many fibered faces in the Thurston norm ball.
\item
Lackenby \cite[p.~320]{Lac06} (see also \cite{Lac11}) gave an approach to the Virtually Fibered Conjecture using
`Heegaard gradients'. This approach was further developed by Lackenby \cite{Lac04}, Maher \cite{Mah05} and Renard \cite{Ren09,Ren10}.
The latter author gave another approach~\cite{Ren11,Ren12}  to the conjecture.
\item
 Agol \cite[Theorem~5.1]{Ag08} showed that aspherical $3$-manifolds with virtually RFRS fundamental groups are virtually fibered.
 (See (E.\ref{E.RFRS}) for the definition of RFRS.)
 The first examples of $3$-manifolds with virtually RFRS fundamental groups were given by Agol \cite[Corollary~2.3]{Ag08}, Bergeron \cite[Theorem~1.1]{Ber08}, Bergeron--Haglund--Wise \cite{BHW11} and Chesebro--DeBlois--Wilton \cite{CDW12}.
\en

\section{Consequences of being virtually (compact) special}\label{section:closed}\label{section:diagram}
\label{section:corakmw}

\noindent
In this section we summarize various consequences of the fundamental group of a $3$-manifold being virtually (compact) special.
As in Section~\ref{section:diagramgeom} we present the results in a diagram.

\medskip

We start out with  further  definitions needed for Diagram~4.
Again the definitions are roughly in the order that they appear in the diagram.

\begin{list}
{{(E.\arabic{itemcounterm})}}
{\usecounter{itemcounterm}\leftmargin=2em}
\item We say that a group $\pi$ \emph{virtually retracts} onto a subgroup $A\subseteq \pi$
if there exists a  finite-index subgroup $\pi'\subseteq \pi$  that contains $A$ and a homomorphism $\pi'\to A$
which is the identity on~$A$.  In this case, we say that $A$ \emph{is a virtual retract of $\pi$}. \index{subgroup!retract!virtual}
\item\label{E.defconjugacyseparable} A group $\pi$ is called \emph{conjugacy separable} if for any two non-conjugate elements $g,h\in \pi$
there exists an epimorphism $\a\colon \pi\to G$ onto a finite group $G$ such that $\a(g)$ and $\a(h)$ are not conjugate. A group $\pi$ is called \emph{hereditarily conjugacy separable} if any (not necessarily normal) finite-index subgroup of $\pi$ is conjugacy separable. \index{group!conjugacy separable} \index{group!hereditarily conjugacy separable}

\item For $N$ a hyperbolic $3$-manifold, we say that $\pi_1(N)$ is GFERF if all geometrically finite subgroups are separable. \index{$3$-manifold group!all geometrically finite subgroups are separable (GFERF)}
\item\label{E.RFRS}
A group $\pi$ is called \emph{residually finite rationally solvable \textup{(}RFRS\textup{)}} if there
exists a filtration of $\pi$ by subgroups $\pi=\pi_0\supseteq \pi_1 \supseteq \pi_2\cdots $
such that: \index{group!residually finite rationally solvable (RFRS)}
\newcounter{itemcountern}
\begin{list}
{{(\arabic{itemcountern})}}
{\usecounter{itemcountern}\leftmargin=2em}
\item $\bigcap_i \pi_i=\{1\}$;
\item   $\pi_i$ is a normal, finite-index subgroup of  $\pi$, for any $i$;
\item for any $i$ the map $\pi_i\to \pi_i/\pi_{i+1}$ factors through $\pi_i\to H_1(\pi_i;\Z)/\mbox{torsion}$.
\end{list}
\item \label{E.fibered} Let $N$ be a compact, orientable 3--manifold.
We say $N$ is \emph{fibered} if $N$ admits the structure of a
surface bundle over $S^1$.\index{$3$-manifold!fibered}\index{fibered!$3$-manifold}
We say that $\phi\in H^1(N;\R)$ is \emph{fibered},\index{fibered!cohomology class}
 if $\phi$ can be represented by a non-degenerate closed 1-form.
Note that by \cite{Tis70} an integral class  $\phi\in H^1(N;\Z) = \mbox{Hom}(\pi_1(N),\Z)$ is fibered if and only if there exists
 a surface bundle $p\colon N\to S^1$ such that the induced map $p_*\colon \pi_1(N)\to \pi_1(S^1)=\mathbb{Z}$ coincides with $\phi$.
\item A group $\pi$ has the \emph{finitely generated intersection property} (or \emph{f.g.i.p.}\ for short) if the intersection of any two finitely generated subgroups of $\pi$ is also finitely generated. \index{group!with the finitely generated intersection property (f.g.i.p.)}
 \item A group $\pi$ is called \emph{poly-free} if it admits a finite sequence of subgroups
\[ \pi=\pi_0\rhd \pi_1\rhd  \pi_2\rhd \dots \rhd \pi_k=\{1\}\]
such that for any $i\in \{0,\dots,k-1\}$ the quotient group $\pi_i/\pi_{i+1}$ is a (not necessarily finitely-generated) free group. \index{group!poly-free}
   \item Let $\pi$ be a  group. We denote its profinite completion by $\widehat{\pi}$.
The group $\pi$ is called \emph{good} if the map $H^*(\widehat{\pi};A)\to H^*(\pi;A)$
is an isomorphism for any finite $\pi$--module $A$. (See \cite[D.2.6~Exercise~2]{Ser97}.)\index{group!good}
\item The unitary dual of a group is defined to be the set of equivalence classes of its irreducible unitary representations.
The unitary dual can be viewed as a topological space with respect to the Fell topology.
A group $\pi$ is said to have \emph{Property FD} if the finite representations of $\pi$ are dense in its unitary dual.
We refer to \cite[Appendix~F.2]{BdlHV08} and \cite{LuSh04} for details.\index{group!with Property FD}
\item
A  group $\pi$ is called \emph{potent} if for any non-trivial $g\in \pi$ and any $n\in \Z$ there exists an epimorphism
$\a\colon \pi\to G$ onto a finite group $G$ such that $\a(g)$ has order $n$. \index{group!potent}
\item A subgroup of a group $\pi$ is called \emph{characteristic} if it is preserved by every automorphism of $\pi$.
Every characteristic subgroup is normal.
 A group $\pi$ is called \emph{characteristically potent}, if given any non-trivial $g\in \pi$ and any $n\in \N$ there exists a finite index characteristic subgroup $\pi'\subseteq \pi$ such that $g$ has order $n$ in $\pi/\pi'$. \index{subgroup!characteristic} \index{group!characteristically potent}
\item A  group $\pi$ is called \emph{weakly characteristically potent} if for any non-trivial $g\in \pi$
 there exists an $r\in \N$ such that for  any $n\in \N$ there exists a characteristic finite-index subgroup $\pi'\subseteq \pi$ such that $g\pi'$ has order $rn$ in $\pi/\pi'$.
\item \label{E.independent}
Let $\pi$ be a torsion-free group. We say that a collection of  elements
\[
g_1,\dots,g_n\in  \pi
\]
is \emph{independent} if distinct pairs of elements do not have conjugate non-trivial powers; that is, if there are $k,l\in \Z\smallsetminus \{0\}$ and $c\in \pi$ with $cg_i^kc^{-1}=g_j^l$, then $i=j$.
The group $\pi$ is called \emph{omnipotent} if given any  independent collection
\[
g_1,\dots,g_n\in \pi
\]
there exists $k\in \N$ such that for any $l_1,\dots,l_n\in \N$  there exists a homomorphism $\a\colon \pi\to G$ to a finite group $G$ such that the order of $\a(g_i)\in G$ is $kl_i$.  This definition was introduced by Wise in \cite[Definition~3.2]{Wis00}. \index{group!omnipotent}
\end{list}

\noindent
Diagram~4 is supposed to be read in the same manner as Diagram~1. For the reader's convenience we recall some of the conventions.

\newcounter{itemcounterii}
\begin{list}
{{(F.\arabic{itemcounterii})}}
{\usecounter{itemcounterii}\leftmargin=2em}
\item In Diagram~4 we mean by  $N$ a compact, orientable, irreducible $3$-manifold  such that
 its boundary consists of a (possibly empty) collection of tori. We furthermore assume throughout Diagram~4  that $\pi:=\pi_1(N)$ is neither solvable
 nor finite.
 Note that without these extra assumptions some of the implications do not hold.
For example the fundamental group of the $3$-torus~$T$ is a RAAG, but $\pi_1(T)$ is not large.
\item In the diagram the top arrow splits into several arrows. In this case exactly one of the possible three conclusions holds.
\item Red arrows indicate that the conclusion holds \emph{virtually}, e.g., if $\pi$ is RFRS, then $N$ is virtually fibered.
\item If a property $\PP$ of groups is written in green, then the following conclusion always holds:
If $N$ is a compact, orientable, irreducible $3$-manifold with empty or toroidal boundary such that the fundamental group of a finite (not necessarily regular) cover of $N$ has Property $\PP$, then $\pi_1(N)$ also has Property~$\PP$.
In (H.1) to (H.8) below we will show that the various properties written in green do indeed have the above property.
\item Note that a concatenation of red and black arrows which leads  to a green property means that the initial group also has the green property.
\item An arrow with a condition next to it indicates that this conclusion only holds if this extra condition is satisfied.
\end{list}

\begin{figure}
\psfrag{1}{\ref{G.1}}\psfrag{2}{\ref{G.2}}\psfrag{3}{\ref{G.3}}\psfrag{4}{\ref{G.4}}
\psfrag{5}{\ref{G.5}}\psfrag{6}{\ref{G.6}}\psfrag{7}{\ref{G.7}}\psfrag{8}{\ref{G.8}}
\psfrag{9}{\ref{G.9}}\psfrag{10}{\ref{G.10}}
\psfrag{11}{\ref{G.11}}\psfrag{12}{\ref{G.12}}\psfrag{13}{\ref{G.13}}\psfrag{14}{\ref{G.14}}
\psfrag{15}{\ref{G.15}}\psfrag{16}{\ref{G.16}}\psfrag{17}{\ref{G.17}}\psfrag{18}{\ref{G.18}}
\psfrag{19}{\ref{G.19}}\psfrag{20}{\ref{G.20}}\psfrag{21}{\ref{G.21}}\psfrag{22}{\ref{G.22}}
\psfrag{23}{\ref{G.23}}\psfrag{24}{\ref{G.24}}\psfrag{25}{\ref{G.25}}\psfrag{26}{\ref{G.26}}
\psfrag{27}{\ref{G.27}}\psfrag{28}{\ref{G.28}}\psfrag{29}{\ref{G.29}}\psfrag{31}{\ref{G.good}}\psfrag{30}{\ref{G.propfd}}
\hspace{-1cm}
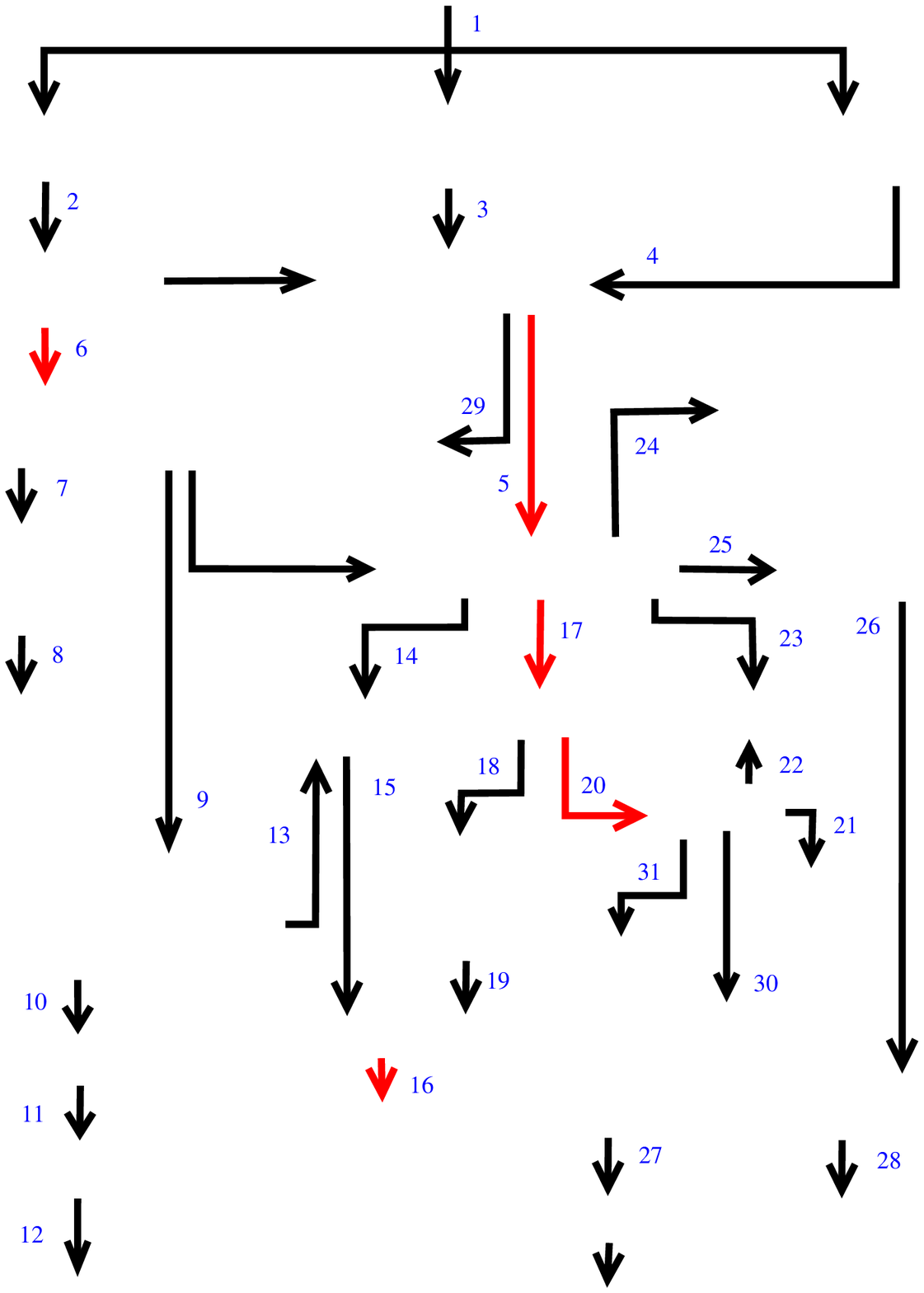
\end{figure}

We now give the justifications for the implications of Diagram~4.
As in Diagram~1 we strive for maximal generality. Unless we say otherwise, we will therefore only assume that $N$ is a connected $3$-manifold and each justification can be read independently of all the other steps.

\newcounter{itemcountero}
\begin{list}
{{(G.\arabic{itemcountero})}}
{\usecounter{itemcountero}\leftmargin=2em}
\item \label{G.1}
Let $N$ be a compact, orientable, irreducible $3$-manifold with empty or toroidal boundary.
 It follows from the Geometrization Theorem (see Theorem~\ref{thm:geom}) that $N$ is either hyperbolic or  a graph manifold,
 or $N$ admits a non-trivial JSJ decomposition with at least one hyperbolic JSJ component.
\item \label{G.2}\label{G.akmw} Let $N$ be a hyperbolic $3$-manifold. The Virtually Compact Special Theorem of  Agol \cite{Ag12}, Kahn--Markovic \cite{KM12} and Wise \cite{Wis12a}
implies that $\pi_1(N)$ is virtually compact special.
We refer to Section~\ref{section:akmw} for details.
\item \label{G.3}\label{G.PW12a}
Let $N$ be an irreducible $3$-manifold with empty or toroidal boundary which is neither hyperbolic nor a graph manifold.
Then by the theorem of Przytycki--Wise (see \cite[Theorem~1.1]{PW12a} and Theorem \ref{thm:pw12})
it follows that $\pi_1(N)$ is virtually special.

\item\label{G.4}\label{G.liu11}
Let $N$ be an aspherical graph manifold.
Liu \cite{Liu11} (see  Theorem \ref{thm:liu11}) showed that $\pi_1(N)$ is virtually special if and only if $N$ is non-positively curved.
By Theorem \ref{thm:leeb} a graph manifold with non-trivial boundary is non-positively curved.  Przytycki--Wise \cite{PW11}
gave an alternative proof that fundamental groups of graph manifolds with non-trivial boundary are virtually special.

\item \label{G.5}\label{G.special=>subgroup of a RAAG} See Corollary \ref{cor: HW special}.

\item \label{G.6} \label{G.compact special=>qc subgroup of a RAAG} See Corollary \ref{cor:HW compact special}.

\item \label{G.7}\label{G.vr of RAAG} Haglund \cite[Theorem F]{Hag08} showed that quasi-convex subgroups of RAAGs are virtual retracts.  In fact, he proved that quasi-convex subgroups of Right-Angled Coxeter Groups are virtual retracts, generalizing earlier results of  Scott \cite{Sco78} (the reflection group of the right-angled hyperbolic pentagon) and Agol--Long--Reid \cite[Theorem~3.1]{ALR01} (reflection groups of arbitrary right-angled hyperbolic polyhedra).
\item \label{G.8} \label{G.minasyan} Minasyan \cite[Theorem~1.1]{Min12} has shown that any  RAAG is hereditarily conjugacy separable.  It follows immediately that virtual retracts of RAAGs are hereditarily conjugacy separable.

Note that combination of Minasyan's result with (G.\ref{G.akmw}), (G.\ref{G.compact special=>qc subgroup of a RAAG}), (G.\ref{G.vr of RAAG})  and  (H.\ref{H.csgreen})
implies that fundamental groups of hyperbolic $3$-manifolds are conjugacy separable.
    In (I.\ref{I.hwz11}) we will see that this  is a key ingredient in the proof of Hamilton--Wilton--Zalesskii \cite{HWZ13} that the fundamental group of any  orientable closed irreducible $3$-manifold   is conjugacy separable.

\item \label{G.9}\label{G.virtualretract} Suppose that $N$ is a hyperbolic $3$-manifold of finite volume and $\pi=\pi_1(N)$ is a quasi-convex subgroup
 of a RAAG $A_\Sigma$.   Let $\Gamma$ be a geometrically finite subgroup of $\pi$.  The idea is that $\Gamma$ should be a quasi-convex subgroup of~$A_\Sigma$.  One could then apply \cite[Theorem F]{Hag08} to deduce that $\Gamma$ is a virtual retract of $A_\Sigma$ and hence of $\pi$.  However, it is not true in full generality that a quasi-convex subgroup of a quasi-convex subgroup is again quasi-convex, and so a careful argument is needed.  In the closed case, the required technical result is \cite[Corollary 2.29]{Hag08}.  In the cusped case, in fact it turns out that $\Gamma$ may not be a quasi-convex subgroup of $A_\Sigma$.  Nevertheless, it is possible to circumvent this difficulty.  We now give detailed references.

If $N$ is closed, then $\pi$ is word-hyperbolic and $\Gamma$ is a quasi-convex subgroup of $\pi$ (see (K.\ref{K.qc})).  The group $\pi$ acts by isometries on $\widetilde{S}_\Sigma$, the universal cover of the Salvetti complex of $A_\Sigma$. Fix a base $0$-cell $x_0\in \widetilde{S}_\Sigma$.  The 1-skeleton of $\widetilde{S}_\Sigma$ is precisely the Cayley graph of $A_\Sigma$ with respect to its standard generating set, and so, by hypothesis, the orbit $\pi.x_0$ is a quasi-convex subset of $\widetilde{S}_\Sigma^{(1)}$.    By \cite[Corollary 2.29]{Hag08}, $\pi$ acts cocompactly on some convex subcomplex $\widetilde{X}\subseteq \widetilde{S}_\Sigma$.  Using the Morse Lemma for geodesics in hyperbolic spaces \cite[Theorem III.D.1.7]{BrH99}, the orbit $\Gamma.x_0$ is a quasi-convex subset of $\widetilde{X}^{(1)}$.  Using \cite[Corollary 2.29]{Hag08} again, it follows that $\Gamma$ acts cocompactly on a convex subcomplex $\widetilde{Y}\subseteq \widetilde{X}$.  The complex $\widetilde{Y}$ is also a convex subcomplex of $\widetilde{S}_\Sigma$, which by a final application of \cite[Corollary 2.29]{Hag08} implies that $\Gamma.x_0$ is a quasi-convex subset of $\widetilde{S}_\Sigma^{(1)}$, or, equivalently, that $\Gamma$ is a quasi-convex subgroup of $A_\Sigma$.  Hence, by \cite[Theorem F]{Hag08}, $\Gamma$ is a virtual retract of $A_\Sigma$ and hence of $\pi$.

If $N$ is not closed, then $\pi$ is not word-hyperbolic, but in any case it is \emph{relatively hyperbolic} and $\Gamma$ is a \emph{relatively quasi-convex} subgroup (see (K.\ref{K.qc}) below for a reference).   One can show in this case that $\Gamma$ is again a virtual retract of $A_\Sigma$ and hence of $\pi$.  The argument is rather more involved than the argument in the word-hyperbolic case; in particular, it is not necessarily true that $\Gamma$ is a quasi-convex subgroup of $A_\Sigma$.  See \cite[Theorem 1.3]{CDW12} for the details; the proof again relies on \cite[Theorem F]{Hag08} together with work of Manning--Martinez-Pedrosa \cite{MMP10}.  See also \cite[Theorem 7.3]{SaW12} for an alternative argument.

\item \label{G.10} \label{G.GFERF}  The following well known argument shows that a virtual retract $G$ of a residually finite group $\pi$ is separable (cf.~\cite[Section~3.4]{Hag08}).  Let $\rho\colon \pi_0\to G$ be a retraction onto $G$ from a subgroup $\pi_0$ of finite index in $\pi$.  Define a map $\delta\colon\pi_0\to \pi_0$ by $g\mapsto g^{-1}\rho(g)$.  It is easy to check that $\delta$ is continuous in the profinite topology on $\pi_0$, and so $G=\delta^{-1}(1)$ is closed. That is to say, $G$ is separable in $\pi_0$, and hence in $\pi$.  In particular, if $N$ is a compact $3$-manifold, then $\pi=\pi_1(N)$ is residually finite by (C.\ref{C.resfinite}), and so if $\pi$ virtually retracts onto geometrically finite subgroups, then $\pi$ is GFERF.
\item \label{G.11}\label{G.LERF}\label{G.tameness}
Now let $N$ be a hyperbolic $3$-manifold such that $\pi=\pi_1(N)$ is GFERF and let $\G\subseteq \pi$ be a finitely generated subgroup.
We want to show that $\G$ is separable.
By the Subgroup Tameness Theorem (see Theorem~\ref{thm:subgroupdichotomy}) and by our assumption we only have to deal with the case that $\G$ is a virtual surface fiber group.  But an  elementary argument shows that in that case $\G$ is separable (see, e.g., (K.\ref{K.fibersep}) for more details).
\item \label{G.12}Let $\pi$ be a group which is word-hyperbolic
with every quasi-convex subgroup separable. Minasyan \cite[Theorem~1.1]{Min06} showed that then any product of finitely many quasi-convex subgroups of $\pi$ is separable. If $\pi=\pi_1(N)$ where $N$ is a closed hyperbolic $3$-manifold and  $\pi$ is GFERF, then because the quasi-convex subgroups are precisely the geometrically finite subgroups  (see (K.\ref{K.qc})), it follows that any product of geometrically finite subgroups is separable.  A direct argument using part (ii) of \cite[Proposition 2.2]{Nib92} shows that any  product of a subgroup with a virtual surface fiber group is separable.  Therefore, by
the Subgroup Tameness Theorem (see Theorem~\ref{thm:subgroupdichotomy}) the group $\pi_1(N)$ is double-coset separable.

    It is expected that the analogue of Minasyan's theorem holds in the relatively hyperbolic setting, in which case the same argument would yield double-coset separability for GFERF fundamental groups of cusped hyperbolic manifolds.  Note that separability of double cosets of abelian subgroups of finite-volume Kleinian groups was proved in \cite{HWZ13}.

\item \label{G.13}Let $N$ be a hyperbolic $3$-manifold such that $\pi=\pi_1(N)$ virtually retracts onto each one of its geometrically finite subgroups.
Let $F\subseteq \pi$ be a geometrically finite non-cyclic free subgroup, such as a Schottky subgroup. (That every non-elementary Kleinian group contains a  Schottky subgroup was first observed by Myrberg \cite{Myr41}.)
 Then by assumption there exists a finite-index subgroup of $\pi$ with
an epimorphism onto $F$. This shows fundamental groups of hyperbolic $3$-manifolds are  large.

\item\label{G.14}\label{G.AM11} Antol\'in--Minasyan \cite[Corollary~1.6]{AM11} showed that every
(not necessarily finitely generated) subgroup of a RAAG is either free abelian of finite rank or maps onto a non-cyclic free group.  This implies directly the fact that if $N$ is a $3$-manifold
    and if $\pi_1(N)$ is virtually special, then either $\pi_1(N)$ is either virtually solvable  or $\pi_1(N)$ is large.
    (Recall that in the diagram we assume throughout that $\pi_1(N)$ is neither finite nor solvable, and     Theorem~\ref{thm:solv} yields that $\pi_1(N)$ is also not virtually solvable.)

\item \label{G.15}We already saw in (C.\ref{C.homlarge}) that a group $\pi$ which is large is homologically large, in particular it has the property that $vb_1(\pi;\Z)=\infty$.
\item \label{G.16}In (C.\ref{C.haken}) we showed that every irreducible, compact $3$-manifold $N$ satisfying $vb_1(N;\Z)\geq 1$ is virtually Haken.
(Furthermore, we saw in (C.\ref{C.locallyindicable}) and (C.\ref{C.left-orderable}) that $\pi_1(N)$ is virtually locally indicable and virtually left-or\-derable.)
\item\label{G.17}\label{G.RAAGvRFRS}  Agol \cite[Theorem~2.2]{Ag08} showed that a RAAG is virtually RFRS.
\index{theorems!Agol's Virtually Fibered Criterion}
It is clear that a subgroup of a RFRS group is again RFRS.
A close inspection of Agol's proof using \cite[Section~1]{DJ00} in fact implies that a RAAG is already RFRS. We will not make use of this fact.
\item \label{G.18}\label{G.RFRSretracttoz} Let $\pi$ be RFRS. It follows easily from the definition that, given any cyclic subgroup $\ll g\rr$, there exists
a finite-index subgroup $\pi'$ such that $g\in \pi'$ and such that $g$ represents a non-trivial element $[g]$ in the torsion-free abelian group $H':=H_1(\pi';\Z)/\mbox{torsion}$.
  There exists a finite-index subgroup $H''$ of $H'$ which contains $g$ and such that $g$ represents a primitive element in $H''$. In particular there exists a homomorphism $\varphi\colon H''\to \Z$ such that $\varphi(g)=1$.  Now
  \[
\ker\{\pi'\to H'/H''\}\to H''\xrightarrow{\varphi}\Z\xrightarrow{1\mapsto g}\ll g\rr
\]
is a virtual retraction onto the cyclic group generated by $g$.

Note that the above together with the argument of (G.\ref{G.GFERF}) and (H.\ref{H.resfinitegreen}) shows that infinite cyclic subgroups
of virtually RFRS groups are separable.
If we combine this observation with (G.\ref{G.akmw}), (G.\ref{G.PW12a}) and (G.\ref{G.liu11}) we see that infinite cyclic subgroups of
compact, orientable $3$-manifolds with empty or toroidal boundary which are not graph manifolds are separable. This was proved earlier for all $3$-manifolds by E.~Hamilton  (see \cite{Hamb01}).

In Proposition~\ref{prop:nonretract} we  show that there exist Seifert fibered manifolds, and also graph manifolds with non-trivial JSJ decomposition, whose fundamental group does not virtually retract onto all cyclic subgroups.

\item \label{G.19}\label{G. vb_1=infty} Let $\pi$ be an infinite torsion-free group which is not virtually abelian and which retracts virtually  onto its cyclic subgroups. An elementary argument using the transfer map shows that    $vb_1(\pi;\Z)=\infty$. (See, e.g., \cite[Theorem~2.14]{LoR08a} for a proof.)
\item \label{G.20}\label{G.agol08}  Let $N$ be a compact aspherical $3$-manifold with empty or toroidal boundary such that $\pi_1(N)$ is  RFRS.
Agol \cite[Theorem~5.1]{Ag08}
(see also \cite[Theorem~5.1]{FKt12}) showed that $N$ is virtually fibered. In fact Agol proved a more refined statement. If $\phi\in H^1(N;\Q)$ is a non-fibered class, then there exists a finite solvable cover $p\colon N'\to N$
(in fact a cover which corresponds to one of the $\pi_i$ in the definition of RFRS) such that
$p^*(\phi)\in H^1(N';\Q)$ lies on the boundary of a fibered cone of the Thurston norm ball of $N'$.
(We refer to \cite{Thu86a} and Section \ref{section:thurstonnorm} for background on the Thurston norm and fibered cones.)

Let $N$ be an irreducible  $3$-manifold with empty or toroidal boundary, which is   not a graph manifold. We will show
in Proposition~\ref{prop:manyfiberedfaces} (see also \cite[Theorem~7.2]{Ag08} for the hyperbolic case) that  there exist  finite covers of $N$ with arbitrarily many  inequivalent fibered faces.

Agol \cite[Theorem~6.1]{Ag08} also proves a corresponding theorem for $3$-manifolds with non-toroidal boundary.
More precisely, if $(N,\gamma)$ is a connected taut-sutured manifold such that $\pi_1(N)$ is virtually RFRS then there exists a finite-sheeted cover $(\ti{N},\ti{\gamma})$ of $(N,\gamma)$ with a depth-one taut-oriented foliation. We refer to \cite{Gab83a,Ag08,CdC03} for background information and precise definitions.

Let $p\colon N\to S^1$ be a fibration with surface fiber $\Sigma$. We obtain a short exact sequence
\[ 1\to \G:=\pi_1(\Sigma)\to \pi=\pi_1(N)\to \Z=\pi_1(S^1)\to 1.\]
This sequence splits and we see that $\pi$ is isomorphic to the semidirect product\index{group!semidirect product}
\[ \Z\ltimes_\varphi \G:=\ll \pi,t\,|\, tgt^{-1}=\varphi(g), g\in \G\rr,\quad\text{for some $\varphi\in \Aut(\G)$.}\]
Let $\varphi,\psi\in \Aut(\G)$. A straightforward argument shows that there exists
an isomorphism $\Z\ltimes_\varphi \G\to \Z\ltimes_\psi \G$ which commutes with the canonical projections to $\Z$
if and only if there exists an $h\in \G$ and a $\a\in \Aut(\G)$ such that
$h\varphi(g)h^{-1}=(\a\circ \psi\circ \a^{-1})(g)$ for all $g \in \G$.

In (C.\ref{C.haken}) we saw that a `generic' closed, orientable $3$-manifold is a rational homology sphere, in particular not fibered.
Dunfield and D. Thurston \cite{DnTb06} showed that a random tunnel number one $3$-manifold, which has one toroidal boundary component, does not fiber over the circle.
\item\label{G.21}
     Let $N$ be a virtually fibered $3$-manifold such that $\pi_1(N)$ is not virtually solvable. Jaco--Evans~\cite[p.~76]{Ja80} showed that $\pi_1(N)$ does not have the~f.g.i.p.

Combining this result with the ones above and with work of So\-ma~\cite{Som92}, we obtain the following:
Let $N$ be a compact $3$-manifold with empty or toroidal boundary.  Then $\pi_1(N)$ has the f.g.i.p.~if and only if $\pi_1(N)$ is finite or solvable.  Indeed, if $\pi_1(N)$ is finite or solvable, then $\pi_1(N)$ is virtually polycyclic and so every subgroup is finitely generated.
(See \cite[Lemma~2]{Som92} for details.) If $N$ is Seifert fibered and $\pi_1(N)$ is neither finite nor solvable,
then  $\pi_1(N)$ does not have  the f.g.i.p.~(See again \cite[Proposition~3]{Som92} for details.) It follows from
the combination of (G.\ref{G.akmw}), (G.\ref{G.special=>subgroup of a RAAG}), (G.\ref{G.RAAGvRFRS}), (G.\ref{G.agol08})
and the above mentioned result of Jaco and Evans, that if $N$ is hyperbolic, then $\pi_1(N)$ does not have the f.g.i.p.\
Finally, if $N$ has non-trivial JSJ decomposition, then by the above already the fundamental group of a JSJ component
does not have the f.g.i.p., hence $\pi_1(N)$ does not have the f.g.i.p.

We refer to \cite[Theorem~1.3]{Hem85b} for examples of $3$-manifolds with non-toroidal boundary which have the f.g.i.p.

\item \label{G.22} Let $N$ be a fibered $3$-manifold. Then there exists an epimorphism $\pi_1(N)\to \Z$ whose kernel equals $\pi_1(\Sigma)$, where $\Sigma$ is a compact surface.
If $\Sigma$ has boundary, then $\pi_1(\Sigma)$ is free and $\pi_1(N)$ is poly-free. If $\Sigma$ is closed, then the kernel of any epimorphism $\pi_1(\Sigma)\to \Z$ is a free group. It follows easily that again $\pi_1(N)$ is poly-free.

\item \label{G.23}Hermiller--\v{S}uni\'c \cite[Theorem~A]{HeS07} have shown that any RAAG is poly-free.  It is clear that any subgroup of a poly-free group is also poly-free.
\item \label{G.good}
If $N$ is a fibered $3$-manifold, then $\pi_1(N)$ is the semidirect product of $\Z$ with a surface group,
and so $\pi_1(N)$ is good by Propositions~3.5 and~3.6 of \cite{GJZ08}.

Let $N$ be a compact, orientable,  irreducible $3$-manifold with empty or toroidal boundary. Wilton--Zalesskii \cite[Corollary~C]{WZ10}  showed that if  $N$ is
a graph manifold, then $\pi_1(N)$ is good.
It follows from (H.\ref{H.goodgreen}), (G.\ref{G.2}), (G.\ref{G.3}), (G.\ref{G.5}), (G.\ref{G.17}) and (G.\ref{G.20}) that if $N$ is
 not a graph manifold, then $\pi_1(N)$ is good.

 Cavendish \cite[Section~3.5,~Lemma~3.7.1]{Cav12}, building on the results of Wise, showed that the fundamental group of any compact $3$-manifold is good.

\item \label{G.propfd} If $N$ is a fibered $3$-manifold, then $\pi_1(N)$ is a semidirect product of $\Z$ with a surface group, and
\cite[Theorem~2.8]{LuSh04} implies that $\pi_1(N)$ has Property~FD.
It follows from (H.\ref{H.fdgreen}), (G.\ref{G.2}), (G.\ref{G.3}), (G.\ref{G.5}), (G.\ref{G.17}) and (G.\ref{G.20}) that if $N$ is a compact, orientable,  irreducible $3$-manifold with empty or toroidal boundary, which
is
not a closed graph manifold, then
 $\pi_1(N)$ has Property~FD.

\item \label{G.24}Hsu and Wise \cite[Corollary~3.6]{HsW99} showed that any RAAG is linear over~$\Z$ (see also \cite[p.~231]{DJ00}). The
idea of the proof is  that any RAAG
embeds in a  right angled Coxeter group, and these are known to be linear over $\Z$ (see for example
 \cite[Chapitre~V,~\S~4,~Section~4]{Bou81}).

\item \label{G.25}\label{G.dk92} The lower central series $(\pi_n)$ of a group $\pi$ is defined inductively via $\pi_1:=\pi$ and $\pi_{n+1}=[\pi,\pi_{n}]$. \index{group!lower central series}
If $\pi$ is a RAAG, then the lower central series $(\pi_n)$ of $\pi$ intersects to the trivial group and the successive quotients $\pi_n/\pi_{n+1}$ are free abelian groups.  This was proved by Duchamp--Krob \cite[p.~387 and p.~389]{DK92} (see also  \cite[Section~III]{Dr83}).
This implies that any RAAG (and hence any subgroup of a RAAG) is residually torsion-free nilpotent.

\item \label{G.26}\label{G.gruenberg} Gruenberg \cite[Theorem~2.1]{Gru57} showed that every torsion-free nilpotent group is residually $p$ for any prime $p$.

\item \label{G.27}Any group $\pi$ which is residually $p$ for all primes $p$ is characteristically potent
(see for example \cite[Proposition~2.2]{BuM06}).
  We  refer to \cite[Section~10]{ADL11} for more information and references on potent groups.

\item \label{G.28}\label{G.rhemtulla} Rhemtulla \cite{Rh73} showed that a group which is residually $p$ for infinitely many primes $p$ is bi-orderable.

Note that the combination of (G.\ref{G.dk92}) and (G.\ref{G.gruenberg}) with \cite{Rh73} implies that
RAAGs are bi-orderable. This result was also proved directly by  Duchamp--Thibon \cite{DpT92}.

Let  $N$ be a compact irreducible orientable $3$-manifold with empty or toroidal boundary.
Boyer--Rolfsen--Wiest \cite[Question~1.10]{BRW05} asked whether
$3$-manifold groups are virtually bi-orderable. Chasing through the diagram we see that the question has an affirmative answer if $N$ is a non-positively curved $3$-manifold.
By Theorem \ref{thm:leeb} it thus remains to address the question for graph manifolds which are not  non-positively curved.
\item \label{G.29}\label{G.omnipotent} Theorem 14.26 of \cite{Wis12a} asserts that word-hyperbolic groups (in particular fundamental groups of closed hyperbolic $3$-manifolds) which are virtually special are omnipotent.

Wise observes in \cite[Corollary~3.15]{Wis00} that if $\pi$ is an omnipotent, torsion-free group and if $g,h\in \pi$
are two elements with $g$ not conjugate to $h^{\pm 1}$, then there exists an epimorphism $\a\colon\pi\to G$ to a finite group  such that
the orders of $\a(g)$ and $\a(h)$ are different. This can be viewed as strong form of conjugacy separability for pairs of elements
$g$, $h$ with $g$ not conjugate to $h^{\pm 1}$.

Wise also states that a corresponding result holds in the cusped case \cite[Remark 14.27]{Wis12a}.  However, it is not the case that cusped hyperbolic manifolds necessarily satisfy the definition of omnipotence given in (E.\ref{E.independent}).  Indeed, it is easy to see that $\Z^2$ is not omnipotent (see \cite[Remark 3.3]{Wis00}), and also that a retract of an omnipotent group is omnipotent.  However, there are many examples of cusped hyperbolic $3$-manifolds $N$ such that $\pi_1(N)$ retracts onto a cusp subgroup (see, for instance, (G.\ref{G.virtualretract})).  Therefore, the fundamental group of such a $3$-manifold $N$ is not omnipotent.
\end{list}


Most of the `green properties' are either green by definition or for elementary reasons.
We thus will only justify the following  statements.
\newcounter{itemcounterp}
\begin{list}
{{(H.\arabic{itemcounterp})}}
{\usecounter{itemcounterp}\leftmargin=2em}
\item \label{H.retracttozgreen} Let  $\pi$ be any group. Long--Reid  \cite[Proof~of~Theorem~4.1.4]{LoR05}
 (or \cite[Proof~of~Theorem~2.10]{LoR08a}) showed that the  ability to retract onto linear subgroups  of a group $\pi$ extends to finite index supergroups of $\pi$. We get the following conclusions:
\bn[(a)]
 \item Since cyclic subgroups are linear it follows that if a finite-index subgroup of $\pi$ retracts onto cyclic subgroups, then $\pi$ also retracts onto cyclic subgroups.
 \item If $N$ is hyperbolic and if $N$ admits a finite cover $N'$ such that $\pi'=\pi_1(N')$ retracts onto geometrically finite subgroups,
 then it follows from the above and from the linearity of $\pi=\pi_1(N)$ that $\pi$ also retracts onto geometrically finite subgroups.
 \en
\item \label{H.resfinitegreen}
 Let $\pi$ be a group which admits a finite-index subgroup $\pi'$ which is LERF; then $\pi$ is LERF itself.
 Indeed, let $\G\subseteq \pi$ be a finitely generated subgroup. Then $\G\cap \pi'\subseteq \pi'$ is separable, i.e., closed in (the profinite topology  of) $\pi'$. It then follows that $\G\cap \pi'$ is closed in $\pi$. Finally  $\G$,
 which can be written as a union of finitely many translates of $\G\cap \pi'$, is also closed in $\pi$, i.e., $\G$ is separable in $\pi$.

The same argument shows that the fundamental group of a hyperbolic $3$-manifold, having a  finite-index subgroup which is GFERF, is GFERF.

\item Niblo \cite[Proposition~2.2]{Nib92} showed that a finite-index subgroup of a group $\pi$ is double-coset separable if and only if $\pi$ is double-coset separable.

\item \label{H.lineargreen}
Let $R$ be a commutative ring and $\pi$ be a group which is linear over $R$. Suppose that $\pi$ is a subgroup of finite index of a group   $\pi'$.
Let  $\a\colon \pi\to \gl(n,R)$ be a faithful representation. Then $\pi'$ acts faithfully on $R[\pi']\otimes_{R[\pi]}R^n\cong R^{n[\pi':\pi]}$ by left-multiplication. It follows that $\pi'$ is also linear over $R$.

\item \label{H.goodgreen} It follows from \cite[Lemma~3.2]{GJZ08} that a group is good if it admits a finite-index subgroup which is good.

\item \label{H.fdgreen} By \cite[Corollary~2.5]{LuSh04}, a group with   a finite-index subgroup which has Property~FD also has Property~FD.

\item Let $\pi$ be a group which admits a finite-index subgroup $\pi'$ which is weakly characteristically potent.
Then $\pi$ is also weakly characteristically potent. To see this,
since subgroups of weakly characteristically potent groups are weakly characteristically potent we can by a standard argument assume that $\pi'$ is in fact a characteristic finite-index subgroup of $\pi$.
Now let $g\in \pi$. We denote by $k\in \N$ the minimal number such that $g^k\in \pi'$.
Since $\pi'$  is weakly characteristically potent there exists an $r'\in \N$ such that for any $n\in \N$ there exists
a characteristic finite-index subgroup $\pi_n\subseteq\pi'$ such that $g^k\pi_n$ has order $rn$ in $\pi'/\pi_n$.
We now let $r=r'k$. Note that $\pi_n\subseteq \pi$ is normal since $\pi_n\subseteq\pi'$ is characteristic.
Clearly $g^{rn}=1\in \pi/\pi_n$. Furthermore, if $m$ is such that $g^m\in \pi_n$, then $g^m\in \pi'$,
hence $m=km'$. It now follows easily that $m$ divides $rn=kr'n$. Finally note that $\pi_n$ is characteristic in $\pi$ since $\pi_n\subseteq \pi'$ and $\pi'\subseteq \pi$ are characteristic. This shows that $\pi$ is  weakly characteristically potent.

\item \label{H.csgreen} In Theorem~\ref{t: Hereditarily CuS upwards} we  show that
the fundamental group of an irreducible $3$-manifold with empty or toroidal boundary, which has a  hereditarily conjugacy separable subgroup  of finite index, is also hereditarily conjugacy separable.
\end{list}

The following gives a list of  further results and alternative arguments which we left out of Diagram~4.
\newcounter{itemcounterq}
\begin{list}
{{(I.\arabic{itemcounterq})}}
{\usecounter{itemcounterq}\leftmargin=2em}
\item \label{I.hwz11}
 Hamilton--Wilton--Zalesskii \cite[Theorem~1.2]{HWZ13} showed that if $N$ is an orientable closed irreducible $3$-manifold   such that the fundamental group of every JSJ piece is conjugacy separable, then $\pi_1(N)$ is conjugacy separable.  By doubling along the boundary and appealing to Lemma~\ref{lem:closed}, the same result holds for compact, irreducible $3$-manifolds with toroidal boundary.

It follows from (G.\ref{G.akmw}), (G.\ref{G.compact special=>qc subgroup of a RAAG}), (G.\ref{G.vr of RAAG}), (G.\ref{G.minasyan}),
and (H.\ref{H.csgreen}),
that fundamental groups of hyperbolic $3$-manifolds are conjugacy separable.
Moreover, fundamental groups of Seifert fibered manifolds are conjugacy separable
 (see \cite{Mao07,AKT05,AKT10}). The aforementioned result from~\cite{HWZ13} now implies that the fundamental group of any orientable, irreducible $3$-manifold with empty or toroidal boundary is conjugacy separable.

Finally note that if a finitely presented group is conjugacy separable (see (E.\ref{E.defconjugacyseparable}) for the definition), then the argument of \cite[Theorem~IV.4.6]{LyS77} also shows that the conjugacy problem is solvable.
The above results thus give another solution to the Conjugacy Problem first solved by Pr\'eaux (see~(C.\ref{C.conjugacyproblem})).

 \item
Let $N$ be an irreducible, non-spherical compact, orientable $3$-manifold with empty or toroidal boundary.
Tracing through the arguments of Diagram~1 and Diagram~4 shows that $vb_1(N)\geq 1$.
It follows from  \cite[p.~444]{Lub96a} that the group $\pi_1(N)$ does not have Property ($\tau$).

This answers in particular the Lubotzky-Sarnak Conjecture (see \cite{Lub96a} and \cite{Lac11} and also Section \ref{section:history} for details) in the affirmative which states
that there exists no hyperbolic $3$-manifold such that its fundamental group has Property~($\tau$).

Note that a group which does not have Property ($\tau$) also does not have Kazhdan's Property~($T$),
see; e.g., \cite[p.~444]{Lub96a} for details and see \cite{BdlHV08} for background on Kazhdan's Property~($T$). This shows that
 the fundamental group of a compact, orientable, irreducible, non-spherical $3$-manifold with empty or toroidal boundary  does not satisfy Kazhdan's Property~($T$). This result was first obtained by Fujiwara \cite{Fuj99}. \index{group!with Property ($T$)}

\item It follows from  the combination of
(G.\ref{G.liu11}), (G.\ref{G.special=>subgroup of a RAAG}), (G.\ref{G.RAAGvRFRS}) and (\ref{G.agol08})
that
non-positively curved graph manifolds (e.g., graph manifolds with non-empty boundary, see \cite[Theorem~3.2]{Leb95}) are virtually fibered.
Wang--Yu \cite[Theorem~0.1]{WY97} proved directly that graph manifolds with non-empty boundary are virtually fibered
(see also \cite{Nemb96}), and Svetlov \cite{Sv04} proved that non-positively curved graph manifolds are virtually fibered.
\item Baudisch \cite[Theorem~1.2]{Bah81} showed that if $\G$ is a $2$-generator subgroup of a RAAG, then $\G$ is either a free abelian group or a free group.
\item \label{I.nonlerf} Fundamental groups of graph manifolds are in general not LERF (see, e.g.,  \cite{BKS87}, \cite[Theorem~5.5]{Mat97a}, \cite[Theorem~2.4]{Mat97b}, \cite{RW98}, \cite[Theorem~4.2]{NW01} and Section~\ref{section:questionsep}). In fact, there are finitely generated surface subgroups of graph manifold groups that are not contained in any proper subgroup of finite index \cite[Theorem~1]{NW98}. On the other hand, Przytycki--Wise \cite[Theorem~1.1]{PW11} have shown that if $N$ is a graph manifold and $\Sigma$ is an oriented incompressible surface which is embedded in~$N$, then $\pi_1(\Sigma)$ is separable in $\pi_1(N)$.
\item Several results on fundamental groups of hyperbolic $3$-manifolds with non-empty boundary can be deduced from the closed case. (Recall that, according to our convention, we only consider hyperbolic $3$-manifolds of finite volume.)  More precisely, the following hold:
\bn
\item[(a)]
Every hyperbolic $3$-manifold $N$ has a closed hyperbolic Dehn filling~$M$, and so $\pi_1(N)$ surjects onto $\pi_1(M)$.  In particular, the fact that the fundamental group of every closed hyperbolic $3$-manifold is large gives a new proof of the theorem of Cooper--Long--Reid that the same is true for fundamental groups of hyperbolic $3$-manifolds with non-empty boundary \cite{CLR97}.
\item[(b)] Further, it follows from the work of Groves--Manning \cite[Corollary~9.7]{GrM08} or Osin \cite[Theorem~1.1]{Osi07}
that given any hyperbolic $3$-manifold $N$ with non-empty boundary and given any finite set $A\subseteq \pi_1(N)$, there exists a hyperbolic Dehn filling $M$ of $N$ such that the induced map $\pi_1(N)\to \pi_1(M)$ is injective when restricted to $A$.
\item[(c)] Manning--Martinez-Pedroza \cite[Proposition~5.1]{MMP10} showed that if the fundamental groups of all closed hyperbolic $3$-manifolds are LERF, then the fundamental groups of all hyperbolic $3$-manifolds with non-empty boundary are also LERF.
 \en
\item Droms \cite[Theorem~2]{Dr87} showed that a RAAG group corresponding to a graph $G$ is the fundamental group of a $3$-dimensional manifold if and only if each connected component of $G$ is either a tree or a triangle.
\item \label{I.fibresp} If $N$ is a fibered $3$-manifold, then $\pi_1(N)=\Z\ltimes F$ where $F$ is  a surface group. Surface groups are well known to be residually $p$, and it is also well known that the semidirect product of a residually~$p$ group with $\Z$ is virtually residually~$p$.
 We refer to \cite[Corollary~4.32]{AF10} and \cite{Kob10} for a full proof.
 \item Bridson \cite[Corollary~5.2]{Brd12} (see also \cite[Proposition~1.3]{Kob12c}) showed that if a  group  has a subgroup of finite index that embeds in a RAAG, then it embeds in the mapping class
 group for infinitely many closed surfaces. By the above results this applies in particular to fundamental groups of
 compact, orientable, closed, irreducible $3$-manifolds with empty or toroidal boundary which are not closed graph manifolds. `Most' $3$-manifold groups thus can be viewed as subgroups of mapping class groups.
\item\label{I.grothendieckrigid}
A finitely generated group $\pi$ and a proper subgroup $\Gamma$ form a \emph{Grothendieck pair} $(\pi,\Gamma)$ if the inclusion map $\Gamma\hookrightarrow\pi$ induces an isomorphism of profinite completions.  A finitely generated group $\pi$ is called \emph{Grothendieck rigid} if $(\pi,\Gamma)$ is never a Grothendieck pair, for each finitely generated subgroup $\Gamma$ of $\pi$. \index{subgroup!Grothendieck pair} \index{group!Grothendieck rigid}

Platonov and Tavgen' exhibited a residually finite group which is not Grothendieck rigid \cite{PT86}.  Bridson and Grunewald \cite{BrGd04} answered a question of Grothendieck \cite[p.~384]{Grk70} by giving an example of a Grothen\-dieck pair $(\pi,\Gamma)$ such that both $\pi$ and $\Gamma$ are residually finite and finitely presented.

Note that LERF groups are Grothendieck rigid: if $\Gamma$ is finitely generated and a proper subgroup of $\pi$ then the inclusion map $\G\hookrightarrow\pi$ does not induce a surjection on profinite completions. It thus follows from the above (see in particular (C.\ref{C.sfssubgroupseparable}) and (G.\ref{G.LERF})) that fundamental groups of Seifert fibered spaces and hyperbolic $3$-manifolds are Grothendieck rigid. This result was first obtained by Long and Reid \cite{LoR11}.  Cavendish \cite[Proposition~3.7.1]{Cav12} used (G.\ref{G.good}) to show that the fundamental group of any closed prime $3$-manifold is Grothendieck rigid.
\item
\label{I.weaklyamenable}
In (C.\ref{C.21}) we saw that most $3$-manifold groups are not amenable.  On the other hand, if $N$ is an irreducible $3$-manifold which is not a closed graph manifold, then it follows from Theorems \ref{thm:akmw}, \ref{thm:liu11} and \ref{thm:pw12} together with work of Mizuta \cite[Theorem~3]{Miz08} and Guentner--Higson \cite{GuH10} \index{group!weakly amenable} that $\pi_1(N)$ is `weakly amenable'.  For closed hyperbolic $3$-manifolds this also follows from~\cite{Oza08}.

\end{list}

\section{Subgroups of $3$-manifold groups}\label{section:subgroups}

\noindent
In this section we collect properties of finitely generated infinite-index subgroups of $3$-manifold groups in a diagram. The study of $3$-manifold groups and the study of their subgroups go hand in hand, and the content of this section therefore partly overlaps with the results mentioned in the previous sections.

\medskip

Most of the definitions required for understanding Diagram~5 have been introduced above.
We therefore need to introduce only the following new definitions.
\newcounter{itemcountertj}
\begin{list}
{{(J.\arabic{itemcountertj})}}
{\usecounter{itemcountertj}\leftmargin=2em}
\item Let $N$ be a $3$-manifold. Let $\G\subseteq \pi_1(N)$ be a subgroup and $X\subseteq N$ a connected subspace.
We say that \emph{$\G$ is carried by $X$} if $\G$ is a subgroup of $\im\{\pi_1(X)\to \pi_1(N)\}$ (up to conjugation). \index{subgroup!carried by a subspace}

\item Let $\pi$ be a finitely generated group and $\G$ be a finitely generated subgroup of $\pi$. We say that the \emph{membership  problem is solvable for $\G$} if, given a finite generating set $g_1,\dots,g_k$ for $\pi$,  there exists an algorithm which can determine whether or not an input word in the generators $g_1,\dots,g_k$ defines an element of $\G$. \index{subgroup!membership problem}

\item Let $N$ be a $3$-manifold. We say that a connected compact surface $\Sigma\subseteq N$ is a \emph{semifiber} if $N$ is the union
of two twisted $I$-bundles over the non-orientable surface $\Sigma$ along their $S^0$-bundles.
(Note that in the literature usually the surface given by the $S^0$-bundle is referred to as a `semifiber'.)
 Note that if $\Sigma\subseteq N$ is a semifiber, then there exists a double cover $p\colon\wti{N}\to N$ such that $p^{-1}(\Sigma)$ consists of two components, each of which is a surface fiber. \index{surface!semifiber} \index{semifiber}

\item  Let $\G$ be  a subgroup of a group $\pi$.  The \emph{width} of $\G$ in $\pi$ was defined in~\cite{GMRS98}.  We say that $g_1,\dots ,g_n\in \pi$ are \emph{essentially distinct} (with respect to $\G$) if $\G g_i\not=\G g_j$ whenever $i\ne j$. Conjugates of $\G$ by essentially distinct elements are called \emph{essentially distinct conjugates}. The \emph{width} of $\G$ in $\pi$ is the maximal $n\in \N\cup\{\infty\}$ such that there exists a collection of $n$ essentially distinct conjugates of $\G$ with the property that the intersection of any two elements of the collection is infinite. The width of $\G$ is $1$ if $\G$ is malnormal.  If $\Gamma$ is normal and infinite, then the width of $\Gamma$ equals its index.\index{subgroup!width}

\item Let $\G$ be a subgroup of a group $\pi$. \index{subgroup!commensurator}
We define the commensurator subgroup of $\G$ to be the subgroup
\[ \operatorname{Comm}_\pi(\G):=\big\{ g\in \pi : \text{$\G \cap g\G g^{-1}$ has finite index in $\G$} \big\}.\]

\item Let $\PP$ be a property of subgroups of a given group.
 We say that a subgroup $\G$ of a group $\pi$ is \emph{virtually $\PP$} (in $\pi$) if $\pi$ admits a (not necessarily normal) subgroup $\pi'$ of finite index
 which contains $\G$ and such that $\G$, viewed as subgroup of $\pi'$,  satisfies $\PP$.
\end{list}
As in Diagrams~1 and~4 we use the convention
that if an arrows splits into several arrows, then exactly one of the possible conclusions holds.
Furthermore, if  an arrow is decorated with a condition, then the conclusion  holds if that condition is satisfied.

\begin{figure}
\psfrag{1}{\ref{K.1}}\psfrag{2}{\ref{K.2}}\psfrag{3}{\ref{K.3}}\psfrag{4}{\ref{K.4}}
\psfrag{5}{\ref{K.5}}\psfrag{6}{\ref{K.6}}\psfrag{7}{\ref{K.7}}\psfrag{8}{\ref{K.8}}
\psfrag{9}{\ref{K.9}}\psfrag{10}{\ref{K.10}}
\psfrag{11}{\ref{K.11}}\psfrag{12}{\ref{K.12}}\psfrag{13}{\ref{K.13}}\psfrag{14}{\ref{K.14}}
\psfrag{15}{\ref{K.15}}\psfrag{16}{\ref{K.16}}\psfrag{17}{\ref{K.17}}\psfrag{18}{\ref{K.18}}
\psfrag{19}{\ref{K.19}}\psfrag{20}{\ref{K.20}}
\psfrag{21}{\ref{K.21}}\psfrag{22}{\ref{K.22}}
\hspace{-1cm}
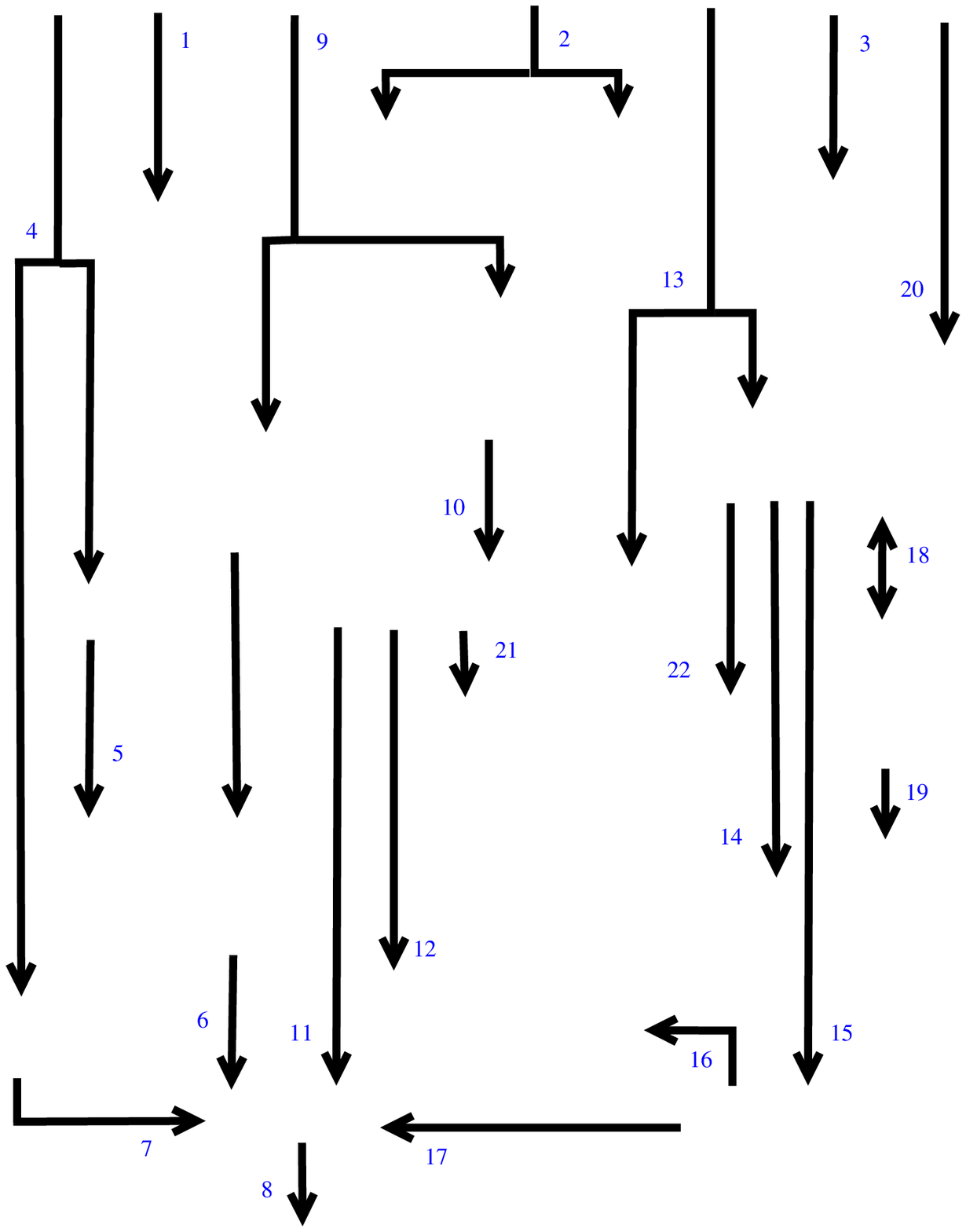
\end{figure}

 In Diagram~5 we put several restrictions on the $3$-manifold~$N$ which we consider.
 Below, in the justifications for the arrows in Diagram~5, we will only assume that~$N$ is connected,
 and we will not put any other blanket restrictions on $N$.
 Before we give the justifications we point out that only (K.\ref{K.virtualretract}) depends on the Virtually Compact Special Theorem.
\newcounter{itemcountert}
\begin{list}
{{(K.\arabic{itemcountert})}}
{\usecounter{itemcountert}\leftmargin=2em}
\item \label{K.1} It follows from the  Sphere Theorem (see Theorem~\ref{thm:sphere}) that each  compact, irreducible $3$-manifold with empty or toroidal boundary, whose fundamental group has
a non-trivial finite subgroup, is spherical.  See (C.\ref{C.pi1torsionfree}) for details.
\item \label{K.2} \label{K.tits} Let $N$ be any compact $3$-manifold and let $\G$ be a finitely generated
subgroup of $\pi=\pi_1(N)$. Then
\bn[(a)]
\item either $\G$ is virtually solvable, or
\item $\G$ contains a non-cyclic free subgroup.
\en
(In other words, $\pi$ satisfies the `Tits Alternative'.) \index{$3$-manifold group!Tits Alternative} Indeed, if $\G$ is a finitely generated subgroup of $\pi=\pi_1(N)$, then by Scott's Core Theorem
  (C.\ref{C.scottcore}), applied to the covering of $N$ corresponding to $\G$,
  there exists a compact $3$-manifold $M$ with $\pi_1(M)=\G$.
  Suppose that $\pi_1(M)$ is not virtually solvable.
  It follows easily from Theorem~\ref{thm:prime}, Lemma~\ref{lem:incomp}, Lemma ~\ref{lem:free subgroup}
  combined with (C.\ref{C.em72free}) and (C.\ref{C.tits}) that $\pi_1(M)$ contains a non-cyclic free subgroup.

\item \label{K.3} \label{K.scottcore} Scott \cite{Sco73b} proved that any finitely generated $3$-manifold group is also finitely presented.  See (C.\ref{C.scottcore}) for more information.

\item \label{K.4} Let $N$ be a compact, orientable, irreducible $3$-manifold with empty or toroidal boundary. Let $\G\subseteq \pi_1(N)$ be an abelian subgroup. It follows either from
Theorems~\ref{thm:centgen} and~\ref{thm:centsfs}, or alternatively from
the  remark after the proof of Theorem~\ref{thm:solv} and the Core Theorem  (C.\ref{C.scottcore}),
that $\G$ is either  cyclic or $\G\cong \Z^2$ or $\G\cong \Z^3$. In the latter case it follows from
the discussion in Section~\ref{section:solv} that $N$ is the $3$-torus and that $\G$ is a finite-index subgroup of $\pi_1(N)$.
\item \label{K.5} Let $N$ be a compact, orientable, irreducible, irreducible $3$-manifold with empty or toroidal boundary
and let $\G\subseteq \pi_1(N)$ be a subgroup isomorphic to $\Z^2$.
 Then there exists
a singular map $f\colon T\to N$ from the $2$-torus to $N$ such that $f_*(\pi_1(T))=\G$.
It now follows from Theorem ~\ref{thm:charsub} that $\G$ is carried by a characteristic submanifold. \index{theorems!Torus Theorem}

The above statement is also known as the `Torus Theorem.'  It was announced by Waldhausen \cite{Wan69}
and the first proof was given by Feustel \cite[p.~29]{Fe76a} and \cite[p.~56]{Fe76b}. We refer to \cite{Wan69,CF76,Milb84,Sco80,Sco84} for information on the closely related `Annulus Theorem.'
Both theorems can be viewed as predecessors of the Characteristic Pair Theorem.

\item \label{K.6} Let $N$ be a compact orientable $3$-manifold with empty or toroidal boundary. Let $M\subseteq N$ be a characteristic submanifold
and let $\G\subseteq \pi_1(M)$ be a finitely generated subgroup.
Then $\G$ is separable in $\pi_1(M)$ by Scott's theorem \cite[Theorem~4.1]{Sco78} (see also  (C.\ref{C.sfssubgroupseparable})),
and $\pi_1(M)$ is separable in $\pi_1(N)$ by  Wilton--Zalesskii \cite[Theorem~A]{WZ10} (see also (C.\ref{C.efficient})).  It follows that $\G\subseteq \pi_1(N)$ is separable.  Note that the same argument also generalizes to hyperbolic JSJ components with LERF fundamental groups.  More precisely, if $M$ is a hyperbolic JSJ component of $N$ such that $\pi_1(M)$ is LERF and if $\G\subseteq \pi_1(M)$ is a finitely generated subgroup, then $\G\subseteq \pi_1(N)$ is separable.

\item \label{K.7}  E. Hamilton \cite{Hamb01} showed that the fundamental group of any compact, orientable $3$-manifold  is abelian subgroup separable.  In particular, any cyclic subgroup is separable.  See also (C.\ref{C.AERF}) for more information.

It follows from (G.\ref{G.RFRSretracttoz}) that
if $\pi_1(N)$ is virtually RFRS, then every infinite cyclic subgroup $\G$ of $\pi_1(N)$ is a virtual retract of $\pi_1(N)$; by (K.\ref{K.retractseparable}), this gives another proof that $\G$ is separable.

\item \label{K.8} \label{K.membership}  The argument of  \cite[Theorem~IV.4.6]{LyS77} can be used to show that if $\G\subseteq \pi$ is a finitely generated separable subgroup of a finitely presented group $\pi$, then the membership problem for $\G$ is solvable.

\item\label{K.9} \label{K.stallings62} Let $N$ be a compact orientable $3$-manifold, and let $\G$ be
a normal finitely generated non-trivial subgroup of $\pi_1(N)$ of infinite index. Work of Hempel--Jaco  \cite[Theorem~3]{HJ72}, the resolution of the Poincar\'e Conjecture, Theorem~\ref{thm:infcyclicnormal}, and (K.\ref{K.scottcore}) imply  that one of the following conclusions hold:
\bn[(a)]
\item $N$ is Seifert fibered and $\G$ is a subgroup of the Seifert fiber subgroup, or
\item $N$ fibers over $S^1$ with surface fiber $\Sigma$ and $\G$ is a finite-index subgroup of $\pi_1(\Sigma)$, or
\item $N$ is the union of two twisted $I$-bundles over a compact connected surface $\Sigma$ which meet in the corresponding $S^0$-bundles
and $\G$ is a finite-index subgroup of $\pi_1(\Sigma)$.
    \en
In particular, if $\G=\ker(\phi)$  for some homomorphism $\phi\colon\pi\to \Z$, then $\phi=p_*$ for some surface bundle  $p\colon N\to S^1$.
This special case was first proved by Stallings \cite{Sta62} and this statement is known as Stallings' Fibration Theorem.
Generalizations to subnormal groups were formulated and proved by Elkalla \cite[Theorem~3.7]{El84} and Bieri--Hillman \cite{BiH91}.

Let $N$ be an compact, orientable, irreducible $3$-manifold.
Heil \cite[p.~147]{Hei81} (see also \cite{Ein76a} and \cite{HeR84}) showed that if $\G\subseteq \pi_1(N)$ is a subgroup
carried by a closed $2$-sided incompressible orientable non-fiber surface,
then   $i_*(\pi_1(\Sigma))$ is its own normalizer in $\pi_1(N)$ unless
 $\Sigma$ bounds in $N$ a twisted line bundle over a closed surface.
 Furthermore Heil \cite[p.~148]{Hei81} showed that if $\G\subseteq \pi_1(N)$ is a subgroup
carried by a $3$-dimensional submanifold $M$, then  $i_*(\pi_1(\Sigma))$ is its own normalizer in $\pi_1(N)$ unless $M\cong \Sigma \times I$ for some surface $\Sigma$.
This generalizes earlier work by Eisner \cite{Ein77a}.

\item \label{K.10} \label{K.Virtual fibre} Let $N$ be a compact $3$-manifold. Let $\G$ be a normal subgroup of $\pi_1(N)$ which is also a finite-index subgroup of $\pi_1(\Sigma)$, where $\Sigma$ is a surface fiber of a surface bundle $N\to S^1$. Since $\G\subseteq \pi_1(N)$ is normal, we have a subgroup
$\wti{\pi}:=\Z\ltimes \G$ of $\Z\ltimes \pi_1(\Sigma)=\pi_1(N)$.
We denote by $\wti{N}$ the finite cover of $N$ corresponding to $\wti{\pi}$. It is clear that $\G$ is a surface fiber subgroup of $\wti{N}$.

\item \label{K.11} \label{K.Separable fibre}
\label{K.fibersep}
Let $N$ be a compact orientable $3$-manifold with empty or toroidal boundary and let $\Sigma$ be a fiber of a surface bundle $N\to S^1$.  Then $\pi_1(N)\cong \Z\ltimes \pi_1(\Sigma)$ and $\pi_1(\Sigma)\subseteq \pi_1(N)$ is therefore separable.  It follows easily that if every virtual surface fiber subgroup of $\pi_1(N)$  is separable.

\item\label{K.12}
Let $N$ be a compact orientable $3$-manifold with empty or toroidal boundary and let $\Sigma$ be a fiber of a surface bundle $N\to S^1$. Then $\pi_1(N)\cong \Z\ltimes \pi_1(\Sigma)$ where $1\in \Z$ acts by some $\Phi\in \aut(\pi_1(\Sigma))$.
Now let $\G\subseteq \pi_1(\Sigma)$ be a finite-index subgroup. Because $\pi_1(\Sigma)$ is finitely generated, there are only finitely many subgroups of index $[\pi_1(\Sigma):\Gamma]$, and so $\Phi^n(\G)=\G$ for some~$n$.  Now $\wti{\pi}:=n\Z\ltimes \G$ is a subgroup of finite index in $\pi_1(N)$ such that $\wti{\pi}\cap \pi_1(\Sigma)=\G$. This shows that $\pi$ induces the full profinite topology on the surface fiber subgroup $\pi_1(\Sigma)$. It follows easily that $\pi$ also induces the full profinite topology on any virtual surface fiber subgroup.

\item\label{K.13} \label{K.tameness}
Let $N$ be a hyperbolic $3$-manifold.
The Subgroup Tameness Theorem (see Theorem~\ref{thm:subgroupdichotomy}) asserts that if
$\G\subseteq \pi_1(N)$ is a finitely generated subgroup, then $\G$ is either geometrically finite
or  $\G$  is a virtual surface fiber subgroup.
See (K.\ref{K.14}), (K.\ref{K.15}), (K.\ref{K.19}), (K.\ref{K.22}) for other formulations of this fundamental dichotomy.

\item \label{K.14} Let $N$ be a hyperbolic $3$-manifold and let $\G\subseteq \pi=\pi_1(N)$ be a geometrically finite subgroup of infinite index.
    Then $\G$ has finite index in $\operatorname{Comm}_\pi(\G)$.  We refer to \cite[Theorem~8.7]{Cay08} for a proof (see also  \cite{KaS96} and \cite[Theorem~2]{Ar01}), and we refer to \cite[Section~5]{Ar01} for more results in this direction.

If $\G\subseteq \pi_1(N)$ is a virtual surface fiber subgroup, then  $\operatorname{Comm}_\pi(\G)$ is easily seen to be a finite-index subgroup of $\pi$, so $\G$ has infinite index in its commensurator.
The commensurator thus gives another way to formulate the dichotomy of (K.\ref{K.tameness}).

\item \label{K.15} \label{K.virtualretract} Let $N$ be a hyperbolic $3$-manifold and $\G\subseteq \pi_1(N)$ be a geometrically finite subgroup.
In (G.\ref{G.virtualretract}) we saw
that it follows from the Virtually Compact Special Theorem  that $\G$ is a virtual retract of $\pi_1(N)$.

On the other hand, it is straightforward to see that if $\G$ is a virtual surface fiber subgroup of $\pi_1(N)$ and if the monodromy of the surface bundle
does not have finite order (e.g., if $N$ is hyperbolic), then
$\G$ is not a virtual retract of $\pi_1(N)$. We thus obtain one more way to formulate the dichotomy of (K.\ref{K.tameness}).

\item \label{K.16} It is easy to prove that every group induces the full profinite topology on each of its virtual retracts.

\item \label{K.17}  \label{K.retractseparable} Let $\pi$ be a residually finite group (e.g., a $3$-manifold group, see (C.\ref{C.resfinite})). If  $\G\subseteq \pi$ is a virtual retract, then the subgroup $\G$ is also separable in $\pi$. See (G.\ref{G.GFERF}) for details.

\item \label{K.18} \label{K.qc}
Let $N$ be a closed hyperbolic $3$-manifold.
As mentioned in Proposition~\ref{prop:hypwordhyp}, it follows that   $\pi=\pi_1(N)$ is word-hyperbolic, and a subgroup of $\pi$ is geometrically finite if and only if it is quasi-convex  (see \cite[Theorem~1.1 and Proposition~1.3]{Swp93} and also \cite[Theorem~2]{KaS96}).

If $N$ has toroidal boundary, then $\pi=\pi_1(N)$ is not word-hyperbolic, but it is hyperbolic \emph{relative to its collection of peripheral subgroups}.    By \cite[Corollary~1.3]{Hr10}, a subgroup $\Gamma$ of $\pi$ is geometrically finite if and only if it is \emph{relatively quasi-convex}.  The reader is referred to \cite{Hr10} for thorough treatments of the various definitions of relative hyperbolicity and of relative quasi-convexity, as well as proofs of their equivalence.
\index{group!relatively hyperbolic} \index{subgroup!relatively quasi-convex}

\item \label{K.19} \label{K.width} Let $\pi$ be the fundamental group of a hyperbolic $3$-manifold $N$  and let $\G$ be a geometrically finite subgroup of $\pi$. By (K.\ref{K.qc}) this means that $\G$ is a relatively quasi-convex subgroup of $\pi$. The main result of \cite{GMRS98} shows that the width of $\G$ is finite when $N$ is closed (so $\pi$ is word-hyperbolic), and the general case follows from \cite{HrW09}.

If, on the other hand, $\G\subseteq \pi_1(N)$ is a virtual surface fiber subgroup, then the width of $\G$ is infinite. The width thus gives another way to formulate the dichotomy of (K.\ref{K.tameness}).

\item \label{K.20} Let $N$ be a compact, orientable, irreducible $3$-manifold and let $\G\subseteq \pi_1(N)$ be a subgroup of infinite index.
The argument of the proof of \cite[Theorem~6.1]{How82} shows that $b_1(\G)\geq 1$.  See also (C.\ref{C.locallyindicable}).
\item \label{K.21} Each surface fiber subgroup corresponding to a surface bundle $p\colon N\to S^1$ is the kernel of the map
$p_*\colon\pi_1(N)\to \pi_1(S^1)=\Z$. It follows immediately from the definition that a virtual surface fiber subgroup is virtually normal.
\item \label{K.22} Let $\Gamma$ be a subgroup of a torsion-free group $\pi$ and suppose that $\Gamma$ is separable and has finite width.  Let $g_1,\ldots,g_n\in\pi$ be a maximal collection of essentially distinct elements of $\pi\smallsetminus\Gamma$ such that $\Gamma\cap\Gamma^{g_i}$ is infinite for all $i$.  Let $\pi'$ be a subgroup with finite index in $\pi$ that contains $\Gamma$ but does not contain $g_i$ for any $i$.  Then $\Gamma$ is easily seen to be malnormal in $\pi'$; in particular, $\Gamma$ is virtually malnormal in $\pi$.

If we combine this argument with (K.\ref{K.15}), (K.\ref{K.17}) and (K.\ref{K.width}) we see that any geometrically finite subgroup of the fundamental group of a hyperbolic manifold $N$ is virtually malnormal.  (In the closed case, this appears as \cite[Lemma 2.3]{Mac12}.)

Together with (K.\ref{K.21}) we thus see that the dichotomy for subgroups of hyperbolic $3$-manifold groups can be rephrased also in terms of being virtually (mal)normal.
\end{list}

We conclude this section with a few more results and references about subgroups of  $3$-manifold groups.

\newcounter{itemcountertt}
\begin{list}
{{(L.\arabic{itemcountertt})}}
{\usecounter{itemcountertt}\leftmargin=2em}
\item A group is called \emph{locally free}\index{group!locally free} if every finitely generated subgroup is free.
If $N$ is a compact, orientable, irreducible $3$-manifold with empty or toroidal boundary,
 then it follows from (K.\ref{K.4}) that every abelian locally free subgroup of $\pi_1(N)$ is already free.
On the other hand Anderson \cite[Theorem~4.1]{Ana02} and Kent \cite[Theorem~1]{Ken04} gave examples of hyperbolic $3$-manifolds which contain non-abelian subgroups which are locally free but not free.\index{subgroup!locally free}
\item  Let $N$ be a hyperbolic $3$-manifold
 and let $\G\subseteq \pi_1(N)$ be a subgroup generated by two elements $x$ and $y$.
  Jaco--Shalen \cite[Theorem~VI.4.1]{JS79} showed that if $x,y\in \pi_1(N)$ do not commute and if the subgroup $\ll x,y\rr\subseteq \pi_1(N)$ has infinite index, then
   $\G$ is a free group. For closed hyperbolic $3$-manifolds this result was  generalized by Gitik \cite[Theorem~1]{Git99a}.
  \item  Let $N$ be a compact orientable $3$-manifold and let $\phi\colon \pi_1(N)\to \Z$ be an epimorphism.
By (K.\ref{K.stallings62}),  $\phi$ is either induced by the projection of a surface bundle or $\ker(\phi)$ is not finitely generated.

This dichotomy can be strengthened in several ways: if $\phi$ is not induced by the projection of a surface bundle,
then the following hold:
\bn
\item[(a)] $\ker(\phi)$ admits uncountably many subgroups of finite index (see \cite[Theorem~5.2]{FV12a}, \cite{SW09a} and \cite[Theorem~3.4]{SW09b}),
\item[(b)] the pair $(\pi_1(N),\phi)$ has `positive rank gradient' (see \cite[Theorem~1.1]{DFV12}),
\item[(c)] $\ker(\phi)$ admits a finite index subgroup which is not normally generated by finitely many elements (see \cite[Theorem~5.1]{DFV12}),
\item[(d)] if $N$ has non-empty toroidal boundary and if the restriction of $\phi$ to each boundary component is non-trivial,
then $\ker(\phi)$ is not locally free (see \cite[Theorem~3]{FF98}).
\en
Bieri--Neu\-mann--Strebel \cite[Corollary~F]{BNS87} showed that the Bieri--Neu\-mann--Strebel invariant $\Sigma(\pi_1(N))$ is symmetric for any compact $3$-manifold. This implies in particular that if $\ker(\phi)$ is infinitely generated, then we can not write $\pi_1(N)$ as a strictly ascending or strictly descending HNN-extension. More precisely, there exists no commutative diagram
\[ \xymatrix{
\pi_1(N)\ar[dr]^\phi \ar[rr] &&
\ll A,t\,|\, t^{\eps} At^{-\eps}=\varphi(A) \rr \ar[dl]_\psi
\\
&\Z&}
\]
where $\varphi\colon A\to A$ is a monomorphism, where $\eps\in \{-1,1\}$ and where $\psi$ is the map given by $\varphi(t)=1$ and $\varphi(a)=0$ for all $a\in A$.
\item  Let  $\G$ be a finitely generated subgroup of a group $\pi$. We say $\G$ is \emph{tight} in $\pi$
 if for any $g\in \pi$ there exists an $n$ such that $g^n\in \G$. Clearly a finite-index subgroup of $\pi$ is tight.
 Let $N$ be a hyperbolic $3$-manifold.
It follows from the Subgroup Tameness Theorem (see Theorem~\ref{thm:subgroupdichotomy}) that any tight subgroup of $\pi_1(N)$ is of finite index.  For $N$ with non-trivial toroidal boundary, this was first shown by Canary \cite[Theorem~6.2]{Cay94}. \index{subgroup!tight}

\item \label{L.ln91} Let $N$ be an orientable, compact irreducible $3$-manifold with (not necessarily toroidal) boundary.
Let $X$ be a connected, incompressible subsurface of the boundary of $N$.  Long--Niblo \cite[Theorem~1]{LoN91} showed that then $\pi_1(X)\subseteq \pi_1(N)$ is separable.
\item Let $N$ be a compact, orientable $3$-manifold with no spherical boundary components. Let $\Sigma$ be an incompressible
connected subsurface of $\partial N$.
If $\pi_1(\Sigma)$ is a finite-index subgroup of $\pi_1(N)$, then by \cite[Theorem~10.5]{Hem76}
one of the following happens:
\bn
\item[(a)] $N$ is a solid torus, or
\item[(b)] $N=\Sigma\times [0,1]$ with $\Sigma=\Sigma\times 0$,  or
\item[(c)] $N$ is a twisted $I$-bundle over a surface with $\Sigma$ the associated $S^0$-bundle.
\en
More generally, if $\G$ is a finite-index subgroup of $\pi_1(N)$ isomorphic to the fundamental group of a closed surface, then  by \cite[Theorem~10.6]{Hem76} $N$ is an $I$-bundle over a closed surface.
(See also \cite[Theorem~3.1]{Broa66} and see \cite{BT74} for an extension to the case of non-compact $N$.)
\item Let $N$ be a compact $3$-manifold and let $\Sigma$ be a connected compact proper subsurface of $\partial N$
such that $\chi(\Sigma)\geq \chi(N)$ and such that $\pi_1(\Sigma)\to \pi_1(N)$ is surjective.
It follows from \cite[Theorem~1]{BrC65} that $\Sigma$ and $\ol{\partial N\setminus \Sigma}$ are strong deformation retracts of $N$.
\item Let $N$ be a compact, orientable, irreducible $3$-manifold and let $\Sigma\ne S^2\subset N$ be a closed incompressible surface.
If $\pi_1(\Sigma)\subseteq \G\subseteq \pi_1(N)$, where $\G$ is isomorphic to the fundamental group of a closed orientable surface, then  $\pi_1(\Sigma)=\G$ by
\cite[Theorem~6]{Ja71}. (See also \cite{Fe70,Fe72b,Hei69b,Hei70} and \cite[Lemma~3.5]{Sco74}.)
\item Button \cite[Theorem~4.1]{But07} showed that if $N$ is a compact $3$-manifold and $\G\subseteq \pi_1(N)$ is a finitely generated subgroup and $t\in \pi_1(N)$
with $t\G t^{-1}\subseteq \G$, then $t\G t^{-1}=\G$.
\item Moon \cite[p.~18]{Moo05} showed that if $N$ is a geometric $3$-manifold and $\G$ a finitely generated subgroup
of infinite index which contains a non-trivial group $G\ne \Z\subset \G$ which is normal in $\pi$, then
$\G$ is commensurable to a virtual surface fiber group.
 (Recall that two subgroups $A$,~$B$ of a group $\pi$ are called commensurable if $A\cap B$ has finite index in $A$ and $B$.)
 In the hyperbolic case this can be seen as a consequence of (K.\ref{K.13}) and (K.\ref{K.22}).
Moon also shows that this conclusion holds for certain non-geometric $3$-manifolds.
\item Given a $3$-manifold $N$ we denote by $\mathcal{K}(N)$ the set of all isomorphism classes of knot groups of $N$.
Here a  \emph{knot group of $N$} is the fundamental group of $N\setminus \nu K$ where $K\subset N$ is a simple closed curve.
Let $N_1$ and $N_2$ be orientable compact $3$-manifolds whose boundaries contain no $2$-spheres.
Jaco--Myers and Row (see  \cite[Corollary~1]{Row79}, \cite[Theorem~6.1]{JM79} and \cite[Theorem~8.1]{Mye82}) showed that  $N_1$ and $N_2$
are diffeomorphic if and only if $\mathcal{K}(N_1)=\mathcal{K}(N_2)$.
We refer to \cite[p.~455]{Fo52}, \cite[p.~181]{Bry60} and \cite{Con70} for some earlier work.
\item Soma \cite{Som91} proved various results on the intersections of conjugates of virtual surface fiber subgroups.
\item We refer to \cite{WW94,WY99} and \cite[Section~7]{BGHM10} for results on finite-index subgroups of $3$-manifold groups.
\end{list}

\section{Proofs}\label{section:proofs}

\noindent
In this section we collect the proofs of several statements that were mentioned in the previous sections.

\subsection{Conjugacy separability}\label{section:conjsep}

It is immediate that a subgroup of a residually finite group is itself residually finite, and it is also easy to prove that a group with a residually finite subgroup of finite index is itself residually finite.  In contrast, the property of conjugacy separability (see (E.\ref{E.defconjugacyseparable})) is more delicate.  Goryaga gave an example of a non-conjugacy-separable group with a conjugacy separable subgroup of finite index \cite{Goa86}.  In the other direction, Martino--Minasyan constructed examples of conjugacy separable groups with non-conjugacy-separable subgroups of finite index \cite[Theorem~1.1]{MMn09}.

For this reason, one defines a group to be \emph{hereditarily conjugacy separable} if every finite-index subgroup is conjugacy separable.   We will now show that, in the $3$-manifold context, one can apply a criterion of Chagas--Zalesskii \cite{ChZ10} to prove that hereditary conjugacy separability passes to finite extensions. \index{group!hereditarily conjugacy separable}

\begin{theorem}\label{t: Hereditarily CuS upwards}\label{thm:conjsep}
Let $N$ be a compact, orientable, irreducible $3$-manifold with toroi\-dal boundary, and let  $K$ be a subgroup of
$\pi=\pi_1(N)$ of finite index.  If~$K$ is hereditarily conjugacy separable, then so is $\pi$.
\end{theorem}

To shows this,
we may assume that $N$ does not admit $\Sol$ geometry, as polycyclic groups are known to be conjugacy separable by a theorem of Remeslennikov \cite{Rev69}.   Furthermore, we may assume that $K$ is normal, corresponding to a regular covering map $N'\to N$ of finite degree.  In particular, $K$ is also the fundamental group of a compact, orientable, irreducible $3$-manifold with toroidal boundary.  We summarize the structure of the centralizers of elements of $\pi$ in the following proposition, which  is an immediate consequence of Theorems~\ref{thm:centgen} and~\ref{thm:centsfs}.

\begin{proposition}
Let $N$ be a compact, orientable, irreducible $3$-manifold with toroidal boundary that does not admit $\Sol$ geometry and
let  $g\in\pi=\pi_1(N)$, $g\neq 1$. Then either $C_\pi(g)$ is free abelian or there is a Seifert fibered piece~$N'$ of the JSJ decomposition of $N$ such that $C_\pi(g)$ is a subgroup of index at most two in $\pi_1(N')$.
\end{proposition}

We will now prove three lemmas about centralizers.  These enable us to apply a result of Chagas--Zalesskii \cite{ChZ10} to finish the proof. In all three lemmas, we let $N$ and $g$ be as in the preceding proposition.

\begin{lemma}\label{l: Lemma 1}
The centralizer $C_\pi(g)$ is conjugacy separable.
\end{lemma}

\begin{proof}
If $C_\pi(g)$ is free abelian, then this is clear.  Otherwise, $C_\pi(g)$ is the fundamental group of a Seifert fibered manifold, which is conjugacy separable by a theorem of Martino \cite{Mao07}.
\end{proof}

For a group $G$, let $\widehat{G}$ denote the profinite completion of $G$, and for a subgroup $H\subseteq G$,  let $\overline{H}$ denote the closure of $H$ in $\widehat{G}$.

\begin{lemma}\label{l: Lemma 2}
The canonical map $\widehat{C_\pi(g)}\to\overline{C_\pi(g)}$ is an isomorphism.
\end{lemma}
\begin{proof}
We need to prove that the profinite topology on $\pi$ induces the full profinite topology on $C_\pi(g)$.  To this end, it is enough to prove that every finite-index subgroup $H$ of $C_\pi(g)$ is separable in $\pi$.

If $C_\pi(g)$ is free abelian, then so is $H$, so $H$ is separable by the main theorem of \cite{Hamb01}.  Therefore, suppose $C_\pi(g)$ is a subgroup of index at most two in $\pi_1(M)$, where $M$ is a Seifert fibered vertex space of $N$, so $H$ is a subgroup of finite index in $\pi_1(M)$. By (C.\ref{C.efficient}), the group $\pi$ induces the full profinite topology on $\pi_1(M)$ and $\pi_1(M)$ is separable in $\pi$. It follows that $H$ is separable in $\pi$.
\end{proof}

The final condition is a direct consequence of Proposition 3.2 and Corollary 12.2 of \cite{Min12}, together with the hypothesis that $K$ is hereditarily conjugacy separable.

\begin{lemma}\label{l: Lemma 3}
The inclusion $\overline{C_\pi(g)}\to C_{\widehat{\pi}}(g)$ is surjective.
\end{lemma}

These lemmas enable us to apply the following useful criterion.

\begin{proposition}[{\cite[Proposition 2.1]{ChZ10}}]\label{p: Chagas--Zalesskii}
Let $\pi$ be a finitely generated group containing a conjugacy separable normal subgroup $K$ of finite index. Let $a \in \pi$ be an element such that there exists a natural number $m$ with $a^m\in K$ and the following conditions hold:
\begin{itemize}
\item[(1)] $C_\pi(a^m)$ is conjugacy separable;
\item[(2)] $\widehat{C_K(a^m)}=\overline{C_K(a^m)}=C_{\widehat{K}}(a^m)$~.
\end{itemize}
Then whenever $b\in \pi$ is not conjugate to $a$, there is a homomorphism $f$ from $\pi$ to a finite group such that $f(a)$ is not conjugate to $f(b)$.
\end{proposition}

We are now in a position to prove Theorem~\ref{t: Hereditarily CuS upwards}.

\begin{proof}[Proof of Theorem~\ref{t: Hereditarily CuS upwards}]
As mentioned earlier, we may assume that $N$ does not admit $\Sol$ geometry. Let $K$ be a
hereditarily conjugacy separable subgroup of $\pi=\pi_1(N)$ of finite index.
By replacing $K$ with the intersection of its conjugates, we may assume that $K$ is normal.  Let $a,b\in\pi$ be non-conjugate and let $m$ be non-zero with $a^m\in K$. By Lemma~\ref{l: Lemma 1}, the centralizer $C_\pi(a^m)$ is conjugacy separable.  Let $N'$ be a finite-sheeted covering space of $N$ with $K=\pi_1(N')$.  Lemma~\ref{l: Lemma 2} applied to $K=\pi_1(N')$ shows that the first equality in condition (2) of Proposition~\ref{p: Chagas--Zalesskii} holds and similarly Lemma~\ref{l: Lemma 3} shows that the second equality holds.   Therefore, Proposition~\ref{p: Chagas--Zalesskii} applies to show that there is a homomorphism from $\pi$ to a finite group under which the images of $a$ and $b$ are non-conjugate.  Therefore, $\pi$ is conjugacy separable.
\end{proof}

\subsection{Fundamental groups of Seifert fibered manifolds are linear over $\Z$} \label{section:sfslinear}

In this section we will give a proof (due to Boyer) of the following theorem.

\begin{theorem}\label{thm:Seifert linear}
Let $N$ be a Seifert fibered manifold. Then  $\pi_1(N)$ is linear over~$\Z$.
\end{theorem}

Before we prove Theorem ~\ref{thm:Seifert linear} we consider  the following two lemmas.
The first lemma is well known, but we include the proof for completeness' sake.

\begin{lemma}\label{lem:seifertcover}\label{lem:seifert}
Let $N$ be a Seifert fibered manifold. Then $N$ is finitely covered by  an $S^1$-bundle over an orientable connected surface.
\end{lemma}

\begin{proof}
We first consider the case that $N$ is closed.
Denote by $B$ the base orbifold of the Seifert fibered manifold $N$.
If $B$ is a `good' orbifold in the sense of \cite[p.~425]{Sco83a}, then $B$ is finitely covered by
an orientable connected surface $F$. This cover $F\to B$ gives rise to a map $Y\to N$ of Seifert fibered manifolds. Since the base orbifold of the Seifert fibered manifold $Y$ is a surface it follows that $Y$ is in fact an $S^1$-bundle over $F$.

The `bad' orbifolds are classified in \cite[p.~425]{Sco83a}, and in the case of base orbifolds the only two classes of bad orbifolds which can arise are $S^2(p)$ and $S^2(p,q)$ (see \cite[p.~430]{Sco83a}).
The former arises from the lens space $L(p,1)$ and the latter from the lens space $L(p,q)$.
But lens spaces are covered by $S^3$ which is an $S^1$-bundle over the sphere.

Now consider the case that $N$ has boundary.
 We consider the double $M=N\cup_{\partial N} N$. Note that $M$ is again a Seifert fibered manifold. By the above there exists a finite-sheeted covering map $p:Y\to M$ such that $Y$ is an $S^1$-bundle over a surface and $p$ preserves the Seifert fibers.
It now follows that $p^{-1}(N)\subseteq Y$ is a sub-Seifert fibered manifold.
In particular any component of $p^{-1}(N)$ is also an $S^1$-bundle over a surface.
\end{proof}

\begin{remark}
Let $N$ be any compact, orientable, irreducible $3$-manifold with empty or toroidal boundary.
A useful generalization of Lemma~\ref{lem:seifertcover} says that $N$ admits a finite cover all of whose Seifert fibered JSJ components   are in fact $S^1$-bundles over a surface. See \cite[Section~4.3]{AF10} and \cite{Hem87,Hamb01}  for details.
\end{remark}

\begin{lemma}\label{lem:S1-bundles}
Let $N$ be an $S^1$-bundle over an orientable surface $F$. Then $\pi_1(N)$ is linear over $\Z$.
\end{lemma}

\begin{proof}
We first recall that  surface groups are linear over $\Z$.
Indeed, if $F$ is a sphere or a torus, then this is obvious.
If $F$ has boundary, then $\pi_1(N)$ is a free group and hence embeds into $\sl(2,\Z)$.
Finally,  Newman \cite[Lemma~1]{New85} showed that if  $F$ is closed, then there exists  an embedding $\pi_1(F)\to\sl(8,\Z)$.
(Alternatively, Scott \cite[Section~3]{Sco78} showed that $\pi_1(F)$ is a subgroup of a right angled Coxeter group on 5 generators,
and hence by  \cite[Chapitre~V,~\S~4,~Section~4]{Bou81} can in fact be embedded  into $\mbox{SL}(5,\Z)$.
Also \cite[p.~576]{DSS89} contains a proof that surface groups embed into RAAGs, and hence are linear over $\Z$ by \cite[Corollary~3.6]{HsW99}.)

We now turn to the proof of the lemma.
We first consider the case  that $N$ is a trivial $S^1$-bundle, i.e., $N\cong S^1\times F$. But then $\pi_1(N)=\Z\times \pi_1(F)$ is the direct product of $\Z$ with a surface  group, so $\pi_1(N)$ is $\Z$-linear by the above.

If $N$  has boundary, then $F$ also has boundary and we obtain $H^2(F;\Z)=0$, so the Euler class of the $S^1$-bundle $N\to F$ is trivial. We therefore conclude that in this case, $N$ is a trivial $S^1$-bundle.

Now assume that $N$ is a non-trivial $S^1$-bundle.
By the above, $F$ is a closed surface.
If $F=S^2$, then the long exact sequence in homotopy theory shows that $\pi_1(N)$ is cyclic, hence linear.
If $F\ne S^2$, then it follows again from the long exact sequence in homotopy theory
that the subgroup $\ll t\rr$ of $\pi_1(N)$ generated by a fiber is normal and infinite cyclic, and that
we have a short exact sequence
\[ 1\to \ll t\rr \to \pi_1(N)\to \pi_1(F)\to 1.\]
Let $e\in H^2(F;\Z)\cong \Z$ be the Euler class of $F$. A presentation for $G:=\pi_1(N)$ is given by
\[ G = \left\langle a_1,b_1,\dots,a_r,b_r,t : \prod_{i=1}^r [a_i,b_i]=t^e, \text{ $t$ central}\right\rangle.\]
Let $G_e$ be the subgroup of $G$ generated by the $a_i$, $b_i$ and $t^e$.
It is straightforward to check that the assignment
\[ \rho(a_1)=\begin{bmatrix} 1&1&0 \\ 0&1&0 \\ 0&0&1\end{bmatrix},\quad
\rho(b_1)=\begin{bmatrix} 1&0&0 \\ 0&1&1 \\ 0&0&1\end{bmatrix}, \quad
\rho(t^e)=\begin{bmatrix} 1&0&1 \\ 0&1&0 \\ 0&0&1\end{bmatrix}\]
and $$\rho(a_i)=\rho(b_i)=\id\qquad\text{for $i\geq 2$}$$
yields a representation $\rho\colon G_e\to \sl(3,\Z)$ such that $\rho(t^e)$ has infinite order.
Now let
$\sigma$ be the composition
$$G_e\to G=\pi_1(N)\to \pi_1(F)\to\sl(n,\Z),$$
where the last homomorphism is a faithful representation
of $\pi_1(F)$, which exists by the above.
Then $$\rho\times\sigma\colon G_e\to\sl(3,\Z)\times\sl(n,\Z)\subseteq \sl(n+3,\Z)$$
is an embedding. This shows that the finite-index subgroup $G_e$ of $G$ is $\Z$-linear.
By (H.\ref{H.lineargreen}) this implies that $G$ is also linear over $\Z$.
\end{proof}

We can now provide a proof of Theorem~\ref{thm:Seifert linear}.

\begin{proof}[Proof of Theorem~\ref{thm:Seifert linear}]
Let $N$ be a Seifert fibered manifold.  By Lemma~\ref{lem:seifertcover}, $N$ is finitely covered by a $3$-manifold $N'$ which is an $S^1$-bundle over an orientable surface $F$. It now follows from Lemma~\ref{lem:S1-bundles}
that $\pi_1(N')$ is linear over $\Z$, and so $\pi_1(N)$ is also linear over $\Z$ by (H.\ref{H.lineargreen}).
\end{proof}

\subsection{Non-virtually-fibered graph manifolds and retractions onto cyclic subgroups}

There exist  Seifert fibered manifolds which are not virtually fibered, and also graph manifolds with non-trivial JSJ decomposition
which are not virtually fibered
(see, e.g., \cite[p.~86]{LuW93} and \cite[Theorem~D]{Nemb96}). The following proposition shows that such examples also have the property that their fundamental groups do not virtually retract onto cyclic subgroups.

\begin{proposition}\label{prop:nonretract}
Let $N$ be a non-spherical graph manifold which is not virtually fibered. Then $\pi_1(N)$ does not virtually retract onto  cyclic subgroups.
\end{proposition}

\begin{proof}
Let $N$ be a non-spherical graph manifold such that $\pi_1(N)$ virtually retracts onto all its cyclic subgroups.  We will show that $N$ is virtually fibered.

By the remark after Lemma~\ref{lem:seifert},  $N$ is finitely covered by a $3$-manifold each of whose JSJ components is an $S^1$-bundle over a surface.
We can therefore  without loss of generality assume that $N$ itself is already of that form.

We  first consider the case that $N$ is  a Seifert fibered manifold, i.e., that $N$ is  an $S^1$-bundle  over a surface $\Sigma$. The assumption that $N$ is non-spherical implies that the regular
fiber generates an infinite cyclic subgroup of $\pi_1(N)$. It is well known, and straightforward to see, that
if $\pi_1(N)$ retracts onto this infinite cyclic subgroup, then $N$ is a product $S^1\times \Sigma$; in particular, $N$ is fibered.

We now consider the case that $N$ has a non-trivial JSJ decomposition.  We denote the JSJ pieces of $N$ by $M_v$, where $v$ ranges over some index set $V$. By hypothesis, each $M_v$ is an $S^1$-bundle over a surface with non-empty boundary, so each $M_v$ is in fact a product. 
We denote by $f_v$ the Seifert fiber of $M_v$. Note that each $f_v$ generates an infinite cyclic subgroup of $\pi_1(N)$.

Since $\pi_1(N)$  virtually retracts onto cyclic subgroups, for each $v$ we can find a finite-sheeted covering space $\wti{N}_v$ of $N$ such that $\ti{\pi}_v=\pi_1(\wti{N}_v)$ retracts onto $\langle f_v\rangle$.  In particular, the image of $f_v$ is non-trivial in $H_1(\wti{N}_v;\Z)/\mbox{torsion}$.  Let $\wti{N}$ be any regular finite-sheeted cover of $N$ that covers every $\wti{N}_v$.  (For instance, $\pi_1(\wti{N})$ could be the intersection of all the conjugates of $\bigcap_{v\in V}\ti{\pi}_v$.)

Let $\tilde{f}$ be a Seifert fiber of a JSJ component of $\wti{N}$.  Up to the action of the deck group of $\wti{N}\to N$, $\tilde{f}$ covers the lift of some $f_v$ in $\wti{N}_v$.  It follows that $\tilde{f}$ is non-trivial in $H_1(\wti{N};\Z)/\mbox{torsion}$.

Therefore, there exists a homomorphism $\phi\colon H_1(\wti{N};\Z)\to \Z$ which is non-trivial on the Seifert fibers of all JSJ components of $\wti{N}$.  Since each JSJ component is a product, the restriction of $\phi$ to each JSJ component of $\wti{N}$ is a fibered class.  By \cite[Theorem~4.2]{EN85}, we conclude that $\wti{N}$ fibers over $S^1$.
\end{proof}

\subsection{(Fibered) faces of the Thurston norm ball of finite covers}\label{section:thurstonnorm}

Let $N$ be a compact, orientable 3--manifold.
Recall that we say that $\phi\in H^1(N;\R)$ is \emph{fibered}, if $\phi$ can be represented by a non-degenerate closed 1-form.

\medskip

The \emph{Thurston norm} of  $\phi\in H^1(N;\mathbb{Z})$  is defined as
 \index{$3$-manifold!Thurston norm}
 \index{Thurston norm!definition}
 \[
\|\phi\|_T=\min \big\{ \chi_-(\Sigma) : \text{$\Sigma \subseteq N$  properly embedded surface dual to $\phi$} \big\}.
\]
Here, given a surface $\Sigma$ with connected components $\Sigma_1\cup\dots \cup \Sigma_k$, we define
$\chi_-(\Sigma)=\sum_{i=1}^k \max\{-\chi(\Sigma_i),0\}$.
 Thurston
\cite[Theorems~2~and~5]{Thu86a} (see also \cite[Chapter~10]{CdC03} and \cite[p.~259]{Oe86}) proved the following results:
\bn
\item $\|{-}\|_T$ defines a seminorm on $H^1(N;\Z)$ which can be extended to a seminorm $\|{-}\|_T$ on $H^1(N;\R)$.
\item The  norm ball \index{Thurston norm!norm ball}
$$\{\phi\in H^1(N;\R) : \|\phi\|_T\leq 1\}$$
is a finite-sided rational polyhedron.
\item  There exist  open top-dimensional faces $F_1,\dots,F_k$ of the Thurston norm ball such that
\[ \{\phi\in H^1(N;\R) : \text{$\phi$  fibered}\}=\bigcup_{i=1}^k \,\R^+ F_i.\]
These  faces are called the \emph{fibered faces} of the Thurston norm ball.
\index{Thurston norm!fibered face}
 \en
 The Thurston norm ball is evidently symmetric in the origin. We say that two faces $F$ and $G$ are \emph{equivalent} if $F=\pm G$.
 Note that a face $F$ is fibered if and only if $-F$ is fibered.

\medskip

The Thurston norm is degenerate in general, e.g., for $3$-manifolds with homologically essential tori.
On the other hand the Thurston norm of a hyperbolic $3$-manifold is non-degenerate, since a hyperbolic $3$-manifold admits no homologically essential surfaces of non-negative Euler characteristic.

\medskip

We start out with the following fact.

\begin{proposition}\label{prop:tnfinitecover}
Let $p\colon M\to N$ be a finite cover.
Then $\phi\in H^1(N;\R)$ is fibered if and only if $p^*\phi\in H^1(M;\R)$ is fibered.
Furthermore
\[ \|p^*\phi\|_T=[M:N]\cdot \|\phi\|_T \quad\text{ for any class $\phi\in H^1(N;\R)$.}\]
In particular, the map $p^*\colon H^1(N;\R)\to H^1(M;\R)$ is, up to a scale factor, an isometry,
and it maps fibered cones into fibered cones.
\end{proposition}

The first statement is an immediate consequence of Stallings' Fibration Theorem (see \cite{Sta62} and (K.\ref{K.stallings62})),
and the second statement follows from
work of Gabai \cite[Corollary~6.13]{Gab83a}.

\medskip

We can now prove the following proposition.

\begin{proposition}\label{prop:manyfaces}
Let $N$ be an irreducible  $3$-manifold with empty or toroidal boundary, which is not a graph manifold.
Then given any $k\in \N$, $N$ has   a finite cover whose Thurston norm ball has at least $k$ inequivalent faces.
\end{proposition}

If $N$ is closed and hyperbolic, then the proposition relies on the Virtually Compact Special Theorem (Theorem \ref{thm:akmw}).
In the other cases it follows from the work of Cooper--Long--Reid \cite{CLR97} and classical facts on the Thurston norm.

\begin{proof}
We first suppose that $N$ is hyperbolic.
It follows from (G.\ref{G.akmw}), (G.\ref{G.compact special=>qc subgroup of a RAAG}), (G.\ref{G.AM11}) and (C.\ref{C.homlarge})
that $N$ admits a finite cover $M$ with $b_1(M)\geq k$. Since the Thurston norm of a hyperbolic $3$-manifold is non-degenerate
it follows that the Thurston norm ball of $M$ has at least $2^{k-1}\geq k$ inequivalent faces.

We now suppose that $N$ is not hyperbolic. By assumption there exists a hyperbolic JSJ component $X$ which is hyperbolic
and which necessarily has non-empty boundary.
It follows from  \cite[Theorem~1.3]{CLR97} (see also (C.\ref{C.large})) that $\pi_1(N)$ is large
and hence by (C.\ref{C.homlarge}) that there exists a finite cover $\wti{X}$ with non-peripheral homology of rank at least $k$.

A standard argument, using (C.\ref{C.efficient}), now shows that there exists a finite cover $M$ of $N$
which admits a hyperbolic JSJ component $Y$ which covers $\wti{X}$. An elementary argument shows that $Y$ also has non-peripheral homology of rank at least $k$.
We consider  $p\colon H_2(Y;\R)\to H_2(Y,\partial Y;\R)$ and $V:=\im p$.
Using Poincar\'e Duality, the Universal Coefficient Theorem and the information on the non-peripheral homology, we see that $\dim(V)\geq k$.

We now  consider $q\colon H_2(Y;\R)\to H_2(M,\partial M;\R)$ and $W:=\im q$.
Since $p$ is the composition of  $q$ and the restriction map $ H_2(M,\partial M;\R)\to  H_2(Y,\partial Y;\R)$, we see that $\dim W\geq \dim V\geq k$.
Since $N$ is hyperbolic, it follows that the Thurston norm of $Y$ is non-degenerate, in particular non-degenerate on $V$.
By \cite[Proposition~3.5]{EN85} the Thurston norm of $p_*\phi$ in $Y$ agrees with the Thurston norm of $q_*\phi$ in $M$.
Thus the Thurston norm of $M$ is non-degenerate on~$W$, in particular
the Thurston norm ball of $M$ has at least $2^{k-1}\geq k$ inequivalent faces.
\end{proof}

We  say that $\phi\in H^1(N;\R)$ is \emph{quasi-fibered} if $\phi$
lies on the closure of a fibered cone of the Thurston norm ball of $N$.
We can now formulate Agol's Virtually Fibered Theorem (see \cite[Theorem~5.1]{Ag08}).

\begin{theorem}\textbf{\emph{(Agol)}}\label{thm:agol08}
 \index{theorems!Agol's Virtually Fibered Theorem}\index{theorems!Virtually Fibered Theorem}
Let $N$ be an irreducible, compact  $3$-manifold with empty or toroidal boundary such that $\pi_1(N)$ is virtually RFRS.
 Then given any $\phi \in H^1(N;\R)$  there exists a finite cover $p\colon M\to N$ such that $p^*\phi$ is quasi-fibered.
\end{theorem}

The following is now a straightforward consequence of Agol's theorem.

\begin{proposition}\label{prop:manyfiberedfacesb}
Let $N$ be an irreducible, compact  $3$-manifold with empty or toroidal boundary
 such that the Thurston norm ball of $N$ has at least $k$ inequivalent faces.
If $\pi_1(N)$ is virtually RFRS and if $N$ is not a graph manifold, then given any $k\in \N$ there exists a finite cover $M$ of $N$ such that the Thurston norm ball of $M$ has at least $k$ inequivalent fibered faces.
\end{proposition}

The proof of the proposition is precisely that of \cite[Theorem~7.2]{Ag08}.
We therefore give just a very quick outline of the proof.

\begin{proof}
We pick classes $\phi_i$ ($i=1,\dots,k$) in $H^1(N;\R)$ which lie in $k$ inequivalent faces.
For $i=1,\dots,k$ we then apply Theorem \ref{thm:agol08}  to the class $\phi_i$ and we obtain a finite cover $\wti{N}_i\to N$ such that the pull-back of $\phi_i$ is quasi-fibered.

Denote by $p\colon M\to N$ the cover corresponding to $\bigcap \pi_1(\wti{N}_i)$.
It follows from Proposition \ref{prop:tnfinitecover} that pull-backs of quasi-fibered classes are quasi-fibered, and that pull-backs of inequivalent faces of the Thurston norm ball lie on
inequivalent faces of the Thurston norm ball. Thus $p^*\phi_1,\dots,p^*\phi_k$ lie on closures of inequivalent fibered faces of $M$,
i.e.,  $M$ has at least $k$ inequivalent fibered faces.
\end{proof}

It is  a natural question to ask in how many different ways
 a (virtually) fibered $3$-manifold (virtually) fibers. We recall the following facts:
\bn
\item
If $\phi \in H^1(N;\Z)$ is a fibered class, then using Stallings' Fibration Theorem (see (K.\ref{K.stallings62}))
one can show that up to isotopy there exists a unique surface bundle representing $\phi$ (see \cite[Lemma~5.1]{EdL83} for details).
\item It follows from the above description of fibered cones that being fibered is an open condition in $H^1(N;\R)$.
We refer to  \cite{Nemb79}
and \cite[Theorem~A]{BNS87} for a group-theoretic proof for classes in $H^1(N;\Q)$, to \cite{To69} and \cite{Nema76} for earlier results
and to \cite{HLMA06} for an explicit discussion for a particular example.
If $b_1(N)\geq 2$ and if the Thurston norm is not identically zero, then a basic Thurston norm argument shows that $N$ admits fibrations with connected fibers of arbitrarily large genus.
(See, e.g., \cite[Theorem~4.2]{But07} for details).
\en
A deeper question is whether a $3$-manifold admits (virtually) inequivalent fibered faces.
The following proposition is now an immediate consequence of Propositions \ref{prop:manyfaces} and \ref{prop:manyfiberedfacesb}
together with (G.\ref{G.akmw}), (G.\ref{G.PW12a}), (G.\ref{G.special=>subgroup of a RAAG}) and (G.\ref{G.RAAGvRFRS}).

\begin{proposition}\label{prop:manyfiberedfaces}
Let $N$ be an irreducible, compact  $3$-manifold with empty or toroidal boundary which is not a graph manifold.
Then for each $k\in\N$, $N$ has a finite cover whose Thurston norm ball has at least $k$ inequivalent fibered faces.
\end{proposition}

\begin{remarks} \mbox{}

\begin{itemize}
\item[(1)]
Let $N$ be an irreducible, compact  $3$-manifold with empty or toroidal boundary such that $\pi_1(N)$ is virtually RFRS but not virtually abelian.
According to \cite[Theorem~7.2]{Ag08} the manifold $N$ admits finite covers with arbitrarily many inequivalent faces in the Thurston norm ball.
At this level of generality the statement does not hold. As an example consider the product manifold $N=S^1\times \Sigma$. Any finite cover $M$ of $N$ is again a product; in particular the Thurston norm ball of $M$ has just two faces.
\item[(2)] It would be interesting to find criteria which decide whether a given graph manifold has virtually arbitrarily many faces in the Thurston norm ball.
\end{itemize}
\end{remarks}

\section{Open questions}\label{sec:Open questions}

\subsection{Separable subgroups in $3$-manifolds with a non-trivial JSJ decomposition}\label{section:questionsep}

Let $N$ be a compact, orientable, irreducible $3$-manifold $N$ with empty or toroidal boundary.  We know that in the hyperbolic case $\pi_1(N)$ is in fact virtually \emph{compact} special, which together with the Tameness Theorem of Agol and Calegari--Gabai and work of Haglund implies that $\pi_1(N)$ is LERF.

The picture is considerably more complicated for non-hyperbolic $3$-manifolds.  Niblo--Wise \cite[Theorem~4.2]{NW01}  showed that the fundamental group of a graph manifold $N$ is LERF if and only if $N$ is geometric (see also (I.\ref{I.nonlerf})). As every known example of a $3$-manifold with non-LERF fundamental group derives from examples of this form, the following conjecture seems reasonable.

\begin{conjecture}\label{Conj: LERF for isolated SFSs}
Let $N$ be a compact, orientable, irreducible $3$-manifold with empty or toroidal boundary such that no torus of the JSJ decomposition bounds a Seifert fibered $3$-manifold on both sides.  Then $\pi_1(N)$ is LERF.
\end{conjecture}

We refer to \cite[Theorem~1.2]{LoR01} for some evidence towards the conjecture.  Note that $\pi_1(N)$ being virtually compact special is in general not enough to deduce that $\pi_1(N)$ is LERF.
 Indeed, there exist graph manifolds with fundamental groups that are compact special but not LERF; for instance, the non-LERF link group exhibited in \cite[Theorem~1.3]{NW01} is a right-angled Artin group.

\medskip

Despite the general failure of LERF, certain families of subgroups are known to be separable.
\bn
\item Let $N$ be an orientable, compact, irreducible $3$-manifold with (not necessarily toroidal) boundary.
Let $X$ be a connected, incompressible subsurface of the boundary of $N$.  Long--Niblo \cite[Theorem~1]{LoN91} showed that  $\pi_1(X)\subseteq \pi_1(N)$ is separable.
\item Let $N$ be a compact $3$-manifold. Hamilton proved that any abelian subgroup of $\pi_1(N)$ is separable (C.\ref{C.AERF}).
\item Hamilton \cite{Hamb03} gave examples of free 2-generator subgroups in non-geometric $3$-manifolds which are separable.
\item Let $N$ be an orientable, compact, irreducible $3$-manifold with empty or toroidal boundary. By (C.\ref{C.efficient}) the manifold $N$ is efficient.  By (C.\ref{C.sfssubgroupseparable}) and (G.\ref{G.LERF}) the fundamental group of any JSJ piece is LERF, and it follows that any subgroup of $\pi_1(N)$ carried by a JSJ piece is separable.
\item
For an arbitray compact $3$-manifold $N$, Przytycki--Wise \cite[Theorem~1.1]{PW12b} have shown that a subgroup carried by an  incompressible properly embedded surface is separable in $\pi_1(N)$.
\en

To bring order to this menagerie of examples, it would be desirable to exhibit some large, intrinsically defined class of subgroups of general $3$-manifold groups which are separable.  In the remainder of this subsection, we propose the class of \emph{fully relatively quasi-convex subgroups} (see below) as a candidate.

Not every separable subgroup listed above is fully relatively quasiconvex.  However,  this proposal captures the aforementioned  fact that all known examples of non-separable subgroups of $3$-manifold groups derive from graph manifolds. In general, a strategy for proving that a given subgroup $\Gamma$ of a $3$-manifold group $\pi$ can be separated from an element $g\in\pi\setminus\Gamma$ is to first use a gluing theorem such as \cite[Theorem~2]{MPS12} to construct a fully relatively quasiconvex subgroup $Q$ such that $\Gamma\subseteq Q$ but $g\notin Q$, and to then argue that $Q$ is separable in $\pi$.

We work in the context of relatively hyperbolic groups.  The following theorem follows quickly from \cite[Theorem~0.1]{Dah03}.
(See also \cite[Corollary~E]{BiW12}.)

\begin{theorem}\label{thm:dah03}
Let $N$ be a compact, irreducible $3$-manifold with empty or toroidal boundary.  Let $M_1,\ldots, M_k$ be the maximal graph manifold pieces of the JSJ decomposition of $N$, let $S_1,\ldots, S_l$ be the tori in the boundary of $N$ that adjoin a hyperbolic piece and let $T_1,\ldots,T_m$ be the tori in the JSJ decomposition of $M$ that separate two \textup{(}not necessarily distinct\textup{)} hyperbolic pieces of the JSJ decomposition.  The fundamental group of $N$ is hyperbolic relative to the set of parabolic subgroups
\[
\{H_i\}=\{\pi_1(M_p)\}\cup\{\pi_1(S_q)\}\cup\{\pi_1(T_r)\}~.
\]
\end{theorem}

In particular, a graph manifold group is hyperbolic relative to itself.
\medskip

There is a notion of a \emph{relatively quasi-convex} subgroup of a relatively hyperbolic group; see (G.\ref{G.virtualretract}) and the references mentioned there for more details.   A subgroup~$\Gamma$ of a relatively hyperbolic group $\pi$ is called \emph{fully relatively quasi-convex} if it is relatively quasi-convex and, furthermore, for each $i$, $\Gamma\cap H_i$ is either trivial or a subgroup of finite index in $H_i$. \index{subgroup!fully relatively quasi-convex}

\begin{conjecture}\label{Conj: Fully relatively quasi-convex subgroups are well behaved}
Let $N$ be an orientable, compact, non-positively curved $3$-manifold with empty or toroidal boundary.  If $\Gamma$ is a subgroup of $\pi=\pi_1(N)$ that is fully relatively quasi-convex with respect to the natural relatively hyperbolic structure on $\pi$, then $\Gamma$ is a virtual retract of $\pi$.  In particular, $\Gamma$ is separable.
\end{conjecture}

Note that, under the hypotheses of Conjecture~\ref{Conj: LERF for isolated SFSs}, the parabolic subgroups of $\pi_1(N)$ in the relatively hyperbolic structure are LERF.
\medskip

Conjecture~\ref{Conj: Fully relatively quasi-convex subgroups are well behaved} would follow, by \cite[Theorem~5.8]{CDW12},
from an affirmative answer to the following extension of the results of Liu and Przytycki--Wise.

\begin{question}\label{Qu: NPC $3$-manifolds are virtually compact special}
Let $N$ be a compact, aspherical $3$-manifold  with empty or toroidal boundary which is non-positively curved.  Is  $\pi_1(N)$ virtually compact special?
\end{question}

It may very well be that the answer to Question~\ref{Qu: NPC $3$-manifolds are virtually compact special} is negative; indeed, the techniques of \cite{Liu11} and \cite{PW11} do not give compact cube complexes.  If this is the case, then it is nevertheless  desirable to either prove Conjecture~\ref{Conj: Fully relatively quasi-convex subgroups are well behaved} or to exhibit a different intrinsically defined class of subgroups that are virtual retracts.

%
%

\subsection{Non-non-positively curved $3$-manifolds}

The above discussion shows that a clear picture of the properties of aspherical non-positively curved $3$-mani\-folds is emerging.
The `last frontier,' oddly enough, seems to be the study of $3$-manifolds which are not non-positively curved.

It is interesting to note that solvable fundamental groups of $3$-manifolds
in some sense have `worse' properties than fundamental groups of hyperbolic $3$-manifolds.
In fact in contrast to the picture we developed in Diagram~4 for hyperbolic $3$-manifold groups, we have the following lemma:

\begin{lemma}\label{lem:sol}
Let $N$ be a $\Sol$-manifold and let $\pi=\pi_1(N)$. Then
\begin{itemize}
\item[(1)] $\pi$ is not virtually RFRS,
\item[(2)] $\pi$ is not virtually special,
\item[(3)] $\pi$ does not admit a finite-index subgroup which is residually $p$ for all primes~$p$, and
\item[(4)] $\pi$ does not virtually retract onto all its cyclic subgroups.
\end{itemize}
\end{lemma}

The first statement was shown by Agol \cite[p.~271]{Ag08}, the second is an immediate consequence of the first statement
and (G.\ref{G.RAAGvRFRS}), the third is proved in \cite[Proposition~1.3]{AF11}, and the fourth statement follows
easily from the fact that any finite cover $N'$ of a $\Sol$-manifold is a $\Sol$-manifold again and so $b_1(N')=1$, contradicting (G.\ref{G. vb_1=infty}).

\medskip

We summarize some known properties in the following theorem.

\begin{theorem}
Let $N$ be an aspherical $3$-manifold with empty or toroidal boundary
which does not admit  a non-positively curved metric.
Then
\begin{itemize}
\item[(1)] $N$ is a closed graph manifold;
\item[(2)] $\pi_1(N)$ is conjugacy separable; and
\item[(3)] for any prime $p$, the group $\pi_1(N)$ is virtually residually $p$.
\end{itemize}
\end{theorem}

The first statement was proved by Leeb \cite[Theorems~3.2 and 3.3]{Leb95} and the other two statements are known to hold for fundamental groups of all graph manifolds by \cite[Theorem~D]{WZ10} and \cite{AF10}, respectively.

\medskip

We saw in Lemma~\ref{lem:sol} and Proposition ~\ref{prop:nonretract} that there are many desirable properties which fundamental groups of $\Sol$-manifolds and some graph manifolds do not have.  Also, recall that there are graph manifolds which are not virtually fibered (cf.~(G.\ref{G.liu11})).  We can nonetheless pose the following question.

\begin{questions}
Let $N$ be an aspherical $3$-manifold with empty or toroidal boundary
which does not admit  a non-positively curved metric.
\begin{itemize}
\item[(1)] Is $\pi_1(N)$ linear over $\C$?
\item[(2)] Is $\pi_1(N)$ linear over $\Z$?
\item[(3)] If $\pi_1(N)$ is not solvable, does  $\pi_1(N)$ admit a finite-index subgroup which is residually $p$ for any prime $p$?
\item[(4)] Is $\pi_1(N)$ virtually bi-orderable?
\end{itemize}
\end{questions}

\subsection{Poincar\'e duality groups and the Cannon Conjecture}\label{section:pdg}

It is natural to ask whether there is an intrinsic, group-theoretic characterization of $3$-manifold groups.
Given $n\in \N$, Johnson--Wall \cite{JW72} introduced the notion of an \emph{$n$-dimensional Poincar\'e duality group} (usually just referred to as a \emph{$\pd_n$-group}).  The fundamental group of any closed, orientable, aspherical $n$-manifold is a $\pd_n$-group.
Now suppose that $\pi$ is a $\pd_n$-group. If $n=1$ and $n=2$, then $\pi$ is the fundamental group of a closed, orientable, aspherical $n$-manifold (the case $n=2$ was proved by
Eckmann, Linnell and M\"uller, see \cite{EcM80,EcL83,Ec84,Ec85,Ec87}).
Davis \cite[Theorem~C]{Davb98} showed that for any $n\geq 4$ there exists a finitely generated  $\pd_n$-group which is not finitely presented
 and hence is not the fundamental group of an aspherical closed $n$-manifold.  However, the following conjecture of Wall is still open. \index{group!$n$-dimensional Poincar\'e duality group ($\pd_n$-group)}

\begin{conjecture}\textbf{\emph{(Wall Conjecture)}}
\index{conjectures!Wall Conjecture}
Let $n\geq 3$. Then every finitely presented $\pd_n$ group is the fundamental group of a closed, orientable, aspherical $n$-manifold.
\end{conjecture}

This conjecture has been studied over many years and a summary of all the results so far exceeds the possibilities of this survey.
We refer to \cite{Tho95,Davb00,Hil11} for some surveys and to
\cite{BiH91,Cas04,Cas07,Cr00,Cr07,Davb00,DuS00,Hil85,Hil87,Hil06,Hil12,Hil11,Kr90b,SS07,Tho84,Tho95,Tur90,Wala04} for more information on the Wall Conjecture in the case $n=3$ and for known results.

\medskip

The Geometrization Theorem implies that the fundamental group of a closed, aspherical non-hyperbolic $3$-manifold $N$ contains a subgroup isomorphic to $\Z^2$, and so the fundamental group of a closed, aspherical $3$-manifold $N$ is word-hyperbolic if and only if $N$ is hyperbolic. It is therefore especially natural to ask which word-hyperbolic groups are $\pd_3$ groups.  Bestvina \cite[Remark~2.9]{Bea96}, extending earlier work of Bestvina--Mess \cite{BeM91}, characterized hyperbolic $\pd_3$ groups in terms of their \emph{Gromov boundaries}. (See \cite[Section~III.H.3]{BrH99} for the definition of the Gromov boundary of a word-hyperbolic group.) He proved that a word-hyperbolic group $\pi$ is $\pd_3$ if and only if its Gromov boundary $\partial\pi$ is homeomorphic to $S^2$.  Therefore, for word-hyperbolic groups the Wall Conjecture is equivalent to the following conjecture of Cannon.

\begin{conjecture} \textbf{\emph{(Cannon Conjecture)}}
\index{conjectures!Cannon Conjecture}
If the boundary of a word-hyperbolic group $\pi$ is homeomorphic to $S^2$, then $\pi$ acts properly discontinuously and cocompactly on $\H^3$ with finite kernel.
\end{conjecture}

This conjecture, which is the $3$-dimensional analogue of the $2$-dimensional results by Casson--Jungreis \cite{CJ94} and Gabai \cite{Gab92}
stated in the remark after Theorem \ref{thm:infcyclicnormal},
was first set out in \cite[Conjecture~5.1]{CaS98} and goes back to earlier work in \cite{Can94} (see also \cite[p.~97]{Man07}).
 We refer to \cite[Section~5]{Bok06} and \cite[Section~9]{BeK02} for a detailed discussion of the conjecture and to \cite{CFP99}, \cite{CFP01}, \cite{BoK05} and \cite{Rus10} for some positive evidence.  Markovic \cite[Theorem~1.1]{Mac12} (see also \cite[Corollary~1.5]{Hai13}) showed that Agol's Theorem (see \cite{Ag12} and Theorem \ref{thm:ag12}) gives a new approach to the Cannon Conjecture.

We finally point out that a high-dimensional analogue to the Cannon Conjecture was proved by Bartels, L\"uck and Weinberger.
More precisely, in \cite[Theorem~A]{BLW10} it is shown that if $\pi$ is a word-hyperbolic group whose boundary is homeomorphic to $S^{n-1}$
with $n\geq 6$, then $\pi$ is the fundamental group of an aspherical closed $n$-dimensional manifold.

\subsection{The Simple Loop Conjecture}

Let $f\colon \Sigma\to N$ be an embedding of  a surface into a compact  $3$-manifold.
If the induced map $f_*\colon \pi_1(\Sigma)\to \pi_1(N)$ is not injective, then
it is a consequence of the Loop Theorem  that there exists an essential simple closed loop on $\Sigma$ which lies in the kernel of $f_*$.
We refer to Theorem~\ref{thm:loop} and \cite[Corollary~3.1]{Sco74} for details.

The Simple Loop Conjecture (see, e.g., \cite[Problem~3.96]{Ki97}) posits that the same conclusion holds for any map of an orientable surface to a compact, orientable $3$-manifold:

\begin{conjecture} \textbf{\emph{(Simple Loop Conjecture)}}
\index{conjectures!Simple Loop Conjecture}
Let $f\colon \Sigma\to N$ be a map from an orientable surface to a compact, orientable $3$-manifold.
If the induced map $f_*\colon \pi_1(\Sigma)\to \pi_1(N)$ is not injective, then there exists an essential simple closed loop on $\Sigma$ which lies in the kernel of $f_*$.
\end{conjecture}

The conjecture was proved for graph manifolds by Rubinstein--Wang \cite[Theorem~3.1]{RW98}, extending earlier work of Gabai \cite[Theorem~2.1]{Gab85} and Hass \cite[Theorem~2]{Has87}.
Minsky \cite[Question~5.3]{Miy00} asked whether the conclusion of the conjecture also holds if the target is replaced by $\sl(2,\C)$.
This was answered in the negative by Louder \cite[Theorem~2]{Lou11} and Cooper--Manning \cite{CoM11} (see also \cite{Cal11}, \cite[Theorem~1.2]{Mnn12} and \cite[Section~7]{But12b}).

\subsection{Homology of finite regular covers and the volume of $3$-manifolds}

Let $N$ be an irreducible, non-spherical $3$-manifold with empty or toroidal boundary. We saw in (C.\ref{C.l2betti}) that for any cofinal regular tower $\{\ti{N}\}_{i\in \N}$ of $N$
we have
\[  \lim_{i\to \infty} \frac{b_1(\ti{N}_i;\Z)}{[\ti{N}_i:N]} = 0.\]
It is natural to ask about the limit behavior of other `measures of complexity' of groups and spaces
for cofinal regular towers of $N$. In particular we propose the following question:

\begin{question}
Let $N$ be an irreducible, non-spherical $3$-manifold with empty or toroidal boundary.
Let $\{\ti{N}_i\}_{i\in \N}$ be a cofinal regular tower of $N$.
\bn
\item[\textup{(1)}]
Does the equality
\[  \lim_{i\to \infty} \frac{b_1(\ti{N}_i;\F_p)}{[\ti{N}_i:N]} = 0\]
hold for any prime $p$?
\item[\textup{(2)}] If \textup{(1)} is answered affirmatively, then  does the following hold?
\[  \liminf_{i\to \infty} \frac{\op{rank}(\pi_1(\ti{N}_i))}{[\ti{N}_i:N]} = 0\]
\en
\end{question}

Note that it is not even clear that the  first limit exists.
The second limit is called the \emph{rank gradient} \index{rank gradient} and was first studied by Lackenby \cite{Lac05}.
Also note that the first question is a particular case of \cite[Question~1.5]{EL12} and that furthermore the second question
is asked in \cite{KN12}. It follows from \cite[Proposition~2.1]{KN12} and \cite[Theorem~4 and Proposition~9]{AJZN11} that (1) and (2) hold for graph manifolds,
and that the general case follows from answering (1) and (2) for hyperbolic $3$-manifolds.

Arguably the most interesting question is about the growth rate of the size of  in the homology of finite covers.
The following question has been raised by several authors  (see, e.g., \cite{BV13},  \cite[Question~13.73]{Lu02} and \cite[Conjecture~1.12]{Lu12}).

\begin{question}
Let $N$ be an irreducible  $3$-manifold with empty or toroidal boundary. We denote by $\operatorname{vol}(N)$ the sum of the volumes of the hyperbolic pieces in the JSJ decomposition of $N$. Does there exist a cofinal regular tower $\{\ti{N}\}_{i\in \N}$ of $N$
such that
\[  \lim_{i\to\infty}  \frac{1}{[\ti{N}_i:N]} \ln| \operatorname{Tor} H_1(\ti{N}_i;\Z)|=\frac{1}{6\pi} \operatorname{vol}(N)~?\]
Or, more optimistically, does the above equality hold for any  cofinal regular tower $\{\ti{N}\}_{i\in \N}$ of $N$?
\end{question}

\bn
\item A good introduction to this question is given in the introduction of \cite{BD13}. There the authors summarize and add to the evidence towards an affirmative answer for arithmetic hyperbolic $3$-manifolds and they also give some evidence that the answer might be negative for general hyperbolic $3$-manifolds.
\item
We refer to \cite[Theorem~1.1]{ACS06}, \cite{Shn07},\cite[Proposition~10.1]{CuS08a}, \cite[Theorem~6.7]{CDS09}, \cite[Theorem~1.2]{DeS09}, \cite[Theorem~9.6]{ACS10}, \cite[Theorem~1.2]{CuS11} for results relating the homology of a hyperbolic $3$-manifold to the hyperbolic volume.
\item
An attractive approach to the question is the result of L\"uck--Schick \cite[Theorem~0.7]{LuS99} that $\op{vol}(N)$ can be expressed in
terms of a certain $L^2$--torsion of $N$.
By \cite[Equation~8.2]{LiZ06} and  \cite[Lemma~13.53]{Lu02} the $L^2$--torsion corresponding to the abelianization
corresponds to the Mahler measure of the  Alexander polynomial.
The relationship between the Mahler measure of the Alexander polynomial  and  the growth of torsion homology is explored
by Silver--Williams \cite[Theorem~2.1]{SW02a} \cite{SW02b} (extending earlier work in \cite{GoS91,Ril90}),
Kitano--Morifuji--Takasawa \cite{KMT03},
Le \cite{Le10} and Raimbault \cite[Theorem~0.2]{Rai12a}.
\item Note that it follows from Gabai--Meyerhoff--Milley \cite[Corollary~1.3]{GMM09} \cite[Theorem~1.3]{Mie09} that for a non-graph manifold $N$ we have $\operatorname{vol}(N)>0.942$.
(See also  \cite[Theorem~3]{Ada87}, \cite{Ada88}, \cite{CaM01}, \cite{GMM10} and \cite[Theorem~3.6]{Ag10b} for more information and more results.)
\en
\medskip

Note that an affirmative answer would imply that the order of torsion in the homology of a hyperbolic $3$-manifold
grows exponentially by going to finite covers. To the best of our knowledge even the following much weaker question is still open:

\begin{question}
Let $N$ be a hyperbolic $3$-manifold. Does $N$ admit a finite cover~$\ti{N}$ with $\op{Tor} H_1(\ti{N};\Z)\ne 0$?
\end{question}

It is also interesting to study  the behavior of the $\F_p$-Betti numbers in finite covers and the number of generators of the first homology group in finite covers.  Little seems to be known about these two problems (but see \cite{LLS11} for some partial results regarding the former problem).
One intriguing question is whether
\[  \lim_{\ti{N}} \frac{ b_1(\ti{N};\F_p)}{[\ti{N}:N]}= \lim_{\ti{N}} \frac{b_1(\ti{N};\Z)}{[\ti{N}:N]} \]
for any compact $3$-manifold.  We also refer to \cite{Lac11} for further questions on $\F_p$-Betti numbers in finite covers.

Given a $3$-manifold $N$, the behavior of the homology in a cofinal regular tower
can depend on the particular choice of sequence. For example, F. Calegari--Dunfield \cite[Theorem~1]{CD06}
together with Boston--Ellenberg \cite{BE06} showed
 that there exists a closed hyperbolic $3$-manifold and a cofinal regular tower  $\{\ti{N}_i\}_{i\in\N}$
such that $b_1(\ti{N}_i)=0$ for any $i$. On the other hand we know by (G.\ref{G.15}) that $vb_1(N)>0$. Another instance of this phenomenon
can be seen in \cite[Theorem~1.2]{LLuR08}.

\subsection{Linear representations of $3$-manifold groups}\label{sec:linear repres}

We now know that the fundamental groups of most $3$-manifolds are linear. It is natural to ask what is the minimal dimension of a faithful representation  for a given $3$-manifold group.  For example,  Thurston \cite[Problem~3.33]{Ki97}
asked whether every finitely generated $3$-manifold group has a faithful representation in $\operatorname{GL}(4,\R)$.
This question was partly motivated by the study of  projective structures on $3$-manifolds,
since a projective structure on a $3$-manifold $N$ naturally gives rise to a (not necessarily faithful) representation $\pi_1(N)\to \op{PGL}(4,\R)$. We refer to \cite{CLT06,CLT07,HP11} for more information on projective structures on $3$-manifolds and to
Cooper and Goldman \cite{CoG12} for a proof that $\R P^3\# \R P^3$ does not admit a projective structure.

Thurston's question was answered in the negative by Button \cite[Corollary~5.2]{But12a}.
More precisely, Button showed that there exists a closed graph manifold $N$ which does not admit a faithful representation $\pi_1(N)\to \gl(4,k)$
for any field~$k$. \medskip

One of the main themes which emerges from this survey is that fundamental groups of closed graph manifolds are at times less well behaved than fundamental groups of irreducible $3$-manifolds which are not closed graph manifolds, e.g.\ which have a hyperbolic JSJ component.
We can therefore ask the following two questions:

\begin{question}
\mbox{}
\bn
\item[\textup{(1)}] Let $N$ be an irreducible $3$-manifolds which is not a closed graph manifold. Does $\pi_1(N)$ admit a faithful representation in $\operatorname{GL}(4,\R)$?
\item[\textup{(2)}] Does there exist an $n$ such that every finitely generated $3$-manifold group has a faithful representation in $\operatorname{GL}(n,\R)$?
\en
 \end{question}

Note though that  we do not even know whether there is an $n$ such that the fundamental group of any Seifert fibered manifold embeds in $\gl(n,\R)$; we refer to (D.\ref{D.Boyer}) for more information.
\medskip

The following was conjectured by Luo \cite[Conjecture~1]{Luo12}.

\begin{conjecture}
Let $N$ be a compact $3$-manifold. Given any non-trivial $g\in \pi_1(N)$,
 there exists
a finite commutative ring $R$ and a homomorphism $\a\colon \pi_1(N)\to \sl(2,R)$
such that $\a(g)$ is non-trivial.
\end{conjecture}

\subsection{$3$-manifold groups which are residually simple}

Long--Reid \cite[Corollary~1.3]{LoR98} showed that the fundamental group of any hyperbolic $3$-manifold is residually simple.
On the other hand, there are examples of $3$-manifold groups which are not residually simple:
 \bn
\item certain  finite fundamental groups like $\Z/4\Z$,
\item non-abelian solvable groups, like fundamental groups of non-trivial torus bundles, and
\item non-abelian groups with non-trivial center, i.e., infinite non-abelian fundamental groups of Seifert fibered spaces.
\en
 We are not aware of any other examples of $3$-manifold groups which are not  residually finite simple. We therefore pose the following question:
\index{$3$-manifold group!residually finite simple}

\begin{question}
Let $N$ be an irreducible $3$-manifold with empty or toroidal boundary. If $N$ is not geometric, is $\pi_1(N)$ residually finite simple?
\end{question}

\subsection{The group ring of a  $3$-manifold group}
We now turn to the study of group rings.
If $\pi$ is a torsion-free group, then the Zero Divisor Conjecture (see, e.g., \cite[Conjecture~10.14]{Lu02})
asserts that the group $\Z[\pi]$ has no non-trivial zero divisors.
This conjecture is still wide open; it is not even known for $3$-manifold groups.
For future reference we record this special case of the Zero Divisor Conjecture: \index{conjectures!Zero Divisor Conjecture}

\begin{conjecture}\label{conj:zerodivisor}
Let $N$ be an aspherical  $3$-manifold with empty or toroidal boundary.
Then $\Z[\pi_1(N)]$ has no non-trivial zero divisors.
\end{conjecture}

Let $\G$ now be any torsion-free group. Then $\Z[\G]$ has no non-trivial zero divisors if  one of the following holds:
\begin{itemize}
\item[(1)] $\G$ is elementary amenable (e.g., solvable-by-finite),
\item[(2)] $\G$ is locally indicable, or
\item[(3)] $\G$ is left-orderable.
\end{itemize}
We refer to \cite[Theorem~1.4]{KLM88}, \cite[Proposition~6]{RoZ98}, \cite[Theorem~4.3]{Lin93} and \cite[Theorem~12]{Hig40} for the proofs.
It is clear that if a group $\G$ is residually a group for which the Zero Divisor Conjecture holds,
then it also holds for $\G$. Thus Conjecture~\ref{conj:zerodivisor} holds
if the following question  is answered in the affirmative:

\begin{questions}\label{qu:restfea}
Let $N$ be an aspherical  $3$-manifold with empty or toroidal boundary.
Is the group $\pi_1(N)$ residually torsion-free elementary amenable?
\end{questions}

A related question arises when one studies Ore localizations (see, e.g., \cite[Section~8.2.1]{Lu02} for a survey).
If $\G$ contains a non-cyclic free group, then  $\Z[\G]$ does not admit an Ore localization (see, e.g., \cite[Proposition~2.2]{Lin06}).
On the other hand, if  $\G$ is an amenable group, then $\Z[\G]$ admits an Ore localization $\K(\G)$ (see, e.g., \cite{Ta57} and \cite[Corollary~6.3]{DLMSY03}).
If $\G$ satisfies furthermore the  Zero Divisor Conjecture,
then the natural map $\Z[\G]\to \K(\G)$ is injective.  We can then view $\Z[\G]$ as a subring of the skew field $\K(\G)$
and $\K(\G)$ is flat over $\Z[\G]$.

\medskip

Let $N$ be an aspherical $3$-manifold with empty or toroidal boundary.  It seems  reasonable to ask whether the group ring $\Z[\pi_1(N)]$ is residually a skew field which is flat over $\Z[\pi_1(N)]$.  Note that an affirmative answer would follow if one of the following holds:
\begin{itemize}
\item[(1)] $\pi_1(N)$ is residually  torsion-free--elementary amenable,
\item[(2)] $\G$ is residually locally indicable--amenable,
\item[(3)] $\G$ is residually left-orderable--amenable.
\end{itemize}
 Maps from $\Z[\pi_1(N)]$ to skew fields played a major role in the work of Cochran--Orr--Teichner \cite{COT03}, Cochran \cite{Coc04} and Harvey \cite{Har05}.

\subsection{Potence}

Recall that a  group $\pi$ is called \emph{potent} if for any non-trivial $g\in \pi$ and any $n\in \N$ there exists an epimorphism
$\a\colon \pi\to G$ onto a finite group $G$ such that $\a(g)$ has order $n$. \index{group!potent}
As we saw above, many $3$-manifold groups are virtually potent.
It is also straightforward to see that fundamental groups of fibered $3$-manifolds are potent.
Also,  Shalen \cite{Shn12} proved the following result: Let $\pi$ be the fundamental group
of a hyperbolic $3$-manifold and let $n>2$ be an integer.  Then there exist finitely many conjugacy classes $C_1,\dots,C_m$ in $\pi$
such that for any $g\not\in C_1\cup \dots \cup C_m$ there exists a homomorphism $\a\colon  \pi_1(N)\to G$ onto a finite group $G$
such that $\a(g)$ has order $n$.

\medskip

The following question naturally arises:

\begin{question}
Let $N$ be an aspherical $3$-manifold with empty or toroidal boundary.  Is $\pi_1(N)$ potent?
\end{question}

\subsection{Left-orderability and Heegaard-Floer $L$-spaces}

Let $N$ be an irreducible $3$-manifold with empty or toroidal boundary.
By (C.\ref{C.locallyindicable}) and (C.\ref{C.left-orderable}) above, if $b_1(N)\geq 1$, then $\pi_1(N)$ is left-orderable. (See also \cite[Theorem~1.1]{BRW05} for a different approach.)
On the other hand if $b_1(N)=0$, i.e., if $N$ is a rational homology sphere, then there is presently no good criterion for determining whether $\pi_1(N)$ is left-orderable or not.
Before we formulate the subsequent conjecture we recall that a rational homology sphere~$N$ is called an \emph{$L$-space}
if the total rank of its Heegaard Floer homology $\widehat{HF}(N)$ equals $|H_1(N;\Z)|$. We refer to the  foundational papers
of Ozsv\'ath--Szab\'o \cite{OzS04a,OzS04b} for details on Heegaard Floer homology and \cite{OzS05} for the definition of $L$-spaces. \index{$3$-manifold!$L$-space}

\medskip

The following conjecture was formulated by Boyer--Gordon--Watson \cite[Conjecture~3]{BGW11}:

\begin{conjecture}\label{conj:bgw}
Let $N$ be an irreducible rational homology sphere. Then $\pi_1(N)$ is left-orderable if and only if $N$ is not an $L$-space.
\end{conjecture}

See \cite{BGW11} for background and   \cite{Pet09,BGW11,CyW12,CyW11,CLW11,LiW11,ClT11,LeL11,Ter11,HaTe12a,HaTe12b,HaTe13,Tra13,MTe13,BoB13} for evidence towards an affirmative answer and for relations of these notions to the existence of taut foliations.

\medskip

A link between left-orderability and $L$-spaces is given by the (non-) existence of certain foliations on $3$-manifolds.
We refer to \cite[Section~7]{CD03} and \cite{RSS03,RoS10} for the interaction between left-orderability and foliations.
Ozsv\'ath and Szab\'o \cite[Theorem~1.4]{OzS04c} on the other hand proved that an $L$-space does not admit a co-orientable taut foliation.
An affirmative answer to Conjecture \ref{conj:bgw} would thus imply that the fundamental group of a rational homology sphere which admits a coorientable taut foliation
is left-orderable. The following theorem can be seen as  evidence towards the conjecture.
\begin{theorem}
Let $N$ be an  irreducible $\Z$-homology sphere which admits a co-orientable taut foliation. Suppose  $N$ that is either a graph manifold or that the JSJ decomposition of $N$ is trivial. Then $\pi_1(N)$ is left-orderable.
\end{theorem}

The case that $N$ is Seifert fibered follows from \cite[Corollary~3.12]{BRW05}, the hyperbolic case is a consequence of \cite[Theorems~6.3~and~7.2]{CD03},
and the graph manifold case is precisely \cite[Theorem~1]{CLW11}.

\subsection{$3$-manifold groups and knot theory}

An \emph{$n$-knot group}\index{group!knot group} is the fundamental groups of the knot exterior $S^{n}\setminus \nu K$
where $K$ is a smoothly embedded $(n-2)$-sphere. Every knot group $\pi$ has the following properties:
\bn
\item $\pi$ is finitely presented.
\item The abelianization of $\pi$ is isomorphic to $\Z$.
\item $H_2(\pi)=0$.
\item The group $\pi$ has weight $1$. Here a group $\pi$ is said to be of \emph{weight $1$} if it admits a normal generator,
i.e., if there exists a $g\in \pi$ such that the smallest normal subgroup containing $g$ equals $\pi$. \index{group!of weight $1$}
\en
The first three properties are obvious, the fourth property follows from the fact that a meridian is a normal generator.
Kervaire \cite{Ker65} showed that for $n\geq 5$ these conditions in fact characterize $n$-knot groups.
This is not true in the case that $n=4$, see, e.g., \cite{Hil77,Lev78,Hil89},
 and it is not true if $n=3$. In the latter case a straightforward example is given by the Baumslag-Solitar group $BS(1)$,
see Section \ref{section:ribbon}. More subtle examples for $n=3$ are given by Rosebrock; see \cite{Bue93,Ros94}.

We now restrict ourselves to the case $n=3$. In particular we henceforth refer to a 3-knot groups a knot group. The following question, which is related to the discussion in Section \ref{section:pdg}, naturally arises.

\begin{question}
Is there a group-theoretic characterization of knot groups?
\end{question}

Knot groups have been studied intensively since the very beginning of  $3$-manifold topology.
They serve partly as a laboratory for the general study of $3$-manifold groups, but of course there are also results and questions specific to knot groups.
We refer to \cite{Neh65,Neh74} for a summary of some early work,
to \cite{GA75,Joh80,JL89} for results on homomorphic images of knot groups,
 and to \cite{Str74,Eim00,KrM04}, and \cite{AL12} for further results.
We will now discus several open questions regarding knot groups.

If $N$ is obtained by Dehn surgery along a knot $K\subseteq S^3$, then the image of the meridian  of $K$ is a normal generator of $\pi_1(N)$,
i.e., $\pi_1(N)$ has weight~$1$.  The converse does not hold, i.e., there exist closed $3$-manifolds $N$ such that $\pi_1(N)$ has weight~$1$, but which are not obtained by Dehn surgery along a knot $K\subseteq S^3$. For example, if $N=P_1\# P_2$ is the connected sum of two copies  of the Poincar\'e homology sphere $P$, then $\pi_1(N)$ is normally generated by $a_1a_2$, where $a_1\in \pi_1(P_1)$ is an element of order 3 and $a_2\in \pi_1(P_2)$ is an element of order 5.  On the other hand it follows from \cite[Corollary~3.1]{GLu89} that $N$ cannot be obtained from Dehn surgery along a knot $K\subseteq S^3$.

\medskip

The following question, which is still open, is a variation of a question asked by Cochran (see \cite[p.~550]{GeS87}).

\begin{question}
Let $N$ be a closed, orientable, irreducible $3$-manifold such that $\pi_1(N)$ has weight~$1$.
Is $N$ the result of Dehn surgery along a knot $K\subseteq S^3$?
\end{question}

Another question concerning fundamental groups of knot complements is the following, due to Cappell--Shaneson (see \cite[Problem~1.11]{Ki97}).

\begin{question}
Let $K\subseteq S^3$ be a knot such that $\pi_1(S^3\setminus \nu K)$ is generated by $n$-meridional generators.
Is $K$  an $n$-bridge knot?
\end{question}

The case $n=1$ is a consequence of the Loop Theorem and the case $n=2$ is a consequence of the work of Boileau and Zimmermann \cite[Corollary~3.3]{BoZi89} together with the Orbifold Geometrization Theorem (see \cite{BMP03,BLP05}).
Bleiler \cite[Problem~1.73]{Ki97} suggested a generalization to knots in general $3$-manifolds
and gave some evidence towards its truth (see \cite{BJ04}),
but  results of Li \cite[Theorem~1.1]{Lia11} can be used to show that Bleiler's conjecture is false in general.

The following question also concerns the relationship between generators and topology of a knot complement.

\begin{question}
Let $K\subset S^3$ be a knot such that $\pi_1(S^3\setminus \nu K)$ is generated by two elements.
Is $K$ a tunnel number one knot?
\end{question}

Here a knot is said to have tunnel number one if there exists a properly embedded arc $A$ in $S^3\setminus \nu K$
such that $S^3\setminus \nu (K\cup A)$ is a handlebody.
Some evidence towards this conjecture is given in \cite{Ble94,BJ04} and \cite[Corollary~7]{BW05}.

It is straightforward to see that
if $K\subseteq S^3$ is a knot, then any meridian normally generates $\pi=\pi_1(S^3\setminus \nu K)$.
An element $g\in \pi$ is called a \emph{pseudo-meridian of $K$}
\index{pseudo-meridian} if it normally generates $\pi$
but if there is no automorphism of $\pi$ which sends $g$ to a meridian.
Examples of pseudo-meridians were first given by Tsau \cite[Theorem~3.11]{Ts85}. Silver--Whitten--Williams \cite[Corollary~1.3]{SWW10} showed that
every non-trivial hyperbolic 2-bridge knot, every torus knot and every hyperbolic knot with unknotting number one admits a pseudo-meridian.
The following conjecture was proposed in \cite[Conjecture~3.3]{SWW10}.\index{knot!pseudo-meridian}

\begin{question}
Does every non-trivial knot $K\subseteq S^3$ have a pseudo-meridian?
\end{question}

\subsection{Ranks of finite-index subgroups}

The rank $\op{rk}(\pi)$ of a finitely generated group $\pi$ is defined as the minimal number of generators of $\pi$.
Reid \cite[p.~212]{Red92} showed that there exists a closed hyperbolic $3$-manifold $N$ such that $N$ admits a finite cover $\ti{N}$
with $\op{rk}(\pi_1(\ti{N}))=\op{rk}(\pi_1(N))-1$.

It is still an open question whether the rank can drop by more than one while going to a finite cover.
More precisely, the following conjecture was formulated by Shalen \cite[Conjecture~4.2]{Shn07}.

\begin{conjecture}
If $N$ is a compact, orientable hyperbolic $3$--manifold, then for any finite cover $\ti{N}$ of $N$
we have $\op{rk}(\pi_1(\ti{N}))\geq \op{rk}(\pi_1(N))-1$.
\end{conjecture}

Note that if $\G$ is a finite-index subgroup of a finitely generated group $\pi$, then it follows from a transfer argument
that $b_1(\G)\geq b_1(\pi)$. More subtle evidence towards the conjecture is given by \cite[Corollary~7.3]{ACS06} which states
that if $N$ is a closed orientable $3$-manifold and $p$ a prime, then  the rank of $\F_p$-homology can drop by at most one by going to a finite cover.

\subsection{$3$-manifold groups and their finite quotients}

As before, we denote by $\hat{\pi}$ the profinite completion of a group $\pi$.
It is natural to ask to what degree a residually finite group is determined by its profinite completion.
This question goes back to Grothendieck \cite{Grk70} and it is studied in the general group-theoretic context in detail in \cite{Pi74,GPS80,GZ11}.
Funar \cite[Corollary~1.4]{Fun11}, using work of Stebe \cite[p.~3]{Ste72}, observed that the
answer is negative for $\Sol$-manifolds; in fact, there exist  non-homeomorphic $\Sol$-manifolds with isomorphic profinite completions. (See (I.\ref{I.grothendieckrigid}).)

However, we have seen throughout this survey that $\Sol$-manifolds play a special role in $3$-manifold topology
and are usually not representative of other $3$-manifolds.
The following question (see also \cite[p.~481]{LoR11}  and  \cite[Remark~3.7]{CFW10})
is still open.

\begin{question}
Let $N_1$ and $N_2$ be  compact, orientable, irreducible  $3$-manifolds with empty or toroidal boundary, which
are not $\Sol$-manifolds.  Does $\widehat{\pi_1(N_1)}\cong\widehat{\pi_1(N_2)}$ imply $\pi_1(N_1)\cong \pi_1(N_2)$?
\end{question}

By \cite[Corollary~3.2.8]{RiZ10} two finitely generated groups have isomorphic profinite completions if and only if they the same finite quotients.  A positive answer to the above question would thus in particular give an alternative solution to the isomorphism problem for such $3$-manifold groups.

Fundamental groups of $\Sol$-manifolds are virtually polycyclic, it thus follows from \cite[p.~155]{GPS80}
that there are at most finitely many $\Sol$-manifolds with the same profinite completion.
On the other hand there are infinite classes of finitely presented groups which have the same profinite completion (see e.g. \cite{Pi74}).
If the answer to the above question is negative, it would therefore be interesting to study the weaker question, whether only finitely
many $3$-manifold groups can have the same profinite completion.

As mentioned in (I.\ref{I.grothendieckrigid}), Cavendish used the fact that $3$-manifold groups are good in the sense of Serre (see (G.\ref{G.good})) to show that  fundamental groups of closed, irreducible $3$-manifolds are Grothendieck rigid.  In fact, one can deduce more.

\begin{proposition}
Let $N_1$,~$N_2$ compact  aspherical $3$-manifolds, and suppose $N_1$ is closed and $N_2$ has non-empty boundary.  Then
\[
\widehat{\pi_1(N_1)}\ncong \widehat{\pi_1(N_2)}~.
\]
\end{proposition}
\begin{proof}
Because $\pi_1(N_1)$ and $\pi_1(N_2)$ are both good,
\[
H^3(\widehat{\pi_1(N_i)};\Z/2)\cong H^3(\pi_1(N_i);\Z/2)
\]
for $i=1,2$.  But
\[
H^3(\pi_1(N_1);\Z/2)\cong H^3(N_1;\Z/2)\cong\Z/2
\]
whereas
\[
H^3(\pi_1(N_2);\Z/2)\cong H^3(N_2;\Z/2)\cong 0
\]
so the two profinite completions cannot be isomorphic.
\end{proof}

\subsection{Free-by-cyclic groups}\label{section:freebycyclic}

A \emph{finitely-generated-free-by-infinite-cyclic} group is a group $\pi$ which admits an epimorphism onto $\Z$ such that the kernel is a finitely generated non-cyclic free group. By a slight abuse of language we refer to such groups henceforth as \emph{free-by-cyclic groups}.
\index{group!free-by-cyclic}

Note that if $\pi$ is a free-by-cyclic group, then the epimorphism $\pi\to \Z$ splits,
and we can thus write $\pi$ as a semidirect product $\Z\ltimes F$ where $F$ is a free group.
Let $F$ be a non-cyclic free group and $\phi\colon F\to F$ an isomorphism.
We say $\phi$ is \emph{topologically realizable} if there exists a surface $\Sigma$ with boundary, a self-diffeomorphism $f\colon \Sigma\to \Sigma$
 and an isomorphism $g\colon F\to \pi_1(\Sigma)$ such that $g^{-1}\circ f_*\circ g=\phi$.
 If $\phi$ is topologically realized by $(\Sigma,f)$, then  the semidirect product $\pi:=\Z\ltimes_\phi F$ is the fundamental group of the mapping torus of $(\Sigma,f)$.
Hence $\pi=\Z\ltimes_\phi F$ is the fundamental group of an irreducible $3$-manifold with non-trivial toroidal boundary.
It now follows that $\pi$ has the following properties:
\bn
\item $\pi$ is coherent by (C.\ref{C.scottcore}).
\item $\pi$ has a 2-dimensional Eilenberg--Mac~Lane space by (C.\ref{C.sphere}).
\item If $N$ is atoroidal, then it follows from Theorem \ref{thm:dah03} that $\pi$ is hyperbolic relative to the subgroups $\pi_1(T_i)$, where $T_1,\dots, T_k$ are the boundary components of $N$.
\item $\pi$ contains a surface group.
\item By Theorem \ref{thm:leeb}, $\pi$ is a $\op{CAT}(0)$-group, i.e., $\pi$ acts properly and cocompactly by isometries on a CAT(0) space.
\index{group!CAT(0) group}
\item $\pi$ is virtually special by Theorems \ref{thm:akmw}, \ref{thm:liu11} and \ref{thm:pw12}.  In particular, it follows from the discussion in Section \ref{section:diagram} that $\pi$
\bn
\item  has a finite-index subgroup which is residually torsion-free nilpotent;
\item  is linear over $\Z$;
\item  is large, in particular $vb_1(\pi)=\infty$;
\item  is LERF, if $N$ is atoroidal.
\en
\item $\pi$ is conjugacy separable.
\en
Not every free-by-cyclic group is the fundamental group of a $3$-manifold.
Indeed, Stallings \cite[p.~22]{Sta82} (see also \cite[Theorem~3.9]{Ge83}) showed that `most' automorphisms of a free group are in fact not topologically realizable.  Bestvina--Handel gave a complete characterization for an automorphism of a free group to be realized by a pseudo-Anosov self-diffeomorphism of a surface with one boundary component (\cite[Theorem~4.1]{BeH92}; see also \cite[Remark 4.2]{BeH92} for a more general statement). The question to which extent  properties of fundamental groups of fibered $3$-manifolds with boundary carry over to the more general case of free-by-cyclic groups has been studied by many authors, see e.g. the references below and also \cite{AlR12,KR12,DKL12}.

We  summarize some known properties of  free-by-cyclic groups. The subsequent list should be compared with the above list of properties of fundamental groups of mapping tori.
\bn
\item[(1)] Every free-by-cyclic group is coherent by the work of Feighn--Handel \cite[Theorem~1.1]{FeH99}.
\item[(2)] It is straightforward to see that any free-by-cyclic group admits a 2--dimensional Eilenberg--Mac~Lane space.
\item[(3)] If $\phi$ is an automorphism of a free group $F$ which is atoroidal, i.e.
 which has no nontrivial
periodic conjugacy classes, then $\pi=\Z\ltimes_\phi F$ is word-hyperbolic by Brinkmann \cite[Theorem~1.2]{Brm00} (see also \cite{BF92,BFH97}).
\item[(5)] Gersten \cite[Proposition~2.1]{Ge94b} exhibited a free-by-cyclic group $\Gamma$ that does not act properly discontinuously by isometries on any CAT(0) space.  It follows that the same holds for any finite-index subgroup of $\Gamma$; in particular, $\Gamma$ is not virtually special.  See also \cite{Bra95,Sam06} for some examples of free-by-cyclic groups that act properly discontinuously and cocompactly by isometries on CAT(0) spaces.  Bridson--Groves \cite{BrGs10} (see also \cite[Theorem~1.1]{Maa00}) showed that $\G$ satisfies a `quadratic isoperimetric inequality', a condition which is also satisfied by CAT(0) groups.
\item[(6a)] For any prime $p$, a free-by-cyclic group is virtually residually $p$ (see (I.\ref{I.fibresp})).
\item[(6d)] Leary--Niblo--Wise \cite[Proposition~4]{LNW99} showed that there exist word-hyperbolic free-by-cyclic groups which are not LERF.
\item[(7)] The conjugacy problem is solvable for  free-by-cyclic groups, by \cite[Theorem~1.1]{BMMV06} and \cite[Corollary~B]{BrGs10}.
\en
Although it is still unknown whether or not all free-by-cyclic groups contain surface subgroups, Calegari and Walker  showed that `most' mapping tori of free group \emph{endomorphisms} contain a surface subgroup \cite[Theorem~8.9]{CW12}.  We thus see that in particular the following questions are open.

\begin{question}\label{qu: Free-by-cyclic group questions}
\mbox{}

\begin{itemize}
\item[(1)] Does every free-by-cyclic group contain a surface subgroup?
\item[(2)] Is every  free-by-cyclic group linear?
\item[(3)] Does every free-by-cyclic group admit a finite-index subgroup with $b_1\geq 2$?
\item[(4)] Is every free-by-cyclic group large?
\item[(5)] Is every free-by-cyclic group conjugacy separable?
\end{itemize}
\end{question}

By Brinkmann \cite[Theorem~1.2]{Brm00}, Question \ref{qu: Free-by-cyclic group questions},~(1) is a special case of a question attributed to Gromov: does every one-ended word-hyperbolic group contain a surface group (see, e.g., \cite{Brd07})?  Note that Question \ref{qu: Free-by-cyclic group questions},~(3) was raised by Casson (see \cite[Question~12.16]{Bea04}).  We refer to \cite[Corollary~3.2]{But07}, \cite[Corollary~4.6]{But08} and \cite[Theorem~3.2]{But11a} for some partial results regarding Casson's question.

\subsection{Ribbon groups}\label{section:ribbon}

A \emph{ribbon group}
\index{group!ribbon}
 is a group $\pi$ with $H_1(\pi;\Z)\cong \Z$ and which admits a Wirtinger presentation of  deficiency~$1$,
 i.e., a presentation
\[ \big\langle g_1,\dots,g_{k+1}\ |\ g_{\sigma(1)}^{\eps_1}g_{1}g_{\sigma(1)}^{-\eps_1}g_{2}^{-1},\ \dots,\ g_{\sigma(k)}^{\eps_k}g_{k}g_{\sigma(k)}^{-\eps_k}g_{k+1}^{-1}\big\rangle\]
where $\sigma\colon \{1,\dots,k\}\to \{1,\dots,k+1\}$ is a map and $\eps_i\in \{-1,1\}$ for $i=1,\dots,k$.
The name ribbon group comes from the fact these groups are precisely the fundamental groups of ribbon disk complements in $D^4$
(see \cite[Theorem~2.1]{FT05} or \cite[p.~22]{Hil02}).

It is well known (see, e.g., \cite[p.~57]{Rol90}) that if $\pi$ is a knot group,
i.e., if $\pi\cong \pi_1(S^3\setminus \nu K)$ where $K\subset S^3$ is a knot,
then $\pi$ is also a ribbon group. Note that knot groups are fundamental groups of irreducible $3$-manifolds with non-trivial toroidal boundary; in particular they have the properties (1) to (7) listed in the beginning of Section \ref{section:freebycyclic}.

 On the other hand not all ribbon groups are $3$-manifold groups, let alone knot groups.  For example, for any $m\in\N$ the Baumslag--Solitar group
\[ BS(m)=\ll a,b\ |\  ba^mb^{-1}=a^{m+1} \rr =\ll a,b\ |\ a^mba^{-m}=ba\rr\]
is a ribbon group but not a knot group. Indeed, following \cite[p.~129]{Kul05}, we see that with $x=ba$ and setting $\ol{g}:=g^{-1}$, the group $BS(m)$ is  isomorphic to
\begin{align*}
& \left\langle x,b, b_1,\dots, b_{m-1}\ |\ (\ol{b}x)b\ol{(\ol{b}x)}=b_1,\dots,(\ol{b}x)b_{m-1}\ol{(\ol{b}x)}=x\right\rangle = \\
&\left\langle x,b, b_1,\dots, b_{m-1}\ |\ xb\ol{x}=bb_1\ol{b},\dots,xb_{m-1}\ol{x}=bx\ol{b}\,\right\rangle = \\
&\left\langle x,b, b_1,\dots, b_{m-1},a_1,\dots,a_{m}\ |\ xb\ol{x}=a_1=bb_1\ol{b},\dots,xb_{m-1}\ol{x}=a_{m}=bx\ol{b}\,\right\rangle,
\end{align*}
which is a ribbon group.
Note that the group $BS(1)$ is isomorphic to the solvable group $\Z\ltimes \Z[1/2]$,
which by Theorem~\ref{thm:virtsolv} implies that $\pi$ is not a $3$-manifold group.
It is shown in \cite[Theorem~1]{BaS62} that  $BS(m)$ is not Hopfian if $m>1$,
which by (C.\ref{C.25}) and (C.\ref{C.26}) then also implies that $BS(m)$ is not a $3$-manifold group.
(See also \cite[Theorem~1]{Shn01} and \cite[Theorem~VI.2.1]{JS79}.)
We refer to \cite[Theorem~3]{Ros94} for more examples of ribbon groups which are not knot groups.

We thus see that ribbon groups, which from the point of view of group presentations look like a mild generalization of knot groups, can exhibit very different behavior.  It is an interesting question whether the `good properties' of knot groups or the `bad properties' of the Baumslag--Solitar groups $BS(m)$ (for $m>1$) are prevalent among ribbon groups.

Very little is known about the general properties of ribbon groups. In particular, the following question is still open.

\begin{question}\label{question:ribbon}
Is the canonical $2$-complex corresponding to a Wirtinger presentation of  deficiency~$1$ of a ribbon group
 an Eilenberg--Mac~Lane space?
\end{question}

An affirmative answer to this question would be an important step towards resolving the question of which knots
bound ribbon disks; see \cite[p.~2136f]{FT05} for details.
Howie (see \cite[Theorem~5.2]{How82} and \cite[Section~10]{How85}) answered this question in the affirmative for certain ribbon groups, e.g., for locally indicable ribbon groups.
See  \cite{IK01,HuR01,HaR03,Ivb05,Bed11,HaR12} for further work.

We conclude this section with a conjecture  due to Whitehead \cite{Whd41b}.

\index{conjectures!Whitehead Conjecture}

\begin{conjecture}\textbf{\emph{(Whitehead)}}
Any subcomplex of an aspherical $2$-complex is also aspherical.
\end{conjecture}

A proof of the Whitehead Conjecture would give an affirmative answer
to Question \ref{question:ribbon}. Indeed, if  $X$ is the canonical 2-complex corresponding to a Wirtinger presentation of  deficiency one of a ribbon group, then the $2$-complex which is given by attaching a 2-cell to any of the generators is easily seen to be aspherical.
We refer to \cite{Bog93,Ros07} for survey articles on the Whitehead Conjecture and to \cite[Theorem~8.7]{BeB97} for some negative evidence.

\subsection{(Non-) Fibered faces in finite covers of $3$-manifolds}

If $N$ is an irreducible $3$-manifold which is not a graph manifold, then by
 Proposition \ref{prop:manyfiberedfaces}, $N$ admits finite covers with an arbitrarily large number of fibered faces. We conclude this survey  with the following two questions on virtual (non-) fiberedness:

\begin{question}
Does every irreducible non-positively curved $3$-manifold admit a finite cover such that all faces of the Thurston norm ball are fibered?
\end{question}

If $N$ is a $3$-manifold which is not finitely covered by a torus bundle and with $vb_1(N)\geq 2$,
then $N$ admits a finite cover $N'$ with non-vanishing Thurston norm and with $b_1(N')\geq 2$.
It then follows from \cite[Theorem~5]{Thu86a} that $N'$ admits a non-trivial class $\phi\in H^1(N';\R)$
which is non-fibered.

Surprisingly, though, the following question is still open.

\begin{question}
Does every irreducible  $3$-manifold which is not a graph manifold admit a finite cover such that at least one top-dimensional  face of the Thurston norm ball is not fibered?
\end{question}

\printindex

\end{document}

%% file: diagram-after-geom-v5.pstex_t
\begin{picture}(0,0)%
\includegraphics{diagram-after-geom-v5.eps}%
\end{picture}%
\setlength{\unitlength}{3355sp}%
\begingroup\makeatletter\ifx\SetFigFont\undefined%
\gdef\SetFigFont#1#2#3#4#5{%
  \reset@font\fontsize{#1}{#2pt}%
  \fontfamily{#3}\fontseries{#4}\fontshape{#5}%
  \selectfont}%
\fi\endgroup%
\begin{picture}(8033,12901)(2841,-13678)
\put(8999,-10926){\makebox(0,0)[lb]{\smash{{\SetFigFont{12}{14.4}{\rmdefault}{\mddefault}{\updefault}{\color[rgb]{0,.56,0}\begin{tabular}{l}for any prime $p$\\$\mbox{vb}_1(N;\Bbb{F}_p)=\infty$\end{tabular}}%
}}}}
\put(8237,-9208){\makebox(0,0)[lb]{\smash{{\SetFigFont{12}{14.4}{\rmdefault}{\mddefault}{\updefault}{\color[rgb]{0,.56,0}\begin{tabular}{l}$\pi$\,admits\\non-abelian\\free\,subgroup\end{tabular}}%
}}}}
\put(9481,-6390){\makebox(0,0)[lb]{\smash{{\SetFigFont{12}{14.4}{\rmdefault}{\mddefault}{\updefault}{\color[rgb]{0,.56,0}$N$ hyperbolic}%
}}}}
\put(6542,-10210){\makebox(0,0)[lb]{\smash{{\SetFigFont{12}{14.4}{\rmdefault}{\mddefault}{\updefault}{\color[rgb]{0,0,0}\begin{tabular}{l}$N$\,admits\,non-separating\\non-fiber\,surface \end{tabular}}%
}}}}
\put(6577,-9099){\makebox(0,0)[lb]{\smash{{\SetFigFont{12}{14.4}{\rmdefault}{\mddefault}{\updefault}{\color[rgb]{0,0,0}$N$\,Haken}%
}}}}
\put(4741,-8272){\makebox(0,0)[lb]{\smash{{\SetFigFont{12}{14.4}{\rmdefault}{\mddefault}{\updefault}{\color[rgb]{0,0,0}$\pi$\,locally\,indicable}%
}}}}
\put(4862,-7427){\makebox(0,0)[lb]{\smash{{\SetFigFont{12}{14.4}{\rmdefault}{\mddefault}{\updefault}{\color[rgb]{0,0,0}$\pi$\,left-orderable}%
}}}}
\put(4851,-9135){\makebox(0,0)[lb]{\smash{{\SetFigFont{12}{14.4}{\rmdefault}{\mddefault}{\updefault}{\color[rgb]{0,0,0}$b_1(N)\geq 1$}%
}}}}
\put(4814,-10142){\makebox(0,0)[lb]{\smash{{\SetFigFont{12}{14.4}{\rmdefault}{\mddefault}{\updefault}{\color[rgb]{0,0,0}$b_1(N)\geq 2$}%
}}}}
\put(6809,-10910){\makebox(0,0)[lb]{\smash{{\SetFigFont{12}{14.4}{\rmdefault}{\mddefault}{\updefault}{\color[rgb]{0,.56,0}$\mbox{vb}_1(N;\Bbb{Z})=\infty$}%
}}}}
\put(7413,-11569){\makebox(0,0)[lb]{\smash{{\SetFigFont{12}{14.4}{\rmdefault}{\mddefault}{\updefault}{\color[rgb]{0,.56,0}$N$\,homologically\,large}%
}}}}
\put(4119,-13601){\makebox(0,0)[lb]{\smash{{\SetFigFont{12}{14.4}{\rmdefault}{\mddefault}{\updefault}{\color[rgb]{0,0,0}Diagram 1. Consequences of the Geometrization Theorem.}%
}}}}
\put(2856,-10709){\makebox(0,0)[lb]{\smash{{\SetFigFont{12}{14.4}{\rmdefault}{\mddefault}{\updefault}{\color[rgb]{0,0,0}\begin{tabular}{c}$\pi$\,linear\\over\,$\Bbb{Z}$\end{tabular}}%
}}}}
\put(4064,-11350){\makebox(0,0)[lb]{\smash{{\SetFigFont{12}{14.4}{\rmdefault}{\mddefault}{\updefault}{\color[rgb]{0,0,0}\begin{tabular}{l}$N$\,contains\\an\,incompressible\,torus\end{tabular}}%
}}}}
\put(8950,-12466){\makebox(0,0)[lb]{\smash{{\SetFigFont{12}{14.4}{\rmdefault}{\mddefault}{\updefault}{\color[rgb]{0,.56,0}$\pi$\,large}%
}}}}
\put(4283,-12466){\makebox(0,0)[lb]{\smash{{\SetFigFont{12}{14.4}{\rmdefault}{\mddefault}{\updefault}{\color[rgb]{0,0,0}$N$\,contains\,a\,separable\,non-fiber\,surface}%
}}}}
\put(3753,-13023){\makebox(0,0)[lb]{\smash{{\SetFigFont{12}{14.4}{\rmdefault}{\mddefault}{\updefault}{\color[rgb]{0,0,0}$\pi$\,double\,coset\,separable}%
}}}}
\put(7226,-13052){\makebox(0,0)[lb]{\smash{{\SetFigFont{12}{14.4}{\rmdefault}{\mddefault}{\updefault}{\color[rgb]{0,0,0}$\pi$\,LERF}%
}}}}
\put(3117,-5192){\makebox(0,0)[lb]{\smash{{\SetFigFont{12}{14.4}{\rmdefault}{\mddefault}{\updefault}{\color[rgb]{0,0,0}$\pi$\,torsion\,free}%
}}}}
\put(3264,-4039){\makebox(0,0)[lb]{\smash{{\SetFigFont{12}{14.4}{\rmdefault}{\mddefault}{\updefault}{\color[rgb]{0,0,0}$N=K(\pi,1)$}%
}}}}
\put(7957,-6164){\makebox(0,0)[lb]{\smash{{\SetFigFont{12}{14.4}{\rmdefault}{\mddefault}{\updefault}{\color[rgb]{0,0,0}\begin{tabular}{l}$\pi$\,residually\\finite\,simple\end{tabular}}%
}}}}
\put(7745,-7344){\makebox(0,0)[lb]{\smash{{\SetFigFont{12}{14.4}{\rmdefault}{\mddefault}{\updefault}{\color[rgb]{0,0,0}\begin{tabular}{r}$\pi$\,discrete\,subgroup\,of\,$\mbox{SL}(2,\overline{\Bbb{Q}})$\end{tabular}}%
}}}}
\put(3001,-1149){\makebox(0,0)[lb]{\smash{{\SetFigFont{12}{14.4}{\rmdefault}{\mddefault}{\updefault}{\color[rgb]{0,0,0}\begin{tabular}{c}$N$ is an irreducible, compact, orientable 3-manifold $N$\\with\,empty\,or\,toroidal\,boundary\,such\,that\,$\pi=\pi_1(N)$\,neither\,finite\,nor\,solvable\end{tabular}}%
}}}}
\put(3863,-2547){\makebox(0,0)[lb]{\smash{{\SetFigFont{12}{14.4}{\rmdefault}{\mddefault}{\updefault}{\color[rgb]{0,0,0}\begin{tabular}{c}deficiency\\of\,$\pi$\,equals\\$1-b_3(N)$\end{tabular}}%
}}}}
\put(4576,-3535){\makebox(0,0)[lb]{\smash{{\SetFigFont{12}{14.4}{\rmdefault}{\mddefault}{\updefault}{\color[rgb]{0,0,0}$\pi$\,coherent}%
}}}}
\put(9591,-5164){\makebox(0,0)[lb]{\smash{{\SetFigFont{12}{14.4}{\rmdefault}{\mddefault}{\updefault}{\color[rgb]{0,0,0}$\underset{\ti{N}}{\lim}\frac{b_1(\ti{N})}{[N:\ti{N}]}=0$}%
}}}}
\put(6077,-5075){\makebox(0,0)[lb]{\smash{{\SetFigFont{12}{14.4}{\rmdefault}{\mddefault}{\updefault}{\color[rgb]{0,0,0}$\pi$\,Hopfian}%
}}}}
\put(3782,-6481){\makebox(0,0)[lb]{\smash{{\SetFigFont{12}{14.4}{\rmdefault}{\mddefault}{\updefault}{\color[rgb]{0,.56,0}\begin{tabular}{c}$N$\,has\,non-trivial\\JSJ\,decomposition\end{tabular}}%
}}}}
\put(6369,-6345){\makebox(0,0)[lb]{\smash{{\SetFigFont{12}{14.4}{\rmdefault}{\mddefault}{\updefault}{\color[rgb]{0,0,0}\begin{tabular}{c}$\Phi(\pi)$\\trivial\end{tabular}}%
}}}}
\put(7624,-5166){\makebox(0,0)[lb]{\smash{{\SetFigFont{12}{14.4}{\rmdefault}{\mddefault}{\updefault}{\color[rgb]{0,0,0}\begin{tabular}{c}$\pi$\,has\,solvable\\word\,problem\end{tabular}}%
}}}}
\put(8305,-3211){\makebox(0,0)[lb]{\smash{{\SetFigFont{12}{14.4}{\rmdefault}{\mddefault}{\updefault}{\color[rgb]{0,0,0}\begin{tabular}{c}conjugacy\\problem\\solvable\end{tabular}}%
}}}}
\put(5915,-2897){\makebox(0,0)[lb]{\smash{{\SetFigFont{12}{14.4}{\rmdefault}{\mddefault}{\updefault}{\color[rgb]{0,0,0}\begin{tabular}{c}$\pi$\,virtually\,$p$\\for\,almost\\all\,primes\,$p$\end{tabular}}%
}}}}
\put(8108,-7949){\makebox(0,0)[lb]{\smash{{\SetFigFont{12}{14.4}{\rmdefault}{\mddefault}{\updefault}{\color[rgb]{0,.56,0}$\pi$\,non-amenable}%
}}}}
\put(9923,-8569){\makebox(0,0)[lb]{\smash{{\SetFigFont{12}{14.4}{\rmdefault}{\mddefault}{\updefault}{\color[rgb]{0,0,0}\begin{tabular}{c}$\pi$\,linear\\over $\Bbb{C}$\end{tabular}}%
}}}}
\put(9000,-2322){\makebox(0,0)[lb]{\smash{{\SetFigFont{12}{14.4}{\rmdefault}{\mddefault}{\updefault}{\color[rgb]{0,0,0}$N$\,efficient}%
}}}}
\put(6399,-4137){\makebox(0,0)[lb]{\smash{{\SetFigFont{12}{14.4}{\rmdefault}{\mddefault}{\updefault}{\color[rgb]{0,0,0}$\pi$\,residually\,finite}%
}}}}
\put(7442,-2201){\makebox(0,0)[lb]{\smash{{\SetFigFont{12}{14.4}{\rmdefault}{\mddefault}{\updefault}{\color[rgb]{0,0,0}$\pi$\,AERF}%
}}}}
\put(3344,-7843){\makebox(0,0)[lb]{\smash{{\SetFigFont{12}{14.4}{\rmdefault}{\mddefault}{\updefault}{\color[rgb]{0,.56,0}\begin{tabular}{c}$N$\,Seifert\\fibered\\manifold\end{tabular}}%
}}}}
\put(9736,-4166){\makebox(0,0)[lb]{\smash{{\SetFigFont{12}{14.4}{\rmdefault}{\mddefault}{\updefault}{\color[rgb]{0,0,0}$b_*^{(2)}(\pi)=0$}%
}}}}
\end{picture}%

%% file: diagram-akmw-v5.pstex_t
\begin{picture}(0,0)%
\includegraphics{diagram-akmw-v5.eps}%
\end{picture}%
\setlength{\unitlength}{3394sp}%
\begingroup\makeatletter\ifx\SetFigFont\undefined%
\gdef\SetFigFont#1#2#3#4#5{%
  \reset@font\fontsize{#1}{#2pt}%
  \fontfamily{#3}\fontseries{#4}\fontshape{#5}%
  \selectfont}%
\fi\endgroup%
\begin{picture}(6885,9472)(8935,-11101)
\put(11239,-9437){\makebox(0,0)[lb]{\smash{{\SetFigFont{12}{14.4}{\rmdefault}{\mddefault}{\updefault}{\color[rgb]{0,0,0}$\pi$ GFERF}%
}}}}
\put(12009,-8813){\makebox(0,0)[lb]{\smash{{\SetFigFont{12}{14.4}{\rmdefault}{\mddefault}{\updefault}{\color[rgb]{0,0,1}Haglund}%
}}}}
\put(11371,-10252){\makebox(0,0)[lb]{\smash{{\SetFigFont{12}{14.4}{\rmdefault}{\mddefault}{\updefault}{\color[rgb]{0,0,0}$\pi$ LERF}%
}}}}
\put(12085,-9812){\makebox(0,0)[lb]{\smash{{\SetFigFont{12}{14.4}{\rmdefault}{\mddefault}{\updefault}{\color[rgb]{0,0,1}Agol, Calegari--Gabai}%
}}}}
\put(8950,-8337){\makebox(0,0)[lb]{\smash{{\SetFigFont{12}{14.4}{\rmdefault}{\mddefault}{\updefault}{\color[rgb]{0,0,0}\begin{tabular}{c}$N$ virtually\\fibered\end{tabular}}%
}}}}
\put(11435,-8343){\makebox(0,0)[lb]{\smash{{\SetFigFont{12}{14.4}{\rmdefault}{\mddefault}{\updefault}{\color[rgb]{0,0,0}$\pi$ virtually a quasi-convex subgroup of a RAAG}%
}}}}
\put(13597,-7892){\makebox(0,0)[lb]{\smash{{\SetFigFont{12}{14.4}{\rmdefault}{\mddefault}{\updefault}{\color[rgb]{0,0,1}Haglund--Wise}%
}}}}
\put(14843,-9355){\makebox(0,0)[lb]{\smash{{\SetFigFont{12}{14.4}{\rmdefault}{\mddefault}{\updefault}{\color[rgb]{0,0,0}$\pi$ large}%
}}}}
\put(9554,-5338){\makebox(0,0)[lb]{\smash{{\SetFigFont{12}{14.4}{\rmdefault}{\mddefault}{\updefault}{\color[rgb]{0,0,1}Thurston}%
}}}}
\put(12739,-4760){\makebox(0,0)[lb]{\smash{{\SetFigFont{12}{14.4}{\rmdefault}{\mddefault}{\updefault}{\color[rgb]{0,0,0}\begin{tabular}{c}$N$\,contains\,a\,dense\,set\,of\\quasi-Fuchsian\,surface\,groups\end{tabular}}%
}}}}
\put(11335,-3498){\makebox(0,0)[lb]{\smash{{\SetFigFont{12}{14.4}{\rmdefault}{\mddefault}{\updefault}{\color[rgb]{0,0,0}\begin{tabular}{r}$N$\,hyperbolic\\with\,boundary\end{tabular}}%
}}}}
\put(13635,-3479){\makebox(0,0)[lb]{\smash{{\SetFigFont{12}{14.4}{\rmdefault}{\mddefault}{\updefault}{\color[rgb]{0,0,0}\begin{tabular}{l}$N$\,closed\\hyperbolic\end{tabular}}%
}}}}
\put(14331,-4047){\makebox(0,0)[lb]{\smash{{\SetFigFont{12}{14.4}{\rmdefault}{\mddefault}{\updefault}{\color[rgb]{0,0,1}Kahn--Markovic}%
}}}}
\put(9737,-4614){\makebox(0,0)[lb]{\smash{{\SetFigFont{12}{14.4}{\rmdefault}{\mddefault}{\updefault}{\color[rgb]{0,0,0}\begin{tabular}{c}$N$\,contains\,a\\geometrically\,finite\\surface\end{tabular}}%
}}}}
\put(12336,-6737){\makebox(0,0)[lb]{\smash{{\SetFigFont{12}{14.4}{\rmdefault}{\mddefault}{\updefault}{\color[rgb]{0,0,1}Wise}%
}}}}
\put(13764,-6865){\makebox(0,0)[lb]{\smash{{\SetFigFont{12}{14.4}{\rmdefault}{\mddefault}{\updefault}{\color[rgb]{0,0,1}Agol}%
}}}}
\put(10671,-7432){\makebox(0,0)[lb]{\smash{{\SetFigFont{12}{14.4}{\rmdefault}{\mddefault}{\updefault}{\color[rgb]{0,0,0}\begin{tabular}{r}$\pi$\,\,is\,\,virtually\,compact\,special\end{tabular}}%
}}}}
\put(9188,-2126){\makebox(0,0)[lb]{\smash{{\SetFigFont{12}{14.4}{\rmdefault}{\mddefault}{\updefault}{\color[rgb]{0,0,0}$\pi$ = fundamental group of a hyperbolic 3-manifold $N$}%
}}}}
\put(9444,-2345){\rotatebox{270.0}{\makebox(0,0)[lb]{\smash{{\SetFigFont{10}{12.0}{\rmdefault}{\mddefault}{\updefault}{\color[rgb]{0,0,0}\begin{tabular}{c}if\,$N$\,closed\\and\,Haken\end{tabular}}%
}}}}}
\put(9939,-2748){\makebox(0,0)[lb]{\smash{{\SetFigFont{12}{14.4}{\rmdefault}{\mddefault}{\updefault}{\color[rgb]{0,0,1}\begin{tabular}{l}Bonahon\\Thurston\end{tabular}}%
}}}}
\put(11366,-1888){\makebox(0,0)[lb]{\smash{{\SetFigFont{14}{16.8}{\rmdefault}{\mddefault}{\updefault}{\color[rgb]{1,1,1}.}%
}}}}
\put(10433,-10741){\makebox(0,0)[lb]{\smash{{\SetFigFont{12}{14.4}{\rmdefault}{\mddefault}{\updefault}{\color[rgb]{0,0,0}Diagram 2. The Virtually Compact Special Theorem.}%
}}}}
\put(12539,-11086){\makebox(0,0)[lb]{\smash{{\SetFigFont{14}{16.8}{\rmdefault}{\mddefault}{\updefault}{\color[rgb]{1,1,1}.}%
}}}}
\put(14333,-10245){\makebox(0,0)[lb]{\smash{{\SetFigFont{12}{14.4}{\rmdefault}{\mddefault}{\updefault}{\color[rgb]{0,0,0}$N$ virtually Haken}%
}}}}
\put(9989,-6911){\makebox(0,0)[lb]{\smash{{\SetFigFont{12}{14.4}{\rmdefault}{\mddefault}{\updefault}{\color[rgb]{0,0,1}Wise}%
}}}}
\put(9726,-6211){\makebox(0,0)[lb]{\smash{{\SetFigFont{12}{14.4}{\rmdefault}{\mddefault}{\updefault}{\color[rgb]{0,0,0}\begin{tabular}{c}$\pi$\,word\,hyperbolic\\with\,quasi-convex\\hierarchy\end{tabular}}%
}}}}
\put(12876,-6324){\makebox(0,0)[lb]{\smash{{\SetFigFont{12}{14.4}{\rmdefault}{\mddefault}{\updefault}{\color[rgb]{0,0,0}\begin{tabular}{c}$\pi$\,is\,fundamental\,group of a\\non-positively\,curved\,cube\,complex\end{tabular}}%
}}}}
\put(13989,-5461){\makebox(0,0)[lb]{\smash{{\SetFigFont{12}{14.4}{\rmdefault}{\mddefault}{\updefault}{\color[rgb]{0,0,1}\begin{tabular}{l}Sageev\\Bergeron--Wise\end{tabular}}%
}}}}
\put(10701,-8074){\makebox(0,0)[lb]{\smash{{\SetFigFont{12}{14.4}{\rmdefault}{\mddefault}{\updefault}{\color[rgb]{0,0,1}Agol}%
}}}}
\end{picture}%

%% file: special_v1.pstex_t
\begin{picture}(0,0)%
\includegraphics{special_v1.eps}%
\end{picture}%
\setlength{\unitlength}{1697sp}%
\begingroup\makeatletter\ifx\SetFigFont\undefined%
\gdef\SetFigFont#1#2#3#4#5{%
  \reset@font\fontsize{#1}{#2pt}%
  \fontfamily{#3}\fontseries{#4}\fontshape{#5}%
  \selectfont}%
\fi\endgroup%
\begin{picture}(13918,6780)(6631,-12797)
\put(7570,-6710){\makebox(0,0)[lb]{\smash{{\SetFigFont{12}{14.4}{\rmdefault}{\mddefault}{\updefault}{\color[rgb]{0,0,1}$\iota_i(Y_i)$}%
}}}}
\put(15160,-6388){\makebox(0,0)[lb]{\smash{{\SetFigFont{12}{14.4}{\rmdefault}{\mddefault}{\updefault}{\color[rgb]{0,0,0}$\iota_j(Y_j)$}%
}}}}
\put(19467,-6693){\makebox(0,0)[lb]{\smash{{\SetFigFont{12}{14.4}{\rmdefault}{\mddefault}{\updefault}{\color[rgb]{0,0,1}$\iota_i(Y_i)$}%
}}}}
\put(7591,-12649){\makebox(0,0)[lb]{\smash{{\SetFigFont{12}{14.4}{\rmdefault}{\mddefault}{\updefault}{\color[rgb]{0,0,0}Figure 1. Directly self-osculating and inter-osculating hyperplanes.}%
}}}}
\end{picture}%

%% file: virtual-properties-v2.pstex_t
\begin{picture}(0,0)%
\includegraphics{virtual-properties-v2.eps}%
\end{picture}%
\setlength{\unitlength}{3394sp}%
\begingroup\makeatletter\ifx\SetFigFont\undefined%
\gdef\SetFigFont#1#2#3#4#5{%
  \reset@font\fontsize{#1}{#2pt}%
  \fontfamily{#3}\fontseries{#4}\fontshape{#5}%
  \selectfont}%
\fi\endgroup%
\begin{picture}(5683,6138)(7201,-11578)
\put(8527,-9192){\makebox(0,0)[lb]{\smash{{\SetFigFont{12}{14.4}{\rmdefault}{\mddefault}{\updefault}{\color[rgb]{0,0,0}$N$ virtually Haken}%
}}}}
\put(10010,-9724){\makebox(0,0)[lb]{\smash{{\SetFigFont{10}{12.0}{\rmdefault}{\mddefault}{\updefault}{\color[rgb]{0,0,0}if $\pi_1(N)$ is LERF}%
}}}}
\put(8576,-8211){\makebox(0,0)[lb]{\smash{{\SetFigFont{12}{14.4}{\rmdefault}{\mddefault}{\updefault}{\color[rgb]{0,0,0}$vb_1(N;\Z) \geq 1$}%
}}}}
\put(7689,-7136){\makebox(0,0)[lb]{\smash{{\SetFigFont{12}{14.4}{\rmdefault}{\mddefault}{\updefault}{\color[rgb]{0,0,0}$vb_1(N;\Z) =\infty$}%
}}}}
\put(7951,-6149){\makebox(0,0)[lb]{\smash{{\SetFigFont{12}{14.4}{\rmdefault}{\mddefault}{\updefault}{\color[rgb]{0,0,0}$\pi_1(N)$ large}%
}}}}
\put(9689,-6174){\makebox(0,0)[lb]{\smash{{\SetFigFont{12}{14.4}{\rmdefault}{\mddefault}{\updefault}{\color[rgb]{0,0,0}$N$ virtually fibered}%
}}}}
\put(7826,-10386){\makebox(0,0)[lb]{\smash{{\SetFigFont{12}{14.4}{\rmdefault}{\mddefault}{\updefault}{\color[rgb]{0,0,0}$\pi_1(N)$ contains a surface group}%
}}}}
\put(9214,-5699){\makebox(0,0)[lb]{\smash{{\SetFigFont{14}{16.8}{\rmdefault}{\mddefault}{\updefault}{\color[rgb]{1,1,1}.}%
}}}}
\put(9218,-5733){\makebox(0,0)[lb]{\smash{{\SetFigFont{14}{16.8}{\rmdefault}{\mddefault}{\updefault}{\color[rgb]{1,1,1}.}%
}}}}
\put(8728,-11004){\makebox(0,0)[lb]{\smash{{\SetFigFont{12}{14.4}{\rmdefault}{\mddefault}{\updefault}{\color[rgb]{0,0,0}Diagram 3. Virtual properties of 3-manifolds.}%
}}}}
\put(10150,-11563){\makebox(0,0)[lb]{\smash{{\SetFigFont{14}{16.8}{\rmdefault}{\mddefault}{\updefault}{\color[rgb]{1,1,1}.}%
}}}}
\put(7216,-8523){\makebox(0,0)[lb]{\smash{{\SetFigFont{14}{16.8}{\rmdefault}{\mddefault}{\updefault}{\color[rgb]{1,1,1}.}%
}}}}
\put(11174,-8321){\makebox(0,0)[lb]{\smash{{\SetFigFont{12}{14.4}{\rmdefault}{\mddefault}{\updefault}{\color[rgb]{0,0,0}\begin{tabular}{c}$\pi_1(N)$ does not have\\Property ($\tau$)\end{tabular}}%
}}}}
\end{picture}%

%% file: diagram-winfig-v32-totally-raag-v7.pstex_t
\begin{picture}(0,0)%
\includegraphics{diagram-winfig-v32-totally-raag-v7.eps}%
\end{picture}%
\setlength{\unitlength}{3355sp}%
\begingroup\makeatletter\ifx\SetFigFont\undefined%
\gdef\SetFigFont#1#2#3#4#5{%
  \reset@font\fontsize{#1}{#2pt}%
  \fontfamily{#3}\fontseries{#4}\fontshape{#5}%
  \selectfont}%
\fi\endgroup%
\begin{picture}(8449,12502)(9689,-16072)
\put(15351,-13744){\makebox(0,0)[lb]{\smash{{\SetFigFont{12}{14.4}{\rmdefault}{\mddefault}{\updefault}{\color[rgb]{0,0,0}$\pi$\,residually\,$p$\,for\,any\,prime\,$p$}%
}}}}
\put(13668,-15495){\makebox(0,0)[lb]{\smash{{\SetFigFont{12}{14.4}{\rmdefault}{\mddefault}{\updefault}{\color[rgb]{0,.56,0}\begin{tabular}{c}$\pi$\,weakly\,characteristically\,potent\end{tabular}}%
}}}}
\put(16808,-14662){\makebox(0,0)[lb]{\smash{{\SetFigFont{12}{14.4}{\rmdefault}{\mddefault}{\updefault}{\color[rgb]{0,0,0}$\pi$\,bi-orderable}%
}}}}
\put(13752,-14677){\makebox(0,0)[lb]{\smash{{\SetFigFont{12}{14.4}{\rmdefault}{\mddefault}{\updefault}{\color[rgb]{0,0,0}\begin{tabular}{c}$\pi$\,characteristically\,potent\end{tabular}}%
}}}}
\put(16322,-5483){\makebox(0,0)[lb]{\smash{{\SetFigFont{12}{14.4}{\rmdefault}{\mddefault}{\updefault}{\color[rgb]{0,.56,0}\begin{tabular}{c}$N$\,graph\,manifold\end{tabular}}%
}}}}
\put(9722,-6712){\makebox(0,0)[lb]{\smash{{\SetFigFont{12}{14.4}{\rmdefault}{\mddefault}{\updefault}{\color[rgb]{0,.56,0}\begin{tabular}{c}$\pi$\,\,virtually\\compact\,special\end{tabular}}%
}}}}
\put(9792,-3965){\makebox(0,0)[lb]{\smash{{\SetFigFont{12}{14.4}{\rmdefault}{\mddefault}{\updefault}{\color[rgb]{0,0,0}\begin{tabular}{c}$N$= irreducible,\,orientable,\,compact\,3-manifold\,with\,empty\,or\,toroidal\,boundary\\such that $\pi=\pi_1(N)$ is neither finite nor solvable\end{tabular}}%
}}}}
\put(11846,-5423){\makebox(0,0)[lb]{\smash{{\SetFigFont{12}{14.4}{\rmdefault}{\mddefault}{\updefault}{\color[rgb]{0,.56,0}\begin{tabular}{c}$N$\,has\,non-trivial\,JSJ\,decomposition\\with\,at\,least\,one\,hyperbolic\,component\end{tabular}}%
}}}}
\put(9891,-5494){\makebox(0,0)[lb]{\smash{{\SetFigFont{12}{14.4}{\rmdefault}{\mddefault}{\updefault}{\color[rgb]{0,.56,0}$N$ hyperbolic}%
}}}}
\put(9724,-7817){\makebox(0,0)[lb]{\smash{{\SetFigFont{12}{14.4}{\rmdefault}{\mddefault}{\updefault}{\color[rgb]{0,0,0}\begin{tabular}{c}$\pi$\,quasiconvex\\subgroup\,of\,a\,RAAG\end{tabular}}%
}}}}
\put(13578,-9096){\makebox(0,0)[lb]{\smash{{\SetFigFont{12}{14.4}{\rmdefault}{\mddefault}{\updefault}{\color[rgb]{0,0,0}$\pi$ subgroup of a RAAG}%
}}}}
\put(10409,-14320){\makebox(0,0)[lb]{\smash{{\SetFigFont{12}{14.4}{\rmdefault}{\mddefault}{\updefault}{\color[rgb]{0,.56,0}$\pi$\,LERF}%
}}}}
\put(9724,-15467){\makebox(0,0)[lb]{\smash{{\SetFigFont{12}{14.4}{\rmdefault}{\mddefault}{\updefault}{\color[rgb]{0,.56,0}\begin{tabular}{c}$\pi$\,double\,coset\,separable\end{tabular}}%
}}}}
\put(9704,-13313){\makebox(0,0)[lb]{\smash{{\SetFigFont{12}{14.4}{\rmdefault}{\mddefault}{\updefault}{\color[rgb]{0,.56,0}\begin{tabular}{c}$N$\,hyperbolic\,and\,$\pi$\,GFERF\end{tabular}}%
}}}}
\put(11073,-14803){\makebox(0,0)[lb]{\smash{{\SetFigFont{10}{12.0}{\rmdefault}{\mddefault}{\updefault}{\color[rgb]{0,0,0}\begin{tabular}{l}if\,$N$\,closed\\\,hyperbolic\end{tabular}}%
}}}}
\put(9704,-12105){\makebox(0,0)[lb]{\smash{{\SetFigFont{12}{14.4}{\rmdefault}{\mddefault}{\updefault}{\color[rgb]{0,.56,0}\begin{tabular}{c}$N$\,hyperbolic\,and\\$\pi$\,virtually\,retracts\,onto\\geometrically\,finite\,subgroups\end{tabular}}%
}}}}
\put(9724,-10636){\makebox(0,0)[lb]{\smash{{\SetFigFont{12}{14.4}{\rmdefault}{\mddefault}{\updefault}{\color[rgb]{0,.56,0}\begin{tabular}{c}$\pi$\,hereditarily\\conjugacy\\separable\end{tabular}}%
}}}}
\put(9845,-9207){\makebox(0,0)[lb]{\smash{{\SetFigFont{12}{14.4}{\rmdefault}{\mddefault}{\updefault}{\color[rgb]{0,0,0}\begin{tabular}{c}$\pi$\,virtual\\retract\\of\,a\,RAAG\end{tabular}}%
}}}}
\put(11899,-9529){\rotatebox{270.0}{\makebox(0,0)[lb]{\smash{{\SetFigFont{10}{12.0}{\rmdefault}{\mddefault}{\updefault}{\color[rgb]{0,0,0}if\,$N$\,hyperbolic}%
}}}}}
\put(12978,-13136){\makebox(0,0)[lb]{\smash{{\SetFigFont{12}{14.4}{\rmdefault}{\mddefault}{\updefault}{\color[rgb]{0,.56,0}$\mbox{vb}_1(N;\Bbb{Z})=\infty$}%
}}}}
\put(13008,-13908){\makebox(0,0)[lb]{\smash{{\SetFigFont{12}{14.4}{\rmdefault}{\mddefault}{\updefault}{\color[rgb]{0,0,0}$N$ Haken}%
}}}}
\put(13235,-6618){\makebox(0,0)[lb]{\smash{{\SetFigFont{12}{14.4}{\rmdefault}{\mddefault}{\updefault}{\color[rgb]{0,.56,0}$\pi$\,virtually\,special}%
}}}}
\put(16548,-7707){\makebox(0,0)[lb]{\smash{{\SetFigFont{12}{14.4}{\rmdefault}{\mddefault}{\updefault}{\color[rgb]{0,.56,0}\begin{tabular}{c}$\pi$\,linear\,over\,$\Bbb{Z}$\end{tabular}}%
}}}}
\put(12434,-7994){\makebox(0,0)[lb]{\smash{{\SetFigFont{12}{14.4}{\rmdefault}{\mddefault}{\updefault}{\color[rgb]{0,0,0}$\pi$\,omnipotent}%
}}}}
\put(16880,-8811){\makebox(0,0)[lb]{\smash{{\SetFigFont{12}{14.4}{\rmdefault}{\mddefault}{\updefault}{\color[rgb]{0,0,0}\begin{tabular}{c}$\pi$\,residually\\torsion-free\\nilpotent\end{tabular}}%
}}}}
\put(13356,-7238){\makebox(0,0)[lb]{\smash{{\SetFigFont{10}{12.0}{\rmdefault}{\mddefault}{\updefault}{\color[rgb]{0,0,0}\begin{tabular}{r}if\,$N$\,closed\\hyperbolic\end{tabular}}%
}}}}
\put(11178,-15995){\makebox(0,0)[lb]{\smash{{\SetFigFont{12}{14.4}{\rmdefault}{\mddefault}{\updefault}{\color[rgb]{0,0,0}Diagram 4. Consequences of being virtually (compact) special.}%
}}}}
\put(16226,-10381){\makebox(0,0)[lb]{\smash{{\SetFigFont{12}{14.4}{\rmdefault}{\mddefault}{\updefault}{\color[rgb]{0,0,0}$\pi$\,poly-free}%
}}}}
\put(12816,-10477){\makebox(0,0)[lb]{\smash{{\SetFigFont{12}{14.4}{\rmdefault}{\mddefault}{\updefault}{\color[rgb]{0,.56,0}$\pi$\,large}%
}}}}
\put(13517,-11935){\makebox(0,0)[lb]{\smash{{\SetFigFont{12}{14.4}{\rmdefault}{\mddefault}{\updefault}{\color[rgb]{0,.56,0}\begin{tabular}{c}$\pi$\,virtually\\[-1mm]retracts\\[-1mm]onto\,cyclic\\[-1mm]subgroups\end{tabular}}%
}}}}
\put(14425,-10367){\makebox(0,0)[lb]{\smash{{\SetFigFont{12}{14.4}{\rmdefault}{\mddefault}{\updefault}{\color[rgb]{0,0,0}$\pi$\,RFRS}%
}}}}
\put(15923,-11233){\makebox(0,0)[lb]{\smash{{\SetFigFont{12}{14.4}{\rmdefault}{\mddefault}{\updefault}{\color[rgb]{0,0,0}$N$ fibered}%
}}}}
\put(15140,-12471){\makebox(0,0)[lb]{\smash{{\SetFigFont{12}{14.4}{\rmdefault}{\mddefault}{\updefault}{\color[rgb]{0,.56,0}$\pi$ good}%
}}}}
\put(15635,-13062){\makebox(0,0)[lb]{\smash{{\SetFigFont{12}{14.4}{\rmdefault}{\mddefault}{\updefault}{\color[rgb]{0,.56,0}$\pi$\,has\,Property\,FD}%
}}}}
\put(16583,-12017){\makebox(0,0)[lb]{\smash{{\SetFigFont{12}{14.4}{\rmdefault}{\mddefault}{\updefault}{\color[rgb]{0,.56,0}\begin{tabular}{c}$\pi$\,does\,not\\have\,f.g.i.p.\end{tabular}}%
}}}}
\put(16139,-6066){\makebox(0,0)[lb]{\smash{{\SetFigFont{10}{12.0}{\rmdefault}{\mddefault}{\updefault}{\color[rgb]{0,0,0}\begin{tabular}{r}if\,$N$\,non-positively\\curved\end{tabular}}%
}}}}
\end{picture}%

%% file: subgroups_v10.pstex_t
\begin{picture}(0,0)%
\includegraphics{subgroups_v10.eps}%
\end{picture}%
\setlength{\unitlength}{3631sp}%
\begingroup\makeatletter\ifx\SetFigFont\undefined%
\gdef\SetFigFont#1#2#3#4#5{%
  \reset@font\fontsize{#1}{#2pt}%
  \fontfamily{#3}\fontseries{#4}\fontshape{#5}%
  \selectfont}%
\fi\endgroup%
\begin{picture}(8978,11198)(2315,-11418)
\put(8713,-6247){\makebox(0,0)[lb]{\smash{{\SetFigFont{11}{13.2}{\rmdefault}{\mddefault}{\updefault}{\color[rgb]{0,0,0}\begin{tabular}{c}$N$\,hyperbolic\\$\Gamma$\,(relatively)\\quasi-convex\end{tabular}}%
}}}}
\put(8773,-7714){\makebox(0,0)[lb]{\smash{{\SetFigFont{11}{13.2}{\rmdefault}{\mddefault}{\updefault}{\color[rgb]{0,0,0}\begin{tabular}{c}$\Gamma$\,has\,finite\\width\end{tabular}}%
}}}}
\put(3149,-2638){\makebox(0,0)[lb]{\smash{{\SetFigFont{11}{13.2}{\rmdefault}{\mddefault}{\updefault}{\color[rgb]{0,0,0}$N$\,spherical}%
}}}}
\put(4448,-1046){\rotatebox{270.0}{\makebox(0,0)[lb]{\smash{{\SetFigFont{11}{13.2}{\rmdefault}{\mddefault}{\updefault}{\color[rgb]{0,0,0}if\,$\G$\,normal}%
}}}}}
\put(2655,-1083){\rotatebox{270.0}{\makebox(0,0)[lb]{\smash{{\SetFigFont{11}{13.2}{\rmdefault}{\mddefault}{\updefault}{\color[rgb]{0,0,0}if\,$\Gamma$\,abelian}%
}}}}}
\put(6205,-479){\makebox(0,0)[lb]{\smash{{\SetFigFont{14}{16.8}{\rmdefault}{\mddefault}{\updefault}{\color[rgb]{1,1,1}.}%
}}}}
\put(3478,-4761){\makebox(0,0)[lb]{\smash{{\SetFigFont{11}{13.2}{\rmdefault}{\mddefault}{\updefault}{\color[rgb]{0,0,0}\begin{tabular}{c}$N$\,Seifert\,fibered\\and\,$\Gamma$\,virtually\\a\,Seifert\,fiber\end{tabular}}%
}}}}
\put(4833,-5530){\makebox(0,0)[lb]{\smash{{\SetFigFont{11}{13.2}{\rmdefault}{\mddefault}{\updefault}{\color[rgb]{0,0,0}\begin{tabular}{l}$\Gamma$\,virtual\,surface\,fiber\,subgroup\end{tabular}}%
}}}}
\put(3387,-1028){\rotatebox{270.0}{\makebox(0,0)[lb]{\smash{{\SetFigFont{11}{13.2}{\rmdefault}{\mddefault}{\updefault}{\color[rgb]{0,0,0}if\,$\Gamma$\,finite}%
}}}}}
\put(4886,-2216){\makebox(0,0)[lb]{\smash{{\SetFigFont{11}{13.2}{\rmdefault}{\mddefault}{\updefault}{\color[rgb]{0,0,0}\begin{tabular}{c}$\Gamma$\,is\,virtually\\solvable\end{tabular}}%
}}}}
\put(6142,-2231){\makebox(0,0)[lb]{\smash{{\SetFigFont{11}{13.2}{\rmdefault}{\mddefault}{\updefault}{\color[rgb]{0,0,0}\begin{tabular}{c}$\Gamma$\,contains\,free\\non-cyclic\,group\end{tabular}}%
}}}}
\put(7775,-4470){\makebox(0,0)[lb]{\smash{{\SetFigFont{11}{13.2}{\rmdefault}{\mddefault}{\updefault}{\color[rgb]{0,0,0}\begin{tabular}{c}$N$\,hyperbolic\,and\\$\Gamma$\,geometrically\,finite\end{tabular}}%
}}}}
\put(8026,-1074){\rotatebox{270.0}{\makebox(0,0)[lb]{\smash{{\SetFigFont{11}{13.2}{\rmdefault}{\mddefault}{\updefault}{\color[rgb]{0,0,0}if\,$N$\,hyperbolic}%
}}}}}
\put(8951,-3774){\makebox(0,0)[lb]{\smash{{\SetFigFont{11}{13.2}{\rmdefault}{\mddefault}{\updefault}{\color[rgb]{0,0,0}$b_1(\Gamma)\geq 1$}%
}}}}
\put(5151,-3824){\makebox(0,0)[lb]{\smash{{\SetFigFont{11}{13.2}{\rmdefault}{\mddefault}{\updefault}{\color[rgb]{0,0,0}\begin{tabular}{c}$\G$\,normal\,in\,$\pi$\,and\\finite\,index\,subgroup\\of a fiber or a\\semifiber\,subgroup\end{tabular}}%
}}}}
\put(2406,-689){\makebox(0,0)[lb]{\smash{{\SetFigFont{11}{13.2}{\rmdefault}{\mddefault}{\updefault}{\color[rgb]{0,0,0}\begin{tabular}{c}$\Gamma$  finitely generated non-trivial subgroup of $\pi=\pi_1(N)$ of infinite index, \\$N$ is an irreducible compact orientable 3-manifold with empty or toroidal boundary\end{tabular}}%
}}}}
\put(8501,-2674){\makebox(0,0)[lb]{\smash{{\SetFigFont{11}{13.2}{\rmdefault}{\mddefault}{\updefault}{\color[rgb]{0,0,0}\begin{tabular}{c}$\Gamma$\,finitely\\presented\end{tabular}}%
}}}}
\put(4189,-9561){\makebox(0,0)[lb]{\smash{{\SetFigFont{11}{13.2}{\rmdefault}{\mddefault}{\updefault}{\color[rgb]{0,0,0}$\Gamma$\,separable}%
}}}}
\put(5949,-11294){\makebox(0,0)[lb]{\smash{{\SetFigFont{14}{16.8}{\rmdefault}{\mddefault}{\updefault}{\color[rgb]{1,1,1}.}%
}}}}
\put(2330,-9010){\makebox(0,0)[lb]{\smash{{\SetFigFont{11}{13.2}{\rmdefault}{\mddefault}{\updefault}{\color[rgb]{0,0,0}$\Gamma$ cyclic}%
}}}}
\put(7775,-9678){\makebox(0,0)[lb]{\smash{{\SetFigFont{11}{13.2}{\rmdefault}{\mddefault}{\updefault}{\color[rgb]{0,0,0}\begin{tabular}{c}$\Gamma$\,virtual\\retract\,of\,$\pi$\end{tabular}}%
}}}}
\put(6989,-7994){\makebox(0,0)[lb]{\smash{{\SetFigFont{11}{13.2}{\rmdefault}{\mddefault}{\updefault}{\color[rgb]{0,0,0}\begin{tabular}{c}$\Gamma$\,finite\,index\,in\\$\operatorname{Comm}_\pi(\G)$\end{tabular}}%
}}}}
\put(5280,-8780){\makebox(0,0)[lb]{\smash{{\SetFigFont{11}{13.2}{\rmdefault}{\mddefault}{\updefault}{\color[rgb]{0,0,0}\begin{tabular}{r}$\pi$\,induces\,full\,profinite\\topology\,on\,$\Gamma$\end{tabular}}%
}}}}
\put(2996,-7782){\makebox(0,0)[lb]{\smash{{\SetFigFont{11}{13.2}{\rmdefault}{\mddefault}{\updefault}{\color[rgb]{0,0,0}\begin{tabular}{c}$\Gamma$ carried by a\\characteristic\\submanifold\end{tabular}}%
}}}}
\put(5733,-6633){\makebox(0,0)[lb]{\smash{{\SetFigFont{11}{13.2}{\rmdefault}{\mddefault}{\updefault}{\color[rgb]{0,0,0}\begin{tabular}{l}$\Gamma$\,virtually\\normal\end{tabular}}%
}}}}
\put(7790,-4803){\rotatebox{270.0}{\makebox(0,0)[lb]{\smash{{\SetFigFont{11}{13.2}{\rmdefault}{\mddefault}{\updefault}{\color[rgb]{0,0,0}if\,$N$\,closed}%
}}}}}
\put(7125,-6648){\makebox(0,0)[lb]{\smash{{\SetFigFont{11}{13.2}{\rmdefault}{\mddefault}{\updefault}{\color[rgb]{0,0,0}\begin{tabular}{r}$\Gamma$ virtually\\malnormal\end{tabular}}%
}}}}
\put(2905,-5695){\makebox(0,0)[lb]{\smash{{\SetFigFont{11}{13.2}{\rmdefault}{\mddefault}{\updefault}{\color[rgb]{0,0,0}$\Gamma \cong \Z^2$}%
}}}}
\put(3707,-10641){\makebox(0,0)[lb]{\smash{{\SetFigFont{11}{13.2}{\rmdefault}{\mddefault}{\updefault}{\color[rgb]{0,0,0}\begin{tabular}{c}membership\,problem\\solvable\,for\,$\Gamma$\end{tabular}}%
}}}}
\put(4387,-11352){\makebox(0,0)[lb]{\smash{{\SetFigFont{11}{13.2}{\rmdefault}{\mddefault}{\updefault}{\color[rgb]{0,0,0}Diagram 5. Subgroups of 3-manifold groups.}%
}}}}
\end{picture}%